%% file: main.tex
\documentclass[letterpaper,10pt]{article}
\usepackage{fullpage}

\usepackage{graphicx}%
\usepackage{multirow}%
\usepackage{amsmath,amssymb,amsfonts}%
\usepackage{amsthm}%
\usepackage{mathrsfs}%
\usepackage[title]{appendix}%
\usepackage{xcolor}%
\usepackage{textcomp}%
\usepackage{manyfoot}%
\usepackage{booktabs}%
\usepackage{algpseudocode}%
\usepackage{listings}%
\usepackage{caption}
\usepackage{hyperref}
\usepackage[capitalize,nameinlink,noabbrev]{cleveref}
\usepackage{mathtools}
\usepackage{cite}

\usepackage{mathabx}

\usepackage{tikz}
\usetikzlibrary{arrows.meta,positioning,calc,decorations.pathreplacing}

\tikzset{
  box/.style   ={draw, thick, minimum width=10mm, minimum height=7mm, align=center},
  bigbox/.style={draw, thick, minimum width=18mm, minimum height=7mm, align=center},
  arr/.style   ={-{Stealth[length=2.2mm]}, thick},
}

\input{preamble_packages.tex}

\input{preamble_symbols.tex}

\input{shortcuts.tex}
\input{shortcuts_matrix.tex}

\hypersetup{colorlinks, hypertexnames=false, pageanchor=true,
            linkcolor=blue, citecolor={green!50!black}, urlcolor=cyan,
            pdftitle={Field Analysis in ADMM for SDP}
            }
\makeatletter
\AddToHook{cmd/appendix/before}{\def\cref@section@alias{appendix}}
\makeatother
\mathtoolsset{centercolon}
\crefname{assumption}{Assumption}{Assumptions}

\theoremstyle{thmstyleone}%
\newtheorem{theorem}{Theorem}%
\newtheorem{proposition}{Proposition}%

\theoremstyle{thmstyletwo}%
\newtheorem{remark}{Remark}%

\theoremstyle{thmstylethree}%
\newtheorem{definition}{Definition}%

\title{
    Local Second-Order Limit Dynamics of the Alternating Direction Method of Multipliers for Semidefinite Programming
}
\author{Shucheng Kang%
\thanks{School of Engineering and Applied Sciences, Harvard University. Email: \texttt{skang1@g.harvard.edu}}
\and Heng Yang%
\thanks{School of Engineering and Applied Sciences, Harvard University. Email: \texttt{hankyang@seas.harvard.edu}}
}

\pdfoutput=1
\begin{document}

\maketitle

\begin{quote}
  \vspace*{-7mm}
  \centering
  \small \textit{Due to arXiv rendering limitations, the block-partitioned structure does not display correctly in the online version. For a fully rendered (“clean”) version, please see the PDF on GitHub: \href{https://github.com/ComputationalRobotics/admmsdp-limitdyn/blob/main/arxiv-paper/arxiv_paper.pdf}{paper}.}
\end{quote}

\input{sections/abstract.tex}

\newpage

\setcounter{tocdepth}{2}

{\small
\tableofcontents
}

\newpage 

\input{sections/introduction.tex}

\input{sections/related_works.tex}

\input{sections/psdproj.tex}

\input{sections/second_order_analysis.tex}

\input{sections/mdvZ_kernel.tex}
\input{sections/mdvZ_range.tex}
\input{sections/mdvZ_continuity.tex}
\input{sections/mdvZ_sigma.tex}
\input{sections/toy_1.tex}
\input{sections/toy_2.tex}

\input{sections/toy_3.tex}
\input{sections/experiments.tex}
\input{sections/conclusion.tex}

\subsection*{Acknowledgments}
We sincerely thank Henry Wolkowicz for valuable discussions.

\clearpage

\appendix
\begin{appendices}

\input{appendices/psdproj_proof.tex}

\end{appendices}

\bibliographystyle{plain}
\bibliography{../../references/refs,../../references/myRefs}

\end{document}

%% file: preamble_packages.tex
\usepackage{comment}
\usepackage{siunitx}
\usepackage{relsize}
\usepackage{ifthen}
\usepackage[colorinlistoftodos]{todonotes}

\usepackage[caption=false]{subfig}

\usepackage[vlined,ruled,linesnumbered]{algorithm2e}
\usepackage{graphics} 
\usepackage{rotating}
\usepackage{color}
\usepackage{enumerate}
\usepackage[T1]{fontenc}
\usepackage{psfrag}
\usepackage{epsfig} 
\usepackage{booktabs}
\usepackage{graphicx,url}
\usepackage{multirow}
\usepackage{array}
\usepackage{latexsym}
\usepackage{amsfonts}
\usepackage{amsmath}
\usepackage{amssymb}
\usepackage{mathtools}
\usepackage{xstring}
\usepackage{multirow}
\usepackage{xcolor}
\usepackage{prettyref}
\usepackage{flexisym}
\usepackage{bigdelim}
\usepackage{breqn} 
\usepackage{listings}

\usepackage{enumitem}
\usepackage{xspace}
\usepackage{bm}
\graphicspath{{./figures/}}
\usepackage{tikz}
\usetikzlibrary{matrix,calc}

\usepackage{amsmath,array,arydshln,xparse}
\usepackage{adjustbox} 
\usepackage{nicematrix}

\usepackage{mdwlist}

\makecompactlist{itemize}{stditemize}


%% file: preamble_symbols.tex
\newrefformat{prob}{Problem\,\ref{#1}}
\newrefformat{def}{Definition\,\ref{#1}}
\newrefformat{sec}{Section\,\ref{#1}}
\newrefformat{sub}{Section\,\ref{#1}}
\newrefformat{prop}{Proposition\,\ref{#1}}
\newrefformat{app}{Appendix\,\ref{#1}}
\newrefformat{alg}{Algorithm\,\ref{#1}}
\newrefformat{cor}{Corollary\,\ref{#1}}
\newrefformat{thm}{Theorem\,\ref{#1}}
\newrefformat{lem}{Lemma\,\ref{#1}}
\newrefformat{fig}{Fig.\,\ref{#1}}
\newrefformat{tab}{Table\,\ref{#1}}

\newtheorem{corollary}{Corollary}
\newtheorem{lemma}{Lemma}
\newtheorem{assumption}{Assumption}

\newcommand{\cf}{\emph{cf.}\xspace}

\newcommand{\bdmath}{\begin{dmath}}
\newcommand{\edmath}{\end{dmath}}
\newcommand{\beq}{\begin{equation}}
\newcommand{\eeq}{\end{equation}}
\newcommand{\bdm}{\begin{displaymath}}
\newcommand{\edm}{\end{displaymath}}
\newcommand{\bea}{\begin{eqnarray}}
\newcommand{\eea}{\end{eqnarray}}
\newcommand{\beal}{\beq \begin{array}{ll}}
\newcommand{\eeal}{\end{array} \eeq}
\newcommand{\beas}{\begin{eqnarray*}}
\newcommand{\eeas}{\end{eqnarray*}}
\newcommand{\ba}{\begin{array}}
\newcommand{\ea}{\end{array}}
\newcommand{\bit}{\begin{itemize}}
\newcommand{\eit}{\end{itemize}}
\newcommand{\ben}{\begin{enumerate}}
\newcommand{\een}{\end{enumerate}}

\newcommand{\calA}{{\cal A}}

\newcommand{\calC}{{\cal C}}
\newcommand{\calD}{{\cal D}}
\newcommand{\calE}{{\cal E}}
\newcommand{\calF}{{\cal F}}

\newcommand{\calH}{{\cal H}}
\newcommand{\calI}{{\cal I}}

\newcommand{\calK}{{\cal K}}
\newcommand{\calL}{{\cal L}}
\newcommand{\calM}{{\cal M}}
\newcommand{\calN}{{\cal N}}
\newcommand{\calO}{{\cal O}}
\newcommand{\calP}{{\cal P}}

\newcommand{\calR}{{\cal R}}
\newcommand{\calS}{{\cal S}}
\newcommand{\calT}{{\cal T}}
\newcommand{\calU}{{\cal U}}

\newcommand{\calX}{{\cal X}}

\newcommand{\calZ}{{\cal Z}}

\newcommand{\eg}{\emph{e.g.,}\xspace}
\newcommand{\ie}{\emph{i.e.,}\xspace}

\newcommand{\hide}[1]{}

\newcommand{\hiddenText}{{\color{gray} hidden text.}}
\newcommand{\hideWithText}[1]{\hiddenText}

\DeclareMathOperator*{\argmin}{arg\,min}

\newcommand{\norm}[1]{\left\| #1 \right\|}

\newcommand{\tran}{^{\mathsf{T}}}

\newcommand{\diag}[1]{\mathrm{diag}\left(#1\right)}
\newcommand{\trace}[1]{\mathrm{tr}\left(#1\right)}

\newcommand{\rank}[1]{\mathrm{rank}(#1)}

\newcommand{\Real}[1]{ { {\mathbb R}^{#1} } }

\newcommand{\blue}[1]{{\color{blue}#1}}

\newcommand{\linkToPdf}[1]{\href{#1}{\blue{(pdf)}}}
\newcommand{\linkToPpt}[1]{\href{#1}{\blue{(ppt)}}}
\newcommand{\linkToCode}[1]{\href{#1}{\blue{(code)}}}
\newcommand{\linkToWeb}[1]{\href{#1}{\blue{(web)}}}
\newcommand{\linkToVideo}[1]{\href{#1}{\blue{(video)}}}
\newcommand{\linkToMedia}[1]{\href{#1}{\blue{(media)}}}
\newcommand{\award}[1]{\xspace} 



%% file: shortcuts.tex
\newcommand{\Sn}{\mathbb{S}^n}

\newcommand{\normlong}[1]{\left\lVert #1 \right\rVert}
\renewcommand{\norm}[1]{\lVert #1 \rVert}
\newcommand{\inprod}[2]{\left\langle #1, #2 \right\rangle}

\newcommand{\mymid}{\ \middle\vert\ }

\newcommand{\bmat}{\left[ \begin{array}}
\newcommand{\emat}{\end{array}\right]}

\newcommand{\abs}[1]{\left|#1\right|}

\newcommand{\bbN}{\mathbb{N}}

\newcommand{\coneT}{\calT_{\Zopt}(\Zbar)}
\newcommand{\coneTnew}{\calT_{\Zopt'}(\Zbar')}
\newcommand{\coneTX}{\calT_{\Xopt}(\Varbar{X})}
\newcommand{\coneTS}{\calT_{\Sopt}(\Varbar{S})}
\newcommand{\coneC}{\calC(\Zbar)}
\newcommand{\coneCX}{\calC_\Primal(\Zbar)}
\newcommand{\coneCS}{\calC_\Dual(\Zbar)}
\newcommand{\coneCnew}{\calC(\Zbar')}

\DeclareDocumentCommand{\coneK}{o}{
  \IfNoValueTF{#1}
    {\calK(\ZHbar)} 
    {\calK(\Zbar; #1)}
}
\DeclareDocumentCommand{\coneKP}{o}{
  \IfNoValueTF{#1}
    {\calK^\circ(\ZHbar)} 
    {\calK^\circ(\Zbar; #1)}
}
\DeclareDocumentCommand{\coneKPX}{o}{
  \IfNoValueTF{#1}
    {\calK^\circ_\Primal(\ZHbar)} 
    {\calK^\circ_\Primal(\Zbar; #1)}
}
\DeclareDocumentCommand{\coneKPS}{o}{
  \IfNoValueTF{#1}
    {\calK^\circ_\Dual(\ZHbar)} 
    {\calK^\circ_\Dual(\Zbar; #1)}
}
\DeclareDocumentCommand{\coneKPXnew}{o}{
  \IfNoValueTF{#1}
    {\calK^\circ_\Primal(\Zbar'; \Hbar')} 
    {\calK^\circ_\Primal(\Zbar'; #1)}
}
\DeclareDocumentCommand{\coneKPSnew}{o}{
  \IfNoValueTF{#1}
    {\calK^\circ_\Dual(\Zbar'; \Hbar')} 
    {\calK^\circ_\Dual(\Zbar'; #1)}
}
\DeclareDocumentCommand{\opeT}{o m}{
  \IfNoValueTF{#1}
    {\Theta(\ZHbar, #2)} 
    {\Theta(\Zbar; #1, #2)}
}
\DeclareDocumentCommand{\opeTP}{o m}{
  \IfNoValueTF{#1}
    {\Thetap(\ZHbar, #2)} 
    {\Thetap(\Zbar; #1, #2)}
}
\DeclareDocumentCommand{\opeE}{o}{
  \IfNoValueTF{#1}
    {\calE(\ZHbar)} 
    {\calE(\Zbar; #1)}
}
\DeclareDocumentCommand{\opeEP}{o}{
  \IfNoValueTF{#1}
    {\calEp(\ZHbar)} 
    {\calEp(\Zbar; #1)}
}
\DeclareDocumentCommand{\opeEnew}{o}{
  \IfNoValueTF{#1}
    {\calE(\Zbar'; \Hbar')} 
    {\calE(\Zbar'; #1)}
}
\DeclareDocumentCommand{\opeEPnew}{o}{
  \IfNoValueTF{#1}
    {\calEp(\Zbar'; \Hbar')} 
    {\calEp(\Zbar'; #1)}
}
\DeclareDocumentCommand{\opeP}{o}{
  \IfNoValueTF{#1}
    {\Psi(\ZHbar)} 
    {\Psi(\Zbar; #1)}
}
\DeclareDocumentCommand{\opeU}{o}{
  \IfNoValueTF{#1}
    {\Upsilon(\ZHbar)} 
    {\Upsilon(\Zbar; #1)}
}
\newcommand{\subindexFirstAPBN}{\substack{a \in \totalsetfirst[+] \\ b \in \totalsetfirst[-]}}
\newcommand{\subindexFirstAPBP}{\substack{a \in \totalsetfirst[+] \\ b \in \totalsetfirst[+]}}
\newcommand{\subindexFirstANBN}{\substack{a \in \totalsetfirst[-] \\ b \in \totalsetfirst[-]}}
\newcommand{\subindexFirstAP}{a \in \totalsetfirst[+]}
\newcommand{\subindexFirstBN}{b \in \totalsetfirst[-]}
\newcommand{\subindexSecondAPBN}{\substack{i \in \totalsetsecond{0}[+] \\ b \in \totalsetsecond{0}[-]}}
\newcommand{\subindexSecondAP}{i \in \totalsetsecond{0}[+]}
\newcommand{\subindexSecondBN}{j \in \totalsetsecond{0}[-]}
\newcommand{\subindexSecondAPBP}{\substack{i \in \totalsetsecond{0}[+] \\ b \in \totalsetsecond{0}[+]}}
\newcommand{\subindexSecondANBN}{\substack{i \in \totalsetsecond{0}[-] \\ b \in \totalsetsecond{0}[-]}}

\newcommand{\accPrefix}{./acceleration/figs}
\newcommand{\sigPrefix}{./sig-wave/figs}
\newcommand{\toyPrefix}{./toy-examples/figs}
\newcommand{\demonPrefix}{./demonstration}
\newcommand{\dataset}[1]{\texttt{#1}}
\newcommand{\Mittelmann}{\texttt{Mittelmann}}
\DeclareRobustCommand{\mymat}[1]{ \begin{bmatrix} #1 \end{bmatrix} }
\DeclareRobustCommand{\mymatplain}[1]{ \begin{matrix} #1 \end{matrix} }
\newcommand{\Xsc}{X_{\mathsf{sc}}}
\newcommand{\Ssc}{S_{\mathsf{sc}}}
\newcommand{\ysc}{y_{\mathsf{sc}}}
\newcommand{\Zsc}{Z_{\mathsf{sc}}}
\newcommand{\Primal}{\mathsf{P}}
\newcommand{\Dual}{\mathsf{D}}
\newcommand{\polar}[1]{{#1}^\circ}
\newcommand{\RomanNum}[1]{\uppercase\expandafter{\romannumeral #1}}
\newcommand{\Zbar}{\Varbar{Z}}
\newcommand{\Hbar}{\Varbar{H}}
\newcommand{\ZHbar}{\Varbar{Z}; \Varbar{H}}
\newcommand{\calEp}{\calE^{\perp}}
\newcommand{\Thetap}{\Theta^{\perp}}

\newcommand{\ppsim}{\Pi_{+}}
\newcommand{\npsim}{\Pi_{-}}

\DeclareDocumentCommand{\mdvZ}{o}{
  \IfNoValueTF{#1}
    {\phi(\Varbar{Z}; \Varbar{H})} 
    {\phi(#1)}
}
\newcommand{\mdvZmap}{\mdvZ[\Varbar{Z}; \cdot]}
\newcommand{\mdvXmap}{\mdvX[\Varbar{Z}; \cdot]}
\newcommand{\mdvSmap}{\mdvS[\Varbar{Z}; \cdot]}
\DeclareDocumentCommand{\mdvX}{o}{
  \IfNoValueTF{#1}
    {\phi_{\Primal}(\Varbar{Z}; \Varbar{H})} 
    {\phi_{\Primal}(#1)}
}
\DeclareDocumentCommand{\mdvXnew}{o}{
  \IfNoValueTF{#1}
    {\phi_{\Primal}(\Varbar{Z}'; \Varbar{H}')} 
    {\phi_{\Primal}(#1)}
}
\DeclareDocumentCommand{\mdvS}{o}{
  \IfNoValueTF{#1}
    {\phi_{\Dual}(\Varbar{Z}; \Varbar{H})} 
    {\phi_{\Dual}(#1)}
}
\DeclareDocumentCommand{\mdvSnew}{o}{
  \IfNoValueTF{#1}
    {\phi_{\Dual}(\Varbar{Z}'; \Varbar{H}')} 
    {\phi_{\Dual}(#1)}
}

\newcommand{\tW}{\Vartilde{W}}
\DeclareDocumentCommand{\DeltaH}{o}{
  \IfNoValueTF{#1}{\Delta_{\Varbar{H}}}{\Delta_{#1}}
}

\newcommand{\psdproj}[1]{\Pi_{\mathbb{S}^{#1}_+}}
\newcommand{\nsdproj}[1]{\Pi_{\mathbb{S}^{#1}_-}}
\newcommand{\eigvalfirst}[1]{\mu_{#1}}
\newcommand{\eigvalsecond}[2]{\eta_{#1, #2}}
\newcommand{\eigvalthird}[3]{\zeta_{#1, #2, #3}}
\newcommand{\indexfirst}[1]{{\alpha_{#1}}}
\newcommand{\indexsecond}[2]{{\beta_{#1, #2}}}
\newcommand{\indexthird}[3]{{\gamma_{#1, #2, #3}}}

\newcommand{\Qfirst}[1]{Q^{#1}}
\newcommand{\Qsecond}[2]{\hat{Q}^{#1, #2}}

\newcommand{\sizeof}[1]{| #1 |}
\newcommand{\range}{\mathrm{ran}}
\newcommand{\Fix}{\mathrm{Fix}}
\newcommand{\closure}{\mathrm{cl}}
\newcommand{\spanning}{\mathrm{span}}

\newcommand{\affinehull}{\mathrm{aff}}
\DeclareDocumentCommand{\totalsetfirst}{o}{
  \IfNoValueTF{#1}{\calI}{\calI_{#1}}
}
\DeclareDocumentCommand{\totalsetsecond}{m o}{
  \IfNoValueTF{#2}{\calI_{#1}}{\calI_{#1, #2}}
}
\DeclareDocumentCommand{\totalsetthird}{m m o}{
  \IfNoValueTF{#3}{\calI_{#1, #2}}{\calI_{#1, #2, #3}}
}
\DeclareDocumentCommand{\deltafirst}{o}{
  \IfNoValueTF{#1}{\delta_{\Varbar{Z}}'}{\delta_{#1}'}
}
\DeclareDocumentCommand{\deltasecond}{o o}{
  \IfNoValueTF{#1}
    {\delta_{\Varbar{Z}, \Varbar{H}}''} 
    {\IfNoValueTF{#2}
       {\delta_{\Varbar{Z}, {#1}}''} 
       {\delta_{{#2}, {#1}}''} 
    }
}
\newcommand{\indexSC}[1]{{\alpha_{\mathsf{#1}}}}
\newcommand{\indexSecondSC}[1]{{\beta_{0, \mathsf{#1}}}}
\newcommand{\indexThirdSC}[1]{{\gamma_{0, 0, \mathsf{#1}}}}
\newcommand{\indexZeroP}{{\alpha_0^{\mathsf{P}}}}
\newcommand{\indexZeroD}{{\alpha_0^{\mathsf{D}}}}
\newcommand{\indexSecondZeroP}{{\beta_{0,0}^{\mathsf{P}}}}
\newcommand{\indexSecondZeroD}{{\beta_{0,0}^{\mathsf{D}}}}
\DeclareDocumentCommand{\SecondOrderLinearTerm}{o m}{
  \IfNoValueTF{#1}
    {\PA \Theta^\perp({#1}; {#2}) + \PAp \Theta({#1}; {#2})}
    {\PA \Theta^\perp(\Varbar{H}; {#2}) + \PAp \Theta(\Varbar{H}; {#2})}
}
\DeclareDocumentCommand{\SecondOrderResidualTerm}{o}{
  \IfNoValueTF{#1}
    {-\PA \calE(\Varbar{H}) - \PAp \calE^\perp(\Varbar{H})}
    {-\PA \calE({#1}) - \PAp \calE^\perp({#1})}
}
\newcommand{\Symp}[1]{\mathbb{S}_{+}^{#1}}
\newcommand{\Symn}[1]{\mathbb{S}_{-}^{#1}}
\newcommand{\Sym}[1]{\mathbb{S}^{#1}}

\newcommand{\PA}{\calP}
\newcommand{\PAp}{\PA^{\perp}}

\newcommand{\Asdp}{\calA}
\newcommand{\AsdpT}{\calA^*}
\newcommand{\Id}{\mathrm{Id}}
\newcommand{\PAb}{\Vartilde{X}}

\newcommand{\invsig}{\frac{1}{\sigma}}
\newcommand{\normtwo}[1]{\norm{#1}_2}

\newcommand{\normF}[1]{\norm{#1}_\mathsf{F}}
\newcommand{\normlongF}[1]{\normlong{#1}_\mathsf{F}}

\newcommand{\Xopt}{\calX_\star}
\newcommand{\Sopt}{\calS_\star}
\newcommand{\Zopt}{\calZ_\star}
\newcommand{\ri}{\mathrm{ri}}
\newcommand{\dist}{\mathrm{dist}}
\newcommand{\simpleadjustbox}[1]{\adjustbox{max width=\linewidth}{\ensuremath{#1}}}
\newcommand{\titlemath}[1]{\texorpdfstring{#1}{M}}

\DeclareDocumentCommand{\Var}{m o o}{
  \IfNoValueTF{#2}
    {{#1}} 
    {\IfNoValueTF{#3}
       {{#1}_{\indexfirst{#2}}} 
       {{#1}_{\indexfirst{#2} \indexfirst{#3}}} 
    }
}

\DeclareDocumentCommand{\Varsecond}{m m o o}{
  \IfNoValueTF{#3}
    {{#1}} 
    {\IfNoValueTF{#4}
       {{#1}_{\indexsecond{#2}{#3}}} 
       {{#1}_{\indexsecond{#2}{#3} \indexsecond{#2}{#4}}} 
    }
}

\DeclareDocumentCommand{\Varhatsecond}{m m o o}{
  \IfNoValueTF{#3}
    {\Varhat{#1}} 
    {\IfNoValueTF{#4}
       {\Varhat{#1}_{\indexsecond{#2}{#3}}} 
       {\Varhat{#1}_{\indexsecond{#2}{#3} \indexsecond{#2}{#4}}} 
    }
}

\DeclareDocumentCommand{\Vark}{m o o}{
  \IfNoValueTF{#2}
    {{#1}^{(k)}} 
    {\IfNoValueTF{#3}
       {{#1}^{(k)}_{\indexfirst{#2}}} 
       {{#1}^{(k)}_{\indexfirst{#2} \indexfirst{#3}}} 
    }
}

\DeclareDocumentCommand{\Varkpo}{m o o}{
  \IfNoValueTF{#2}
    {{#1}^{(k+1)}} 
    {\IfNoValueTF{#3}
       {{#1}^{(k+1)}_{\indexfirst{#2}}} 
       {{#1}^{(k+1)}_{\indexfirst{#2} \indexfirst{#3}}} 
    }
}

\DeclareDocumentCommand{\Varbar}{m o o}{
  \IfNoValueTF{#2}
    {\widebar{#1}} 
    {\IfNoValueTF{#3}
       {\widebar{#1}_{\indexfirst{#2}}} 
       {\widebar{#1}_{\indexfirst{#2} \indexfirst{#3}}} 
    }
}

\DeclareDocumentCommand{\Varhat}{m o o}{
  \IfNoValueTF{#2}
    {\widehat{#1}} 
    {\IfNoValueTF{#3}
       {\widehat{#1}_{\indexfirst{#2}}} 
       {\widehat{#1}_{\indexfirst{#2} \indexfirst{#3}}} 
    }
}

\DeclareDocumentCommand{\Vartilde}{m o o}{
  \IfNoValueTF{#2}
    {\widetilde{#1}} 
    {\IfNoValueTF{#3}
       {\widetilde{#1}_{\indexfirst{#2}}} 
       {\widetilde{#1}_{\indexfirst{#2} \indexfirst{#3}}} 
    }
}

\DeclareDocumentCommand{\VarSC}{m o o}{
  \IfNoValueTF{#2}
    {{#1}} 
    {\IfNoValueTF{#3}
       {{#1}_{\indexSC{#2}}} 
       {{#1}_{\indexSC{#2} \indexSC{#3}}} 
    }
}

\ExplSyntaxOn
\NewDocumentCommand{\VarFirstZeroP}{ m m o }{%
  \ensuremath{%
    \str_case:nnF {#2} {
      {1}{ {#1}_{\indexZeroP} } 
      {2}{ {#1}_{\indexZeroP \indexZeroP} } 
      {c}{ {#1}_{\indexZeroP \indexZeroD} }
      {l}{
        \IfNoValueTF{#3}
          { \PackageError{VarFirstZeroP}{Missing third argument for 'l'}%
            {Use \string\VarFirstZeroP\{x\}\{l\}[i] to get x_{i, A}.} }
          { {#1}_{ \indexZeroP \indexfirst{#3}  } }
      }
      {r}{
        \IfNoValueTF{#3}
          { \PackageError{VarFirstZeroP}{Missing third argument for 'r'}%
            {Use \string\VarFirstZeroP\{x\}\{r\}[i] to get x_{A, i}.} }
          { {#1}_{ \indexfirst{#3} \indexZeroP  } }
      }
    }{
      \PackageError{VarFirstZeroP}{Invalid second argument '#2'}%
        {Allowed values: 1, 2, c, l, r.}
    }%
  }%
}
\ExplSyntaxOff

\ExplSyntaxOn
\NewDocumentCommand{\VarFirstZeroD}{ m m o }{%
  \ensuremath{%
    \str_case:nnF {#2} {
      {1}{ {#1}_{\indexZeroD} } 
      {2}{ {#1}_{\indexZeroD \indexZeroD} } 
      {c}{ {#1}_{\indexZeroD \indexZeroP} }
      {l}{
        \IfNoValueTF{#3}
          { \PackageError{VarFirstZeroD}{Missing third argument for 'l'}%
            {Use \string\VarFirstZeroD\{x\}\{l\}[i] to get x_{i, A}.} }
          { {#1}_{ \indexZeroD \indexfirst{#3}  } }
      }
      {r}{
        \IfNoValueTF{#3}
          { \PackageError{VarFirstZeroD}{Missing third argument for 'r'}%
            {Use \string\VarFirstZeroD\{x\}\{r\}[i] to get x_{A, i}.} }
          { {#1}_{  \indexfirst{#3} \indexZeroD } }
      }
    }{
      \PackageError{VarFirstZeroD}{Invalid second argument '#2'}%
        {Allowed values: 1, 2, c, l, r.}
    }%
  }%
}
\ExplSyntaxOff

\ExplSyntaxOn
\NewDocumentCommand{\VarSecondZeroP}{ m m o }{%
  \ensuremath{%
    \str_case:nnF {#2} {
      {1}{ {#1}_{\indexSecondZeroP} } 
      {2}{ {#1}_{\indexSecondZeroP \indexSecondZeroP} } 
      {c}{ {#1}_{\indexSecondZeroP \indexSecondZeroD} }
      {l}{
        \IfNoValueTF{#3}
          { \PackageError{VarSecondZeroP}{Missing third argument for 'l'}%
            {Use \string\VarSecondZeroP\{x\}\{l\}[i] to get x_{i, A}.} }
          { {#1}_{ \indexSecondZeroP \indexsecond{0}{#3}  } }
      }
      {r}{
        \IfNoValueTF{#3}
          { \PackageError{VarSecondZeroP}{Missing third argument for 'r'}%
            {Use \string\VarFirstZeroP\{x\}\{r\}[i] to get x_{A, i}.} }
          { {#1}_{ \indexsecond{0}{#3} \indexSecondZeroP  } }
      }
    }{
      \PackageError{VarSecondZeroP}{Invalid second argument '#2'}%
        {Allowed values: 1, 2, c, l, r.}
    }%
  }%
}
\ExplSyntaxOff

\ExplSyntaxOn
\NewDocumentCommand{\VarSecondZeroD}{ m m o }{%
  \ensuremath{%
    \str_case:nnF {#2} {
      {1}{ {#1}_{\indexSecondZeroD} } 
      {2}{ {#1}_{\indexSecondZeroD \indexSecondZeroD} } 
      {c}{ {#1}_{\indexSecondZeroD \indexSecondZeroP} }
      {l}{
        \IfNoValueTF{#3}
          { \PackageError{VarSecondZeroD}{Missing third argument for 'l'}%
            {Use \string\VarSecondZeroD\{x\}\{l\}[i] to get x_{i, A}.} }
          { {#1}_{ \indexSecondZeroD \indexsecond{0}{#3}  } }
      }
      {r}{
        \IfNoValueTF{#3}
          { \PackageError{VarSecondZeroD}{Missing third argument for 'r'}%
            {Use \string\VarFirstZeroD\{x\}\{r\}[i] to get x_{A, i}.} }
          { {#1}_{ \indexsecond{0}{#3} \indexSecondZeroD  } }
      }
    }{
      \PackageError{VarSecondZeroD}{Invalid second argument '#2'}%
        {Allowed values: 1, 2, c, l, r.}
    }%
  }%
}
\ExplSyntaxOff

\ExplSyntaxOn
\NewDocumentCommand{\VectorTwoRow}{ m }{%
  \begingroup
    \setlength{\arrayrulewidth}{0.2pt}%
    \renewcommand{\arraystretch}{1.15}%
    \ifcsname dashlinedash\endcsname
      \setlength{\dashlinedash}{1.2pt}%
      \setlength{\dashlinegap}{1.0pt}%
    \fi
    \seq_set_split:Nnn \l_tmpa_seq { ; } { #1 }
    \int_compare:nF { \seq_count:N \l_tmpa_seq = 2 }
      { \PackageWarning{VectorTwoRow}{Expected 2 items separated by `;`} }
    \ensuremath{%
      \left[
        \begin{array}{c| c}
          \seq_item:Nn \l_tmpa_seq {1} & \seq_item:Nn \l_tmpa_seq {2} \\ 
        \end{array}
      \right]
    }%
  \endgroup
}
\ExplSyntaxOff

\ExplSyntaxOn
\NewDocumentCommand{\VectorTwoColumn}{ m }{%
  \begingroup
    \setlength{\arrayrulewidth}{0.2pt}%
    \renewcommand{\arraystretch}{1.15}%
    \ifcsname dashlinedash\endcsname
      \setlength{\dashlinedash}{1.2pt}%
      \setlength{\dashlinegap}{1.0pt}%
    \fi
    \seq_set_split:Nnn \l_tmpa_seq { ; } { #1 }
    \int_compare:nF { \seq_count:N \l_tmpa_seq = 2 }
      { \PackageWarning{VectorTwoColumn}{Expected 2 items separated by `;`} }
    \ensuremath{%
      \left[
        \begin{array}{c}
          \seq_item:Nn \l_tmpa_seq {1}  \\
          \hline
          \seq_item:Nn \l_tmpa_seq {2}
        \end{array}
      \right]
    }%
  \endgroup
}
\ExplSyntaxOff

\ExplSyntaxOn
\NewDocumentCommand{\MatrixFourSC}{ m }{%
  \begingroup
    \setlength{\arrayrulewidth}{0.2pt}%
    \renewcommand{\arraystretch}{1.15}%
    \ifcsname dashlinedash\endcsname
      \setlength{\dashlinedash}{1.2pt}%
      \setlength{\dashlinegap}{1.0pt}%
    \fi
    \seq_set_split:Nnn \l_tmpa_seq { ; } { #1 }
    \int_compare:nF { \seq_count:N \l_tmpa_seq = 4 }
      { \PackageWarning{MatrixFourSC}{Expected 4 items separated by `;`} }
    \ensuremath{%
      \left[
        \begin{array}{c| c}
          \seq_item:Nn \l_tmpa_seq {1} & \seq_item:Nn \l_tmpa_seq {2} \\
          \hline
          \seq_item:Nn \l_tmpa_seq {3} & \seq_item:Nn \l_tmpa_seq {4} 
        \end{array}
      \right]
    }%
  \endgroup
}
\ExplSyntaxOff

\ExplSyntaxOn
\NewDocumentCommand{\MatrixFour}{ m }{%
  \begingroup
    \setlength{\arrayrulewidth}{0.2pt}%
    \renewcommand{\arraystretch}{1.15}%
    \ifcsname dashlinedash\endcsname
      \setlength{\dashlinedash}{1.2pt}%
      \setlength{\dashlinegap}{1.0pt}%
    \fi
    \seq_set_split:Nnn \l_tmpa_seq { ; } { #1 }
    \int_compare:nF { \seq_count:N \l_tmpa_seq = 4 }
      { \PackageWarning{MatrixFour}{Expected 4 items separated by `;`} }
    \ensuremath{%
      \left[
        \begin{array}{c: c}
          \seq_item:Nn \l_tmpa_seq {1} & \seq_item:Nn \l_tmpa_seq {2} \\
          \hdashline
          \seq_item:Nn \l_tmpa_seq {3} & \seq_item:Nn \l_tmpa_seq {4} 
        \end{array}
      \right]
    }%
  \endgroup
}
\ExplSyntaxOff

\ExplSyntaxOn
\NewDocumentCommand{\MatrixNine}{ s m }{%
  \begingroup
    \setlength{\arrayrulewidth}{0.2pt}%
    \renewcommand{\arraystretch}{1.15}%
    \ifcsname dashlinedash\endcsname
      \IfBooleanTF{#1}
        {
          \setlength{\arrayrulewidth}{0.2pt}
          \setlength{\dashlinedash}{3pt}
          \setlength{\dashlinegap}{1.5pt}
        }
        {
          \setlength{\dashlinedash}{1.2pt}%
          \setlength{\dashlinegap}{1.0pt}%
        }%
    \fi
    \seq_set_split:Nnn \l_tmpa_seq { ; } { #2 }
    \int_compare:nF { \seq_count:N \l_tmpa_seq = 9 }
      { \PackageWarning{MatrixNine}{Expected 9 items separated by `;`} }
    \ensuremath{%
      \left[
        \begin{array}{c: c: c}
          \seq_item:Nn \l_tmpa_seq {1} & \seq_item:Nn \l_tmpa_seq {2} & \seq_item:Nn \l_tmpa_seq {3} \\
          \hdashline
          \seq_item:Nn \l_tmpa_seq {4} & \seq_item:Nn \l_tmpa_seq {5} & \seq_item:Nn \l_tmpa_seq {6} \\
          \hdashline
          \seq_item:Nn \l_tmpa_seq {7} & \seq_item:Nn \l_tmpa_seq {8} & \seq_item:Nn \l_tmpa_seq {9}
        \end{array}
      \right]
    }%
  \endgroup
}
\ExplSyntaxOff

\ExplSyntaxOn
\NewDocumentCommand{\MatrixNineFirst}{ O{a} O{b} m }{%
  \begingroup
    \setlength{\arrayrulewidth}{0.2pt}%
    \renewcommand{\arraystretch}{1.15}%
    \ifcsname dashlinedash\endcsname
      \setlength{\dashlinedash}{1.2pt}%
      \setlength{\dashlinegap}{1.0pt}%
    \fi
    \seq_set_split:Nnn \l_tmpa_seq { ; } { #3 }
    \int_compare:nF { \seq_count:N \l_tmpa_seq = 9 }
      { \PackageWarning{MatrixNineFirst}{Expected 9 items separated by `;`} }
    \ensuremath{%
      \left[
        \begin{array}{c: c: c}
          \left\{\seq_item:Nn \l_tmpa_seq {1}\right\}_{\substack{#1 \in \totalsetfirst[+]\\ #2 \in \totalsetfirst[+]}} &
          \left\{\seq_item:Nn \l_tmpa_seq {2}\right\}_{\substack{#1 \in \totalsetfirst[+]\\ #2 \in \totalsetfirst[0]}} &
          \left\{\seq_item:Nn \l_tmpa_seq {3}\right\}_{\substack{#1 \in \totalsetfirst[+]\\ #2 \in \totalsetfirst[-]}} \\
          \hdashline
          \left\{\seq_item:Nn \l_tmpa_seq {4}\right\}_{\substack{#1 \in \totalsetfirst[0]\\ #2 \in \totalsetfirst[+]}} &
          \left\{\seq_item:Nn \l_tmpa_seq {5}\right\}_{\substack{#1 \in \totalsetfirst[0]\\ #2 \in \totalsetfirst[0]}} &
          \left\{\seq_item:Nn \l_tmpa_seq {6}\right\}_{\substack{#1 \in \totalsetfirst[0]\\ #2 \in \totalsetfirst[-]}} \\
          \hdashline
          \left\{\seq_item:Nn \l_tmpa_seq {7}\right\}_{\substack{#1 \in \totalsetfirst[-]\\ #2 \in \totalsetfirst[+]}} &
          \left\{\seq_item:Nn \l_tmpa_seq {8}\right\}_{\substack{#1 \in \totalsetfirst[-]\\ #2 \in \totalsetfirst[0]}} &
          \left\{\seq_item:Nn \l_tmpa_seq {9}\right\}_{\substack{#1 \in \totalsetfirst[-]\\ #2 \in \totalsetfirst[-]}}
        \end{array}
      \right]
    }%
  \endgroup
}
\ExplSyntaxOff

\ExplSyntaxOn
\NewDocumentCommand{\MatrixSixteenSC}{ m }{%
  \begingroup
    \setlength{\arrayrulewidth}{0.2pt}%
    \renewcommand{\arraystretch}{1.15}%
    \ifcsname dashlinedash\endcsname
      \setlength{\dashlinedash}{1.2pt}%
      \setlength{\dashlinegap}{1.0pt}%
    \fi
    \seq_set_split:Nnn \l_tmpa_seq { ; } { #1 }
    \int_compare:nF { \seq_count:N \l_tmpa_seq = 16 }
      { \PackageWarning{MatrixSixteenSC}{Expected 16 items separated by `;`} }
    \ensuremath{%
      \left[
        \begin{array}{c: c|c: c}
          \seq_item:Nn \l_tmpa_seq {1}  & \seq_item:Nn \l_tmpa_seq {2}  & \seq_item:Nn \l_tmpa_seq {3}  & \seq_item:Nn \l_tmpa_seq {4}  \\
          \cdashline{1-2}\cdashline{3-4}
          \seq_item:Nn \l_tmpa_seq {5}  & \seq_item:Nn \l_tmpa_seq {6}  & \seq_item:Nn \l_tmpa_seq {7}  & \seq_item:Nn \l_tmpa_seq {8}  \\
          \hline
          \seq_item:Nn \l_tmpa_seq {9}  & \seq_item:Nn \l_tmpa_seq {10} & \seq_item:Nn \l_tmpa_seq {11} & \seq_item:Nn \l_tmpa_seq {12} \\
          \cdashline{1-2}\cdashline{3-4}
          \seq_item:Nn \l_tmpa_seq {13} & \seq_item:Nn \l_tmpa_seq {14} & \seq_item:Nn \l_tmpa_seq {15} & \seq_item:Nn \l_tmpa_seq {16}
        \end{array}
      \right]
    }%
  \endgroup
}
\ExplSyntaxOff

\ExplSyntaxOn
  \prop_new:N \g_multiindexfirst_prop
  \prop_put:Nnn \g_multiindexfirst_prop {1,1} { \indexfirst{+} \indexfirst{+} }
  \prop_put:Nnn \g_multiindexfirst_prop {1,2} { \indexfirst{+} \indexfirst{0} }
  \prop_put:Nnn \g_multiindexfirst_prop {1,3} { \indexfirst{+} \indexfirst{-} }
  \prop_put:Nnn \g_multiindexfirst_prop {2,1} { \indexfirst{0} \indexfirst{+} }
  \prop_put:Nnn \g_multiindexfirst_prop {2,2} { \indexfirst{0} \indexfirst{0} }
  \prop_put:Nnn \g_multiindexfirst_prop {2,3} { \indexfirst{0} \indexfirst{-} }
  \prop_put:Nnn \g_multiindexfirst_prop {3,1} { \indexfirst{-} \indexfirst{+} }
  \prop_put:Nnn \g_multiindexfirst_prop {3,2} { \indexfirst{-} \indexfirst{0} }
  \prop_put:Nnn \g_multiindexfirst_prop {3,3} { \indexfirst{-} \indexfirst{-} }
  \NewDocumentCommand{\MultiIndexFirst}{mm}{
    \prop_get:NnNTF \g_multiindexfirst_prop {#1,#2} \l_tmpa_tl
      { \ensuremath{\l_tmpa_tl} }
      { \textbf{[undefined (#1,#2)]} }
  }
\ExplSyntaxOff

\ExplSyntaxOn
  \prop_new:N \g_multiindexfirstsc_prop
  \prop_put:Nnn \g_multiindexfirstsc_prop {1,1} { \indexfirst{+} \indexfirst{+} }
  \prop_put:Nnn \g_multiindexfirstsc_prop {1,2} { \indexfirst{+} \indexZeroP }
  \prop_put:Nnn \g_multiindexfirstsc_prop {1,3} { \indexfirst{+} \indexZeroD }
  \prop_put:Nnn \g_multiindexfirstsc_prop {1,4} { \indexfirst{+} \indexfirst{-} }
  \prop_put:Nnn \g_multiindexfirstsc_prop {2,1} { \indexZeroP \indexfirst{+} }
  \prop_put:Nnn \g_multiindexfirstsc_prop {2,2} { \indexZeroP \indexZeroP }
  \prop_put:Nnn \g_multiindexfirstsc_prop {2,3} { \indexZeroP \indexZeroD }
  \prop_put:Nnn \g_multiindexfirstsc_prop {2,4} { \indexZeroP \indexfirst{-} }
  \prop_put:Nnn \g_multiindexfirstsc_prop {3,1} { \indexZeroD \indexfirst{+}}
  \prop_put:Nnn \g_multiindexfirstsc_prop {3,2} { \indexZeroD \indexZeroP }
  \prop_put:Nnn \g_multiindexfirstsc_prop {3,3} { \indexZeroD \indexZeroD }
  \prop_put:Nnn \g_multiindexfirstsc_prop {3,4} { \indexZeroD \indexfirst{-} }
  \prop_put:Nnn \g_multiindexfirstsc_prop {4,1} { \indexfirst{-} \indexfirst{+} }
  \prop_put:Nnn \g_multiindexfirstsc_prop {4,2} { \indexfirst{-} \indexZeroP }
  \prop_put:Nnn \g_multiindexfirstsc_prop {4,3} { \indexfirst{-} \indexZeroD }
  \prop_put:Nnn \g_multiindexfirstsc_prop {4,4} { \indexfirst{-} \indexfirst{-} }
  \NewDocumentCommand{\MultiIndexFirstSC}{mm}{
    \prop_get:NnNTF \g_multiindexfirstsc_prop {#1,#2} \l_tmpa_tl
      { \ensuremath{\l_tmpa_tl} }
      { \textbf{[undefined (#1,#2)]} }
  }
\ExplSyntaxOff

\ExplSyntaxOn
  \prop_new:N \g_multiindexsecond_prop
  \prop_put:Nnn \g_multiindexsecond_prop {1,1} { \indexsecond{0}{+} \indexsecond{0}{+} }
  \prop_put:Nnn \g_multiindexsecond_prop {1,2} { \indexsecond{0}{+} \indexsecond{0}{0} }
  \prop_put:Nnn \g_multiindexsecond_prop {1,3} { \indexsecond{0}{+} \indexsecond{0}{-} }
  \prop_put:Nnn \g_multiindexsecond_prop {2,1} { \indexsecond{0}{0} \indexsecond{0}{+} }
  \prop_put:Nnn \g_multiindexsecond_prop {2,2} { \indexsecond{0}{0} \indexsecond{0}{0} }
  \prop_put:Nnn \g_multiindexsecond_prop {2,3} { \indexsecond{0}{0} \indexsecond{0}{-} }
  \prop_put:Nnn \g_multiindexsecond_prop {3,1} { \indexsecond{0}{-} \indexsecond{0}{+} }
  \prop_put:Nnn \g_multiindexsecond_prop {3,2} { \indexsecond{0}{-} \indexsecond{0}{0} }
  \prop_put:Nnn \g_multiindexsecond_prop {3,3} { \indexsecond{0}{-} \indexsecond{0}{-} }
  \NewDocumentCommand{\MultiIndexSecond}{mm}{
    \prop_get:NnNTF \g_multiindexsecond_prop {#1,#2} \l_tmpa_tl
      { \ensuremath{\l_tmpa_tl} }
      { \textbf{[undefined (#1,#2)]} }
  }
\ExplSyntaxOff

\ExplSyntaxOn
  \prop_new:N \g_multiindexsecondsc_prop
  \prop_put:Nnn \g_multiindexsecondsc_prop {1,1} { \indexsecond{0}{+} \indexsecond{0}{+} }
  \prop_put:Nnn \g_multiindexsecondsc_prop {1,2} { \indexsecond{0}{+} \indexSecondZeroP }
  \prop_put:Nnn \g_multiindexsecondsc_prop {1,3} { \indexsecond{0}{+} \indexSecondZeroD }
  \prop_put:Nnn \g_multiindexsecondsc_prop {1,4} { \indexsecond{0}{+} \indexsecond{0}{-} }
  \prop_put:Nnn \g_multiindexsecondsc_prop {2,1} { \indexSecondZeroP \indexsecond{0}{+} }
  \prop_put:Nnn \g_multiindexsecondsc_prop {2,2} { \indexSecondZeroP \indexSecondZeroP }
  \prop_put:Nnn \g_multiindexsecondsc_prop {2,3} { \indexSecondZeroP \indexSecondZeroD }
  \prop_put:Nnn \g_multiindexsecondsc_prop {2,4} { \indexSecondZeroP \indexsecond{0}{-} }
  \prop_put:Nnn \g_multiindexsecondsc_prop {3,1} { \indexSecondZeroD \indexsecond{0}{+}}
  \prop_put:Nnn \g_multiindexsecondsc_prop {3,2} { \indexSecondZeroD \indexSecondZeroP }
  \prop_put:Nnn \g_multiindexsecondsc_prop {3,3} { \indexSecondZeroD \indexSecondZeroD }
  \prop_put:Nnn \g_multiindexsecondsc_prop {3,4} { \indexSecondZeroD \indexsecond{0}{-} }
  \prop_put:Nnn \g_multiindexsecondsc_prop {4,1} { \indexsecond{0}{-} \indexsecond{0}{+} }
  \prop_put:Nnn \g_multiindexsecondsc_prop {4,2} { \indexsecond{0}{-} \indexSecondZeroP }
  \prop_put:Nnn \g_multiindexsecondsc_prop {4,3} { \indexsecond{0}{-} \indexSecondZeroD }
  \prop_put:Nnn \g_multiindexsecondsc_prop {4,4} { \indexsecond{0}{-} \indexsecond{0}{-} }
  \NewDocumentCommand{\MultiIndexSecondSC}{mm}{
    \prop_get:NnNTF \g_multiindexsecondsc_prop {#1,#2} \l_tmpa_tl
      { \ensuremath{\l_tmpa_tl} }
      { \textbf{[undefined (#1,#2)]} }
  }
\ExplSyntaxOff

\ExplSyntaxOn
  \prop_new:N \g_multiindexfirstpd_prop
  \prop_put:Nnn \g_multiindexfirstpd_prop {1,1} { \indexSC{P} \indexSC{P} }
  \prop_put:Nnn \g_multiindexfirstpd_prop {1,2} { \indexSC{P} \indexSC{D} }
  \prop_put:Nnn \g_multiindexfirstpd_prop {2,1} { \indexSC{D} \indexSC{P} }
  \prop_put:Nnn \g_multiindexfirstpd_prop {2,2} { \indexSC{D} \indexSC{D} }
  \NewDocumentCommand{\MultiIndexFirstPD}{mm}{
    \prop_get:NnNTF \g_multiindexfirstpd_prop {#1,#2} \l_tmpa_tl
      { \ensuremath{\l_tmpa_tl} }
      { \textbf{[undefined (#1,#2)]} }
  }
\ExplSyntaxOff

\ExplSyntaxOn
  \prop_new:N \g_multiindexsecondpd_prop
  \prop_put:Nnn \g_multiindexsecondpd_prop {1,1} { \indexSecondSC{P} \indexSecondSC{P} }
  \prop_put:Nnn \g_multiindexsecondpd_prop {1,2} { \indexSecondSC{P} \indexSecondSC{D} }
  \prop_put:Nnn \g_multiindexsecondpd_prop {2,1} { \indexSecondSC{D} \indexSecondSC{P} }
  \prop_put:Nnn \g_multiindexsecondpd_prop {2,2} { \indexSecondSC{D} \indexSecondSC{D} }
  \NewDocumentCommand{\MultiIndexSecondPD}{mm}{
    \prop_get:NnNTF \g_multiindexsecondpd_prop {#1,#2} \l_tmpa_tl
      { \ensuremath{\l_tmpa_tl} }
      { \textbf{[undefined (#1,#2)]} }
  }
\ExplSyntaxOff

%% file: shortcuts_matrix.tex
\newcommand{\SplitHDashTwoTwo}{\cdashline{1-2}\cdashline{3-4}}   

\newcommand{\BMpad}{\hspace{1pt}}

\ExplSyntaxOn
\RenewDocumentCommand{\VectorTwoRow}{ m }{%
  \begingroup
    \setlength{\arrayrulewidth}{0.7pt}%
    \renewcommand{\arraystretch}{1.15}%
    \setlength{\arraycolsep}{4pt}%
    \seq_set_split:Nnn \l_tmpa_seq { ; } { #1 }%
    \int_compare:nF { \seq_count:N \l_tmpa_seq = 2 }
      { \PackageWarning{VectorTwoRow}{Expected 2 items separated by `;`} }%
    \ensuremath{%
      \left[
        \begin{array}{@{\BMpad}c|c@{\BMpad}}
          \seq_item:Nn \l_tmpa_seq {1} & \seq_item:Nn \l_tmpa_seq {2}
        \end{array}
      \right]
    }%
  \endgroup
}
\ExplSyntaxOff

\ExplSyntaxOn
\RenewDocumentCommand{\VectorTwoColumn}{ m }{%
  \begingroup
    \setlength{\arrayrulewidth}{0.7pt}%
    \renewcommand{\arraystretch}{1.15}%
    \setlength{\arraycolsep}{4pt}%
    \seq_set_split:Nnn \l_tmpa_seq { ; } { #1 }%
    \int_compare:nF { \seq_count:N \l_tmpa_seq = 2 }
      { \PackageWarning{VectorTwoColumn}{Expected 2 items separated by `;`} }%
    \ensuremath{%
      \left[
        \begin{array}{@{\BMpad}c@{\BMpad}}
          \seq_item:Nn \l_tmpa_seq {1} \\
          \hline
          \seq_item:Nn \l_tmpa_seq {2}
        \end{array}
      \right]
    }%
  \endgroup
}
\ExplSyntaxOff

\ExplSyntaxOn
\RenewDocumentCommand{\MatrixFourSC}{ m }{%
  \begingroup
    \setlength{\arrayrulewidth}{0.7pt}%
    \renewcommand{\arraystretch}{1.15}%
    \setlength{\arraycolsep}{4pt}%
    \seq_set_split:Nnn \l_tmpa_seq { ; } { #1 }%
    \int_compare:nF { \seq_count:N \l_tmpa_seq = 4 }
      { \PackageWarning{MatrixFourSC}{Expected 4 items separated by `;`} }%
    \ensuremath{%
      \left[
        \begin{array}{@{\BMpad}c|c@{\BMpad}}
          \seq_item:Nn \l_tmpa_seq {1} & \seq_item:Nn \l_tmpa_seq {2} \\
          \hline
          \seq_item:Nn \l_tmpa_seq {3} & \seq_item:Nn \l_tmpa_seq {4}
        \end{array}
      \right]
    }%
  \endgroup
}
\ExplSyntaxOff

\ExplSyntaxOn
\RenewDocumentCommand{\MatrixFour}{ m }{%
  \begingroup
    \setlength{\arrayrulewidth}{0.75pt}%
    \renewcommand{\arraystretch}{1.15}%
    \ifcsname ADLsomewide\endcsname \ADLsomewide \fi
    \ifcsname dashlinedash\endcsname
      \setlength{\dashlinedash}{2.2pt}%
      \setlength{\dashlinegap}{1.1pt}%
    \fi
    \seq_set_split:Nnn \l_tmpa_seq { ; } { #1 }%
    \int_compare:nF { \seq_count:N \l_tmpa_seq = 4 }
      { \PackageWarning{MatrixFour}{Expected 4 items separated by `;`} }%
    \ensuremath{%
      \left[
        \begin{array}{@{\BMpad}c:c@{\BMpad}}
          \seq_item:Nn \l_tmpa_seq {1} & \seq_item:Nn \l_tmpa_seq {2} \\
          \hdashline
          \seq_item:Nn \l_tmpa_seq {3} & \seq_item:Nn \l_tmpa_seq {4}
        \end{array}
      \right]
    }%
  \endgroup
}
\ExplSyntaxOff

\ExplSyntaxOn
\RenewDocumentCommand{\MatrixNine}{ s m }{%
  \begingroup
    \setlength{\arrayrulewidth}{0.75pt}%
    \renewcommand{\arraystretch}{1.15}%
    \ifcsname ADLsomewide\endcsname \ADLsomewide \fi
    \ifcsname dashlinedash\endcsname
      \IfBooleanTF{#1}
        {
          \setlength{\dashlinedash}{4pt}%
          \setlength{\dashlinegap}{1.6pt}%
        }
        {
          \setlength{\dashlinedash}{2.2pt}%
          \setlength{\dashlinegap}{1.1pt}%
        }%
    \fi
    \seq_set_split:Nnn \l_tmpa_seq { ; } { #2 }%
    \int_compare:nF { \seq_count:N \l_tmpa_seq = 9 }
      { \PackageWarning{MatrixNine}{Expected 9 items separated by `;`} }%
    \ensuremath{%
      \left[
        \begin{array}{@{\BMpad}c:c:c@{\BMpad}}
          \seq_item:Nn \l_tmpa_seq {1} & \seq_item:Nn \l_tmpa_seq {2} & \seq_item:Nn \l_tmpa_seq {3} \\
          \hdashline
          \seq_item:Nn \l_tmpa_seq {4} & \seq_item:Nn \l_tmpa_seq {5} & \seq_item:Nn \l_tmpa_seq {6} \\
          \hdashline
          \seq_item:Nn \l_tmpa_seq {7} & \seq_item:Nn \l_tmpa_seq {8} & \seq_item:Nn \l_tmpa_seq {9}
        \end{array}
      \right]
    }%
  \endgroup
}
\ExplSyntaxOff

\ExplSyntaxOn
\RenewDocumentCommand{\MatrixNineFirst}{ O{a} O{b} m }{%
  \begingroup
    \setlength{\arrayrulewidth}{0.75pt}%
    \renewcommand{\arraystretch}{1.15}%
    \ifcsname ADLsomewide\endcsname \ADLsomewide \fi
    \ifcsname dashlinedash\endcsname
      \setlength{\dashlinedash}{2.2pt}%
      \setlength{\dashlinegap}{1.1pt}%
    \fi
    \seq_set_split:Nnn \l_tmpa_seq { ; } { #3 }%
    \int_compare:nF { \seq_count:N \l_tmpa_seq = 9 }
      { \PackageWarning{MatrixNineFirst}{Expected 9 items separated by `;`} }%
    \ensuremath{%
      \left[
        \begin{array}{@{\BMpad}c:c:c@{\BMpad}}
          \left\{\seq_item:Nn \l_tmpa_seq {1}\right\}_{\substack{#1 \in \totalsetfirst[+]\\ #2 \in \totalsetfirst[+]}} &
          \left\{\seq_item:Nn \l_tmpa_seq {2}\right\}_{\substack{#1 \in \totalsetfirst[+]\\ #2 \in \totalsetfirst[0]}} &
          \left\{\seq_item:Nn \l_tmpa_seq {3}\right\}_{\substack{#1 \in \totalsetfirst[+]\\ #2 \in \totalsetfirst[-]}} \\
          \hdashline
          \left\{\seq_item:Nn \l_tmpa_seq {4}\right\}_{\substack{#1 \in \totalsetfirst[0]\\ #2 \in \totalsetfirst[+]}} &
          \left\{\seq_item:Nn \l_tmpa_seq {5}\right\}_{\substack{#1 \in \totalsetfirst[0]\\ #2 \in \totalsetfirst[0]}} &
          \left\{\seq_item:Nn \l_tmpa_seq {6}\right\}_{\substack{#1 \in \totalsetfirst[0]\\ #2 \in \totalsetfirst[-]}} \\
          \hdashline
          \left\{\seq_item:Nn \l_tmpa_seq {7}\right\}_{\substack{#1 \in \totalsetfirst[-]\\ #2 \in \totalsetfirst[+]}} &
          \left\{\seq_item:Nn \l_tmpa_seq {8}\right\}_{\substack{#1 \in \totalsetfirst[-]\\ #2 \in \totalsetfirst[0]}} &
          \left\{\seq_item:Nn \l_tmpa_seq {9}\right\}_{\substack{#1 \in \totalsetfirst[-]\\ #2 \in \totalsetfirst[-]}}
        \end{array}
      \right]
    }%
  \endgroup
}
\ExplSyntaxOff

\ExplSyntaxOn
\RenewDocumentCommand{\MatrixSixteenSC}{ m }{%
  \begingroup
    \setlength{\arrayrulewidth}{0.75pt}%
    \renewcommand{\arraystretch}{1.15}%
    \ifcsname ADLsomewide\endcsname \ADLsomewide \fi
    \ifcsname dashlinedash\endcsname
      \setlength{\dashlinedash}{2.2pt}%
      \setlength{\dashlinegap}{1.1pt}%
    \fi
    \seq_set_split:Nnn \l_tmpa_seq { ; } { #1 }%
    \int_compare:nF { \seq_count:N \l_tmpa_seq = 16 }
      { \PackageWarning{MatrixSixteenSC}{Expected 16 items separated by `;`} }%
    \ensuremath{%
      \left[
        \begin{array}{@{\BMpad}c:c|c:c@{\BMpad}}
          \seq_item:Nn \l_tmpa_seq {1}  & \seq_item:Nn \l_tmpa_seq {2}  & \seq_item:Nn \l_tmpa_seq {3}  & \seq_item:Nn \l_tmpa_seq {4}  \\
          \SplitHDashTwoTwo
          \seq_item:Nn \l_tmpa_seq {5}  & \seq_item:Nn \l_tmpa_seq {6}  & \seq_item:Nn \l_tmpa_seq {7}  & \seq_item:Nn \l_tmpa_seq {8}  \\
          \hline
          \seq_item:Nn \l_tmpa_seq {9}  & \seq_item:Nn \l_tmpa_seq {10} & \seq_item:Nn \l_tmpa_seq {11} & \seq_item:Nn \l_tmpa_seq {12} \\
          \SplitHDashTwoTwo
          \seq_item:Nn \l_tmpa_seq {13} & \seq_item:Nn \l_tmpa_seq {14} & \seq_item:Nn \l_tmpa_seq {15} & \seq_item:Nn \l_tmpa_seq {16}
        \end{array}
      \right]
    }%
  \endgroup
}
\ExplSyntaxOff

%% file: sections/abstract.tex

\begin{abstract}
    The alternating direction method of multipliers (ADMM) is widely used for solving large-scale semidefinite programs (SDPs), yet on instances with multiple primal--dual optimal solution pairs, it often enters prolonged slow-convergence regions where the Karush--Kuhn--Tucker (KKT) residuals nearly stall. To explain and predict the fine-grained dynamical behavior inside these regions, we develop a local second-order \emph{limit dynamics} framework for ADMM near an \emph{arbitrary} KKT point---not necessarily the eventual limit point of the iterates.
    Assuming the existence of a strictly complementary primal--dual solution pair, we derive a second-order local expansion of the ADMM dynamics by leveraging a refined and simplified variational characterization of the (parabolic) second-order directional derivative of the PSD projection operator. This expansion reveals a closed convex cone of directions along which the local first-order update vanishes, and it induces a second-order \emph{limit map} that governs the persistent drift after transient effects are filtered out. We characterize fundamental properties of this mapping, including its kernel, range, and continuity. A primal--dual decoupling further yields a clean scaling law for the effect of the penalty parameter in ADMM.
    We connect these properties to second-order dynamical features of ADMM, including fixed points, almost-invariant sets, and microscopic phases. Three empirical phenomena in slow-convergence regions are then explained or predicted: (\romannumeral1) angles between consecutive iterate differences are small yet nonzero, except for sparse spikes; (\romannumeral2) primal and dual infeasibilities are insensitive to penalty-parameter updates; and (\romannumeral3) iterates can be transiently trapped in a low-dimensional subspace for an extended period. 
    Extensive numerical experiments on the \Mittelmann\ dataset corroborate our theoretical predictions. 
    
\end{abstract}

%% file: sections/introduction.tex

\section{Introduction}
\label{sec:intro}

Consider the following pair of primal--dual semidefinite programs (SDPs) in standard form:
\begin{equation} 
    \label{eq:intro:sdp}
    \begin{array}{llllll}
        \text{Primal:} \;\; & \text{minimize} & \inprod{C}{X} & \qquad \qquad \quad \text{Dual:} \;\; & \text{maximize} & b\tran y \\
        & \text{subject to} & \Asdp X = b & & \text{subject to} & \AsdpT y + S = C \\
        & & X \in \Symp{n} & & & S \in \Symp{n},
    \end{array}
\end{equation}
with primal variable $X \in \Sym{n}$ and dual variables $S \in \Sym{n}$, $y \in \mathbb{R}^m$. $\Sym{n}$ is the set of real symmetric $n \times n$ matrices and $\Symp{n}$ is the set of positive semidefinite (PSD) matrices in $\Sym{n}$. The linear operator $\Asdp : \Sym{n} \to \mathbb{R}^m$ is defined as $\Asdp X := (\inprod{A_1}{X}, \cdots, \inprod{A_m}{X} )$. $\AsdpT y := \sum_{i=1}^m y_i A_i$ is its adjoint operator. The coefficients $C, A_1, \ldots, A_m$ are symmetric $n \times n$ matrices, and $b \in \mathbb{R}^m$. It is assumed that $\{A_i\}_{i=1}^m$ are linearly independent so that $\Asdp\AsdpT$ is an invertible operator.

As the need to solve large-scale SDPs continues to grow---\eg those stemming from the moment and sums-of-squares (SOS) relaxations in polynomial optimization~\cite{lasserre01siopt-global,parrilo03mp-semidefinite,wang21siopt-tssos,yang22pami-certifiably-plain,kang24wafr-strom,kang25rss-spot}---first-order methods (FOMs) have attracted increasing interest due to their low per-iteration cost and their ability to exploit problem structures such as sparsity. Among these methods, the Alternating Direction Method of Multipliers (ADMM) has become a particularly popular choice, supported by a wide range of implementations, applications, and algorithmic variants~\cite{wen10mpc-admmsdp,zheng17ifac-cdcs-sdpsolver,yang15mpc-sdpnalplus-sdpsolver,garstka21jota-cosmo,rontsis22jota-approximate-admm}.

\paragraph{ADMM for SDP.} Starting from $(X^{(0)}, y^{(0)}, S^{(0)})$, the classical three-step ADMM iteration for the SDP~\eqref{eq:intro:sdp} reads~\cite{wen10mpc-admmsdp}:
\begin{subequations}
    \label{eq:intro:admm-three-step}
    \begin{align}
        & \Varkpo{y} = (\Asdp \AsdpT)^{-1} \left(\sigma^{-1} b - \Asdp \left( \sigma^{-1} \Vark{X} + \Vark{S} - C \right) \right), \\
        & \Varkpo{S} = \Pi_{\Symp{n}} \left( C - \AsdpT \Varkpo{y} - \sigma^{-1} \Vark{X} \right), \\
        & \Varkpo{X} = \Vark{X} + \sigma \left( \Varkpo{S} + \AsdpT \Varkpo{y} - C \right),
    \end{align}
\end{subequations}
where $\Pi_{\Symp{n}}(\cdot)$ denotes the orthogonal projection onto the PSD cone $\Symp{n}$ and $\sigma>0$ is the penalty parameter. Under mild conditions, $(\Vark{X}, \Vark{y}, \Vark{S} )$ is convergent to $(\Varbar{X}, \Varbar{y}, \Varbar{S})$, one of the optimal solution pairs satisfying the Karush--Kuhn--Tucker (KKT) conditions~\cite[Theorem 2]{wen10mpc-admmsdp}:
\begin{equation}
    \label{eq:intro:kkt}
    \Asdp \Varbar{X} = b, \ \AsdpT \Varbar{y} + \Varbar{S} = C, \ \inprod{\Varbar{X}}{\Varbar{S}} = 0, \ \Varbar{X} \in \Symp{n}, \ \Varbar{S} \in \Symp{n}.
\end{equation} 
The ADMM iteration applied to the dual SDP is equivalent to the Douglas--Rachford splitting (DRS) method applied to the primal SDP~\cite{li18sisc-ssn-sdp}:
\begin{align}
    \label{eq:intro:one-step-admm}
    \Varkpo{Z}  
    = \Vark{Z} -\PA ( \psdproj{n}(\Vark{Z})  - \PAb ) + \PAp ( \psdproj{n}(\Vark{-Z}) - \sigma C ),
\end{align}
where $\PA := \AsdpT (\Asdp \AsdpT)^{-1} \Asdp$ denotes the orthogonal projection onto the range space of $\Asdp$. $\PAp := \Id - \PA$ ($\Id$ denotes the identity mapping). $\PAb$ is any constant matrix satisfying $\Asdp \PAb = b$. We can recover the primal and dual variables from~\eqref{eq:intro:one-step-admm} as: $\Vark{X} := \psdproj{n}(\Vark{Z})$ and $\Vark{S} := \invsig \psdproj{n}(-\Vark{Z})$. Thus, each primal--dual optimal solution pair $(\Varbar{X}, \Varbar{S})$ corresponds to one optimal auxiliary variable $\Varbar{Z} := \Varbar{X} - \sigma \Varbar{S}$. We shall also call~\eqref{eq:intro:one-step-admm} the \emph{one-step} ADMM for solving the SDP.  Define the primal optimal set $\Xopt$ (resp. dual optimal set $\Sopt$) as the collection of $X$ (resp. $S$) satisfying KKT conditions in~\eqref{eq:intro:kkt}. We further define the \emph{difference} optimal set $\Zopt$ as $\Zopt := \Xopt - \sigma \Sopt$. 

\paragraph{One-dimensional criteria.} Both the three-step and the one-step ADMM for solving the SDP are high-dimensional dynamical systems. In practice, however, we often observe them through one-dimensional quantities. For the three-step ADMM~\eqref{eq:intro:admm-three-step} in particular, the primal infeasibility, dual infeasibility, and relative gap---collectively, the KKT residuals---are defined as
\begin{align}
    \label{eq:intro:kkt-residual}
    \Vark{r}_p := \frac{\normtwo{\Asdp \Vark{X} - b}}{1 + \normtwo{b}}, \qquad
    \Vark{r}_d := \frac{\normF{\AsdpT \Vark{y} + \Vark{S} - C}}{1 + \normF{C}}, \qquad
    \Vark{r}_g := \frac{\abs{\inprod{C}{\Vark{X}} - b\tran \Vark{y}}}{1 + \abs{\inprod{C}{\Vark{X}}} + \abs{b\tran \Vark{y}}},
\end{align}
with $\Vark{r}_{\max} := \max\{\Vark{r}_p, \Vark{r}_d, \Vark{r}_g\}$ the maximum KKT residual. Since $\Vark{X} \succeq 0$ and $\Vark{S} \succeq 0$ at all iterations, we omit the PSD-violation terms from~\eqref{eq:intro:kkt-residual}.

For the one-step ADMM~\eqref{eq:intro:one-step-admm}, we denote $\Delta \Vark{Z} := \Varkpo{Z} - \Vark{Z}$, which is tightly related to $\Vark{r}_{\max}$~\cite{kang25arxiv-admm}. We write $\normF{\Delta \Vark{Z}}$ for its Frobenius norm. Similarly, we define $\Delta \Vark{X}$ and $\Delta \Vark{S}$ with their Frobenius norms $\normF{\Delta \Vark{X}}$ and $\normF{\Delta \Vark{S}}$. The angle between two consecutive $\Delta \Vark{Z}$ is denoted by $\angle(\Delta \Vark{Z}, \Delta \Varkpo{Z})$, defined as
\begin{align*}
    \angle(\Delta \Varkpo{Z}, \Delta \Vark{Z}) := \arccos \left( 
        \frac{
            \inprod{\Delta \Vark{Z}}{\Delta \Varkpo{Z}}
        }{
            \normF{\Delta \Vark{Z}} \cdot \normF{\Delta \Varkpo{Z}}
        }
     \right).
\end{align*}
We will frequently use these one-dimensional criteria in the subsequent analysis. 

\subsection{ADMM for SDP: Empirical Slow-Convergence Patterns}

Despite its growing popularity and wide adoption, ADMM often suffers from slow-convergence issues when solving SDPs~\cite{kang24wafr-strom,zheng17ifac-cdcs-sdpsolver,han24arxiv-culoras}: after several thousand iterations, progress often slows down dramatically and may nearly stall. This empirical observation \emph{almost} aligns with existing theory. In general, ADMM for SDPs is widely understood to have sublinear convergence. Under additional regularity at the limiting KKT point---such as two-sided constraint nondegeneracy~\cite{chan08siopt-constraint,han18mor-linear} and strict complementarity~\cite{kang25arxiv-admm}---one can establish local linear convergence.
Although these two regularity conditions hold \emph{generically}~\cite{alizadeh9mp-complementarity-nondegeneracy}, they may both fail in SDPs involving multiple KKT points, such as the important SDP instances frequently arising from Moment-SOS relaxation with finite convergence~\cite{lasserre01siopt-global}.
For these SDPs, metric subregularity of the KKT operator at the limiting point, which is required for local linear convergence of primal--dual splitting methods, may easily fail to hold~\cite[Example~1]{cui16arxiv-superlinear-alm-sdp}. Consequently, slow-convergence regions are generally unavoidable for ADMM on SDPs with multiple KKT points, and characterizing these regions is of both practical and theoretical importance. 

\paragraph{Empirical patterns in slow-convergence regions.}
A major motivation for this paper comes from the empirical observation that these slow-convergence regions exhibit remarkably consistent patterns. While a comprehensive numerical study is provided in \S\ref{sec:exp}, here we focus on four representative SDPs from the \Mittelmann\ dataset\footnote{\href{https://plato.asu.edu/ftp/sparse_sdp.html}{https://plato.asu.edu/ftp/sparse_sdp.html}}: \dataset{cnhil10}, \dataset{foot}, \dataset{neu1g}, and \dataset{texture}. These small- to medium-scale instances are among those for which ADMM struggles to reach high accuracy (\eg $r_{\max}\le 10^{-10}$) within $10^6$ iterations~\cite{kang25arxiv-admm}.

\paragraph{Experiment \RomanNum{1}.}
We first run three-step ADMM for about $10^6$ iterations. The initial guesses $(X^{(0)}, y^{(0)}, S^{(0)})$ are all zero and the initial $\sigma$ is set to $1$. In the first $20000$ iterations, $\sigma$ is updated using the classical heuristic that balances the primal and dual infeasibilities~\cite{wen10mpc-admmsdp}; afterward, we fix $\sigma$. Figure~\ref{fig:intro:kkt_ang} reports the trajectories of $\angle(\Delta \Varkpo{Z}, \Delta \Vark{Z})$, $\normF{\Delta \Vark{Z}}$, and $\Vark{r}_{\max}$. We observe the {first} noticeable pattern:
\begin{quote}
    During the prolonged period where $\normF{\Delta \Vark{Z}}$ and $\Vark{r}_{\max}$ nearly stall, $\angle(\Delta \Varkpo{Z}, \Delta \Vark{Z})$ tends to be small yet nonzero (typically around $10^{-3}$ to $10^{-5}$), except for a few ``sparse spikes''.
\end{quote}
This is unusual because even the smallest decision-variable dimension among these four SDPs exceeds $5000$. In such high dimensions, two randomly generated vectors are typically nearly orthogonal, not nearly parallel.

\input{figs/introduction/kkt_ang.tex}

\paragraph{Experiment \RomanNum{2}.}
We next perform a more delicate experiment. Taking $(X^{(40000)}, y^{(40000)}, S^{(40000)})$ as a new initialization, we gradually increase $\sigma$ by a factor of $10$ over the next $5000$ iterations, mimicking the effect of $\sigma$ updating in practice. Figure~\ref{fig:intro:sigma} shows the trajectories of $\normF{\Delta \Vark{X}}$, $\normF{\Delta \Vark{S}}$, $\Vark{r}_p$, and $\Vark{r}_d$ as functions of $\sigma$. We observe the {second} noticeable pattern:
\begin{quote}
    As $\sigma$ changes, $\Vark{r}_p$ and $\Vark{r}_d$ remain almost unchanged. Meanwhile, $\log_{10}(\normF{\Delta \Vark{X}})$ (resp.\ $\log_{10}(\normF{\Delta \Vark{S}})$) increases (resp.\ decreases) approximately linearly with $\log_{10}(\sigma)$, with slope close to $+1$ (resp.\ $-1$).
\end{quote}
This observation conflicts with the common wisdom behind updating $\sigma$ in practice, which aims to balance the primal and dual infeasibilities~\cite{wen10mpc-admmsdp}. The apparent insensitivity of $\Vark{r}_p$ and $\Vark{r}_d$ to $\sigma$ therefore poses a significant challenge for designing effective $\sigma$-update rules.

\input{figs/introduction/sigma.tex}

In this paper, we aim to understand the mechanisms underlying ADMM's slow-convergence regions, with two goals: (\romannumeral1)~to explain the two empirical patterns above; (\romannumeral2)~to predict additional qualitative behaviors in the slow-convergence regions and to shed light on algorithmic design for ADMM on SDPs with multiple KKT points.

\subsection{Contributions}
\input{figs/introduction/demonstration.tex}

Assuming the \emph{existence} of a strictly complementary primal--dual solution pair, we view ADMM for SDPs as a structured nonlinear dynamical system and study its \emph{limiting} behavior in a neighborhood of an arbitrary $\Zbar\in\Zopt$. Focusing on the cone of directions along which the local first-order update vanishes, we construct a local second-order \emph{limit map} $\mdvZmap$ as a vector field. The induced local second-order \emph{limit dynamics} is
\begin{align}
    \label{eq:intro:limitdyn}
    \Varkpo{Z} = \Vark{Z} + \frac{1}{2}\mdvZ[\Zbar; \Vark{Z}-\Zbar] + o(\normF{\Vark{Z}-\Zbar}^2).
\end{align}
We focus on this surrogate model for two reasons: (i) it captures ADMM's local limiting behavior while filtering out transient effects; (ii) it concentrates the complexity of the ADMM dynamics into the limit map $\mdvZmap$, which we show admits clean and useful structure. We then analyze the fundamental properties of $\mdvZmap$ (\eg kernel, range, continuity, and primal--dual partition) and connect them to qualitative features of the limit dynamics (\eg fixed points, almost-invariant sets, phase transitions, and the role of $\sigma$). Two notable aspects of our framework are:
\begin{enumerate}
    \item Rather than proposing new sufficient conditions to guarantee fast local linear convergence of ADMM for SDPs, we focus on \emph{modeling} and \emph{analyzing} the mechanisms that drive ADMM's slow-convergence regions, thereby bridging theory and practice. This physics-driven viewpoint complements the existing literature on local linear convergence of ADMM for SDPs~\cite{han18mor-linear,kang25arxiv-admm} and introduces new tools for understanding ADMM's local dynamical behavior when fast convergence does \emph{not} occur.

    \item Our analysis does \emph{not} require $\Zbar$ to be the limiting point to which ADMM eventually converges. This shifts the emphasis from a \emph{pointwise, asymptotic} paradigm to a \emph{region-wise, transient} one. Such a perspective is fundamentally different from existing second-order analyses for (nonlinear) SDPs, which are typically developed around a fixed limiting solution~\cite{shapiro97mp-first-order-nlsdp,sun06mor-sosc-nlsdp,feng25mp-ssn-nlsdp-without-generalized-jacobian-regularity}.
\end{enumerate}
Concretely, our contributions are as follows.


\paragraph{A refined and simplified formula for the second-order directional derivative of $\psdproj{n}(\cdot)$.}
A central technical ingredient in building our second-order analysis is the (parabolic) second-order directional derivative of the PSD projection $\psdproj{n}(\cdot)$. Our derivation builds on~\cite[Theorem~4.1]{zhang13svva-second-order-directional-derivative-symmatric-matrix-valued} and~\cite[Propositions~3.1--3.2]{liu22svva-second-order-sdcmpcc}, with two key refinements: (\romannumeral1) we correct several minor typos in both references, which yields a cleaner and more streamlined expression; (\romannumeral2) we expose a \emph{self-similar} structure between the first- and (parabolic) second-order directional derivatives of $\psdproj{n}(\cdot)$. This self-similarity is repeatedly exploited in our second-order analysis. We expect the refined variational characterization to be useful beyond the present setting.

\paragraph{A local second-order limiting model for ADMM near any $\Zbar \in \Zopt$.}
Starting from any $\Zbar\in\Zopt$, we expand the one-step ADMM dynamics~\eqref{eq:intro:one-step-admm} around it up to second order, using the first- and (parabolic) second-order directional derivatives of $\psdproj{n}(\cdot)$. For the operator governing the local first-order dynamics, we prove its firm nonexpansiveness and give a detailed characterization of its nonempty fixed-point set $\coneC$. For any stalled first-order direction $\Hbar\in\coneC$, we show that the operator associated with the local second-order dynamics is also firmly nonexpansive but, in general, does \emph{not} admit fixed points. Instead, we prove the existence of the limit of the iterate difference for the second-order dynamics and denote it by $\mdvZ$. By varying $\Hbar$ over $\coneC$, we obtain the local second-order \emph{limit map} $\mdvZmap:\coneC\mapsto\Sym{n}$, which becomes the central object of the paper, and we define the induced limit dynamics accordingly. See Figure~\ref{fig:intro:demonstration} for an illustration. After uncovering a primal--dual decoupling structure hidden in $\mdvZ$, we study four core properties of the limit map $\mdvZmap$ and their physical implications:

\begin{itemize}
    \item \textbf{Kernel of $\mdvZmap$:}
    \begin{itemize}
        \item \emph{Mathematical proof.} We prove that $\ker(\mdvZmap)$ coincides with $\coneT$, the tangent cone to $\Zopt$ at $\Zbar$. This ties ADMM's local dynamics to Sturm's square-root error bound under strict complementarity~\cite{sturm00siopt-error-bound-lmi}.
        \item \emph{Physical interpretation.} From the limit-dynamics viewpoint, if ADMM is initialized with $Z^{(0)}$ satisfying $Z^{(0)}-\Zbar\in\coneC\backslash\coneT$, then $\Delta\Vark{Z}$ transiently tracks $\frac{1}{2}\mdvZ[\Zbar;\Vark{Z}-\Zbar]$. This mechanism partially explains the ``small yet nonzero'' behavior of $\angle(\Delta\Vark{Z},\Delta\Varkpo{Z})$ observed in Experiment~\RomanNum{1}.
    \end{itemize}

    \item \textbf{Range of $\mdvZmap$:}
    \begin{itemize}
        \item \emph{Mathematical proof.} We clarify the relationship between $\range(\mdvZmap)$ and $\affinehull(\coneC)$: (\romannumeral1) in general, $\range(\mdvZmap)\nsubseteq \affinehull(\coneC)$; (\romannumeral2) if, in addition, either the primal or the dual optimal solution is unique, then $\range(\mdvZmap)\subseteq \affinehull(\coneC)$.
        \item \emph{Physical interpretation.} Interpreted through the limit dynamics, these results illuminate how $\coneC$ can act as a local almost-invariant structure and why second-order updates may remain confined to a low-dimensional subspace for a long time.
    \end{itemize}

    \item \textbf{Continuity of $\mdvZmap$:}
    \begin{itemize}
        \item \emph{Mathematical proof.} We first construct explicit points of discontinuity of $\mdvZmap$ on $\coneC$, and then establish an almost-sure type continuity statement for $\mdvZmap$ with respect to the Lebesgue measure on $\affinehull(\coneC)$.
        \item \emph{Physical interpretation.} In terms of limit dynamics, the ``sparse'' discontinuities of $\mdvZmap$ provide a concrete explanation for the ``sparse spikes'' in $\angle(\Delta\Vark{Z},\Delta\Varkpo{Z})$ observed in Experiment~\RomanNum{1}, and enable accurate predictions of these microscopic phase transitions.
    \end{itemize}

    \item \textbf{Effect of $\sigma$ on $\mdvZmap$:}
    \begin{itemize}
        \item \emph{Mathematical proof.} We show that, under the local second-order {limit} dynamics model, the limitations of $\Delta \Vark{X}$ (resp. $\Delta \Vark{S}$) scales \emph{exactly} in proportion to $\sigma$ (resp. $\frac{1}{\sigma}$). We further prove that the second-order limits of both $\Vark{r}_p$ and $\Vark{r}_d$ are irrelevant to $\sigma$. 
        \item \emph{Physical interpretation.} This result directly explains the response curves in Experiment~\RomanNum{2}. We also discuss the implications for designing $\sigma$-update strategies in second-order-dominant regimes.
    \end{itemize}
\end{itemize}

\paragraph{Numerical verification.}
We validate our theory on three (small-scale) SDP examples with multiple KKT points, where first- and second-order quantities can be computed explicitly. We further conduct experiments on the \Mittelmann\ dataset. Across a substantial subset of hard instances, we observe empirical patterns that are explained by our local second-order {limit} dynamics, supporting the generality of the proposed framework. All codes, data, and results can be found in \href{https://github.com/ComputationalRobotics/admmsdp-limitdyn}{https://github.com/ComputationalRobotics/admmsdp-limitdyn}.

\subsection{Limitations}
\label{sec:intro:limitations}

Since our framework is physics-driven, it prioritizes explanatory power over complete mathematical closure in a few places.
The main open issue is that, while the second-order limit map and dynamics~\eqref{eq:intro:limitdyn} exhibit rich and clean structure, it is generally difficult to quantify the approximation error between the limit model~\eqref{eq:intro:limitdyn} and the true ADMM dynamics~\eqref{eq:intro:one-step-admm}, since three coupled layers of approximation are involved. In addition, our analysis assumes the existence of a strictly complementary solution pair, which may fail for certain problem classes. We discuss these issues in more detail in \S\ref{sec:future}. Despite these compromises, we hope that our work can initiate a systematic study of the ubiquitous slow-convergence phenomena in first-order splitting methods for SDPs.

\paragraph{Scope and interpretation.}
Throughout the paper, we carefully separate rigorous mathematical results from empirical/physical interpretation. All formal statements (theorems/propositions/lemmas) are proved rigorously under Assumption~\ref{ass:soa:sc} and Definitions~\ref{def:soa:fod-sod}--\ref{def:soa:mdvZ-dynamics}. The ``Discussion'' subsections (\S\ref{sec:mdvZ-kernel-3}, \S\ref{sec:mdvZ-range-3}, \S\ref{sec:mdvZ-continuity-3}, \S\ref{sec:mdvZ-sigma-3}) and the ``Numerical Experiments'' section (\S\ref{sec:exp}) are explicitly interpretive: they connect the limit-dynamics framework to observed ADMM behavior, and are separated from---and not required for---the theoretical developments.


\subsection{Notation}

Given a finite-dimensional Hilbert space $(\calH,\inprod{\cdot}{\cdot})$ and a convex set $C\subset\calH$, we write $\ri(C)$ for the relative interior of $C$, $\affinehull(C)$ for its affine hull, and $\closure(C)$ for its closure. If, in addition, $C$ is closed and convex, we denote by $\calT_C(x)$ the tangent cone to $C$ at any $x\in C$. Correspondingly, we denote $\calN_C(x)$ as the normal cone to $C$ at $x$. For a convex cone $\calK\subset\calH$, we write $\polar{\calK}$ for its polar cone. Given a mapping $\calT:\calU\mapsto\calH$ (with $\calU\subset\calH$), we define $\ker(\calT):=\{x\in\calU\mid \calT(x)=0\}$, $\range(\calT):=\{\calT(x)\mid x\in\calU\}$, and $\Fix(\calT):=\{x\in\calU\mid \calT(x)=x\}$. We denote by $\Pi_C(x)$ the orthogonal projection of $x\in\calH$ onto $C$. We define $\mathbb{B}_r(x):=\{y\in\calH\mid \norm{y-x}\le r\}$, where $\norm{\cdot}:=\sqrt{\inprod{\cdot}{\cdot}}$. Let $\dist(x, C) := \inf_{y \in C} \norm{x - y}$.

For the space $\Sym{n}$, the inner product is $\inprod{A}{B}=\trace{A\tran B}=\trace{AB}$ for all $A,B\in\Sym{n}$. For any $H\in\Sym{n}$, we denote by $\calO^{n}(H)$ the set of orthonormal matrices that diagonalize $H$, and we write $\normF{H}$ for its Frobenius norm. We denote by $\Symn{n}$ the set of negative semidefinite (NSD) matrices in $\Sym{n}$. When the matrix size $n$ is clear from the context, we abbreviate $\psdproj{n}(H)$ (resp.\ $\nsdproj{n}(H)$) as $\ppsim(H)$ (resp.\ $\npsim(H)$). We further denote $\lambda_{\max}(H)$ (resp. $\lambda_{\min}(H)$) as the maximum (resp. minimum) eigenvalue of $H$. For symmetric matrices, we display only the upper-triangular part for simplicity; the symmetric entries are indicated by ``$\sim$''. We denote by $I_k$ the $k\times k$ identity matrix. For any $v\in\Real{n}$, we write $\normtwo{v}$ for its Euclidean norm.

\subsection{Outline}
After a brief overview of related work in \S\ref{sec:related-works}, we present in \S\ref{sec:psd} a simplified formula for the (parabolic) second-order directional derivative of $\psdproj{n}(\cdot)$. Building on this result, \S\ref{sec:soa} develops a detailed second-order analysis around an arbitrary $\Zbar\in\Zopt$, which naturally leads to the definition of the local second-order limit map $\mdvZmap$ and its induced dynamics---the core concepts of this paper. \S\ref{sec:decouple} analyzes the primal-dual decoupling structures of the limit map immediately afterwards. We then investigate four fundamental properties of $\mdvZmap$ in parallel: (\romannumeral1) its kernel in \S\ref{sec:mdvZ-kernel}; (\romannumeral2) its range in \S\ref{sec:mdvZ-range}; (\romannumeral3) its continuity in \S\ref{sec:mdvZ-continuity}; and (\romannumeral4) the effect of $\sigma$ in \S\ref{sec:mdvZ-sigma}. The connection between each property of $\mdvZmap$ and ADMM's dynamical behavior is discussed at the end of the corresponding section. In \S\ref{sec:toy}, we present three SDP examples, which serve as sanity checks and illustrations of the theory and also contribute to our proofs. In \S\ref{sec:exp}, we report numerical experiments on the \Mittelmann\ dataset. \S\ref{sec:future} lists several future directions and open problems. Finally, \S\ref{sec:conclusion} concludes the paper.

%% file: figs/introduction/kkt_ang.tex

\begingroup

\newcommand{\FieldFig}[1]{
    \begin{minipage}{0.23\textwidth}
        \centering
        \dataset{#1}
        \includegraphics[width=\columnwidth]{\accPrefix/#1/original_ang.png}
        \includegraphics[width=\columnwidth]{\accPrefix/#1/original_kkt.png}
    \end{minipage}
}

\begin{figure}[htbp]
    \centering
    \begin{minipage}{\textwidth}
        \centering
        \hspace{3mm}
        \includegraphics[width=0.5\columnwidth]{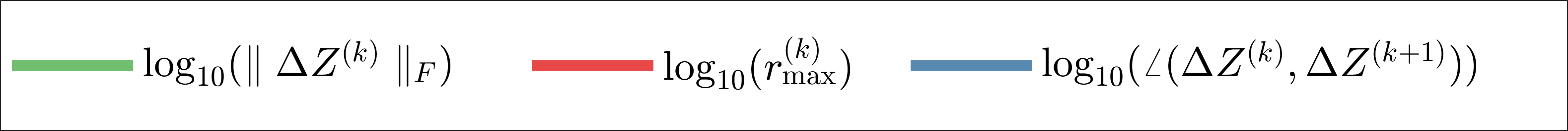}
        \vspace{1mm}
    \end{minipage}

    \begin{minipage}{\textwidth}
        \centering
        \begin{tabular}{cccc}
            \FieldFig{cnhil10}
            \FieldFig{foot}
            \FieldFig{neu1g}
            \FieldFig{texture}
        \end{tabular}
    \end{minipage}

    \caption{\label{fig:intro:kkt_ang} Trajectories of $\Vark{r}_{\max}$, $\normF{\Delta \Vark{Z}}$, and $\angle(\Delta \Vark{Z}, \Delta \Varkpo{Z})$ in Experiment \RomanNum{1}. 
    }
\end{figure}

\endgroup

%% file: figs/introduction/sigma.tex

\begingroup

\newcommand{\FieldFig}[1]{
    \begin{minipage}{0.23\textwidth}
        \centering
        \dataset{#1}
        \includegraphics[width=\columnwidth]{\sigPrefix/#1/four_items.png}
    \end{minipage}
}

\begin{figure}[h]
    \centering
    \begin{minipage}{\textwidth}
        \centering
        \hspace{2mm}
        \includegraphics[width=0.5\columnwidth]{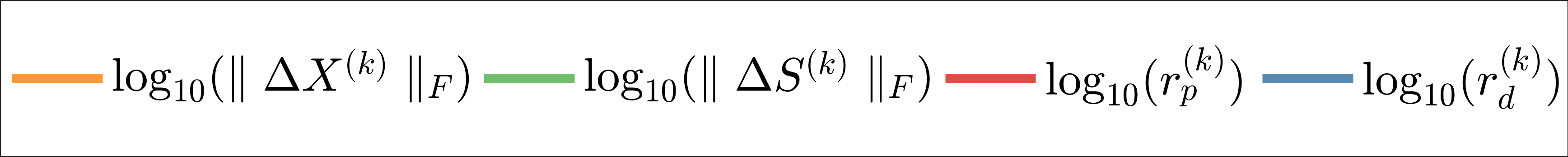}
        \vspace{3mm}
    \end{minipage}

    \begin{minipage}{\textwidth}
        \centering
        \begin{tabular}{cccc}
            \FieldFig{cnhil10}
            \FieldFig{foot}
            \FieldFig{neu1g}
            \FieldFig{texture}
        \end{tabular}
    \end{minipage}

    \caption{\label{fig:intro:sigma} Trajectories of $\normF{\Delta \Vark{X}}$, $\normF{\Delta \Vark{S}}$, $\Vark{r}_p$, and $\Vark{r}_d$ w.r.t $\sigma$ in Experiment \RomanNum{2}.
    }
\end{figure}

\endgroup

%% file: figs/introduction/demonstration.tex

\begin{figure}[htbp]
    \centering
    \includegraphics[width=0.8\columnwidth]{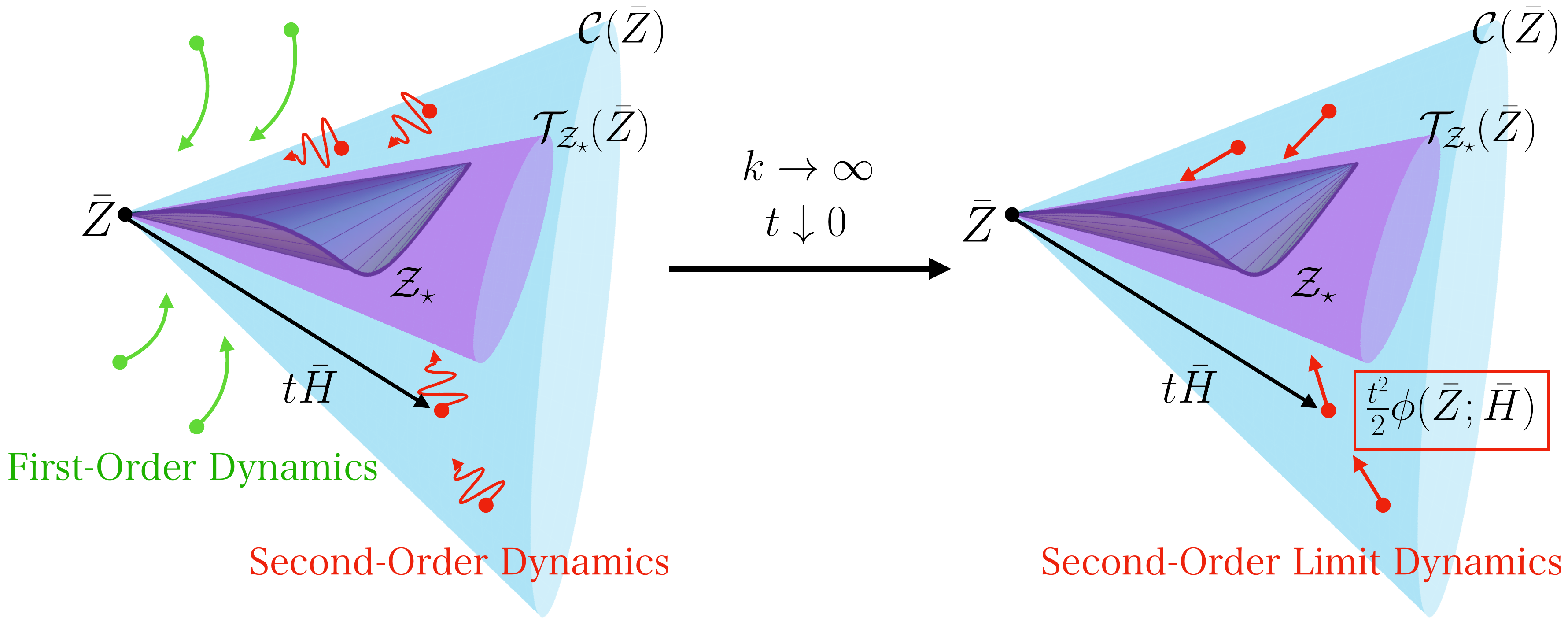}
    \caption{\label{fig:intro:demonstration} Illustration of the local second-order limit dynamics of ADMM for SDPs. The spectrahedron represents the optimal solution set $\Zopt$. The blue cone depicts $\coneC$, the cone of directions along which ADMM's local first-order update vanishes. The purple cone depicts $\coneT$, the tangent cone to $\Zopt$ attached at $\Zbar$. In the left panel, the green points and flow indicate the transient local first-order dynamics, which vanishes as $k\to\infty$ and converges to $\coneC$. The red points and wavy trajectories illustrate the transient local second-order dynamics. For each point of the form $\Zbar+t\Hbar$ with a stalled first-order direction $\Hbar\in\coneC$, the second-order iterate difference converges to $\frac{t^2}{2}\mdvZ$ (red arrows in the right panel), capturing ADMM's limiting behavior up to second order.}
\end{figure}

%% file: sections/related_works.tex

\section{Related Work}
\label{sec:related-works}

Our second-order analysis targets ADMM for SDPs and draws on tools from convex analysis, matrix analysis, monotone operator theory, and dynamical systems.

\paragraph{First-order proximal methods for SDP.}
ADMM can be viewed as a representative primal--dual proximal method arising from monotone operator theory~\cite{ryu22book-monotone}. Beyond the classical ADMM approach for SDPs~\cite{wen10mpc-admmsdp}, symmetric Gauss--Seidel (sGS)-ADMM~\cite{chen17mp-sgsadmm} has attracted increasing attention as an efficient scheme for solving general SDPs to medium accuracy. Other proximal methods, such as the primal--dual hybrid gradient (PDHG) method~\cite{jiang22coa-bregman-sparsesdp}, have likewise been investigated for SDP-type formulations. More recently, these proximal frameworks have been integrated with low-rank factorization schemes to better exploit problem structure and scalability~\cite{han24arxiv-culoras,wang25arxiv-dual-riemannian-admm-low-rank}. On the theoretical side, sufficient conditions guaranteeing fast local linear convergence have been established, including two-sided constraint nondegeneracy~\cite{han18mor-linear} and strict complementarity~\cite{kang25arxiv-admm} at the limiting KKT point. Numerical evidence supporting fast local convergence under such conditions can be found in~\cite{kang25arxiv-admm,wang25arxiv-dual-riemannian-admm-low-rank}.

\paragraph{Variational properties of $\psdproj{n}(\cdot)$.}
The PSD cone projection operator~\cite{higham88laa-computing-psdcone-projection} can be viewed as a spectral function generated by the $\mathsf{ReLU}$ map $\max\{x,0\}$. For spectral functions that are twice (continuously) Fr\'echet differentiable, explicit formulas for the first- and second-order Fr\'echet derivatives are provided in~\cite{lewis01simaa-twice-differentiable-spectral-functions,lewis96dmor-derivatives-spectral-functions}. The broader class of second-order directional differentiability is systematically treated in~\cite{zhang13svva-second-order-directional-derivative-symmatric-matrix-valued}, and~\cite{liu22svva-second-order-sdcmpcc} further leverages these results to derive the (parabolic) second-order directional derivative of $\psdproj{n}(\cdot)$. Additional variational properties, including strong semismoothness, are studied in~\cite{sun02mor-semismooth-matrix-valued-function}, with extensions to more general spectral operators discussed in~\cite{ding18mp-spectral-operator}. Recent work also explores approximating the PSD cone projection via composite polynomial filtering motivated by homomorphic encryption considerations~\cite{kang25arxiv-psdprojection}.

\paragraph{Second-order analysis for (nonlinear) SDP.}
Second-order variational analysis for (nonlinear) SDPs provides a systematic language for curvature, constraint qualifications, and stability. Foundational developments include first-order optimality and sensitivity frameworks~\cite{shapiro97mp-first-order-nlsdp} and second-order sufficient conditions together with constraint nondegeneracy-type regularity~\cite{sun06mor-sosc-nlsdp}. More recent studies introduce weaker second-order conditions~\cite{feng25mp-ssn-nlsdp-without-generalized-jacobian-regularity} and extend such analyses to stratum-restricted settings~\cite{bao26arxiv-stratification-nlsdp}. These second-order conditions also play a role in characterizing critical points arising from reformulations such as the squared-variable approach versus the original nonlinear SDP~\cite{ding25arxiv-squared-variable-formulatio-nlsdp}. In contrast, our work adopts a different viewpoint, emphasizing transient dynamical behavior in ADMM for SDPs rather than asymptotic optimality conditions at a single limiting point.

\paragraph{Optimization algorithms as dynamical systems.}
Many optimization algorithms---including primal--dual splitting methods for conic programming---can be naturally viewed as highly structured iterative maps~\cite{bauschke17springer-convex-analysis-hilbert-spaces,ryu22book-monotone}. Compared with complexity analyses, however, dynamical features such as phases and (almost-)invariant sets remain relatively under-explored. Within the existing literature, partial smoothness and active-set identification~\cite{lewis02siopt-active} provide a powerful mechanism for explaining pronounced phase transitions from slow convergence to faster local regimes~\cite{liang17jota-local,liang14nips-local-linconv-forward-backward}. In the case of SDP, this typically requires the limiting KKT point satisfies strict complementarity. For first-order methods in linear programming, dedicated geometric tools have also been developed to explain phase-transition phenomena even without partial smoothness assumptions~\cite{lu24mp-geometry,xiong24arxiv-accessible-complexity-bounds-pdlp}. In this paper, we investigate multiple dynamical features of ADMM for SDPs through the lens of ``limit dynamics''. This perspective is motivated by dynamical systems theory, where understanding limiting behaviors (\eg limit cycles~\cite{perko13springer-differential-equations-dynamical-systems} and center manifolds~\cite{carr12springer-centre-manifold}) is a standard approach to analyzing complicated trajectories.

%% file: sections/psdproj.tex

\section{Refined Second-Order Directional Derivative of \titlemath{$\psdproj{n}(\cdot)$}}
\label{sec:psd}

Let $f: \mathbb{R} \mapsto \mathbb{R}$ be a (parabolically) second-order directionally differentiable scalar function. Its first- and (parabolic) second-order directional derivatives are defined by
\begin{subequations}
    \label{eq:psd:f-dd}
    \begin{eqnarray}
        & \displaystyle f'(z; h) := \lim_{t \downarrow 0} \frac{f(z + th) - f(z)}{t}, \label{psd:eq:psd:f-dd-first} \\
        & \displaystyle f''(z; h, w) := \lim_{t \downarrow 0} \frac{f(z + th + \frac{t^2}{2} w) - f(z) - t f'(z; h)}{\frac{1}{2}t^2}. \label{psd:eq:psd:f-dd-second}
    \end{eqnarray}
\end{subequations}
Next, let $F: \Sym{n} \mapsto \Sym{n}$ be a (parabolically) second-order directionally differentiable spectral function generated by $f$. Namely, for any $X \in \Sym{n}$ with $Q \in \calO^{n}(X)$ and $\{\lambda_i\}_{i=1}^n$ the eigenvalues of $X$,
\begin{align*}
    F(X) = F(Q \diag{\{\lambda_i\}_{i=1}^n} Q\tran) = Q \diag{\{f(\lambda_i)\}_{i=1}^n} Q\tran.
\end{align*}
The first- and (parabolic) second-order directional derivatives of $F$ are defined as
\begin{subequations}
    \label{eq:psd:general-dd}
    \begin{eqnarray}
        & \displaystyle F'(Z; H) := \lim_{t \downarrow 0} \frac{F(Z + tH) - F(Z)}{t}, \label{eq:psd:general-dd-first} \\
        & \displaystyle F''(Z; H, W) := \lim_{t \downarrow 0} \frac{F(Z + tH + \frac{1}{2}t^2 W) - F(Z) - t F'(Z; H)}{\frac{1}{2}t^2}. \label{eq:psd:general-dd-second}
    \end{eqnarray}
\end{subequations}
In particular, we are interested in the case when $F = \psdproj{n}(\cdot)$, the PSD cone projection operator generated by $f(x) = \max\{x, 0\}$. We know that $\psdproj{n}(\cdot)$ is second-order directionally differentiable; see~\cite{zhang13svva-second-order-directional-derivative-symmatric-matrix-valued}.

\paragraph{Nested eigen-structure description.} To present the first- and second-order directional derivatives of $\psdproj{n}(\cdot)$ succinctly, we introduce a nested eigen-structure description that recursively captures the eigenvalue structures of the first- and second-order perturbation matrices $H$ and $W$. We adopt the notation of~\cite{zhang13svva-second-order-directional-derivative-symmatric-matrix-valued}, with adjustments tailored to the PSD cone projection.
\begin{enumerate}
    \item \emph{First-level description.} Start with a diagonal matrix $Z$. Denote its distinct positive eigenvalues (if any) by $\{\eigvalfirst{a}\}_{a \in \totalsetfirst[+]}$ and its distinct negative eigenvalues (if any) by $\{\eigvalfirst{b}\}_{b \in \totalsetfirst[-]}$. If $Z$ has a zero eigenvalue, denote it by $\eigvalfirst{0}=0$. Define $\totalsetfirst := \totalsetfirst[+] \cup \totalsetfirst[-] \cup \{0\}$. For each sub-block $k \in \totalsetfirst$, let the corresponding index set be $\indexfirst{k}$. That is,
    \begin{align*}
        Z_{\indexfirst{k} \indexfirst{l}} = \begin{cases}
            \eigvalfirst{k} \cdot I_{\sizeof{\indexfirst{k}}}, & \quad k = l \in \totalsetfirst, \\
            0, & \quad k \ne l,\ \ k,l \in \totalsetfirst .
        \end{cases}
    \end{align*}
    Here, for any matrix $A$, $A_{\indexfirst{k}\indexfirst{l}}$ denotes the sub-block of $A$ with row indices $\indexfirst{k}$ and column indices $\indexfirst{l}$. We also write $A_{\indexfirst{k}}$ for the \emph{columns} of $A$ indexed by $\indexfirst{k}$. Finally, define $\indexfirst{+} := \cup_{a \in \totalsetfirst[+]} \indexfirst{a}$ and $\indexfirst{-} := \cup_{b \in \totalsetfirst[-]} \indexfirst{b}$.

    \item \emph{Second-level description.} Let $Z \in \Sym{n}$ be given by the first-level description, and let $H \in \Sym{n}$ be another symmetric matrix. For any $k \in \totalsetfirst$, extract the corresponding sub-block of $H$, denoted by $\Var{H}[k][k]$. Denote its distinct positive eigenvalues (if any) by $\{\eigvalsecond{k}{i}\}_{i \in \totalsetsecond{k}[+]}$ and its distinct negative eigenvalues (if any) by $\{\eigvalsecond{k}{j}\}_{j \in \totalsetsecond{k}[-]}$. If it has a zero eigenvalue, denote it by $\eigvalsecond{k}{0}=0$. Define $\totalsetsecond{k} := \totalsetsecond{k}[+] \cup \totalsetsecond{k}[-] \cup \{0\}$. For $\Var{H}[k][k]$, let each eigen-block $i \in \totalsetsecond{k}$ be indexed by a set $\indexsecond{k}{i}$. Equivalently, there exists $\Qfirst{k} \in \calO^{\sizeof{\indexfirst{k}}}(H_{\indexfirst{k} \indexfirst{k}})$ such that
    \begin{align*}
        (\Qfirst{k}_{\indexsecond{k}{i}})\tran H_{\indexfirst{k} \indexfirst{k}} \Qfirst{k}_{_{\indexsecond{k}{j}}}
        = \begin{cases}
            \eigvalsecond{k}{i} \cdot I_{\sizeof{\indexsecond{k}{i}}}, & \quad i = j \in \totalsetsecond{k}, \\
            0, & \quad i \ne j,\ \ i,j \in \totalsetsecond{k}.
        \end{cases}
    \end{align*}
    Finally, define $\indexsecond{k}{+} := \cup_{i \in \totalsetsecond{k}[+]} \indexsecond{k}{i}$ and $\indexsecond{k}{-} := \cup_{j \in \totalsetsecond{k}[-]} \indexsecond{k}{j}$.

    \item \emph{Third-level description.} Let $Z \in \Sym{n}$ be given by the first-level description, and let $H \in \Sym{n}$ be given by the second-level description. Let $W \in \Sym{n}$. For any $k \in \totalsetfirst$, define 
    \begin{align}
        \label{eq:psd:V}
        V_k(H, W) := W_{\indexfirst{k} \indexfirst{k}} + \sum_{l \in \totalsetfirst \backslash \{k\}} \frac{2}{\eigvalfirst{k} - \eigvalfirst{l}} \cdot H_{\indexfirst{k} \indexfirst{l}} H_{\indexfirst{l} \indexfirst{k}}.
    \end{align}
    We may abbreviate $V_k(H, W)$ as $V_k$ if there is no ambiguity. Abbreviate $(\Qfirst{k}_{\indexsecond{k}{i}})\tran V_k \Qfirst{k}_{\indexsecond{k}{j}}$ as $\hat{V}_k^{i,j}$. For any $k \in \totalsetfirst, i \in \totalsetsecond{k}$, we denote $\hat{V}_k^{i,i}$'s distinct positive eigenvalues (if any) by $\{\eigvalthird{k}{i}{i'}\}_{i' \in \totalsetthird{k}{i}[+]}$ and its distinct negative eigenvalues (if any) by $\{\eigvalthird{k}{i}{j'}\}_{j' \in \totalsetthird{k}{i}[-]}$. If $\hat{V}_k^{i,i}$ has a zero eigenvalue, denote it by $\eigvalthird{k}{i}{0}=0$. Define $\totalsetthird{k}{i} := \totalsetthird{k}{i}[+] \cup \totalsetthird{k}{i}[-] \cup \{0\}$. For $\hat{V}_k^{i,i}$, let each eigen-block $i' \in \totalsetthird{k}{i}$ be indexed by a set $\indexthird{k}{i}{i'}$. Equivalently, there exists $\Qsecond{k}{i} \in \calO^{\sizeof{\indexsecond{k}{i}}}(\hat{V}_k^{i,i})$ such that
    \begin{align*}
        (\Qsecond{k}{i}_{\indexthird{k}{i}{i'}})\tran \hat{V}_k^{i,i} \Qsecond{k}{i}_{\indexthird{k}{i}{j'}}
        = \begin{cases}
            \eigvalthird{k}{i}{i'} \cdot I_{\sizeof{\indexthird{k}{i}{i'}}}, & \quad i' = j' \in \totalsetthird{k}{i}, \\
            0, & \quad i' \ne j', \ i', j' \in \totalsetthird{k}{i}.
        \end{cases}
    \end{align*}
    Finally, define $\indexthird{k}{i}{+} := \cup_{i' \in \totalsetthird{k}{i}[+]} \indexthird{k}{i}{i'}$ and $\indexthird{k}{i}{-} := \cup_{j' \in \totalsetthird{k}{i}[-]} \indexthird{k}{i}{j'}$.
\end{enumerate}
Now suppose we are given a triplet $(Z, H, W)$ from the above three-level description. 
For visualization, we partition the $n \times n$ matrix into $3 \times 3$ sub-blocks, based on $Z$'s positive-zero-negative eigenvalue structures. The partition is represented by \emph{dashed} lines. For instance, 
\begin{align*}
    H = \MatrixNine{
        H_{\MultiIndexFirst{1}{1}} ; H_{\MultiIndexFirst{1}{2}} ; H_{\MultiIndexFirst{1}{3}} ;
        \sim ; H_{\MultiIndexFirst{2}{2}} ; H_{\MultiIndexFirst{2}{3}} ; 
        \sim ; \sim ; H_{\MultiIndexFirst{3}{3}}
    } = \MatrixNine{
        \left\{ \Var{H}[a][b] \right\}_{\subindexFirstAPBP} ; 
        \left\{ \Var{H}[a][0] \right\}_{\subindexFirstAP} ; 
        \left\{ \Var{H}[a][b] \right\}_{\subindexFirstAPBN} ; 
        \sim ; 
        \Var{H}[0][0] ; 
        \left\{ \Var{H}[0][b] \right\}_{\subindexFirstBN} ;
        \sim ; 
        \sim ; 
        \left\{ \Var{H}[a][b] \right\}_{\subindexFirstANBN}
    }.
\end{align*}
Similarly, for the $\MultiIndexFirst{2}{2}$ block of $W$, we partition it into $3 \times 3$ sub-blocks following $\Var{H}[0][0]$'s positive-zero-negative eigenvalue structures. Since $\Var{H}[0][0]$ is no longer diagonal, a basis change is necessary. For instance,
\begin{align*}
    \simpleadjustbox{
    \Var{W}[0][0] = \Qfirst{0} \MatrixNine{
        \Varhat{W}_{\MultiIndexSecond{1}{1}} ; \Varhat{W}_{\MultiIndexSecond{1}{2}} ; \Varhat{W}_{\MultiIndexSecond{1}{3}} ;
        \sim ; \Varhat{W}_{\MultiIndexSecond{2}{2}} ; \Varhat{W}_{\MultiIndexSecond{2}{3}} ;
        \sim ; \sim ; \Varhat{W}_{\MultiIndexSecond{3}{3}}
    } (\Qfirst{0})\tran = \Qfirst{0} \MatrixNine{
        \left\{ \Varhatsecond{W}{0}[i][j] \right\}_{\subindexSecondAPBP} ; 
        \left\{ \Varhatsecond{W}{0}[i][0] \right\}_{\subindexSecondAP} ; 
        \left\{ \Varhatsecond{W}{0}[i][j] \right\}_{\subindexSecondAPBN} ; 
        \sim ;
        \Varhatsecond{W}{0}[0][0] ;
        \left\{ \Varhatsecond{W}{0}[0][j] \right\}_{\subindexSecondBN} ;
        \sim ; 
        \sim ; 
        \left\{ \Varhatsecond{W}{0}[i][j] \right\}_{\subindexSecondANBN}
    } (\Qfirst{0})\tran.
    }
\end{align*}
where $\Varhat{W} := (\Qfirst{0})\tran \Var{W}[0][0] \Qfirst{0}$. 

\paragraph{First-order directional derivative of $\psdproj{n}(\cdot)$.} We provide the following classical result from~\cite[Theorem 4.7]{sun02mor-semismooth-matrix-valued-function} that gives the first-order directional derivative of $\psdproj{n}(\cdot)$.
\begin{theorem}[$\psdproj{n}'(Z; H)$]
    \label{thm:psd:first-dd}
    Let $Z \in \Sym{n}$ be given by the first-level description. For any $H \in \Sym{n}$ given by the second-level description,
    \begin{align}
        \label{eq:psd:first-dd-diagonal-Z}
        \ppsim'(Z; H) = \MatrixNine{
            \Var{H}[+][+] ; 
            \Var{H}[+][0] ; 
            \left\{ \frac{\eigvalfirst{a}}{\eigvalfirst{a} - \eigvalfirst{b}} \Var{H}[a][b] \right\}_{\subindexFirstAPBN} ; 
            \sim ; 
            \ppsim(\Var{H}[0][0]) ; 
            0 ;
            \sim ; 
            \sim ; 
            0
        }.
    \end{align} 
    For a non-diagonal $Z \in \Sym{n}$: Pick $Q \in \calO^n(Z)$. Denote $\Vartilde{Z} := Q\tran Z Q$ diagonal and $\Vartilde{H} := Q\tran H Q$:
    \begin{align}
        \label{eq:psd:first-dd-nondiagonal-Z}
        \ppsim'(Z; H) = Q \ppsim'(\Vartilde{Z}; \Vartilde{H}) Q\tran. 
    \end{align} 
\end{theorem}

\paragraph{(Parabolic) second-order directional derivative of $\psdproj{n}(\cdot)$.} Our result builds on~\cite[Theorem~4.1]{zhang13svva-second-order-directional-derivative-symmatric-matrix-valued} and~\cite[Propositions~3.1--3.2]{liu22svva-second-order-sdcmpcc}, with two key refinements: (\romannumeral1) we correct several minor typos in~\cite{zhang13svva-second-order-directional-derivative-symmatric-matrix-valued} and~\cite{liu22svva-second-order-sdcmpcc}, which in turn yields a simplified formula; (\romannumeral2) we reveal a \emph{self-similar} structure between the $\MultiIndexFirst{2}{2}$ block of $\ppsim''(Z; H, W)$ and $\ppsim'(Z; H)$, which serves as a key ingredient in the subsequent second-order analysis.

\begin{theorem}[$\psdproj{n}''(Z; H, W)$]
    \label{thm:psd:second-dd}
    Let the triplet $(Z, H, W)$ be given by the three-level description. Then, 
    \begin{align}
        \label{eq:psd:second-dd-diagonal-Z}
        & \ppsim''(Z; H, W) = \\
        & \simpleadjustbox{
            \MatrixNine{
                \left\{ \substack{
                    \Var{W}[a][b] \\
                    + 2 \sum\limits_{c \in \totalsetfirst[-]} \frac{-\eigvalfirst{c} }{(\eigvalfirst{c} - \eigvalfirst{a}) (\eigvalfirst{c} - \eigvalfirst{b})} \Var{H}[a][c] \Var{H}[c][b]
                } \right\}_{\subindexFirstAPBP} ;
                \left\{ \substack{
                    \Var{W}[a][0] \\
                    + 2 \sum\limits_{c \in \totalsetfirst[-]} \frac{1}{\eigvalfirst{a} - \eigvalfirst{c}} \Var{H}[a][c] \Var{H}[c][0] \\
                    - 2 \frac{1}{\eigvalfirst{a}} \Var{H}[a][0] \ppsim(-\Var{H}[0][0])
                } \right\}_{\subindexFirstAP} ;
                \left\{ \substack{
                    \frac{\eigvalfirst{a}}{\eigvalfirst{a} - \eigvalfirst{b}} \Var{W}[a][b] \\
                    + 2 \sum\limits_{c \in \totalsetfirst[+]} \frac{-\eigvalfirst{b}}{(\eigvalfirst{b} - \eigvalfirst{a})(\eigvalfirst{b} - \eigvalfirst{c})} \Var{H}[a][c] \Var{H}[c][b] \\
                    + 2 \frac{1}{\eigvalfirst{a} - \eigvalfirst{b}} \Var{H}[a][0] \Var{H}[0][b] \\
                    + 2 \sum\limits_{c \in \totalsetfirst[-]} \frac{\eigvalfirst{a}}{(\eigvalfirst{a} - \eigvalfirst{b})(\eigvalfirst{a} - \eigvalfirst{c})} \Var{H}[a][c] \Var{H}[c][b]
                } \right\}_{\subindexFirstAPBN} ;
                \sim ;
                \substack{
                    2 \sum\limits_{c \in \totalsetfirst[+]} \frac{1}{\eigvalfirst{c}} \Var{H}[0][c] \Var{H}[c][0] \\
                    + \ppsim'(\Var{H}[0][0]; V_{0}(\Var{H}, W))
                } ;
                \left\{ \substack{
                    2 \sum\limits_{c \in \totalsetfirst[+]} \frac{1}{\eigvalfirst{c} - \eigvalfirst{b}} \Var{H}[0][c] \Var{H}[c][b] \\
                    + 2 \frac{1}{-\eigvalfirst{b}} \ppsim(\Var{H}[0][0]) \Var{H}[0][b]
                } \right\}_{\subindexFirstBN} ;
                \sim ; 
                \sim ; 
                \left\{ \substack{
                    2 \sum\limits_{c \in \totalsetfirst[+]} \frac{\eigvalfirst{c} }{(\eigvalfirst{c} - \eigvalfirst{a}) (\eigvalfirst{c} - \eigvalfirst{b})} \Var{H}[a][c] \Var{H}[c][b]
                } \right\}_{\subindexFirstANBN}
            }.
        } \nonumber
    \end{align}
    where $V_0(H, W)$ is defined in~\eqref{eq:psd:V}. $\ppsim'(\Var{H}[0][0]; V_0(\Var{H}, W))$ in the $\MultiIndexFirst{2}{2}$ block is calculated by~\eqref{eq:psd:first-dd-nondiagonal-Z}, since $\Var{H}[0][0]$ is not diagonal.
    For a non-diagonal $Z \in \Sym{n}$: Pick $Q \in \calO^n(Z)$. Denote $\Vartilde{Z} := Q\tran Z Q$ diagonal,  $\Vartilde{H} := Q\tran H Q, \ \Vartilde{W} := Q\tran W Q$:
    \begin{align}
        \label{eq:psd:second-dd-nondiagonal-Z}
        \ppsim''(Z; H, W) = Q \ppsim''(\Vartilde{Z}; \Vartilde{H}, \Vartilde{W}) Q\tran. 
    \end{align}
\end{theorem}
For readability, we postpone the proof and discussion of Theorem~\ref{thm:psd:second-dd} to Appendix~\ref{app:sec:proof-second-dd}. One may have already noticed that $\psdproj{n}''(Z; H, W)_{\MultiIndexFirst{2}{2}}$ exhibits a strong structural resemblance to $\psdproj{n}'(Z; H)$. This is not a coincidence; rather, it stems from the \emph{self-similarity} between the first- and (parabolic) second-order directional derivatives of $f(x)=\max\{x,0\}$:
\begin{align*}
    f'(h; w) = f''(0; h, w) = \begin{cases}
        w, & \quad h > 0 \\
        \max\{w,0\}, & \quad h = 0 \\
        0, & \quad h < 0
    \end{cases}.
\end{align*}

\paragraph{First- and (parabolic) second-order directional derivatives of $\nsdproj{n}(\cdot)$.} For convenience and further use, we also derive $\nsdproj{n}'(Z; H)$ and $\nsdproj{n}''(Z; H, W)$.
\begin{theorem}[$\nsdproj{n}'(Z; H)$]
    \label{thm:psd:nsd-first-dd}
    Let $Z \in \Sym{n}$ be given by the first-level description. For any $H \in \Sym{n}$ given by the second-level description,
    \begin{align}
        \label{eq:psd:nsd-first-dd-diagonal-Z}
        \npsim'(Z; H) = \MatrixNine{
            0 ; 
            0 ; 
            \left\{ \frac{-\eigvalfirst{b}}{\eigvalfirst{a} - \eigvalfirst{b}} \Var{H}[a][b] \right\}_{\subindexFirstAPBN} ; 
            \sim ; 
            \npsim(\Var{H}[0][0]) ; 
            \Var{H}[0][-] ;
            \sim ; 
            \sim ; 
            \Var{H}[-][-]
        }.
    \end{align}
    For a non-diagonal $Z \in \Sym{n}$: Pick $Q \in \calO^n(Z)$. Denote $\Vartilde{Z} := Q\tran Z Q$ diagonal and $\Vartilde{H} := Q\tran H Q$:
    \begin{align}
        \label{eq:psd:nsd-first-dd-nondiagonal-Z}
        \npsim'(Z; H) = Q \npsim'(\Vartilde{Z}; \Vartilde{H}) Q\tran .
    \end{align} 
\end{theorem}
\begin{proof}
    Since $\ppsim(Z) + \npsim(Z) = Z$, we get 
    \begin{align*}
        \ppsim'(Z; H) + \npsim'(Z; H) = H.
    \end{align*}
    Then,~\eqref{eq:psd:nsd-first-dd-diagonal-Z} is derived from~\eqref{eq:psd:first-dd-diagonal-Z} by calculating $H - \ppsim'(Z; H)$. 
\end{proof}

\begin{theorem}[$\nsdproj{n}''(Z; H, W)$]
    \label{thm:psd:nsd-second-dd}
    Let the triplet $(Z, H, W)$ be given by the three-level description. Then, 
    \begin{align}
        \label{eq:psd:nsd-second-dd-diagonal-Z}
        & \npsim''(Z; H, W) = \\
        & \simpleadjustbox{
            \MatrixNine{
                \left\{ \substack{
                    2 \sum\limits_{c \in \totalsetfirst[-]} \frac{\eigvalfirst{c} }{(\eigvalfirst{c} - \eigvalfirst{a}) (\eigvalfirst{c} - \eigvalfirst{b})} \Var{H}[a][c] \Var{H}[c][b]
                } \right\}_{\subindexFirstAPBP} ;
                \left\{ \substack{
                    2 \sum\limits_{c \in \totalsetfirst[-]} \frac{1}{\eigvalfirst{c} - \eigvalfirst{a}} \Var{H}[a][c] \Var{H}[c][0] \\
                    + 2 \frac{1}{-\eigvalfirst{a}} \Var{H}[a][0] \npsim(\Var{H}[0][0])
                } \right\}_{\subindexFirstAP} ;
                \left\{ \substack{
                    \frac{-\eigvalfirst{b}}{\eigvalfirst{a} - \eigvalfirst{b}} \Var{W}[a][b] \\
                    + 2 \sum\limits_{c \in \totalsetfirst[+]} \frac{\eigvalfirst{b}}{(\eigvalfirst{b} - \eigvalfirst{a})(\eigvalfirst{b} - \eigvalfirst{c})} \Var{H}[a][c] \Var{H}[c][b] \\
                    + 2 \frac{1}{\eigvalfirst{b} - \eigvalfirst{a}} \Var{H}[a][0] \Var{H}[0][b] \\
                    + 2 \sum\limits_{c \in \totalsetfirst[-]} \frac{-\eigvalfirst{a}}{(\eigvalfirst{a} - \eigvalfirst{b})(\eigvalfirst{a} - \eigvalfirst{c})} \Var{H}[a][c] \Var{H}[c][b]
                } \right\}_{\subindexFirstAPBN} ;
                \sim ;
                \substack{
                    2 \sum\limits_{c \in \totalsetfirst[-]} \frac{1}{\eigvalfirst{c}} \Var{H}[0][c] \Var{H}[c][0] \\
                    + \npsim'(\Var{H}[0][0]; V_{0}(\Var{H}, W))
                } ;
                \left\{ \substack{
                    \Var{W}[a][b] \\
                    + 2 \sum\limits_{c \in \totalsetfirst[+]} \frac{1}{\eigvalfirst{b} - \eigvalfirst{c}} \Var{H}[0][c] \Var{H}[c][b] \\
                    - 2 \frac{1}{\eigvalfirst{b}} \npsim(-\Var{H}[0][0]) \Var{H}[0][b]
                } \right\}_{\subindexFirstBN} ;
                \sim ; 
                \sim ; 
                \left\{ \substack{
                    \Var{W}[a][b] \\
                    + 2 \sum\limits_{c \in \totalsetfirst[+]} \frac{-\eigvalfirst{c} }{(\eigvalfirst{c} - \eigvalfirst{a}) (\eigvalfirst{c} - \eigvalfirst{b})} \Var{H}[a][c] \Var{H}[c][b]
                } \right\}_{\subindexFirstANBN}
            }.
        } \nonumber
    \end{align}
    where $V_0(H, W)$ is defined in~\eqref{eq:psd:V}. $\npsim'(\Var{H}[0][0]; V_0(\Var{H}, W))$ in the $\MultiIndexFirst{2}{2}$ block is calculated by~\eqref{eq:psd:nsd-first-dd-nondiagonal-Z}, since $\Var{H}[0][0]$ is not diagonal.
    For a non-diagonal $Z \in \Sym{n}$: Pick $Q \in \calO^n(Z)$. Denote $\Vartilde{Z} := Q\tran Z Q$ diagonal,  $\Vartilde{H} := Q\tran H Q, \ \Vartilde{W} := Q\tran W Q$:
    \begin{align}
        \label{eq:psd:nsd-second-dd-nondiagonal-Z}
        \npsim''(Z; H, W) = Q \npsim''(\Vartilde{Z}; \Vartilde{H}, \Vartilde{W}) Q\tran. 
    \end{align}
\end{theorem}
\begin{proof}
    Since $\ppsim(Z) + \npsim(Z) = Z$, we get $\ppsim'(Z; H) + \npsim'(Z; H) = H$ and 
    \begin{align*}
        \ppsim''(Z; H, W) + \npsim''(Z; H, W) = W.
    \end{align*}
    Then, for $\npsim''(Z; H, W)$'s $\MultiIndexFirst{2}{2}$ block:
    \begin{align*}
        & \npsim''(Z; H, W)_{\MultiIndexFirst{2}{2}} \\
        = & \Var{W}[0][0] -2 \sum\limits_{c \in \totalsetfirst[+]} \frac{1}{\eigvalfirst{c}} \Var{H}[0][c] \Var{H}[c][0] 
        - \ppsim'(\Var{H}[0][0]; V_0(\Var{H}, W)) \\
        = & \Var{W}[0][0]-2 \sum\limits_{c \in \totalsetfirst[+]} \frac{1}{\eigvalfirst{c}} \Var{H}[0][c] \Var{H}[c][0]
        + \npsim'(\Var{H}[0][0]; V_0(\Var{H}, W)) - V_0(\Var{H}, W) \\
        = & \Var{W}[0][0]-2 \sum\limits_{c \in \totalsetfirst[+]} \frac{1}{\eigvalfirst{c}} \Var{H}[0][c] \Var{H}[c][0]
        + \npsim'(\Var{H}[0][0]; V_0(\Var{H}, W)) - \Var{W}[0][0] \\
        & + 2 \sum\limits_{c \in \totalsetfirst[+]} \frac{1}{\eigvalfirst{c}} \Var{H}[0][c] \Var{H}[c][0]
        + 2 \sum\limits_{c \in \totalsetfirst[-]} \frac{1}{\eigvalfirst{c}} \Var{H}[0][c] \Var{H}[c][0] \\
        = & 2 \sum\limits_{c \in \totalsetfirst[-]} \frac{1}{\eigvalfirst{c}} \Var{H}[0][c] \Var{H}[c][0] 
                + \npsim'(\Var{H}[0][0]; V_0(\Var{H}, W)),
    \end{align*}
    where we use~\eqref{eq:psd:V} and the fact that $\ppsim'(\Var{H}[0][0], V_0) + \npsim'(\Var{H}[0][0], V_0) = V_0$. The other blocks in $\npsim''(Z; H, W)$ can be derived from~\eqref{eq:psd:second-dd-diagonal-Z} and simple calculation.
\end{proof}

%% file: sections/second_order_analysis.tex

\section{Local Second-Order Limit Dynamics}
\label{sec:soa}

As shown in~\cite{kang25arxiv-admm}, ADMM for SDPs converges locally at a linear rate when the iterates converge to a nonsingular KKT point $\Zsc$ whose primal--dual components satisfy strict complementarity. In this section, we study ADMM's finer dynamical behavior near an arbitrary, possibly singular KKT point $\Zbar$. In contrast to~\cite{kang25arxiv-admm}, we do \emph{not} assume that the iterates converge to $\Zbar$, which allows us to shift the analysis from a {pointwise, asymptotic} paradigm to a {region-wise, transient} one. We will show that, in this regime, the local dynamics can be effectively described by a \emph{second-order limit map}~$\mdvZmap$.

In \S\ref{sec:soa:ass}, we state the standing assumptions used throughout the paper. In \S\ref{sec:soa:expansion}, we expand the one-step ADMM update~\eqref{eq:intro:one-step-admm} up to second order around $\Zbar$, leveraging the expression for $\ppsim''(Z; H, W)$ in Theorem~\ref{thm:psd:second-dd}. In \S\ref{sec:soa:coneC}, we examine the geometry of $\coneC$, a closed convex cone along which the first-order updates vanish, and discuss its relationship with $\coneT$, the tangent cone to the set of KKT points at $\Zbar$.
In \S\ref{sec:soa:mdvZ}, we show that, under the local second-order expansion model, for every first-order direction $H \in \coneC$, the limit of second-order drifting, denoted as $\mdvZ[\Zbar; H]$, 
exists and need not vanish. This nonzero second-order effect motivates the central object of the paper: the \emph{second-order limit map} $\mdvZmap:\coneC \mapsto \Sym{n}$, viewed as a vector field. Specifically, at the points $Z$ with $Z - \Zbar \in \coneC$ and $\normF{Z - \Zbar} \rightarrow 0$, we associate the second-order displacement $\frac{1}{2} \mdvZ[\Zbar; Z - \Zbar]$. The definition of the corresponding \emph{second-order limit dynamics} then follows immediately. As a local surrogate for the nonlinear dynamics~\eqref{eq:intro:one-step-admm} near $\Zbar$, it captures the limiting behavior of~\eqref{eq:intro:one-step-admm}.

\subsection{Assumptions}
\label{sec:soa:ass}

\begin{assumption}
    \label{ass:soa:sc}
    Two assumptions are made throughout the paper:
    \begin{enumerate}
        \item The linear operator $\Asdp: \Sym{n} \mapsto \Real{m}$ is surjective.
        \item There exists a KKT point satisfying strict complementarity, \ie $\exists (\Xsc, \ysc, \Ssc)$ satisfying~\eqref{eq:intro:kkt}, s.t. $\rank{\Xsc} + \rank{\Ssc} = n$. Equivalently, $\Zsc = \Xsc - \sigma \Ssc$ is nonsingular.
    \end{enumerate}
\end{assumption}
The surjectivity of $\Asdp$ in Assumption~\ref{ass:soa:sc} guarantees that $\Asdp\AsdpT$ is invertible, which in turn ensures that the ADMM iterations~\eqref{eq:intro:admm-three-step} and~\eqref{eq:intro:one-step-admm} are well defined.
Note that Assumption~\ref{ass:soa:sc} requires neither a Slater condition nor constraint nondegeneracy.
The requirement that a strictly complementary solution pair exists is mild and standard in the SDP literature, including analyses of interior-point methods (IPMs)~\cite{alizadeh98siopt-aho,luo98anp-superlinear-convergence-path-following-sdp} and augmented Lagrangian methods (ALMs)~\cite{cui16arxiv-superlinear-alm-sdp,liao24neurips-inexact-alm}.
Moreover, many SDPs arising in real-world applications (including instances with multiple KKT points) admit a strictly complementary primal--dual solution pair~\cite{kang25arxiv-admm}.
On the other hand, the existence of a strictly complementary solution pair rules out pathological SDPs whose optimal set has singularity degree greater than one~\cite{sturm00siopt-error-bound-lmi,waki12coa-strange,sremac19siopt-error-bounds-singularity-degree-sdp}; such problem classes can be challenging even for IPMs~\cite{sremac19siopt-error-bounds-singularity-degree-sdp}.

\subsubsection{\titlemath{A $4 \times 4$ matrix block partition under strict complementarity}}
We first give a characterization of $\Xopt$ and $\Sopt$. 
\begin{proposition}[$\Xopt$ and $\Sopt$]
    \label{prop:soa:optimal-sets}
    Under Assumption~\ref{ass:soa:sc} (1), the primal and dual optimal solution sets $\Xopt$ and $\Sopt$ can be expressed as 
    \begin{subequations}
        \label{eq:soa:optimal-sets}
        \begin{align}
            \Xopt & := \{ X \mid \PA (X - \Vartilde{X}) = 0, \ X \in \Symp{n}, \ \inprod{X}{\Ssc} = 0 \}, \\
            \Sopt & := \{ S \mid \PAp (S - C) = 0, \ S \in \Symp{n}, \ \inprod{S}{\Xsc} = 0 \},
        \end{align}
    \end{subequations}
    where $\Vartilde{X}$ can be any constant matrix satisfying $\Asdp \Vartilde{X} = b$.
    Actually, the fixed strictly complementary solution pair $(\Xsc, \Ssc)$ in~\eqref{eq:soa:optimal-sets} can be replaced by any other fixed primal--dual optimal solution pair $(\Varbar{X},\Varbar{S})\in \Xopt \times \Sopt$.
\end{proposition}
\begin{proof}
    Since $\Asdp$ is surjective by Assumption~\ref{ass:soa:sc}, fixing any constant matrix $\Vartilde{X}$ satisfying $\Asdp \Vartilde{X} = b$, we have 
    \begin{align*}
        \Asdp X = b \Longleftrightarrow \Asdp (X - \Vartilde{X}) = 0 \Longleftrightarrow \PA (X - \Vartilde{X}) = 0.
    \end{align*}
    Symmetrically, $\exists y \in \Real{m}$, s.t. $\AsdpT y + S = C$ is equivalent to $\PAp (S - C) = 0$. Thus, from~\eqref{eq:intro:kkt}, the set of primal-dual optimal solution pairs $(X, S)$ can be expressed as 
    \begin{align*}
        \calC := \left\{ 
            (X, S) \mymid \mymatplain{
                \Asdp X = b, \ X \succeq 0, \\ 
                \exists y \in \Real{m},~\text{s.t.}~\AsdpT y + S = C, \ S \succeq 0, \\
                \inprod{X}{S} = 0
            }
         \right\} = \left\{ 
            (X, S) \mymid \mymatplain{
                \PA (X - \Vartilde{X}) = 0, \ X \succeq 0, \\ 
                \PAp (S - C) = 0, \ S \succeq 0, \\
                \inprod{X}{S} = 0
            }
         \right\}.
    \end{align*} 
    Then, by Assumption~\ref{ass:soa:sc} (2), there exists a strictly complementary solution pair $(\Xsc, \Ssc) \in \calC$, \ie 
    \begin{align*}
        \PA (\Xsc - \Vartilde{X}) = 0, \ \Xsc \succeq 0, \ \PAp (\Ssc - C) = 0, \ \Ssc \succeq 0, \ \inprod{\Xsc}{\Ssc} = 0.
    \end{align*}
    Now all we need to prove is 
    \begin{align*}
        \Xopt = \{X \mid \exists S \in \Sym{n},~\text{s.t.}~(X, S) \in \calC \} \quad \text{and} \quad 
        \Sopt = \{S \mid \exists X \in \Sym{n},~\text{s.t.}~(X, S) \in \calC \}.
    \end{align*}
    We shall only prove the primal part $\Xopt = \{X \mid \exists S \in \Sym{n},~\text{s.t.}~(X, S) \in \calC \}$, since the dual part can be proven symmetrically.

    (\romannumeral1) The ``$\subseteq$'' part. Take any $X \in \Xopt$. Then, by the definition of $\Ssc$, $(X, \Ssc) \in \calC$. 

    (\romannumeral2) The ``$\supseteq$'' part. Take any $(X, S) \in \calC$. Then, by definition, 
    \begin{align*}
        & \PA (X - \Vartilde{X}) = \PA (\Xsc - \Vartilde{X}) = 0, \ X, \Xsc \succeq 0, \\
        & \PAp (S - C) = \PAp (\Ssc - C) = 0, \ S, \Ssc \succeq 0. 
    \end{align*}
    We shall prove $\inprod{X}{\Ssc} = 0$. To see this, on the one hand, 
    \begin{align*}
        \inprod{X - \Xsc}{S - \Ssc} = \inprod{\PA (X - \Xsc)}{\PA (S - \Ssc)} + \inprod{\PAp (X - \Xsc)}{\PAp (S - \Ssc)} = 0.
    \end{align*}
    On the other hand, 
    \begin{align*}
        \inprod{X - \Xsc}{S - \Ssc} = \inprod{X}{S} + \inprod{\Xsc}{\Ssc} - \inprod{X}{\Ssc} - \inprod{\Xsc}{S}. 
    \end{align*}
    Combining these two, we get $\inprod{X}{\Ssc} + \inprod{\Xsc}{S} = 0$. Since $X, \Xsc, S, \Ssc \succeq 0$, the only possibility is $\inprod{X}{\Ssc} = 0$ and $\inprod{\Xsc}{S} = 0$. Thus, $X \in \Xopt$.

    Clearly, the strictly complementary pair $(\Xsc, \Ssc)$ can be replaced by any fixed pairs $(\Varbar{X}, \Varbar{S}) \in \calC$.
\end{proof}

Given Proposition~\ref{prop:soa:optimal-sets}, without loss of generality, we can partition $\Sym{n}$ into $2\times 2$ blocks indexed by $(\indexSC{P},\indexSC{D})\times(\indexSC{P},\indexSC{D})$, where $\indexSC{P}\cup \indexSC{D}=\{1,\dots,n\}$ and $\indexSC{P} \cap \indexSC{D} = \emptyset$. Then, up to a change of basis,
\begin{align}
    \label{eq:soa:mat-four-partition}
    X = \MatrixFourSC{X_{\MultiIndexFirstPD{1}{1}} ; 0 ; \sim ; 0}, \quad
    S = \MatrixFourSC{0 ; 0 ; \sim ; S_{\MultiIndexFirstPD{2}{2}}}, \quad
    \forall (X,S)\in \Xopt \times \Sopt,
\end{align}
with $[\Xsc]_{\MultiIndexFirstPD{1}{1}} \succ 0$ and $[\Ssc]_{\MultiIndexFirstPD{2}{2}} \succ 0$.  Notice that the block partition based on strict complementarity is represented by \emph{solid} lines, which distinguishes it from \emph{dashed} lines partitioning the positive-zero-negative eigenvalue structures.

Now fix an optimal solution pair $(\Varbar{X},\Varbar{S})\in \Xopt \times \Sopt$. Without loss of generality, we further assume that both $\Varbar{X}$ and $\Varbar{S}$ are diagonal. Indeed, if this is not the case, then by the $2\times 2$ block structure in~\eqref{eq:soa:mat-four-partition}, there exists an orthonormal matrix $Q$ of the form
\begin{align*}
    Q := \MatrixFourSC{
        Q_{\MultiIndexFirstPD{1}{1}} ; 0 ; \sim ; Q_{\MultiIndexFirstPD{2}{2}}
    },
\end{align*}
which simultaneously diagonalizes $\Varbar{X}$ and $\Varbar{S}$:
\begin{align*}
    Q\tran \Varbar{X} Q = \MatrixFourSC{
        \MatrixFour{\Lambda_\Primal ; 0 ; \sim ; 0} ; 0 ;
        \sim ; 0
    }, \quad
    Q\tran \Varbar{S} Q = \MatrixFourSC{
        0 ; 0 ;
        \sim ; \MatrixFour{0 ; 0 ; \sim ; \Lambda_\Dual}
    }.
\end{align*}
Moreover, for any $(X,S)\in \Xopt \times \Sopt$, the same $2\times 2$ block partition~\eqref{eq:soa:mat-four-partition} is preserved under this congruence transformation, since
\begin{align*}
    Q\tran X Q = \MatrixFourSC{
        {Q_{\MultiIndexFirstPD{1}{1}}\tran X_{\MultiIndexFirstPD{1}{1}} Q_{\MultiIndexFirstPD{1}{1}}} ; 0 ;
        \sim ; 0
    }, \quad
    Q\tran S Q = \MatrixFourSC{
        0 ; 0 ;
        \sim ; {Q_{\MultiIndexFirstPD{2}{2}}\tran S_{\MultiIndexFirstPD{2}{2}} Q_{\MultiIndexFirstPD{2}{2}}}
    }.
\end{align*}
Accordingly, apply the following congruent transformation to the SDP data $\Asdp$ and $C$:
\begin{align*}
    \left\{ A_i \right\}_{i = 1}^{m} \leftarrow \left\{ Q\tran A_i Q \right\}_{i = 1}^{m}, \quad
    C \leftarrow Q\tran C Q.
\end{align*}
Under the transformed data $(\Asdp,b,C)$, the optimal sets $(\Xopt,\Sopt)$ still satisfy the $2\times 2$ block partition~\eqref{eq:soa:mat-four-partition}, while the chosen pair $(\Varbar{X},\Varbar{S})$ becomes diagonal. 
For $\Zbar := \Varbar{X} - \sigma \Varbar{S}$, we further assume that it satisfies the first-level description in \S\ref{sec:psd}. Under this assumption, the $2\times 2$ partition in~\eqref{eq:soa:mat-four-partition} refines the $3\times 3$ block partition in \S\ref{sec:psd} into a $4\times 4$ one as follows. Define $\indexZeroP := \indexSC{P} \backslash \indexfirst{+}$ and $\indexZeroD := \indexSC{D} \backslash \indexfirst{-}$. Then, for any $H \in \Sym{n}$,
\begin{align*}
    H = \MatrixSixteenSC{
        H_{\MultiIndexFirstSC{1}{1}} ; H_{\MultiIndexFirstSC{1}{2}} ; H_{\MultiIndexFirstSC{1}{3}} ; H_{\MultiIndexFirstSC{1}{4}} ;
        \sim ; H_{\MultiIndexFirstSC{2}{2}} ; H_{\MultiIndexFirstSC{2}{3}} ; H_{\MultiIndexFirstSC{2}{4}} ;
        \sim ; \sim ; H_{\MultiIndexFirstSC{3}{3}} ; H_{\MultiIndexFirstSC{3}{4}} ;
        \sim ; \sim ; \sim ; H_{\MultiIndexFirstSC{4}{4}}
    }.
\end{align*}
If we further assume that $H$ satisfies the second-level description in \S\ref{sec:psd}, then for any $W \in \Sym{n}$ the $\MultiIndexFirst{2}{2}$ block of $W$ can be expressed---after a change of basis since $\Var{H}[0][0]$ is not diagonal---as
\begin{align*}
    \Var{W}[0][0] = \Qfirst{0} \MatrixSixteenSC{
        \Varhat{W}_{\MultiIndexSecondSC{1}{1}} ; \Varhat{W}_{\MultiIndexSecondSC{1}{2}} ; \Varhat{W}_{\MultiIndexSecondSC{1}{3}} ; \Varhat{W}_{\MultiIndexSecondSC{1}{4}} ;
        \sim ; \Varhat{W}_{\MultiIndexSecondSC{2}{2}} ; \Varhat{W}_{\MultiIndexSecondSC{2}{3}} ; \Varhat{W}_{\MultiIndexSecondSC{2}{4}} ;
        \sim ; \sim ; \Varhat{W}_{\MultiIndexSecondSC{3}{3}} ; \Varhat{W}_{\MultiIndexSecondSC{3}{4}} ;
        \sim ; \sim ; \sim ; \Varhat{W}_{\MultiIndexSecondSC{4}{4}}
    } (\Qfirst{0})\tran,
\end{align*}
where $(\indexSecondSC{P}, \indexSecondSC{D})$ is the primal--dual block partition for $\Var{H}[0][0]$ with $\indexSecondSC{P} \cup \indexSecondSC{D} = \indexfirst{0}$. $\Varhat{W} := (\Qfirst{0})\tran \Var{W}[0][0] \Qfirst{0}$, $\indexSecondZeroP := \indexSecondSC{P} \backslash \indexsecond{0}{+}$, and $\indexSecondZeroD := \indexSecondSC{D} \backslash \indexsecond{0}{-}$.

\begin{remark}
    \label{rem:soa:sc-partition}
    One may have already noticed the role played by the existence of a strictly complementary solution pair in enabling the $4\times 4$ block partition. Without strict complementarity, the block $\MultiIndexFirst{2}{2}$ can no longer be cleanly subdivided into four pieces, which would significantly increase the complexity of the analysis.
\end{remark}

\subsection{Second-Order Local Expansion}
\label{sec:soa:expansion}

Let $\Zbar \in \Zopt$ be diagonal and satisfy the first-level description in \S\ref{sec:psd}. Suppose that the ADMM iterate $\Vark{Z}$ admits the following expansion in a neighborhood of $\Zbar$:
\begin{align}
    \label{eq:soa:Vark-Z}
    \Vark{Z} = \Zbar + t \Vark{H} + \frac{t^2}{2} \Vark{W} + o(t^2),
\end{align}
where $t \downarrow 0$ is a scale parameter.
Unlike~\cite{kang25arxiv-admm}, this local expansion does not require $\Zbar$ to be the eventual limit point of $\Vark{Z}$. Consequently, our local framework is more flexible, shifting from a \emph{pointwise, asymptotic} perspective to a \emph{region-wise, transient} one.

Recall the one-step ADMM update~\eqref{eq:intro:one-step-admm} and rewrite it in finite-difference form as $\Varkpo{Z}-\Vark{Z}=\delta(\Vark{Z})$, where the residual map $\delta(\cdot):\Sym{n}\mapsto\Sym{n}$ is defined by
\begin{align}
    \label{eq:soa:residual}
    \delta(Z) := -\PA (\ppsim(Z) - \Vartilde{X}) + \PAp(\ppsim(-Z) - \sigma C)
    = -\PA (\ppsim(Z) - \Vartilde{X}) - \PAp(\npsim(Z) + \sigma C).
\end{align}
Clearly, $\delta(\Zbar)=0$.
Since both $\ppsim(\cdot)$ and $\npsim(\cdot)$ are (parabolically) second-order directionally differentiable around $\Zbar$, the mapping $\delta(\cdot)$ is also (parabolically) second-order directionally differentiable at $\Zbar$, with
\begin{subequations}
    \begin{align}
        \delta'(\Zbar; H) & = - \PA \ppsim'(\Zbar; H) - \PAp \npsim'(\Zbar; H), \label{eq:soa:delta-first-dd} \\
        \delta''(\Zbar; H, W) & = - \PA \ppsim''(\Zbar; H, W) - \PAp \npsim''(\Zbar; H, W). \label{eq:soa:delta-second-dd}
    \end{align}
\end{subequations}
Expanding $\Varkpo{Z}$ up to second order then yields
\begin{align}
    \label{eq:soa:second-order-expansion}
    \Varkpo{Z} =\ & \Vark{Z} + \delta(\Zbar) + t\,\delta'(\Zbar; \Vark{H}) + \frac{t^2}{2}\,\delta''(\Zbar; \Vark{H}, \Vark{W}) + o(t^2) \nonumber \\
    =\ & \Zbar + t\,\underbrace{\Big\{ \Vark{H} - \PA \ppsim'(\Zbar; \Vark{H}) - \PAp \npsim'(\Zbar; \Vark{H}) \Big\}}_{=: \Varkpo{H}} \nonumber \\
    &\quad + \frac{t^2}{2}\,\underbrace{\Big\{ \Vark{W} - \PA \ppsim''(\Zbar; \Vark{H}, \Vark{W}) - \PAp \npsim''(\Zbar; \Vark{H}, \Vark{W}) \Big\}}_{=: \Varkpo{W}} + o(t^2).
\end{align}

This expansion motivates the following definitions of local first- and second-order dynamics.

\begin{definition}[Local first- and second-order dynamics]
    \label{def:soa:fod-sod}
    Around $\Zbar$, define the local first-order dynamics as
    \begin{align}
        \label{eq:soa:fod}
        \Varkpo{H} = (\Id + \delta'(\Zbar; \cdot))(\Vark{H}) = \Vark{H} - \PA \ppsim'(\Zbar; \Vark{H}) - \PAp \npsim'(\Zbar; \Vark{H}),
    \end{align}
    and the local second-order dynamics as
    \begin{align}
        \label{eq:soa:sod}
        \Varkpo{W} = (\Id + \delta''(\Zbar; \Vark{H}, \cdot))(\Vark{W}) = \Vark{W} - \PA \ppsim''(\Zbar; \Vark{H}, \Vark{W}) - \PAp \npsim''(\Zbar; \Vark{H}, \Vark{W}).
    \end{align}
    A notable feature of the second-order dynamics is that $\Varkpo{W}$ depends on both $\Vark{H}$ and $\Vark{W}$.
\end{definition}


\subsection{\titlemath{$\coneC$}: the Cone where First-Order Updates Vanish}
\label{sec:soa:coneC}

In this section, we analyze the local first-order dynamics~\eqref{eq:soa:fod}. We shall first see that $\Varkpo{H} = (\Id + \delta'(\Zbar; \cdot))(\Vark{H})$ will converge to one of the fixed points of $\Id + \delta'(\Zbar; \cdot)$. Recall that an operator $\calT: \Sym{n} \mapsto \Sym{n}$ is firmly nonexpansive on $(\Sym{n}, \inprod{\cdot}{\cdot})$, if 
\begin{align*}
    \normF{\calT(H) - \calT(G)}^2 + \normF{(\Id - \calT)(H) - (\Id - \calT)(G)}^2 \le \normF{H - G}^2, \quad \forall H, G \in \Sym{n}.
\end{align*}
\begin{lemma}[Convergent first-order dynamics]
    \label{lem:soa:convergent-fod}
    Under Assumption~\ref{ass:soa:sc}, $\Id + \delta'(\Zbar; \cdot)$ is firmly nonexpansive on $(\Sym{n}, \inprod{\cdot}{\cdot})$. Moreover, for any $H^{(0)}$, $\Varkpo{H} = (\Id + \delta'(\Zbar; \cdot))(\Vark{H})$ converges to a fixed point of $\Id + \delta'(\Zbar; \cdot)$.
\end{lemma}
\begin{proof}
    For ease of notation, for any $H \in \Sym{n}$, we denote the mappings $(\Id + \delta'(\Zbar; \cdot))(H)$ as $\calT(H)$, $\ppsim'(\Zbar; H)$ as $\Omega(H)$, and $\npsim'(\Zbar; H)$ as $\Omega^\perp(H)$. 

    (\romannumeral1) We first prove the firmly nonexpansiveness of $\calT$. We have $\forall H, G \in \Sn$, 
    \begin{align*}
        & \normF{\calT(H) - \calT(G)}^2 + \normF{(\Id - \calT)(H) - (\Id - \calT)(G)}^2 \\
        = & \normF{\PAp [\Omega(H) - \Omega(G)]}^2 + \normF{\PA [\Omega^\perp(H) - \Omega^\perp(G)]}^2
        + \normF{\PA [\Omega(H) - \Omega(G)]}^2 + \normF{\PAp [\Omega^\perp(H) - \Omega^\perp(G)]}^2 \\
        = & \normF{\Omega(H) - \Omega(G)}^2 + \normF{\Omega^\perp(H) - \Omega^\perp(G)}^2 \\
        = & \normF{H - G}^2 - 2 \inprod{
            \Omega(H) - \Omega(G)
        }{
            \Omega^\perp(H) - \Omega^\perp(G)
        }.
    \end{align*}
    All we need to show is $\inprod{\Omega(H) - \Omega(G)}{\Omega^\perp(H) - \Omega^\perp(G)} \ge 0$. Denote $U := \Omega(H) - \Omega(G)$ and $V := \Omega^\perp(H) - \Omega^\perp(G)$. Only consider the upper triangular parts of the symmetric matrix:
    \begin{itemize}
        \item If (1) $a \in \totalsetfirst[+], b \in \totalsetfirst[+]$; or (2)  $a \in \totalsetfirst[+], b = 0$; or (3) $a = 0, b \in \totalsetfirst[-]$; or (4) $a \in \totalsetfirst[-], b \in \totalsetfirst[-]$: 
        \begin{align*}
            \inprod{\Var{U}[a][b]}{\Var{V}[a][b]} = 0
        \end{align*}

        \item (5) $a = 0, b = 0$:
        \begin{align*}
            & \inprod{\Var{U}[0][0]}{\Var{V}[0][0]} 
            = \inprod{
                \ppsim(\Var{H}[0][0]) - \ppsim(\Var{G}[0][0])
            }{
                \npsim(\Var{H}[0][0]) - \npsim(\Var{G}[0][0])
            } \\
            = & \inprod{
                \ppsim(\Var{H}[0][0])
            }{
                - \npsim(\Var{G}[0][0])
            } + 
            \inprod{
                \ppsim(\Var{G}[0][0])
            }{
                - \npsim(\Var{H}[0][0])
            } \ge 0
        \end{align*}

        \item (6) $a \in \totalsetfirst[+], b \in \totalsetfirst[-]$: 
        \begin{align*}
            \inprod{\Var{U}[a][b]}{\Var{V}[a][b]} =
            \inprod{
                \frac{\eigvalfirst{a}}{\eigvalfirst{a} - \eigvalfirst{b}} \Var{(H-G)}[a][b]
            }{
                \frac{-\eigvalfirst{b}}{\eigvalfirst{a} - \eigvalfirst{b}} \Var{(H-G)}[a][b]
            } \ge 0
        \end{align*}
    \end{itemize}

    (\romannumeral2) We second show $\Fix(\calT) \ne \emptyset$. Since $\delta'(\Zbar; 0) = 0$, we have $0 \in \Fix(\calT)$. Therefore, by~\cite[Example 5.18]{bauschke17springer-convex-analysis-hilbert-spaces}, $\Varkpo{H} = \calT(\Vark{H})$ converges to one of the points in $\Fix(\calT)$.
\end{proof}

Denote $\coneC$ as $\Fix(\Id + \delta'(\Zbar; \cdot))$ (or equivalently, $\ker(\delta'(\Zbar; \cdot))$):
\begin{align}
    \label{eq:soa:coneC}
    \coneC := \left\{ H \in \Sym{n} \mymid (\Id + \delta'(\Zbar; \cdot))(H) = H \right\} = \left\{ H \in \Sym{n} \mymid \delta'(\Zbar; H) = 0 \right\}.
\end{align}
We need an important lemma before starting to characterize $\coneC$'s structures.
\begin{lemma}
    \label{lem:soa:X-inrpod-S}
    For $G \in \Sn$, under Assumption~\ref{ass:soa:sc}:
    \begin{enumerate}
        \item If $\normF{\PA G} \le \epsilon$ and $\inprod{G}{\Varbar{S}} = 0$, then 
        \begin{align*}
            \abs{\inprod{G}{S}} \le \normF{S - \Varbar{S}} \cdot \epsilon, \quad \forall S \in \Sopt.
        \end{align*}
        \item If $\normF{\PAp G} \le \epsilon$ and $\inprod{G}{\Varbar{X}} = 0$, then 
        \begin{align*}
            \abs{\inprod{G}{X}} \le \normF{X - \Varbar{X}} \cdot \epsilon, \quad \forall X \in \Xopt.
        \end{align*}
    \end{enumerate}
\end{lemma}
\begin{proof}
    (1) Since $\Varbar{S}, S \in \Sopt$, we have $\PAp \Varbar{S} = \PAp S = \PAp C$ from Proposition~\ref{prop:soa:optimal-sets}. Thus, 
    \begin{align*}
        & \abs{\inprod{G}{S}} = \abs{\inprod{\PA G}{\PA S} + \inprod{\PAp G}{\PAp S}} 
        =  \abs{\inprod{\PA G}{\PA S} + \inprod{\PAp G}{\PAp \Varbar{S}}} \\
        = &  \abs{\inprod{\PA G}{\PA \Varbar{S}} + \inprod{\PAp G}{\PAp \Varbar{S}} + \inprod{\PA G}{\PA (\Varbar{S} - S)}} = \abs{\inprod{G}{\Varbar{S}} + \inprod{\PA G}{\PA (\Varbar{S} - S)}} \\
        = & \abs{\inprod{\PA G}{\PA (\Varbar{S} - S)}} \\
        \le & \normF{\PA G} \normF{\PA (\Varbar{S} - S)} \le \epsilon \normF{\PA (\Varbar{S} - S)}
    \end{align*}
    On the other hand, $\Varbar{S} - S = \PAp (\Varbar{S} - S) + \PA (\Varbar{S} - S) = \PA (\Varbar{S} - S)$.

    (2) By primal-dual symmetry.
\end{proof}

\begin{proposition}[Structure of $\coneC$]
    \label{prop:soa:coneC}
    Under Assumption~\ref{ass:soa:sc}:
    \begin{enumerate}
        \item $\coneC$ is a nonempty, closed and convex cone. 
        \item $\coneC = \coneCX + \coneCS$, where  
        \begin{subequations}
            \label{eq:soa:coneC-XS}
            \begin{align}
                & \coneCX := \left\{ 
                    H = \MatrixSixteenSC{
                            H_{\MultiIndexFirstSC{1}{1}} ; H_{\MultiIndexFirstSC{1}{2}} ; H_{\MultiIndexFirstSC{1}{3}} ; 0 ;
                            \sim ; H_{\MultiIndexFirstSC{2}{2}} ; 0 ; 0 ;
                            \sim ; \sim; 0 ; 0 ;
                            \sim ; \sim ; \sim ; 0
                        } \mymid \mymatplain{
                            \PA H = 0, \\
                            H_{\MultiIndexFirstSC{2}{2}} \succeq 0
                        }
                \right\}, \\
                & \coneCS := \left\{ 
                    H = \MatrixSixteenSC{
                            0 ; 0 ; 0 ; 0 ;
                            \sim ; 0 ; 0 ; H_{\MultiIndexFirstSC{2}{4}} ;
                            \sim ; \sim; H_{\MultiIndexFirstSC{3}{3}} ; H_{\MultiIndexFirstSC{2}{4}} ;
                            \sim ; \sim ; \sim ; H_{\MultiIndexFirstSC{4}{4}}
                        } \mymid \mymatplain{
                            \PAp H = 0, \\
                            H_{\MultiIndexFirstSC{3}{3}} \preceq 0
                        }
                \right\}.
            \end{align}
        \end{subequations}
    \end{enumerate}
\end{proposition}
\begin{proof}
    (1) is directly from Lemma~\ref{lem:soa:convergent-fod} and~\cite[Proposition 4.23]{bauschke17springer-convex-analysis-hilbert-spaces}.

    (2) For ease of notation, denote $\ppsim'(\Zbar; \cdot)$ (resp. $\npsim'(\Zbar; \cdot)$) as $\Omega(\cdot)$ (resp. $\Omega^\perp(\cdot)$). Then from~\eqref{eq:soa:coneC}, $H \in \coneC$ if and only if $\PA \Omega(H) = 0$ and $\PAp \Omega^\perp(H) = 0$. 

    (\romannumeral1) We first prove $H_{\MultiIndexFirstSC{1}{4}} = 0, \forall H \in \coneC$. Notice that 
    \begin{align*}
        \inprod{\Omega(H)}{\Omega^\perp(H)} = \inprod{\PA \Omega(H)}{\PA \Omega^\perp(H)} + \inprod{\PAp \Omega(H)}{\PAp \Omega^\perp(H)} = 0.
    \end{align*}
    On the other hand, 
    \begin{align*}
        \inprod{\Omega(H)}{\Omega^\perp(H)} = 2 \sum_{a \in \totalsetfirst[+], b \in \totalsetfirst[-]}
        \frac{\eigvalfirst{a} \cdot (-\eigvalfirst{b})}{\eigvalfirst{a} - \eigvalfirst{b}} \normF{\Var{H}[a][b]}^2.
    \end{align*}
    Thus, $\Var{H}[a][b] = 0, \forall a \in \totalsetfirst[+], b \in \totalsetfirst[-]$, \ie $\Var{H}[+][-] = 0$.

    (\romannumeral2) We second prove $H_{\MultiIndexFirstSC{2}{3}} = 0, \forall H \in \coneC$. Since $\Omega(H)_{\MultiIndexFirst{3}{3}} = 0$ from Theorem~\ref{thm:psd:first-dd}, we get $\inprod{\Omega(H)}{\Varbar{S}} = 0$. 
    Together with $\PA \Omega(H) = 0$ and $\Ssc \in \Sopt$, we have $\inprod{\Omega(H)}{\Ssc} = 0$ from Lemma~\ref{lem:soa:X-inrpod-S}. On the other hand, 
    \begin{align*}
        \simpleadjustbox{
        \Omega(\Var{H}) = \MatrixSixteenSC{
            \Var{H}[+][+] ; 
            \VarFirstZeroP{\Var{H}}{r}[+] ; 
            \VarFirstZeroD{\Var{H}}{r}[+] ; 
            \Var{H}[+][-] ;
            \sim ; 
            [\ppsim(\Var{H}[0][0])]_{\indexSecondSC{P} \indexSecondSC{P}} ; 
            [\ppsim(\Var{H}[0][0])]_{\indexSecondSC{P} \indexSecondSC{D}} ; 
            0 ;
            \sim ;
            \sim ; 
            [\ppsim(\Var{H}[0][0])]_{\indexSecondSC{D} \indexSecondSC{D}} ;
            0 ; 
            \sim ;
            \sim ; 
            \sim ; 
            0 
        }, \quad 
        \Ssc = \MatrixSixteenSC{
            0 ; 0 ; 0 ; 0 ; 
            \sim ; 0 ; 0 ; 0 ;
            \sim ; \sim ; \Var{[\Ssc]}_{\indexZeroD \indexZeroD} ; \Var{[\Ssc]}_{\indexZeroD \indexfirst{-}} ;
            \sim ; \sim ; \sim ; \Var{[\Ssc]}_{\indexfirst{-} \indexfirst{-}}
        }
        },
    \end{align*}
    we have $\inprod{\Omega(H)}{\Ssc} = \inprod{
        [\ppsim(\Var{H}[0][0])]_{\indexSecondSC{D} \indexSecondSC{D}}
    }{
        \Var{[\Ssc]}_{\indexZeroD \indexZeroD}
    } = 0$. Since $\Var{[\Ssc]}_{\indexZeroD \indexZeroD} \succ 0$ and $[\ppsim(\Var{H}[0][0])]_{\indexSecondSC{D} \indexSecondSC{D}} \succeq 0$, we get $[\ppsim(\Var{H}[0][0])]_{\indexSecondSC{D} \indexSecondSC{D}} = 0$. This further implies $[\ppsim(\Var{H}[0][0])]_{\indexSecondSC{P} \indexSecondSC{D}} = 0$. Symmetrically, $[\npsim(\Var{H}[0][0])]_{\indexSecondSC{P} \indexSecondSC{D}} = 0$. Thus, 
    \begin{align*}
        H_{\MultiIndexFirstSC{2}{3}} = 
        [\Var{H}[0][0]]_{\indexSecondSC{P} \indexSecondSC{D}} 
        = [\ppsim(\Var{H}[0][0])]_{\indexSecondSC{P} \indexSecondSC{D}}
        + [\npsim(\Var{H}[0][0])]_{\indexSecondSC{P} \indexSecondSC{D}}
        = 0.
    \end{align*}

    (\romannumeral3) From (\romannumeral1) and (\romannumeral2), $\forall H \in \coneC$, it should be of the following form:
    \begin{align*}
        H = \underbrace{
            \MatrixSixteenSC{
                H_{\MultiIndexFirstSC{1}{1}} ; H_{\MultiIndexFirstSC{1}{2}} ; H_{\MultiIndexFirstSC{1}{3}} ; 0 ;
                \sim ; H_{\MultiIndexFirstSC{2}{2}} ; 0 ; 0 ;
                \sim ; \sim; 0 ; 0 ;
                \sim ; \sim ; \sim ; 0
            }
        }_{=: U} + \underbrace{
            \MatrixSixteenSC{
                0 ; 0 ; 0 ; 0 ;
                \sim ; 0 ; 0 ; H_{\MultiIndexFirstSC{2}{4}} ;
                \sim ; \sim; H_{\MultiIndexFirstSC{3}{3}} ; H_{\MultiIndexFirstSC{2}{4}} ;
                \sim ; \sim ; \sim ; H_{\MultiIndexFirstSC{4}{4}}
            }
        }_{=: V}, \quad \text{with}~U_{\MultiIndexFirstSC{2}{2}} \succeq 0, V_{\MultiIndexFirstSC{3}{3}} \preceq 0
    \end{align*}
    Moreover, 
    \begin{align*}
        \PA \Omega(H) = \PA U = 0, \quad \PAp \Omega^\perp(H) = \PAp V = 0.
    \end{align*}
    Thus, $U \in \coneCX, V \in \coneCS$. This proves the ``$\subseteq$'' part. For the ``$\supseteq$'' part: take any $U \in \coneCX, V \in \coneCS$. Then, 
    \begin{align*}
        \PA \Omega(U + V) = \PA U = 0, \quad \PAp \Omega^\perp(U+V) = \PAp V = 0,
    \end{align*}
    which closes the proof.
\end{proof}

\subsubsection{\titlemath{Relationships between $\coneC$ and $\coneT$}} 
The cone $\coneC$ consists of directions $H$ along which $\delta(\Varbar{Z}+tH)$ (the backward error~\cite{sturm00siopt-error-bound-lmi}, \ie the KKT residual) vanishes to first order, whereas $\coneT$ consists of directions $H$ along which $\dist(\Varbar{Z}+tH,\Zopt)$ (the forward error, \ie the distance to the optimal set) vanishes to first order. As we show below, these two cones are closely related.
\begin{proposition}[Structure of $\coneT$]
    \label{prop:soa:coneT}
    Under Assumption~\ref{ass:soa:sc}, 
    \begin{enumerate}
        \item $\coneT = \coneTX - \coneTS$, where 
        \begin{subequations}
            \label{eq:soa:coneT-XS}
            \begin{align}
                \coneTX = & \left\{ 
                    H = \MatrixFourSC{
                        \MatrixFour{
                            H_{\MultiIndexFirstSC{1}{1}} ; H_{\MultiIndexFirstSC{1}{2}} ;
                            \sim ; H_{\MultiIndexFirstSC{2}{2}} 
                        } ; 0 ; 
                        \sim ; 0 
                    } \mymid \mymatplain{\PA H = 0, \\ H_{\MultiIndexFirstSC{2}{2} \succeq 0}}
                \right\}, \\
                \coneTS = & \left\{ 
                    H = \MatrixFourSC{
                        0 ; 0 ; 
                        \sim ; \MatrixFour{
                            H_{\MultiIndexFirstSC{3}{3}} ; H_{\MultiIndexFirstSC{3}{4}} ;
                            \sim ; H_{\MultiIndexFirstSC{4}{4}} 
                        } 
                    } \mymid \mymatplain{\PAp H = 0, \\ H_{\MultiIndexFirstSC{3}{3} \preceq 0}}
                 \right\}.
            \end{align}
        \end{subequations}

        \item $\coneT = \coneC \cap \{H \in \Sym{n} \mid H_{\MultiIndexFirstSC{1}{3}} = 0, H_{\MultiIndexFirstSC{2}{4}} = 0\}$.
    \end{enumerate}
\end{proposition}
\begin{proof}
    (1) We first calculate $\calT_{\Xopt}(\Varbar{X})$. Regularize $\Xopt$ to its affine hull: picking $(\Xsc, \Ssc)$ as a maximal-rank primal--dual optimal solution pair, 
    \begin{align*}
        \Xopt = & \left\{ X \mymid X \in \Symp{n} \cap \left\{ X \mymid \PA X = \PA \Vartilde{X}, \inprod{X}{\Ssc} = 0 \right\} \right\} \\
        = & \underbrace{
            \left\{ X \mymid X \in \Symp{n}, X_{\MultiIndexFirstPD{1}{2}} = 0, X_{\MultiIndexFirstPD{2}{2}} = 0  \right\}
        }_{=: \calC_1} \cap \underbrace{
            \left\{ X \mymid \PA X = \PA \Vartilde{X} \right\}
        }_{=: \calC_2}.
    \end{align*}
    Since $\Xsc \in \ri(\calC_1) \cap \ri(\calC_2)$, by~\cite[Theorem 6.42]{rockafellar98book-variational-analysis},
    \begin{align*}
        & \calT_{\Xopt}(\Varbar{X}) = \calT_{\calC_1}(\Varbar{X}) \cap \calT_{\calC_2}(\Varbar{X}) \\
        = & \left\{ 
            H \mymid \VarSC{H}[P][P] \in \calT_{\Symp{\sizeof{\indexSC{P}}}}(\VarSC{\bar{X}}[P][P]), \ 
            \calP H = 0, \ \VarSC{H}[P][D] = \VarSC{H}[D][D] = 0
         \right\} \\
         = & \left\{ 
            H \mymid \VarFirstZeroP{H}{2} \succeq 0, \ 
            \calP H = 0, \ \VarSC{H}[P][D] = \VarSC{H}[D][D] = 0
         \right\}.
    \end{align*}
    Symmetrically, $\coneTS$ is of the form in~\eqref{eq:soa:coneT-XS}.
    We notice that $\calT_{\Xopt}(\Varbar{X}) \cap \calT_{\Sopt}(\Varbar{S}) = \{0\}$. Thus, via~\cite[Corollary 4.8 (v)]{bauschke21jmaa-angles-convex-sets} and~\cite[Exercise 6.44]{rockafellar98book-variational-analysis}, 
    \begin{align*}
        \calT_{\Zopt}(\Varbar{Z}) = \closure (\calT_{\Xopt}(\Varbar{X}) - \calT_{\Sopt}(\Varbar{S})) = \calT_{\Xopt}(\Varbar{X}) - \calT_{\Sopt}(\Varbar{S}).
    \end{align*}

    (2) Denote $\{H \in \Sym{n} \mid H_{\MultiIndexFirstSC{1}{3}} = 0, H_{\MultiIndexFirstSC{2}{4}} = 0\}$ as $\calM$. Suppose $H \in \coneC \cap \calM$. Then, from Proposition~\ref{prop:soa:coneC} (2), $H = U + V$, with $\PA U = 0$ and $\PAp V = 0$, with
    \begin{align*}
        U = \MatrixSixteenSC{
                U_{\MultiIndexFirstSC{1}{1}} ; U_{\MultiIndexFirstSC{1}{2}} ; 0 ; 0 ;
                \sim ; U_{\MultiIndexFirstSC{2}{2}} \succeq 0 ; 0 ; 0 ;
                \sim ; \sim; 0 ; 0 ;
                \sim ; \sim ; \sim ; 0
            }, \quad 
        V = \MatrixSixteenSC{
                0 ; 0 ; 0 ; 0 ;
                \sim ; 0 ; 0 ; 0 ;
                \sim ; \sim; V_{\MultiIndexFirstSC{3}{3}} \preceq 0 ; V_{\MultiIndexFirstSC{3}{4}} ;
                \sim ; \sim ; \sim ; V_{\MultiIndexFirstSC{4}{4}}
        }.
    \end{align*}
    Thus, $U \in \coneTX, V \in -\coneTS$. This proves the ``$\subseteq$'' part. For the ``$\supseteq$'' part, take any $U \in \coneTX, V \in -\coneTS$. Then, $U \in \coneCX \cap \calM$ and $V \in \coneCS \cap \calM$. Thus, $U + V \in (\coneCX + \coneCS) \cap \calM = \coneC \cap \calM$. 
\end{proof}

There are several special scenarios when $\coneC = \coneT$. 
\begin{corollary}[Special cases when $\coneC = \coneT$]
    \label{cor:soa:coneC-equal-coneT}
    Under Assumption~\ref{ass:soa:sc}, $\coneC = \coneT$ under any of the following conditions:
    \begin{enumerate}
        \item If $\Varbar{X}$ satisfies primal constraint nondegeneracy and $\Varbar{S}$ satisfies dual constraint nondegeneracy; 
        \item If $(\Varbar{X}, \Varbar{S})$ is a strictly complementary solution pair.
        \item If the linear growth condition holds locally, \ie $\exists \gamma > 0, r > 0$, s.t. $\forall Z \in \mathbb{B}_r(\Zbar)$, $\gamma \dist(Z, \Zopt) \le \normF{\delta(Z)}$.
    \end{enumerate}
\end{corollary}
\begin{proof}
    (1) Under the two-sided nondegeneracy conditions,~\cite[Theorem 5]{kang25arxiv-admm} has already proven $\Fix(\Id + \delta'(\Zbar; \cdot)) = \{0\}$. Thus, $\coneC = \coneT = \{0\}$.

    (2) When $(\Varbar{X}, \Varbar{S})$ is of maximal rank, $H_{\MultiIndexFirstPD{1}{2}} = H_{\MultiIndexFirst{1}{3}} = 0$ for any $H \in \coneC$. Thus, $H_{\MultiIndexFirstSC{1}{3}} = 0, H_{\MultiIndexFirstSC{2}{4}} = 0$ naturally holds. By Proposition~\ref{prop:soa:coneT} (2), $\coneC = \coneT$.

    (3) Prove by contradiction. Suppose $\coneT \subsetneq \coneC$. Then, pick $H \in \coneC \backslash \coneT$. Since $\Zopt$ is closed and convex, there exists a non-zero $V \in \calN_{\Zopt}(\Zbar)$, such that $\Zopt \subseteq \{Z \mid \inprod{V}{Z - \Zbar} \le 0\} =: \calH$. Since $H \notin \coneT$, we have $\inprod{V}{H} > 0$ and 
    \begin{align*}
        \dist(\Zbar + t H, \Zopt) \ge \dist(\Zbar + t H, \calH) = \frac{t \inprod{V}{H}}{\normF{V}}
    \end{align*}
    for all $t > 0$. Thus, by the local linear growth condition, for all $0 < t \le \frac{r}{\normF{H}}$, 
    \begin{align*}
        \normF{\delta(\Zbar + t H)} \ge \gamma \dist(\Zbar + t H, \Zopt) \ge \gamma \cdot \frac{t \inprod{V}{H}}{\normF{V}} \Longrightarrow \frac{\normF{\delta(\Zbar + t H)}}{t} \ge \gamma \cdot \frac{\inprod{V}{H}}{\normF{V}} > 0.
    \end{align*}
    On the other hand, since $H \in \coneC$, $\lim_{t \downarrow 0} \frac{\normF{\delta(\Zbar + t H)}}{t} =0$, which leads to a contradiction.
\end{proof}

It turns out that all three sufficient conditions in Corollary~\ref{cor:soa:coneC-equal-coneT} for $\coneT=\coneC$ are closely tied to local linear convergence of ADMM for SDPs. Indeed, if we additionally assume that $\Zbar$ is the limiting point of the ADMM iterates, then: (1) local linear convergence under two-sided nondegeneracy can be derived from~\cite{chan08siopt-constraint,han18mor-linear}; (2) local linear convergence of ADMM for SDPs under strict complementarity at the \emph{limiting} KKT point has been established recently in~\cite{kang25arxiv-admm}; (3) a local linear growth condition is known to guarantee local linear convergence for a broad class of primal--dual splitting methods in more general nonsmooth convex optimization settings~\cite{han18mor-linear,yuan20jmlr-discerning}. In the convex quadratic SDP setting, the local linear growth condition is also closely related to metric subregularity of the KKT operator~\cite[Theorem~3.2]{cui16arxiv-superlinear-alm-sdp}, which is generally difficult to characterize.

On the other hand, as we will see in \S\ref{sec:mdvZ-kernel}, any direction $H\in\coneC\backslash\coneT$ can lead to a second-order dominant phenomenon. Such directions are easy to construct even for small-scale SDPs involving multiple KKT points, as illustrated by the examples in \S\ref{sec:toy}. Based on these observations, we conjecture that $\coneC = \coneT$ is a necessary condition to establish local linear convergence, given the additional assumption that $\Zbar$ is the final convergent point of one-step ADMM's iterations. 

\subsection{Second-Order Limit Map \titlemath{$\mdvZmap$}}
\label{sec:soa:mdvZ}
In this section, we will develop the core concepts in the paper: the local second-order \emph{limit} map and its induced local second-order \emph{limit} dynamics. 

By Lemma~\ref{lem:soa:convergent-fod}, the local first-order dynamics~\eqref{eq:soa:fod} eventually vanishes and drives the iterates toward $\coneC$ for any initialization $H^{(0)} \in \Sym{n}$. In contrast, the \emph{true} one-step ADMM dynamics~\eqref{eq:intro:one-step-admm} need not vanish. This motivates us to investigate the local second-order dynamics~\eqref{eq:soa:sod} in the regime where the first-order dynamics has stalled.
Fix an arbitrary $\Hbar \in \coneC$. Then, by~\eqref{eq:soa:fod} and Lemma~\ref{lem:soa:convergent-fod}, we have $\Vark{H} \equiv \Hbar$ for all $k \in \bbN$ if $H^{(0)}$ is set to $\Hbar$. Consequently, the local second-order dynamics~\eqref{eq:soa:sod} reduces to
\begin{align}
    \label{eq:soa:sod-Hbar}
    \Varkpo{W} = (\Id + \delta''(\Zbar; \Hbar, \cdot))(\Vark{W})
    = \Vark{W} - \PA \ppsim''(\Zbar; \Hbar, \Vark{W}) - \PAp \npsim''(\Zbar; \Hbar, \Vark{W}).
\end{align}
As we will see later, a fundamental difference between the second-order dynamics~\eqref{eq:soa:sod-Hbar} and the first-order dynamics~\eqref{eq:soa:fod} is that the second-order sequence $\{\Vark{W}\}$ in~\eqref{eq:soa:sod-Hbar} need not converge.

\subsubsection{\titlemath{$\delta''(\Zbar; \Hbar, W)$'s simplification under $\Hbar \in \coneC$}}
Since $\Varbar{H} \in \coneC$, we have $\Varbar{H}[+][-] = 0$ from Proposition~\ref{prop:soa:coneC}. Therefore, $\ppsim''(\Varbar{Z}; \Varbar{H}, W)$ in~\eqref{eq:psd:second-dd-diagonal-Z} and $\npsim''(\Varbar{Z}; \Varbar{H}, W)$ in~\eqref{eq:psd:nsd-second-dd-diagonal-Z} can be simplified as:
\begin{subequations}
    \label{eq:soa:psd-nsd-second-dd-simplified}
    \begin{align}
        & \ppsim''(\Varbar{Z}; \Varbar{H}, W) = 
        \MatrixNine{
            \substack{
                \Var{W}[+][+] 
            } 
            ;
            \left\{ 
            \substack{
                \Var{W}[a][b] \\
                - 2 \frac{1}{\eigvalfirst{a}} \Varbar{H}[a][0] \ppsim(-\Varbar{H}[0][0])
            } 
            \right\}_{a \in \totalsetfirst[+]} 
            ;
            \left\{ 
            \substack{
                \frac{\eigvalfirst{a}}{\eigvalfirst{a} - \eigvalfirst{b}} \Var{W}[a][b] \\
                + 2 \frac{1}{\eigvalfirst{a} - \eigvalfirst{b}} \Varbar{H}[a][0] \Varbar{H}[0][b] 
            } 
            \right\}_{\substack{a \in \totalsetfirst[+] \\ b \in \totalsetfirst[-]}} ;
            \sim ; 
            \substack{
                2 \sum\limits_{c \in \totalsetfirst[+]} \frac{1}{\eigvalfirst{c}} \Varbar{H}[0][c] \Varbar{H}[c][0] \\
                + \ppsim'(\Varbar{H}[0][0]; V_0(\Varbar{H}, W))
            } 
            ;
            \left\{ 
            \substack{
                2 \frac{1}{-\eigvalfirst{b}} \ppsim(\Varbar{H}[0][0]) \Varbar{H}[0][b]
            }
            \right\}_{b \in \totalsetfirst[-]} ;
            \sim ; 
            \sim ; 
            \substack{
                0
            }
        }, \\
        & \npsim''(\Zbar; \Hbar, W) = 
        \MatrixNine{
            \substack{
                0
            } 
            ;
            \left\{ 
            \substack{
                2 \frac{1}{-\eigvalfirst{a}} \Varbar{H}[a][0] \npsim(\Varbar{H}[0][0])
            } 
            \right\}_{a \in \totalsetfirst[+]} 
            ;
            \left\{ 
            \substack{
                \frac{-\eigvalfirst{b}}{\eigvalfirst{a} - \eigvalfirst{b}} \Var{W}[a][b] \\
                + 2 \frac{1}{\eigvalfirst{b} - \eigvalfirst{a}} \Varbar{H}[a][0] \Varbar{H}[0][b] 
            } 
            \right\}_{\substack{a \in \totalsetfirst[+] \\ b \in \totalsetfirst[-]}} 
            ;
            \sim
            ;
            \substack{
                2 \sum\limits_{c \in \totalsetfirst[-]} \frac{1}{\eigvalfirst{c}} \Varbar{H}[0][c] \Varbar{H}[c][0] \\
                + \npsim'(\Varbar{H}[0][0]; V_0(\Varbar{H}, W))
            } 
            ;
            \left\{ 
            \substack{
                \Var{W}[a][b] \\
                -2 \frac{1}{\eigvalfirst{b}} \npsim(-\Varbar{H}[0][0]) \Varbar{H}[0][b]
            }
            \right\}_{b \in \totalsetfirst[-]} 
            ;
            \sim 
            ;
            \sim 
            ;
            \substack{
                \Var{W}[-][-]
            }
        }.
    \end{align}
\end{subequations}
We notice that $V_0(\Varbar{H}, W)$ is linear in $W$. Therefore, define $\tW$ as:
\begin{align}
    \label{eq:soa:tW}
    \Var{\tW}[a][b] := \begin{cases}
        V_0 
        = \Var{W}[0][0] 
        - 2 \sum\limits_{c \in \totalsetfirst[+] \cup \totalsetfirst[-]} \frac{1}{\eigvalfirst{c}} \Varbar{H}[0][c] \Varbar{H}[c][0]
        , & \quad a = 0, b = 0 \\
        \Var{W}[a][b], & \quad \text{Otherwise}
    \end{cases}.
\end{align}
Define
\begin{align}
    \label{eq:soa:Upsilon}
    \opeU:= \MatrixNine{
        0 ;
        0 ; 
        0 ;
        \sim ; 
        2 \sum\limits_{c \in \totalsetfirst[+] \cup \totalsetfirst[-]} \frac{1}{\eigvalfirst{c}} \Varbar{H}[0][c] \Varbar{H}[c][0] ;
        0 ;
        \sim ;
        \sim ;
        0 
    }.
\end{align}
Then, $\Var{W} - \Var{\tW} \equiv \opeU$. Furthermore, define 
\begin{subequations}
    \label{eq:soa:Theta}
    \begin{align}
        \opeT{\tW} & := \MatrixNine{
            \Var{\tW}[+][+] 
            ;
            \Var{\tW}[+][0] 
            ;
            \left\{ 
                \frac{\eigvalfirst{a}}{\eigvalfirst{a} - \eigvalfirst{b}} \Var{\tW}[a][b]
            \right\}_{\substack{a \in \totalsetfirst[+] \\  b \in \totalsetfirst[-]}}
            ;
            \sim
            ;
            {\ppsim'(\Varbar{H}[0][0]; \Var{\tW}[0][0])}
            ;
            0
            ;
            \sim
            ;
            \sim
            ;
            0
        }, \\ 
        \opeTP{\tW} & := \MatrixNine{
            0
            ;
            0
            ;
            \left\{ 
                \frac{-\eigvalfirst{b}}{\eigvalfirst{a} - \eigvalfirst{b}} \Var{\tW}[a][b]
            \right\}_{\substack{a \in \totalsetfirst[+] \\  b \in \totalsetfirst[-]}}
            ;
            \sim
            ;
            {\npsim'(\Varbar{H}[0][0]; \Var{\tW}[0][0])}
            ;
            \Var{\tW}[0][-]
            ;
            \sim
            ;
            \sim
            ;
            \Var{\tW}[-][-]
        },
    \end{align}
\end{subequations}
and 
\begin{subequations}
    \label{eq:soa:calE}
    \begin{align}
        \simpleadjustbox{
            \opeE := \MatrixNine{
                0
                ;
                \left\{ 
                    - 2 \frac{1}{\eigvalfirst{a}} \Varbar{H}[a][0] \ppsim(-\Varbar{H}[0][0])
                \right\}_{a \in \totalsetfirst[+]} 
                ;
                \left\{ 
                    2 \frac{1}{\eigvalfirst{a} - \eigvalfirst{b}} \Varbar{H}[a][0] \Varbar{H}[0][b] 
                \right\}_{\substack{a \in \totalsetfirst[+] \\ b \in \totalsetfirst[-]}} 
                ;
                \sim
                ;
                2 \sum\limits_{c \in \totalsetfirst[+]} \frac{1}{\eigvalfirst{c}} \Varbar{H}[0][c] \Varbar{H}[c][0] 
                ;
                \left\{ 
                    2 \frac{1}{-\eigvalfirst{b}} \ppsim(\Varbar{H}[0][0]) \Varbar{H}[0][b]
                \right\}_{b \in \totalsetfirst[-]} 
                ;
                \sim
                ;
                \sim
                ;
                0
            },
        } \\
        \simpleadjustbox{
            \opeEP := \MatrixNine{
                0
                ;
                \left\{ 
                    2 \frac{1}{-\eigvalfirst{a}} \Varbar{H}[a][0] \npsim(\Varbar{H}[0][0])
                \right\}_{a \in \totalsetfirst[+]} 
                ;
                \left\{ 
                    2 \frac{1}{\eigvalfirst{b} - \eigvalfirst{a}} \Varbar{H}[a][0] \Varbar{H}[0][b] 
                \right\}_{\substack{a \in \totalsetfirst[+] \\ b \in \totalsetfirst[-]}} 
                ;
                \sim 
                ;
                    2 \sum\limits_{c \in \totalsetfirst[-]} \frac{1}{\eigvalfirst{c}} \Varbar{H}[0][c] \Varbar{H}[c][0] 
                ;
                \left\{ 
                    -2 \frac{1}{\eigvalfirst{b}} \npsim(-\Varbar{H}[0][0]) \Varbar{H}[0][b]
                \right\}_{b \in \totalsetfirst[-]} 
                ;
                \sim
                ;
                \sim
                ;
                0
            }.
        }
    \end{align}
\end{subequations}
For all $\Varbar{H} \in \coneC$ and $\tW \in \Sn$, the following relationships hold:
\begin{subequations}
    \label{eq:soa:relation-Theta-calE}
    \begin{align}
        & \ppsim''(\Varbar{Z}; \Varbar{H}, W) = \opeT{\tW} + \opeE, \\
        & \npsim''(\Varbar{Z}; \Varbar{H}, W) = \opeTP{\tW} + \opeEP, \\
        & \opeT{\tW}  + \opeTP{\tW}  = \tW,\\
        & \opeE + \opeEP = \opeU.
    \end{align}
\end{subequations}
Now we are ready to simplify $\delta''(\Zbar; \Hbar, W)$:
\begin{align}
    \label{eq:soa:deltasecond-simplified}
    & \delta''(\Zbar; \Hbar, W) = -\PA \ppsim''(\Varbar{Z}; \Varbar{H}, W) - \PAp \npsim''(\Varbar{Z}; \Varbar{H}, W) \nonumber \\
    = & -\PA \opeT{\tW} - \PAp \opeTP{\tW} - \PA \opeE - \PAp \opeEP, 
\end{align}
where $\opeT{\tW}, \opeTP{\tW}$ are defined in~\eqref{eq:soa:Theta} and $\opeE, \opeEP$ are defined in~\eqref{eq:soa:calE}.  

\subsubsection{\titlemath{$\Varkpo{W} - \Vark{W}$ is convergent}}
From~\eqref{eq:soa:deltasecond-simplified} and~\eqref{eq:soa:tW}, 
\begin{align*}
    \Varkpo{\tW} - \Vark{\tW} = \Varkpo{W} - \Vark{W} 
    = -\PA \opeT{\Vark{\tW}} - \PAp \opeTP{\Vark{\tW}} - \PA \opeE - \PAp \opeEP.
\end{align*}
From~\eqref{eq:soa:relation-Theta-calE}, $\Vark{\tW} = \opeT{\Vark{\tW}} + \opeTP{\Vark{\tW}}$. Thus,
\begin{align}
    \label{eq:soa:sod-simplified}
    \Varkpo{\tW} = \left\{ 
        \PAp \opeT{\Vark{\tW}} + \PA \opeTP{\Vark{\tW}}
     \right\}
    + \underbrace{
        \left\{ 
            - \PA \opeE - \PAp \opeEP
        \right\}
    }_{=: \opeP}.
\end{align}
From now on, suppose $\Hbar$ follows the second-level description in \S\ref{sec:psd}.
\begin{lemma}[Firmly nonexpansiveness of $\PAp \opeT{\cdot} + \PA \opeTP{\cdot}$]
    \label{lem:soa:sod-firmly-nonexp}
    $\PAp \opeT{\cdot} + \PA \opeTP{\cdot}$ in~\eqref{eq:soa:sod-simplified} is firmly nonexpansive on $(\Sym{n}, \inprod{\cdot}{\cdot})$.
\end{lemma}
\begin{proof}
    The proof procedure is similar to the one in Lemma~\ref{lem:soa:convergent-fod}. For ease of notation, we abbreviate $\PAp \opeT{\cdot} + \PA \opeTP{\cdot}$ as $\calT(\cdot)$, $\opeT{\cdot}$ as $\calF(\cdot)$, and $\opeTP{\cdot}$ as $\calF^\perp(\cdot)$: 
    \begin{align*}
        & \normF{\calT(U) - \calT(V)}^2 + \normF{(\Id - \calT)(U) - (\Id - \calT)(V)}^2 \\
        = & \normF{\PAp [\calF(U) - \calF(V)]}^2 + \normF{\PA [\calF^\perp(U) - \calF^\perp(V)]}^2
        + \normF{\PA [\calF(U) - \calF(V)]}^2 + \normF{\PAp [\calF^\perp(U) - \calF^\perp(V)]}^2 \\
        = & \normF{\calF(U) - \calF(V)}^2 + \normF{\calF^\perp(U) - \calF^\perp(V)}^2 \\
        = & \normF{U - V}^2 - 2 \inprod{
            \calF(U) - \calF(V)
        }{
            \calF^\perp(U) - \calF^\perp(V)
        }.
    \end{align*}
    Thus, all we need to show is $\inprod{
            \calF(U) - \calF(V)
        }{
            \calF^\perp(U) - \calF^\perp(V)
        } \ge 0$. 
    Since 
    \begin{align*}
        & \inprod{
            \calF(U) - \calF(V)
        }{
            \calF^\perp(U) - \calF^\perp(V)
        } \\
        = & \underbrace{
            2 \sum\limits_{a \in \totalsetfirst[+], b \in \totalsetfirst[-]}  
            \frac{\eigvalfirst{a}}{\eigvalfirst{a} - \eigvalfirst{b}} 
            \cdot \frac{-\eigvalfirst{b}}{\eigvalfirst{a} - \eigvalfirst{b}}
            \normF{\Var{U}[a][b] - \Var{V}[a][b]}^2 
        }_{\ge 0} \\
        & + \underbrace{
            \inprod{
                \ppsim'(\Varbar{H}[0][0]; \Var{U}[0][0]) 
                - \ppsim'(\Varbar{H}[0][0]; \Var{V}[0][0])
            }{
                \npsim'(\Varbar{H}[0][0]; \Var{U}[0][0]) 
                - \npsim'(\Varbar{H}[0][0]; \Var{V}[0][0])
            } 
        }_{=:\text{LHS}}.
    \end{align*}
    It boils down to prove $\text{LHS} \ge 0$. Abbreviate $\Varhat{U}$ as $(\Qfirst{0})\tran \Var{U}[0][0] \Qfirst{0}$ and $\Varhat{V}$ as $(\Qfirst{0})\tran \Var{V}[0][0] \Qfirst{0}$:
    \begin{subequations}
        \begin{align*}
            \ppsim'(\Varbar{H}[0][0]; \Var{U}[0][0]) = & \Qfirst{0} \MatrixNine{
                \Varsecond{\Varhat{U}}{0}[+][+] 
                ;
                \Varsecond{\Varhat{U}}{0}[+][0] 
                ;
                \left\{ 
                    \frac{
                        \eigvalsecond{0}{i}
                    }{
                        \eigvalsecond{0}{i} - \eigvalsecond{0}{j}
                    } \Varsecond{\Varhat{U}}{0}[i][j]
                \right\}_{\substack{i \in \totalsetsecond{0}[+] \\ j \in \totalsetsecond{0}[-]}}
                ;
                \sim
                ; 
                \ppsim(
                    \Varsecond{\Varhat{U}}{0}[0][0]
                )
                ;
                0 
                ;
                \sim ;
                \sim 
                ;
                0
            } (\Qfirst{0})\tran,  \\
            \npsim'(\Varbar{H}[0][0]; \Var{V}[0][0]) = & \Qfirst{0} \MatrixNine{
                0 
                ;
                0 
                ;
                \left\{ 
                    \frac{
                        -\eigvalsecond{0}{j}
                    }{
                        \eigvalsecond{0}{i} - \eigvalsecond{0}{j}
                    } \Varsecond{\Varhat{V}}{0}[i][j]
                \right\}_{\substack{i \in \totalsetsecond{0}[+] \\ j \in \totalsetsecond{0}[-]}}
                ;
                \sim 
                ; 
                \npsim(
                    \Varsecond{\Varhat{V}}{0}[0][0]
                )
                ;
                \Varsecond{\Varhat{V}}{0}[0][-] 
                ;
                \sim 
                ;
                \sim   
                ;
                \Varsecond{\Varhat{V}}{0}[-][-] 
            } (\Qfirst{0})\tran.
        \end{align*}
    \end{subequations}
    Thus,
    \begin{align*}
        \text{LHS} = & \underbrace{
            2 \sum\limits_{i \in \totalsetsecond{0}[+], j \in \totalsetsecond{0}[-]}
        \frac{\eigvalsecond{0}{i}}{\eigvalsecond{0}{i} - \eigvalsecond{0}{j}} 
        \cdot \frac{-\eigvalsecond{0}{j}}{\eigvalsecond{0}{i} - \eigvalsecond{0}{j}} 
        \normF{\Varsecond{\Varhat{U}}{0}[i][j] - \Varsecond{\Varhat{V}}{0}[i][j]}^2
        }_{\ge 0} \\
        & + \inprod{
            \ppsim(
                    \Varsecond{\Varhat{U}}{0}[0][0]
            ) - \ppsim(
                    \Varsecond{\Varhat{V}}{0}[0][0]
            )
        }{
            \npsim(
                    \Varsecond{\Varhat{U}}{0}[0][0]
            ) - \npsim(
                    \Varsecond{\Varhat{V}}{0}[0][0]
            )
        } \\
        \ge & -\inprod{
            \ppsim(
                    \Varsecond{\Varhat{U}}{0}[0][0]
            )
        }{
            \npsim(
                    \Varsecond{\Varhat{V}}{0}[0][0]
            )
        } - \inprod{
            \ppsim(
                    \Varsecond{\Varhat{V}}{0}[0][0]
            )
        }{
            \npsim(
                    \Varsecond{\Varhat{U}}{0}[0][0]
            )
        } \ge 0,
    \end{align*}
    which concludes the proof. 
\end{proof}
One may have already noticed that the operator $\PAp \opeT{\cdot} + \PA \opeTP{\cdot}$ in~\eqref{eq:soa:sod-simplified} closely resembles $\Id + \delta'(\Zbar; \cdot)$ in Lemma~\ref{lem:soa:convergent-fod}. For instance, $\PAp \opeT{\cdot} + \PA \opeTP{\cdot}$ is also positively homogeneous, and hence $0 \in \Fix(\PAp \opeT{\cdot} + \PA \opeTP{\cdot})$. The essential difference between the local first- and second-order dynamics, however, lies in the presence of the ``constant term'' $\opeP$.

\begin{theorem}[Convergent $\Varkpo{\tW} - \Vark{\tW}$]
    \label{thm:soa:mdvZ}
    For the dynamical system~\eqref{eq:soa:sod-simplified}, 
    \begin{align}
        \label{eq:soa:mdvZ}
        \Varkpo{\tW} - \Vark{\tW} \rightarrow \mdvZ := \opeP - \Pi_{\coneK}(\opeP) = \Pi_{\coneKP}(\opeP),
        \quad \text{as } k \rightarrow \infty  
    \end{align}
    where the closed convex cone $\coneK$ is defined as 
    \begin{align}
        \label{eq:soa:calK}
        \coneK := \closure\left( 
            \left\{ 
                \PA \opeT{W} + \PAp \opeTP{W} \mymid W \in \Sn
             \right\}
         \right)
    \end{align}
    and its polar cone: 
    \begin{align}
        \label{eq:soa:polar-calK}
        \coneKP := \left\{ 
            Y \in \Sym{n} \mymid \inprod{
                \PA \opeT{W} + \PAp \opeTP{W}
            }{
                Y
            } \le 0, \ \forall W \in \Sn 
         \right\}.
    \end{align}
\end{theorem}
\begin{proof}
    This is a standard result from Monotone operator theory. For ease of notation, abbreviate the operator $\PAp \opeT{\cdot} + \PA \opeTP{\cdot}$ as $\calT(\cdot)$, $\opeP$ as $\Psi$, and $\coneK$ as $\calK$.
     Since $\calT$ is firmly nonexpansive and positive homogeneous, $\calK := \closure(\range(\Id - \calT))$ is a nonempty, closed, and convex cone from~\cite[Lemma 4]{pazy71ijm-asymptotic-behavior-contractions}. Denote $\calS(\cdot) := \calT(\cdot) + \Psi$. Since $\calT$ is firmly nonexpansive on $(\Sym{n}, \inprod{\cdot}{\cdot})$ and $\Psi$ is a constant drift, $\calS$ is also firmly nonexpansive on $(\Sym{n}, \inprod{\cdot}{\cdot})$, yet it may not yield any fixed point. 
     
     On the other hand, from~\cite[Corollary 2.3]{baillon78hjm-asymptotic-behavior-nonexpansive-mappings}, we have the following weaker result for the dynamical system $\Varkpo{\tW} = \calS(\Vark{\tW})$ in~\eqref{eq:soa:sod-simplified}:
    \begin{align*}
        \Varkpo{\tW} - \Vark{\tW} \rightarrow -\Pi_{\closure(\range(\Id - \calS))}(0) 
        \quad \text{as } k \rightarrow \infty.
    \end{align*}
    Since $\closure(\range(\Id - \calS)) = -\Psi + \closure(\range(\Id - \calT)) = -\Psi + \calK$, we get
    \begin{align*}
        \Pi_{\closure(\range(\Id - \calS))}(0) = \Pi_{-\Psi + \calK}(0) = \Pi_{\calK}(\Psi) - \Psi.
    \end{align*}
    From polar cone's definition, $\Psi = \Pi_{\calK}(\Psi) + \Pi_{\polar{\calK}}(\Psi)$. The closure in $\calK$ does not affect $\polar{\calK}$, since for an arbitrary convex cone $\calC$, $\polar{(\closure(\calC))} = \polar{\calC}$. 
\end{proof}

\subsubsection{\titlemath{Local second-order {limit} dynamics}}
By Lemma~\ref{lem:soa:sod-firmly-nonexp}, the increment $\Varkpo{W}-\Vark{W}$ eventually converges to a constant second-order ``drift'' $\mdvZ$, and this limit is independent of the initialization $W^{(0)}$. From the viewpoint of time-scale separation, this convergence manifests as a second-order effect (scaled by $\frac{t^2}{2}$ with $t\downarrow 0$), whereas the evolution of $\Vark{H}$ is a first-order effect (scaled by $t$ with $t\downarrow 0$). It is therefore reasonable to assume that, by the time $\Varkpo{W}-\Vark{W}$ has converged, $\Vark{Z}$ remains unchanged to first order. Consequently, by~\eqref{eq:soa:second-order-expansion}, the \emph{limit} dynamics after $\Varkpo{W}-\Vark{W}\rightarrow \mdvZ$ is
\begin{align*}
    \Varkpo{Z} = \Vark{Z} + \frac{t^2}{2}\,\mdvZ[\Zbar; \Hbar] + o(t^2).
\end{align*}
The update above produces a ray in $\Sym{n}$ with a constant second-order ``drift'' as $t\downarrow 0$. Moreover, $\Vark{Z} - \Zbar = t \Hbar + o(t)$ for any \emph{fixed} $k$. However, one must account for the cumulative effect of this drift over many iterations: when $k \sim O(\frac{1}{t})$, the accumulated second-order displacement can become non-negligible compared to the first-order term $t\Hbar$. In this regime, the original separation of first- and second-order dynamics in Definition~\ref{def:soa:fod-sod} and~\eqref{eq:soa:sod-simplified} may no longer be accurate, and the effective direction $\Hbar$ may need to be re-identified because $t$ is fixed and positive. We address this issue by replacing the constant term $t\Hbar$ with the ``feedback'' term $\Vark{Z}-\Zbar$. To this end, we use the following elementary scaling property of $\mdvZ[\Zbar;\cdot]$.

\begin{proposition}[Positive 2-homogeneity of \titlemath{$\mdvZ[\Zbar;\cdot]$}]
    \label{prop:soa:mdvZ-2-homogeneity}
    For any $t>0$ and any $\Hbar\in\coneC$, it holds that $\mdvZ[\Zbar; t\Hbar]=t^2\,\mdvZ[\Zbar;\Hbar]$.
\end{proposition}
\begin{proof}
    By~\eqref{eq:soa:calE}, we have $\opeE[t\Hbar]=t^2\opeE$ and $\opeEP[t\Hbar]=t^2\opeEP$. Next, we show that $\coneK[\Hbar]=\coneK[t\Hbar]$ for all $t>0$. Indeed, for any $W\in\Sym{n}$ and any $t > 0$,
    \begin{align*}
        \ppsim'(t \Varbar{H}[0][0]; \Var{W}[0][0]) 
        = \MatrixNine{
            \Varhatsecond{W}{0}[+][+] ; 
            \Varhatsecond{W}{0}[+][0] ;
            \left\{ 
                \frac{t \eigvalsecond{0}{i}}{t \eigvalsecond{0}{i} - t \eigvalsecond{0}{j}}\Varhatsecond{W}{0}[i][i]
             \right\}_{\substack{
                i \in \totalsetsecond{0}[+] \\
                i \in \totalsetsecond{0}[-]
            }} ;
            \sim ; 
            \ppsim(\Varhatsecond{W}{0}[0][0]) ;
            0 ; 
            \sim ; 
            \sim ; 
            0 
        }
        = \ppsim'(\Varbar{H}[0][0]; \Var{W}[0][0]),
    \end{align*}
    where $\Varhat{W} := (\Qfirst{0})\tran \Var{W}[0][0] \Qfirst{0}$.
    Hence, by~\eqref{eq:soa:Theta}, 
    \begin{align*}
        \opeT[\Hbar]{W} = 
        \MatrixNine{
            \Var{W}[+][+] ; 
            \Var{W}[+][0] ;
            \left\{ 
                \frac{t \eigvalfirst{a}}{t \eigvalfirst{a} - t \eigvalfirst{b}} \Var{W}[a][b]
             \right\}_{\substack{
                a \in \totalsetfirst[+] \\
                b \in \totalsetfirst[-]
            }} ;
            \sim ; 
            {\ppsim'(t \Varbar{H}[0][0]; \Var{W}[0][0])} ;
            0 ; 
            \sim ; 
            \sim ; 
            0 
        }
        =\opeT[t\Hbar]{W}
    \end{align*}
    for all $W\in\Sym{n}$ and $t > 0$. By symmetry, $\opeTP[\Hbar]{W}=\opeTP[t\Hbar]{W}$ for all $W\in\Sym{n}$. It then follows from~\eqref{eq:soa:calK} that $\coneK[\Hbar]=\coneK[t\Hbar]$. Finally, using~\eqref{eq:soa:mdvZ},
    \begin{align*}
        \mdvZ[\Zbar; t\Hbar]
        = \opeP[t\Hbar] - \Pi_{\coneK[t\Hbar]}(\opeP[t\Hbar])
        = t^2 \opeP[\Hbar] - \Pi_{\coneK[\Hbar]}(t^2 \opeP[\Hbar])
        = t^2\,\mdvZ[\Zbar;\Hbar],
    \end{align*}
    which concludes the proof.
\end{proof}

With Proposition~\ref{prop:soa:mdvZ-2-homogeneity}, we have $\frac{t^2}{2}\mdvZ[\Zbar;\Hbar]=\frac{1}{2}\mdvZ[\Zbar; t\Hbar]$ for all $t>0$, which motivates the following definition.

\begin{definition}[Local second-order limit map and limit dynamics]
    \label{def:soa:mdvZ-dynamics}
    At a point $\Zbar\in\Zopt$, the local second-order limit map $\mdvZ[\Zbar;\cdot]:\coneC\mapsto\Sym{n}$ is defined by
    \begin{align}
        \label{eq:soa:mdvZ-def}
        \mdvZ[\Zbar;\cdot] := \opeP[\cdot] - \Pi_{\coneK[\cdot]}(\opeP[\cdot]) = \Pi_{\coneKP[\cdot]}(\opeP[\cdot]).
    \end{align}
    Here $\opeP[\cdot]$ is defined in~\eqref{eq:soa:sod-simplified}, $\coneK[\cdot]$ in~\eqref{eq:soa:calK}, and $\coneKP[\cdot]$ in~\eqref{eq:soa:polar-calK}. 
    
    The local second-order limit dynamics is defined as
    \begin{align}
        \label{eq:soa:limiting-dynamics-def}
        \Varkpo{Z} = \Vark{Z} + \frac{1}{2}\mdvZ[\Zbar; \Vark{Z}-\Zbar] + o(\normF{\Vark{Z}-\Zbar}^2).
    \end{align}
\end{definition}

\begin{remark}
    \label{rem:soa:three-level-approx}
    One may have noticed that three layers of approximation are involved before arriving at the local second-order limit dynamics~\eqref{eq:soa:limiting-dynamics-def}. At the first layer, we perform a local expansion governed by the scale parameter $t\downarrow 0$ and obtain the local second-order dynamics~\eqref{eq:soa:sod}. At the second layer, we take the iteration limit of~\eqref{eq:soa:sod} as $k\to\infty$ to obtain the local second-order limit map $\mdvZ$. At the third layer, we replace the constant term $t\Hbar$ by the feedback term $\Vark{Z}-\Zbar$ to arrive at the local second-order limit dynamics~\eqref{eq:soa:limiting-dynamics-def}. These coupled approximations allow us to focus on the persistent limiting behavior of ADMM's local dynamics and lead to a model with clean and useful structure. Admittedly, however, they come at the cost of mathematical rigor. For this reason, rather than claiming~\eqref{eq:soa:limiting-dynamics-def} as a fully rigorous consequence of~\eqref{eq:intro:one-step-admm}, we introduce it as part of the definition of a new mathematical object whose properties we then study.
\end{remark}

From the perspective of a vector field, the limit dynamics~\eqref{eq:soa:limiting-dynamics-def} assigns to each point $Z-\Zbar\in\coneC$ a displacement vector $\frac{1}{2}\mdvZ[\Zbar; Z-\Zbar]$, which satisfies $\frac{1}{2}\mdvZ[\Zbar; Z-\Zbar]\sim \calO(\normF{Z-\Zbar}^2)$ up to higher-order terms. Therefore, understanding the mapping $\mdvZ[\Zbar;\cdot]$ in~\eqref{eq:soa:mdvZ-def} becomes key to understanding the associated limit dynamics~\eqref{eq:soa:limiting-dynamics-def}. In the subsequent section, we will see that fundamental properties of $\mdvZ[\Zbar;\cdot]$ (\eg kernel, range, continuity, and primal--dual partition) are tightly linked to dynamical features of~\eqref{eq:soa:limiting-dynamics-def} (\eg fixed points, almost-invariant sets, phases, and the effect of $\sigma$), which in turn explain and predict the limiting behavior of the one-step ADMM update~\eqref{eq:intro:one-step-admm} around $\Zbar$. Before analyzing $\mdvZmap$'s properties, we first exploit several structural properties of $\mdvZ[\Zbar;\cdot]$.


\section{Polar Description and Primal--Dual Decoupling}
\label{sec:decouple}
In this section, we simplify $\mdvZmap$ by exposing its primal--dual decoupling structure, which serves as a foundation for the subsequent characterization of $\mdvZmap$. In \S\ref{sec:decouple:coneKP}, we simplify $\coneKP$ by revealing the self-similar structure of $\ppsim''(Z; H, W)$. As a result, $\coneKP$ can be expressed as a Minkowski sum of two simpler cones, $\coneKPX$ and $\coneKPS$. In \S\ref{sec:decouple:XS}, we prove that the second-order limits of both $\Delta \Vark{X}$ and $\Delta \Vark{S}$ exist. Finally, in \S\ref{sec:decouple:mdvZ}, we link these limiting drifts tightly to $\coneKPX$ and $\coneKPS$, thereby revealing a clean primal--dual decoupling mechanism in the second-order-dominant regimes.

\subsection{\titlemath{Simplification of $\coneKP$}}
\label{sec:decouple:coneKP}

We first simplify the structure of $\Qfirst{0} \in \calO^{\sizeof{\indexfirst{0}}}(\Varbar{H}[0][0])$ when $\Hbar \in \coneC$.

\begin{lemma}[Block-diagonal structure of $\Qfirst{0}$]
    \label{lem:soa:Q0-block-structure}
    Under Assumption~\ref{ass:soa:sc}, fix any $\Zbar \in \Zopt$ and $\Hbar \in \coneC$. Suppose $\Zbar$ follows the first-level description in \S\ref{sec:psd} and $\Hbar$ follows the second-level description. Then, $\Qfirst{0}_{\MultiIndexSecondPD{1}{2}} \equiv 0$ for all $\Qfirst{0} \in \calO^{\sizeof{\indexfirst{0}}}(\Varbar{H}[0][0])$. 
\end{lemma}
\begin{proof}
    From Proposition~\ref{prop:soa:coneC} (2), we get 
    \begin{align*}
        \Hbar = \MatrixSixteenSC{
            \Hbar_{\MultiIndexFirstSC{1}{1}} ; \Hbar_{\MultiIndexFirstSC{1}{2}} ; \Hbar_{\MultiIndexFirstSC{1}{3}} ; 0 ; 
            \sim ; \Hbar_{\MultiIndexFirstSC{2}{2}} ; 0 ; \Hbar_{\MultiIndexFirstSC{2}{4}} ;
            \sim ; \sim ; \Hbar_{\MultiIndexFirstSC{3}{3}} ; \Hbar_{\MultiIndexFirstSC{3}{4}} ;
            \sim ; \sim ; \sim ; \Hbar_{\MultiIndexFirstSC{4}{4}}
        }
    \end{align*}
    for any $\Hbar \in \coneC$. Thus, $[\Varbar{H}[0][0]]_{\MultiIndexSecondPD{1}{2}} = 0$, which closes the proof. 
\end{proof}

It turns out $\coneKP$ has the following nice structures. 

\begin{proposition}[Structure of $\coneKP$]
    \label{prop:soa:polarK-XS}
    Under Assumption~\ref{ass:soa:sc}, fix any $\Zbar \in \Zopt$ and $\Hbar \in \coneC$. Suppose $\Zbar$ follows the first-level description in \S\ref{sec:psd} and $\Hbar$ follows the second-level description. Then, $\coneKP = \coneKPX + \coneKPS$, where 
    \begin{subequations}
        \label{eq:soa:polarK-XS}
        \begin{align}
            & \simpleadjustbox{
            \coneKPX 
            := \left\{ 
                W = \MatrixNine{
                    W_{\MultiIndexFirst{1}{1}} ; W_{\MultiIndexFirst{1}{2}} ; 0 ;
                    \sim ; 
                    {\Qfirst{0} \MatrixSixteenSC{
                        \Varhat{W}_{\MultiIndexSecondSC{1}{1}} ; \Varhat{W}_{\MultiIndexSecondSC{1}{2}} ; \Varhat{W}_{\MultiIndexSecondSC{1}{3}} ; 0 ; 
                        \sim ; \Varhat{W}_{\MultiIndexSecondSC{2}{2}} ; 0 ; 0 ;
                        \sim ; \sim ; 0 ; 0 ; 
                        \sim ; \sim ; \sim ; 0 
                    } (\Qfirst{0})\tran}
                    ; 0 ; 
                    \sim ; \sim ; 0 
                } \mymid 
                \mymatplain{
                    \PA W = 0, \\ 
                    \Varhat{W} = (\Qfirst{0})\tran W_{\MultiIndexFirst{2}{2}} \Qfirst{0}, \\
                    \Varhat{W}_{\MultiIndexSecondSC{2}{2}} \succeq 0 
                }
            \right\},
            } \\
            & \simpleadjustbox{
            \coneKPS
            := \left\{ 
                W = \MatrixNine{
                    0 ; 0 ; 0 ;
                    \sim ; 
                    {\Qfirst{0} \MatrixSixteenSC{
                        0 ; 0 ; 0 ; 0 ; 
                        \sim ; 0 ; 0 ; \Varhat{W}_{\MultiIndexSecondSC{2}{4}} ;
                        \sim ; \sim ; \Varhat{W}_{\MultiIndexSecondSC{3}{3}} ; \Varhat{W}_{\MultiIndexSecondSC{3}{4}} ; 
                        \sim ; \sim ; \sim ; \Varhat{W}_{\MultiIndexSecondSC{4}{4}} 
                    } (\Qfirst{0})\tran}
                    ; W_{\MultiIndexFirst{2}{3}} ; 
                    \sim ; \sim ; W_{\MultiIndexFirst{3}{3}} 
                } \mymid 
                \mymatplain{
                    \PAp W = 0, \\ 
                    \Varhat{W} = (\Qfirst{0})\tran W_{\MultiIndexFirst{2}{2}} \Qfirst{0}, \\
                    \Varhat{W}_{\MultiIndexSecondSC{3}{3}} \preceq 0 
                }
            \right\}.
            }
        \end{align}
    \end{subequations}
\end{proposition}
\begin{proof}
    ``$\subseteq$'': From~\eqref{eq:soa:polar-calK}, $Y \in \coneKP$ if and only if 
    \begin{align}
        \label{eq:soa:proof-polarK-0}
        & \inprod{Y}{\PA \opeT{W} + \PAp \opeTP{W}} \le 0, \ \forall W \in \Sym{n} \nonumber \\
        \Longleftrightarrow & \inprod{\PA Y}{\PA \opeT{W}} + \inprod{\PAp Y}{\PAp \opeT{W}} \le 0, \ \forall W \in \Sym{n} \nonumber \\
        \Longleftrightarrow & Y = U + V, \ \PA U = 0, \ \PAp V = 0, \ \inprod{V}{\PA \opeT{W}} + \inprod{U}{\PAp \opeTP{W}} \le 0, \ \forall W \in \Sym{n} \nonumber \\
        \Longleftrightarrow & Y = U + V, \ \PA U = 0, \ \PAp V = 0, \ \inprod{V}{\opeT{W}} + \inprod{U}{\opeTP{W}} \le 0, \ \forall W \in \Sym{n}.
    \end{align} 

    (\romannumeral1) Set $\Var{W}[+][-] = 0, \Var{W}[0][0] = 0, \Var{W}[0][-] = 0, \Var{W}[-][-] = 0$. Then, from~\eqref{eq:soa:Theta},
    \begin{align*}
        \Theta(\Varbar{H}; \Var{W}) = \MatrixNine{
            \Var{W}[+][+] ; \Var{W}[+][0] ; 0 ;
            \sim ; 0 ; 0 ;
            \sim ; \sim ; 0 
        }, \quad \Theta^\perp(\Varbar{H}; \Var{W}) = 0.
    \end{align*}
    Since $\Var{W}[+][+]$ and $\Var{W}[+][0]$ can be chosen arbitrarily, we have $\Var{V}[+][+] = 0, \Var{V}[+][0] = 0$. Symmetrically, $\Var{U}[-][-] = 0, \Var{U}[0][-] = 0$. 
    
    (\romannumeral2) Set everything except $\Var{W}[+][-]$ to be zero:
    \begin{align*}
        \inprod{V}{\opeT{W}} + \inprod{U}{\opeTP{W}}
        = 2 \sum_{a \in \totalsetfirst[+]} \sum_{b \in \totalsetfirst[-]}  \inprod{\frac{\eigvalfirst{a}}{\eigvalfirst{a} - \eigvalfirst{b}}\Var{V}[a][b] + \frac{-\eigvalfirst{b}}{\eigvalfirst{a} - \eigvalfirst{b}} \Var{U}[a][b]}{\Var{W}[a][b]}.
    \end{align*}
    Since $\Var{W}[+][-]$ can be arbitrarily chosen, we have 
    \begin{align}
        \label{eq:soa:proof-polarK-1}
        \frac{\eigvalfirst{a}}{\eigvalfirst{a} - \eigvalfirst{b}}\Var{V}[a][b] + \frac{-\eigvalfirst{b}}{\eigvalfirst{a} - \eigvalfirst{b}} \Var{U}[a][b] = 0, \quad \forall a \in \totalsetfirst[+], b \in \totalsetfirst[-].
    \end{align}

    (\romannumeral3) Now we zoom in to the $\MultiIndexFirst{2}{2}$ block. Set everything except $\Var{W}[0][0]$ to be zero. Then from~\eqref{eq:soa:Theta},
    \begin{align*}
        \inprod{V}{\opeT{W}} + \inprod{U}{\opeTP{W}}
        = \inprod{\Var{V}[0][0]}{\ppsim'(\Varbar{H}[0][0]; \Var{W}[0][0])} + 
        \inprod{\Var{U}[0][0]}{\npsim'(\Varbar{H}[0][0]; \Var{W}[0][0])}.
    \end{align*}
    Further denote $\Varhat{U} = (\Qfirst{0})\tran \Var{U}[0][0] \Qfirst{0}, \Varhat{V} = (\Qfirst{0})\tran \Var{V}[0][0] \Qfirst{0}, \Varhat{W} = (\Qfirst{0})\tran \Var{W}[0][0] \Qfirst{0}$:
    \begin{align*}
        & \inprod{V}{\opeT{W}} + \inprod{U}{\opeTP{W}} = \\
        & \simpleadjustbox{
         \inprod{\Varhat{V}}{
            \MatrixNine{
                \Varsecond{\Varhat{W}}{0}[+][+] ; 
                \Varsecond{\Varhat{W}}{0}[+][0] ; 
                \left\{ 
                    \frac{\eigvalsecond{0}{i}}{\eigvalsecond{0}{i} - \eigvalsecond{0}{j}} \Varsecond{\Varhat{W}}{0}[i][j]
                 \right\}_{\substack{i \in \totalsetsecond{0}[+] \\ j \in \totalsetsecond{0}[-]}} ;
                \sim ;
                \ppsim(\Varsecond{\Varhat{W}}{0}[0][0]) ;
                0 ;
                \sim ;
                \sim ; 
                0 
            }
        } + \inprod{\Varhat{U}}{
            \MatrixNine{
                0 ; 
                0 ; 
                \left\{ 
                    \frac{-\eigvalsecond{0}{j}}{\eigvalsecond{0}{i} - \eigvalsecond{0}{j}} \Varsecond{\Varhat{W}}{0}[i][j]
                 \right\}_{\substack{i \in \totalsetsecond{0}[+] \\ j \in \totalsetsecond{0}[-]}} ;
                \sim ;
                \npsim(\Varsecond{\Varhat{W}}{0}[0][0]) ;
                \Varsecond{\Varhat{W}}{0}[0][-] ;
                \sim ;
                \sim ; 
                \Varsecond{\Varhat{W}}{0}[-][-] 
            }
        }
        }.
    \end{align*}

    (a) Similar to (\romannumeral1), for $\Varhat{W}$, set everything to $0$ except $\Varhatsecond{W}{0}[+][+]$ and $\Varhatsecond{W}{0}[+][0]$. We get $\Varhatsecond{V}{0}[+][+] = 0$ and $\Varhatsecond{V}{0}[+][0] = 0$. Symmetrically, we get $\Varhatsecond{U}{0}[-][-] = 0$ and $\Varhatsecond{U}{0}[0][0] = 0$.

    (b) Similar to (\romannumeral2), for $\Varhat{W}$, set everything to $0$ except $\Varhatsecond{W}{0}[+][-]$. We get 
    \begin{align}
        \label{eq:soa:proof-polarK-2}
        \frac{\eigvalsecond{0}{i}}{\eigvalsecond{0}{i} - \eigvalsecond{0}{j}}\Varhatsecond{V}{0}[i][j] + \frac{-\eigvalsecond{0}{j}}{\eigvalsecond{0}{i} - \eigvalsecond{0}{j}}\Varhatsecond{U}{0}[i][j] = 0, \quad \forall i \in \totalsetsecond{0}[+], j \in \totalsetsecond{0}[-].
    \end{align}

    (c) For $\Varhat{W}$, set everything to $0$ except $\Varhatsecond{W}{0}[0][0]$. We get 
    \begin{align*}
        \inprod{\Varhatsecond{V}{0}[0][0]}{\ppsim(\Varhatsecond{W}{0}[0][0])} 
        + \inprod{\Varhatsecond{U}{0}[0][0]}{\npsim(\Varhatsecond{W}{0}[0][0])} \le 0, \ \forall \Varhatsecond{W}{0}[0][0].
    \end{align*}
    Transversing $\Varhatsecond{W}{0}[0][0]$ through $\Symp{\sizeof{\indexsecond{0}{0}}}$, we get $\Varhatsecond{V}{0}[0][0] \preceq 0$. Symmetrically, we get $\Varhatsecond{U}{0}[0][0] \succeq 0$.

    (\romannumeral4) Since $\PA U = 0, \PAp V = 0$, we have $\inprod{U}{V} = 0$. On the other hand, combining (\romannumeral1) - (\romannumeral3):
    \begin{align*}
        & \inprod{U}{V} = 2 \sum_{a \in \totalsetfirst[+]} \sum_{b \in \totalsetfirst[-]} \inprod{\Var{U}[a][b]}{\Var{V}[a][b]} + 2 \sum_{i \in \totalsetsecond{0}[+]} \sum_{j \in \totalsetsecond{0}[-]} \inprod{\Varhatsecond{U}{0}[i][j]}{\Varhatsecond{V}{0}[i][j]} + \inprod{\Varhatsecond{U}{0}[0][0]}{\Varhatsecond{V}{0}[0][0]} \\
        = & 2 \sum_{a \in \totalsetfirst[+]} \sum_{b \in \totalsetfirst[-]} \frac{\eigvalfirst{b}}{\eigvalfirst{a}} \normF{\Var{U}[a][b]}^2 + 2 \sum_{i \in \totalsetsecond{0}[+]} \sum_{j \in \totalsetsecond{0}[-]} \frac{\eigvalsecond{0}{j}}{\eigvalsecond{0}{i}}\normF{\Varhatsecond{U}{0}[i][j]}^2 + \inprod{\Varhatsecond{U}{0}[0][0]}{\Varhatsecond{V}{0}[0][0]} = 0.
    \end{align*}
    where the last equality comes from~\eqref{eq:soa:proof-polarK-1} and~\eqref{eq:soa:proof-polarK-2}. Since $\eigvalfirst{a} > 0, \eigvalfirst{b} < 0, \eigvalsecond{0}{i} > 0, \eigvalsecond{0}{j} < 0$ and $\Varhatsecond{U}{0}[0][0] \succeq 0, \Varhatsecond{V}{0}[0][0] \preceq 0$, we get 
    \begin{align*}
        \Var{U}[+][-] = 0, \ \Var{V}[+][-] = 0, \ \Varhatsecond{U}{0}[+][-] = 0, \ \Varhatsecond{V}{0}[+][-] = 0, \ \inprod{\Varhatsecond{U}{0}[0][0]}{\Varhatsecond{V}{0}[0][0]} = 0.
    \end{align*}

    (\romannumeral5) Combining (\romannumeral1) - (\romannumeral4), we know for any $Y = U + V \in \coneK$, $U$ and $V$ have the following structure:
    \begin{align*}
        \simpleadjustbox{
        U = \MatrixNine{
            \Var{U}[+][+] ; \Var{U}[+][0] ; 0 ; 
            \sim ; \Qfirst{0} \MatrixNine{
                \Varhatsecond{U}{0}[+][+] ; \Varhatsecond{U}{0}[+][0] ; 0 ;
                \sim ; \Varhatsecond{U}{0}[0][0] \succeq 0 ; 0 ;
                \sim ; \sim ; 0 
            } (\Qfirst{0})\tran ; 0 ;
            \sim ; \sim ; 0 
        }, \ V = \MatrixNine{
            0 ; 0 ; 0 ; 
            \sim ; \Qfirst{0} \MatrixNine{
                0 ; 0 ; 0 ;
                \sim ; \Varhatsecond{V}{0}[0][0] \preceq 0 ; \Varhatsecond{V}{0}[0][-] ;
                \sim ; \sim ; \Varhatsecond{V}{0}[-][-]
            } (\Qfirst{0})\tran ; \Var{V}[0][-] ;
            \sim ; \sim ; \Var{V}[-][-] 
        }
        }.
    \end{align*}
    Thus, $\inprod{U}{\Varbar{S}} = 0$. Since $\PA U = 0$, we get $\inprod{U}{\Ssc} = 0$ from Lemma~\ref{lem:soa:X-inrpod-S}. In addition with Lemma~\ref{lem:soa:Q0-block-structure}, we get 
    \begin{align*}
        \inprod{U}{\Ssc} = \inprod{
            \Qfirst{0}_{\indexSecondSC{D} \indexSecondSC{D}} \MatrixFour{
                \Varhat{U}_{\MultiIndexSecondSC{3}{3}} ; 0 ;
                \sim ; 0 
            } (\Qfirst{0}_{\indexSecondSC{D} \indexSecondSC{D}})\tran 
        }{\Var{[\Ssc]}_{\MultiIndexFirstSC{3}{3}}} = 0.
    \end{align*}
    Since $\Var{[\Ssc]}_{\MultiIndexFirstSC{3}{3}} \succ 0$ and $\Varhat{U}_{\MultiIndexSecondSC{3}{3}} \succeq 0$, we get $\Varhat{U}_{\MultiIndexSecondSC{3}{3}} = 0$. Symmetrically, we get $\Varhat{V}_{\MultiIndexSecondSC{2}{2}} = 0$. Since $\Varhatsecond{U}{0}[0][0] \succeq 0$ and $\Varhatsecond{V}{0}[0][0] \preceq 0$, their fine-grained structures can be defined: 
    \begin{align*}
        \Varhatsecond{U}{0}[0][0] = \MatrixFourSC{
            \Varhat{U}_{\MultiIndexSecondSC{2}{2}} \succeq 0 ; 0 ;
            \sim ; 0 
        }, \quad \Varhatsecond{V}{0}[0][0] = \MatrixFourSC{
            0 ; 0 ;
            \sim ; \Varhat{V}_{\MultiIndexSecondSC{3}{3}} \preceq 0
        }.
    \end{align*}
    Notice that under this complementary structure, $\inprod{\Varhatsecond{U}{0}[0][0]}{\Varhatsecond{V}{0}[0][0]} = 0$ automatically holds. This finishes the ``$\subseteq$'' part. 

    ``$\supseteq$'': Now we shall prove that for any $U \in \coneKPX$ and $V \in \coneKPS$,~\eqref{eq:soa:proof-polarK-0} holds. To see this, $\forall W \in \Sym{n}$:
    \begin{align*}
        \inprod{V}{\opeT{W}} + \inprod{U}{\opeTP{W}}
        = \inprod{
            \Varhatsecond{V}{0}[0][0]
        }{
            \ppsim(\Varhatsecond{W}{0}[0][0])
        } + \inprod{
            \Varhatsecond{U}{0}[0][0]
        }{
            \npsim(\Varhatsecond{W}{0}[0][0])
        } \le 0 .
    \end{align*}
    This finishes the ``$\supseteq$'' part.
\end{proof}

\begin{corollary}[Relationship between $\coneC$ and $\coneKP$]
    \label{cor:soa:coneC-coneKP}
    Under Assumption~\ref{ass:soa:sc}, for any $\Zbar \in \Zopt$ and $\Hbar \in \coneC$, $\coneCX \subseteq \coneKPX$, $\coneCS \subseteq \coneKPS$. 
\end{corollary}
\begin{proof}
    Take any $H \in \coneCX$. From Proposition~\ref{prop:soa:coneC} (2) and Lemma~\ref{lem:soa:Q0-block-structure}, 
    \begin{align*}
        H = \MatrixNine{
            \Var{H}[+][+] ; \Var{H}[+][0] ; 0 ;
            \sim ; \Qfirst{0} \MatrixFourSC{
                (\Qfirst{0}_{\MultiIndexSecondPD{1}{1}})\tran H_{\MultiIndexFirstSC{2}{2}} \Qfirst{0}_{\MultiIndexSecondPD{1}{1}} ; 0 ;
                \sim ; 0 
            } (\Qfirst{0})\tran ; 0 ; 
            \sim ; \sim ; 0 
        },
    \end{align*}
    with $H_{\MultiIndexFirstSC{2}{2}} \succeq 0$ and $\PA H = 0$. Thus, $H \in \coneKPX$ from Proposition~\ref{prop:soa:polarK-XS}. The relationship between $\coneCS$ and $\coneKPS$ can be proven symmetrically. 
\end{proof}

\subsection{\titlemath{Second-Order Limits of $\Delta \Vark{X}$ and $\Delta \Vark{S}$}}
\label{sec:decouple:XS}
Under the local first- and second-order dynamics in Definition~\ref{def:soa:fod-sod}, if we initialize with $H^{(0)}=\Hbar\in\coneC$, then $\Vark{H}\equiv \Hbar$ and $\Varkpo{W}-\Vark{W}\rightarrow \mdvZ$ by Theorem~\ref{thm:soa:mdvZ}. Within this local model, it is natural to ask whether the second-order limits of the primal and dual increments also exist. We give an affirmative answer.

For the primal variable $\Vark{X}=\ppsim(\Vark{Z})$, we have
\begin{align*}
    & \Varkpo{X}-\Vark{X} = \ppsim(\Varkpo{Z})-\ppsim(\Vark{Z}) \\
    =\ & t\big(\ppsim'(\Varbar{Z}; \Varkpo{H})-\ppsim'(\Varbar{Z}; \Vark{H})\big)
    + \frac{t^2}{2}\big(\ppsim''(\Varbar{Z}; \Varkpo{H}, \Varkpo{W})-\ppsim''(\Varbar{Z}; \Vark{H}, \Vark{W})\big) + o(t^2),
\end{align*}
where $\Vark{Z}$ is of the form~\eqref{eq:soa:Vark-Z}. In the present regime, the first-order updates have stalled, \ie $\Varkpo{H}=\Vark{H}=\Varbar{H}$. Hence, it suffices to analyze the limit of
\[
\ppsim''(\Varbar{Z}; \Varbar{H}, \Varkpo{W})-\ppsim''(\Varbar{Z}; \Varbar{H}, \Vark{W})
\quad \text{as } k\rightarrow \infty.
\]
Similarly, for the dual variable $\Vark{S}=-\frac{1}{\sigma}\npsim(\Vark{Z})$, it suffices to study the limit of
\[
-\frac{1}{\sigma}\big(\npsim''(\Varbar{Z}; \Varbar{H}, \Varkpo{W})-\npsim''(\Varbar{Z}; \Varbar{H}, \Vark{W})\big) \quad \text{as } k\rightarrow \infty.
\]
To proceed, we first establish the following auxiliary lemma.

\begin{lemma}[Convergent difference of $\ppsim(\cdot)$]
    \label{lem:soa:psdcone-difference-converge}
    For a symmetric matrix sequence $\{X_k\}_{k=0}^\infty$, we have
    \begin{align*}
        \lim\limits_{k \rightarrow \infty} \big(X_{k+1}-X_k\big) = \Delta
        \Longrightarrow
        \lim\limits_{k \rightarrow \infty} \big(\ppsim(X_{k+1})-\ppsim(X_{k})\big) = \ppsim(\Delta).
    \end{align*}
\end{lemma}
\begin{proof}
    Let $Y_k = \frac{X_k}{k}$ for all $k\ge 1$, and denote $\Delta_k:=X_{k+1}-X_k$. Then
    \[
        \lim_{k\to\infty} Y_k
        = \lim_{k\to\infty}\left(\frac{X_0}{k} + \frac{1}{k}\sum_{i=0}^{k-1}\Delta_i\right)
        = \Delta.
    \]
    Since $\ppsim(\cdot)$ is positively homogeneous,
    \begin{align*}
        \ppsim(X_{k+1})-\ppsim(X_k)
        = (k+1)\ppsim(Y_{k+1}) - k\ppsim(Y_k) 
        = \ppsim(Y_{k+1}) + k\big(\ppsim(Y_{k+1})-\ppsim(Y_k)\big).
    \end{align*}
    The first term satisfies $\ppsim(Y_{k+1})\to \ppsim(\Delta)$ by continuity of $\ppsim(\cdot)$. For the second term, note that
    \[
        Y_{k+1}-Y_k = \frac{1}{k+1}\big(\Delta_k - Y_k\big),
        \qquad
        \Delta_k - Y_k \to 0,
    \]
    hence $Y_{k+1}-Y_k = o(\frac{1}{k})$ as $k\to\infty$. Since $\ppsim(\cdot)$ is $1$-Lipschitz,
    \begin{align}
        \normF{\ppsim(Y_{k+1})-\ppsim(Y_k)} \le \normF{Y_{k+1}-Y_k} = o\!\left(\frac{1}{k}\right).
    \end{align}
    Therefore, $k\big(\ppsim(Y_{k+1})-\ppsim(Y_k)\big)\to 0$, and the claim follows.
\end{proof}

\begin{theorem}[$\mdvX$ and $\mdvS$]
    \label{thm:soa:mdvX-dmvS-def}
    Under the local first- and second-order dynamics in Definition~\ref{def:soa:fod-sod}, with initialization $H^{(0)}=\Hbar\in\coneC$, the local second-order limit of $\Varkpo{X}-\Vark{X}$ is
    \begin{align}
        \label{eq:soa:mdvX}
        \mdvX := \lim_{k \rightarrow \infty} \left\{
            \ppsim''(\Varbar{Z}; \Varbar{H}, \Varkpo{W}) - \ppsim''(\Varbar{Z}; \Varbar{H}, \Vark{W})
            \right\}
            = \opeT{\mdvZ},
    \end{align}
    and the local second-order limit of $\Varkpo{S}-\Vark{S}$ is
    \begin{align}
        \label{eq:soa:mdvS}
        \mdvS := -\frac{1}{\sigma}\lim_{k \rightarrow \infty} \left\{
            \npsim''(\Varbar{Z}; \Varbar{H}, \Varkpo{W}) - \npsim''(\Varbar{Z}; \Varbar{H}, \Vark{W})
            \right\}
            = -\frac{1}{\sigma} \opeTP{\mdvZ},
    \end{align}
    where $\opeT[\Hbar]{\cdot}$ and $\opeTP[\Hbar]{\cdot}$ are defined in~\eqref{eq:soa:Theta}.
    Moreover, $\mdvZ = \mdvX - \sigma \mdvS$.
\end{theorem}
\begin{proof}
    (\romannumeral1) For the primal part, by~\eqref{eq:soa:relation-Theta-calE},
    \begin{align*}
        & \ppsim''(\Varbar{Z}; \Varbar{H}, \Varkpo{W}) - \ppsim''(\Varbar{Z}; \Varbar{H}, \Vark{W})
        = \opeT{\Varkpo{\tW}} - \opeT{\Vark{\tW}} \\
        =\ & \simpleadjustbox{
        \MatrixNine{
            \Varkpo{\tW}[+][+] - \Vark{\tW}[+][+] ;
            \Varkpo{\tW}[+][0] - \Vark{\tW}[+][0] ;
            \left\{
                \frac{\eigvalfirst{a}}{\eigvalfirst{a} - \eigvalfirst{b}} (\Varkpo{\tW}[a][b] - \Vark{\tW}[a][b])
            \right\}_{\substack{a \in \totalsetfirst[+] \\ b \in \totalsetfirst[-]}} ;
            \sim ;
            {\ppsim'(\Varbar{H}[0][0]; \Varkpo{\tW}[0][0])
            - \ppsim'(\Varbar{H}[0][0]; \Vark{\tW}[0][0])} ;
            0 ;
            \sim ;
            \sim ;
            0
        }.
        }
    \end{align*}
    Since $\Varkpo{\tW}-\Vark{\tW}\rightarrow \mdvZ$ as $k\rightarrow \infty$ by Theorem~\ref{thm:soa:mdvZ}, we have
    \begin{align*}
        \Varkpo{\tW}[a][b]-\Vark{\tW}[a][b] \rightarrow \Var{\mdvZ}[a][b], \quad
        \forall a \in \totalsetfirst,\ \forall b \in \totalsetfirst.
    \end{align*}
    Thus, it remains to handle the only nonlinear term
    $\ppsim'(\Varbar{H}[0][0]; \Varkpo{\tW}[0][0]) - \ppsim'(\Varbar{H}[0][0]; \Vark{\tW}[0][0])$.
    For any $W\in\Sym{n}$,
    \begin{align*}
        \simpleadjustbox{
        \ppsim'(\Varbar{H}[0][0]; \Var{W}[0][0])
        = \Qfirst{0} \MatrixNine{
            \Varsecond{\Varhat{W}}{0}[+][+] ;
            \Varsecond{\Varhat{W}}{0}[+][0] ;
            \left\{
                \frac{\eigvalsecond{0}{i}}{\eigvalsecond{0}{i} - \eigvalsecond{0}{j}} (\Varsecond{\Varhat{W}}{0}[i][j])
            \right\}_{\substack{i \in \totalsetsecond{0}[+] \\ j \in \totalsetsecond{0}[-]}} ;
            \sim ;
            \ppsim(\Varsecond{\Varhat{W}}{0}[0][0]) ;
            0 ;
            \sim ;
            \sim ;
            0
        } (\Qfirst{0})\tran
        }
    \end{align*}
    where $\Varhat{W} = (\Qfirst{0})\tran \Var{W}[0][0] \Qfirst{0}$. Again, the only nonlinear component is the PSD projector located at the $\MultiIndexSecond{2}{2}$ block.
    By Lemma~\ref{lem:soa:psdcone-difference-converge},
    \[
        \ppsim\!\big(\Varsecond{\Varhat{W}}{0}[0][0]^{(k+1)}\big)
        - \ppsim\!\big(\Varsecond{\Varhat{W}}{0}[0][0]^{(k)}\big)
        \rightarrow \ppsim\!\big(\Varhatsecond{\phi}{0}[0][0]\big)
        \quad \text{as } k\rightarrow \infty,
    \]
    where $\Varhat{\phi}:=(\Qfirst{0})\tran \mdvZ_{\MultiIndexFirst{2}{2}} \Qfirst{0}$. Therefore, as $k \rightarrow \infty$,
    \begin{align*}
        & \ppsim'(\Varbar{H}[0][0]; \Varkpo{\Vartilde{W}}[0][0]) - \ppsim'(\Varbar{H}[0][0]; \Vark{\Vartilde{W}}[0][0]) \rightarrow \\
        & \Qfirst{0} \MatrixNine{
            \Varsecond{\Varhat{\phi}}{0}[+][+] ;
            \Varsecond{\Varhat{\phi}}{0}[+][0] ;
            \left\{
                \frac{\eigvalsecond{0}{i}}{\eigvalsecond{0}{i} - \eigvalsecond{0}{j}} (\Varsecond{\Varhat{\phi}}{0}[i][j])
            \right\}_{\substack{i \in \totalsetsecond{0}[+] \\ j \in \totalsetsecond{0}[-]}} ;
            \sim ;
            \ppsim(\Varsecond{\Varhat{\phi}}{0}[0][0]) ;
            0 ;
            \sim ;
            \sim ;
            0
        } (\Qfirst{0})\tran
        = \ppsim'(\Varbar{H}[0][0]; \Var{\mdvZ}[0][0]),
    \end{align*}
    which implies that
    \[
        \ppsim''(\Varbar{Z}; \Varbar{H}, \Varkpo{W}) - \ppsim''(\Varbar{Z}; \Varbar{H}, \Vark{W})
        \rightarrow \opeT{\mdvZ}.
    \]

    (\romannumeral2) The dual part follows by symmetry: one can similarly show that
    \[
        \npsim''(\Varbar{Z}; \Varbar{H}, \Varkpo{W}) - \npsim''(\Varbar{Z}; \Varbar{H}, \Vark{W})
        \rightarrow \opeTP{\mdvZ}
        \quad \text{as } k\rightarrow \infty.
    \]
    The final identity $\mdvZ = \mdvX - \sigma \mdvS$ follows from
    $\mdvZ = \opeT{\mdvZ} + \opeTP{\mdvZ}$ in~\eqref{eq:soa:relation-Theta-calE}.
\end{proof}

\subsection{\titlemath{Primal--Dual Decoupling of $\mdvZ$}}
\label{sec:decouple:mdvZ}

Theorem~\ref{thm:soa:mdvX-dmvS-def} connects $\mdvX$ (resp.\ $\mdvS$) with the limiting behavior of $\Varkpo{X}-\Vark{X}$ (resp.\ $\Varkpo{S}-\Vark{S}$). The next theorem further reveals a deeper connection between $\mdvX$ (resp.\ $\mdvS$) and $\coneKPX$ (resp.\ $\coneKPS$).

\begin{theorem}[Primal--dual decoupling of $\mdvZ$]
    \label{thm:soa:primal-dual-decouple}
    Let $\mdvX$ be defined in~\eqref{eq:soa:mdvX} and $\mdvS$ in~\eqref{eq:soa:mdvS}. Let $\coneKPX$ and $\coneKPS$ be defined in~\eqref{eq:soa:polarK-XS}. Then, under Assumption~\ref{ass:soa:sc}, 
    \begin{subequations}
        \label{eq:soa:primal-dual-decoupling}
        \begin{align}
            & \mdvX = \argmin_{W \in \coneKPX} \normF{W + \opeEP}^2 = \Pi_{\coneKPX}(-\opeEP), \\
            & \mdvS = -\frac{1}{\sigma} \argmin_{W \in \coneKPS} \normF{W + \opeE}^2 = -\frac{1}{\sigma} \Pi_{\coneKPS}(-\opeE).
        \end{align}
    \end{subequations}
    where $\opeE$ and $\opeEP$ are defined in~\eqref{eq:soa:calE}.
\end{theorem}
\begin{proof}
    Since $\coneKP$ is closed and convex, the optimal solution of $\inf_{W \in \coneKP} \normF{W - \opeP}^2$ can be attained and is unique, where $\opeP$ is defined in~\eqref{eq:soa:sod-simplified}. Additionally, with Proposition~\ref{prop:soa:polarK-XS}, we get 
    \begin{align*}
        & \mdvZ = \Pi_{\coneKP} \opeP = \argmin_{W \in \coneKP} \normF{W - \opeP}^2 \\
        = & \argmin_{
            \substack{
                W = U + V, \\
                U \in \coneKPX, V \in \coneKPS
            }
        } \normF{U + V - \opeP}^2 \\
        = & \argmin_{
            \substack{
                W = U + V, \\
                U \in \coneKPX, V \in \coneKPS
            }
        } \normF{U + V + \PA \opeE + \PAp \opeEP}^2.
    \end{align*}
    Since $U \in \coneKPX$, we get $\PA U = 0$. Symmetrically, $\PAp V = 0$. Therefore, 
    \begin{align*}
        & \normF{U + V + \PA \opeE + \PAp \opeEP}^2 \\
        = & \normF{\PAp U + \PA V + \PA \opeE + \PAp \opeEP}^2 
        = \normF{\PAp U + \PAp \opeEP}^2 + \normF{\PA V + \PA \opeE}^2 \\
        = & \normF{U + \opeEP}^2 + \normF{V + \opeE}^2 - \normF{\PA \opeEP}^2 - \normF{\PAp \opeE}^2,
    \end{align*}
    where in the last equality, we use the property that
    \begin{align*}
        & \simpleadjustbox{
            \normF{U + \opeEP}^2 = \normF{\PA U + \PA \opeEP}^2 + \normF{\PAp U + \PAp \opeEP}^2 
        = \normF{\PA \opeEP}^2 + \normF{\PAp U + \PAp \opeEP}^2,
         } \\
        & \simpleadjustbox{
            \normF{V + \opeE}^2 = \normF{\PA V + \PA \opeE}^2 + \normF{\PAp V + \PAp \opeE}^2 
        = \normF{\PA V + \PA \opeE}^2 + \normF{\PAp \opeE}^2.
        }
    \end{align*}
    Notice that $- \normF{\PA \opeEP}^2 - \normF{\PAp \opeE}^2$ is a constant and does not affect the $\argmin$.
    After observing that $U \in \coneKPX$ and $V \in \coneKPS$ are totally decoupled in the objective, we get
    \begin{align*}
        \mdvZ = \underbrace{
            \argmin_{U \in \coneKPX} \normF{U + \opeEP}^2
        }_{=: \Varbar{U}} + 
        \underbrace{
            \argmin_{V \in \coneKPS} \normF{V + \opeE}^2
        }_{=: \Varbar{V}},
    \end{align*}
    where $\Varbar{U}$ (resp. $\Varbar{V}$) is attainable and unique since $\coneKPX$ (resp. $\coneKPS$) is closed and convex. Now we shall prove that $\Varbar{U} = \mdvX$ and $\Varbar{V} = -\sigma \mdvS$. From Proposition~\ref{prop:soa:polarK-XS}, 
    \begin{align*}
        \simpleadjustbox{
        \mdvZ = \Varbar{U} + \Varbar{V} = \MatrixNine{
                \Varbar{U}_{\MultiIndexFirst{1}{1}} ; \Varbar{U}_{\MultiIndexFirst{1}{2}} ; 0 ;
                \sim ; 
                {\Qfirst{0} \MatrixSixteenSC{
                    \Varhat{U}_{\MultiIndexSecondSC{1}{1}} ; \Varhat{U}_{\MultiIndexSecondSC{1}{2}} ; \Varhat{U}_{\MultiIndexSecondSC{1}{3}} ; 0 ; 
                    \sim ; \Varhat{U}_{\MultiIndexSecondSC{2}{2}} \succeq 0 ; 0 ; \Varhat{V}_{\MultiIndexSecondSC{2}{4}} ;
                    \sim ; \sim ; \Varhat{V}_{\MultiIndexSecondSC{3}{3}} \preceq 0 ; \Varhat{V}_{\MultiIndexSecondSC{3}{4}} ; 
                    \sim ; \sim ; \sim ; \Varhat{V}_{\MultiIndexSecondSC{4}{4}} 
                } (\Qfirst{0})\tran}
                 ; \Varbar{V}_{\MultiIndexFirst{2}{3}} ; 
                \sim ; \sim ; \Varbar{V}_{\MultiIndexFirst{3}{3}} 
            },
        }
    \end{align*}
    where $\Varhat{U} = (\Qfirst{0})\tran \Varbar{U}_{\MultiIndexFirst{2}{2}} \Qfirst{0}$ and $\Varhat{V} = (\Qfirst{0})\tran \Varbar{V}_{\MultiIndexFirst{2}{2}} \Qfirst{0}$. Then, from Theorem~\ref{thm:psd:first-dd},
    \begin{align*}
        & \ppsim'(\Varbar{H}[0][0]; \Var{\mdvZ}[0][0]) 
        = \ppsim'(\Varbar{H}[0][0]; \Varbar{U}[0][0] + \Varbar{V}[0][0]) 
        = \Qfirst{0} \ppsim'(\Varhat{H}; \Varhat{U} + \Varhat{V}) (\Qfirst{0})\tran \\
        = & \Qfirst{0} \MatrixSixteenSC{
            \Varhat{U}_{\MultiIndexSecondSC{1}{1}} ; \Varhat{U}_{\MultiIndexSecondSC{1}{2}} ; \Varhat{U}_{\MultiIndexSecondSC{1}{3}} ; 0 ; 
            \sim ; \Varhat{U}_{\MultiIndexSecondSC{2}{2}} ; 0 ; 0 ; 
            \sim ; \sim ; 0 ; 0 ; 
            \sim ; \sim ; \sim ; 0 
        } (\Qfirst{0})\tran,
    \end{align*}
    where $\Varhat{H} := (\Qfirst{0})\tran \Varbar{H}_{\MultiIndexFirst{2}{2}} \Qfirst{0}$ is diagonal. Thus, from~\eqref{eq:soa:Theta},
    \begin{align*}
        \opeT[\Hbar]{\mdvZ} = \MatrixNine{
            \Varbar{U}_{\MultiIndexFirst{1}{1}} ; \Varbar{U}_{\MultiIndexFirst{1}{2}} ; 0 ;
            \sim ; 
            {\Qfirst{0} \MatrixSixteenSC{
                \Varhat{U}_{\MultiIndexSecondSC{1}{1}} ; \Varhat{U}_{\MultiIndexSecondSC{1}{2}} ; \Varhat{U}_{\MultiIndexSecondSC{1}{3}} ; 0 ; 
                \sim ; \Varhat{U}_{\MultiIndexSecondSC{2}{2}} \succeq 0 ; 0 ; 0 ;
                \sim ; \sim ; 0 ; 0 ; 
                \sim ; \sim ; \sim ; 0
            } (\Qfirst{0})\tran}
                ; 0 ; 
            \sim ; \sim ; 0 
        } = \Varbar{U},
    \end{align*}
    where in the last equality, we use $\Varbar{U} \in \coneKPX$ and Proposition~\ref{prop:soa:polarK-XS} again. Symmetrically, $\Varbar{V} = \opeTP[\Hbar]{\mdvZ}$. We close the proof by recalling Theorem~\ref{thm:soa:mdvX-dmvS-def}: $\opeT[\Hbar]{\mdvZ} = \mdvX$ and $\opeTP[\Hbar]{\mdvZ} = -\sigma \mdvS$. 
\end{proof}

%% file: sections/mdvZ_kernel.tex

\section{Kernel of \titlemath{$\mdvZ[\Zbar; \cdot]$}}
\label{sec:mdvZ-kernel}

The first property of $\mdvZmap$ that we study is its kernel $\ker(\mdvZmap)$, \ie $\{\Hbar \in \coneC \mid \mdvZ = 0\}$. The set $\ker(\mdvZmap)$ directly characterizes when the local second-order limiting dynamics~\eqref{eq:soa:limiting-dynamics-def} is \emph{effective}, in the sense that the higher-order term $o(\normF{\Vark{Z}-\Zbar}^2)$ in~\eqref{eq:soa:limiting-dynamics-def} can be neglected. When $\Hbar \in \coneC \backslash \ker(\mdvZmap)$, we have $\mdvZ \ne 0$, and the second-order term dominates the evolution in~\eqref{eq:soa:limiting-dynamics-def}. Conversely, if $\Hbar \in \ker(\mdvZmap)$ (\ie the second-order term vanishes) yet the true one-step ADMM iteration~\eqref{eq:intro:one-step-admm} does not converge, then higher-order terms must be taken into account. It turns out that $\ker(\mdvZmap)$ admits a clean characterization.

\begin{proposition}[Kernel of $\mdvZmap$]
    \label{prop:mdvZ-kernel:kernel}
    Under Assumption~\ref{ass:soa:sc}, for any $\Zbar \in \Zopt$, $\ker(\mdvZmap) = \coneT$.
\end{proposition}
The proof of Proposition~\ref{prop:mdvZ-kernel:kernel} is divided into two parts. First, we show that $\mdvZ \ne 0$ for any $\Hbar \in \coneC \backslash \coneT$ (Lemma~\ref{lemma:mdvZ-kernel:direction-1} in \S\ref{sec:mdvZ-kernel-1}). Second, we show that $\mdvZ = 0$ for any $\Hbar \in \coneT$ (Lemma~\ref{lemma:mdvZ-kernel:direction-2} in \S\ref{sec:mdvZ-kernel-2}).

In \S\ref{sec:mdvZ-kernel-3}, we discuss one implication of $\ker(\mdvZmap)$. In slow-convergence regions of the one-step ADMM iteration~\eqref{eq:intro:one-step-admm}, a typical pattern is that $\angle(\Delta \Vark{Z}, \Delta \Varkpo{Z})$ tends to be very small yet is generally nonzero. We explain this phenomenon using our local second-order limiting dynamics model~\eqref{eq:soa:limiting-dynamics-def}, with the initialization $Z^{(0)}$ chosen in $\Zbar + (\coneC \backslash \coneT)$.

\subsection{Proof of \titlemath{``$\Hbar \in \coneC \backslash \coneT \Longrightarrow \mdvZ \ne 0$''}}
\label{sec:mdvZ-kernel-1}

The motivation for this part comes from Sturm's square-root error bound under the existence of a strictly complementary primal--dual pair~\cite{sturm00siopt-error-bound-lmi}. Since $\Varbar{H} \in \coneC \backslash \coneT$, the forward error $\dist(\Varbar{Z} + t\Varbar{H}, \Zopt)$ is of order $t$. Consequently, under Assumption~\ref{ass:soa:sc}, the backward error $\delta(\Varbar{Z} + t\Varbar{H})$ must exhibit a nonzero response at order $t^2$.

\begin{lemma}[$\Hbar \in \coneC \backslash \coneT \Longrightarrow \mdvZ \ne 0$]
    \label{lemma:mdvZ-kernel:direction-1}
    Under Assumption~\ref{ass:soa:sc}, pick any $\Varbar{H} \in \coneC$. If $\mdvZ = 0$, then $\Varbar{H} \in \coneT$.
\end{lemma}
\begin{proof}    
    We aim to show that if $\mdvZ = 0$, then $\VarFirstZeroD{\Varbar{H}}{r}[+] = 0$ and $\VarFirstZeroP{\Varbar{H}}{l}[-] = 0$. From Theorem~\ref{thm:soa:mdvZ}, if $\mdvZ = \opeP - \Pi_{\coneK}(\opeP) = 0$, we have $\opeP \in \coneK$. By~\eqref{eq:soa:calK}, there exists a convergent sequence $\{\Psi^i\}_{i=1}^\infty \rightarrow \opeP$, such that for each $\Psi^i$, there exists $W^i \in \Sn$ with:
    \begin{align*}
        \PA \opeT{W^i} + \PAp \opeTP{W^i} = \Psi^i.
    \end{align*}
    By the definition of $\{\Psi^i\}_{i=1}^\infty$, $\forall \epsilon > 0$, $\exists N_\epsilon \in \mathbb{N}$, such that $\forall i \ge N_\epsilon$, $\normF{\opeP - \Psi^i} \le \epsilon$. Substituting $\opeP$'s formula from~\eqref{eq:soa:sod-simplified}:
    \begin{align*}
        & \normF{
            \PA \opeT{W^i} + \PAp \opeTP{W^i} 
            - (-\PA \opeE - \PAp \opeEP)
        } \le \epsilon \\
        \Longrightarrow & \begin{cases}
            & \normF{\PA \{ 
                \opeT{W^i} + \opeE
             \}} \le \epsilon \\
            & \normF{\PAp \{ 
                \opeTP{W^i} + \opeEP
             \}} \le \epsilon
        \end{cases}.
    \end{align*}
    We first focus on the primal part: $\normF{\PA \left\{ 
        \opeT{W^i} + \opeE
    \right\}} \le \epsilon$. Perform an expansion for $\opeT{W^i} + \opeE$ from~\eqref{eq:soa:Theta} and~\eqref{eq:soa:calE}:
    \begin{align*}
        \simpleadjustbox{
        \opeT{W^i} + \opeE = 
            \MatrixNine{
                \substack{
                    \Var{W}[+][+]^i
                } 
                ;
                \left\{ 
                \substack{
                    \Var{W}[a][0]^i \\
                    - 2 \frac{1}{\eigvalfirst{a}} \Varbar{H}[a][0] \ppsim(-\Varbar{H}[0][0])
                } 
                \right\}_{a \in \totalsetfirst[+]} 
                ; 
                \left\{ 
                \substack{
                    \frac{\eigvalfirst{a}}{\eigvalfirst{a} - \eigvalfirst{b}} \Var{W}[a][b]^i \\
                    + 2 \frac{1}{\eigvalfirst{a} - \eigvalfirst{b}} \Varbar{H}[a][0] \Varbar{H}[0][b] 
                } 
                \right\}_{\substack{a \in \totalsetfirst[+] \\ b \in \totalsetfirst[-]}} 
                ;
                \sim
                ;
                \substack{
                    \ppsim'(\Varbar{H}[0][0]; \Var{W}[0][0]^i) \\
                    + 2 \sum\limits_{c \in \totalsetfirst[+]} \frac{1}{\eigvalfirst{c}} \Varbar{H}[0][c] \Varbar{H}[c][0] 
                } 
                ;
                \left\{ 
                \substack{
                    2 \frac{1}{-\eigvalfirst{b}} \ppsim(\Varbar{H}[0][0]) \Varbar{H}[0][b]
                }
                \right\}_{b \in \totalsetfirst[-]} 
                ;
                \sim 
                ;
                \sim 
                ;
                \substack{
                    0
                }
            }.
        } 
    \end{align*}
    Now our goal is to show 
    \begin{align}
        \label{eq:mdvZ-kernel:direction1-1}
        \simpleadjustbox{
        [\opeT{W^i} + \opeE]_{\indexSC{D} \indexSC{D}}
        = \MatrixFour{
            {[\opeT{W^i} + \opeE]_{\indexZeroD \indexZeroD}} ; 
            0 ;
            \sim ; 
            0 
        } \text{ with } [\opeT{W^i} + \opeE]_{\indexZeroD \indexZeroD} \succeq 0.
        }
    \end{align}
    (a) For $2 \frac{1}{-\eigvalfirst{b}} \ppsim(\Varbar{H}[0][0]) \Varbar{H}[0][b], \ \forall b \in \totalsetfirst[-]$, we notice from Proposition~\ref{prop:soa:coneC} (2) that 
    \begin{align*}
        \ppsim(\Varbar{H}[0][0]) = \MatrixFourSC{
            [\ppsim(\Varbar{H}[0][0])]_{\indexSecondSC{P} \indexSecondSC{P}} ;
            0 ;
            \sim ;
            0 
        }.
    \end{align*}
    Thus,
    \begin{align*}
        \ppsim(\Varbar{H}[0][0]) \Varbar{H}[0][b] 
        = \MatrixFourSC{
            [\ppsim(\Varbar{H}[0][0])]_{\indexSecondSC{P} \indexSecondSC{P}} ;
            0 ;
            \sim ;
            0 
        } \VectorTwoColumn{
            \Varbar{H}_{\indexZeroP \indexfirst{b}} ;
            \Varbar{H}_{\indexZeroD \indexfirst{b}}
        } = \VectorTwoColumn{
            [\ppsim(\Varbar{H}[0][0])]_{\indexSecondSC{P} \indexSecondSC{P}} \Varbar{H}_{\indexZeroP \indexfirst{b}} ; 
            0
        },
    \end{align*}
    which directly implies $[\opeT{W^i} + \opeE]_{\indexZeroD \indexfirst{-}} = 0$.
    \\
    (b) For $[\opeT{W^i} + \opeE]_{\indexZeroD \indexZeroD} = 
            [\ppsim'(\Varbar{H}[0][0]; \Var{W}[0][0]^i) 
            + 2 \sum\limits_{c \in \totalsetfirst[+]} \frac{1}{\eigvalfirst{c}} \Varbar{H}[0][c] \Varbar{H}[c][0]]_{\indexSecondSC{D} \indexSecondSC{D}} $:\\
            since $2 \sum\limits_{c \in \totalsetfirst[+]} \frac{1}{\eigvalfirst{c}} \Varbar{H}[0][c] \Varbar{H}[c][0] \succeq 0$, all we need to prove is $[\ppsim'(\Varbar{H}[0][0]; \Var{W}[0][0]^i)]_{\indexSecondSC{D} \indexSecondSC{D}} \succeq 0$. From~\eqref{eq:psd:first-dd-nondiagonal-Z}, we have 
    \begin{align*}
        & \ppsim'(\Varbar{H}[0][0]; \Var{W}[0][0]^i) = 
        \Qfirst{0} \left\{ 
            \ppsim'\left(
                (\Qfirst{0})\tran \Varbar{H}[0][0] \Qfirst{0},
                (\Qfirst{0})\tran \Var{W}[0][0]^i \Qfirst{0}
            \right)
         \right\} (\Qfirst{0})\tran = 
        \\
        & \simpleadjustbox{
        \Qfirst{0}
        \MatrixSixteenSC{
            \Varhat{W}_{\indexsecond{0}{+} \indexsecond{0}{+}} ; 
            \Varhat{W}_{\indexsecond{0}{+} \indexSecondZeroP} ; 
            \Varhat{W}_{\indexsecond{0}{+} \indexSecondZeroD} ; 
            \left\{
                \frac{\eigvalsecond{0}{i}}{\eigvalsecond{0}{i} - \eigvalsecond{0}{j}}
                \Varhat{W}_{\indexsecond{0}{i} \indexsecond{0}{j}}
            \right\}_{\substack{
                i \in \totalsetsecond{0}[+] \\ j \in \totalsetsecond{0}[-]
            }} ;
            \sim ;
            [
                \ppsim(\Varhat{W}_{\indexsecond{0}{0} \indexsecond{0}{0}})
            ]_{\indexThirdSC{P} \indexThirdSC{P}} ;
            [
                \ppsim(\Varhat{W}_{\indexsecond{0}{0} \indexsecond{0}{0}})
            ]_{\indexThirdSC{P} \indexThirdSC{D}} ; 
            0 ;
            \sim ;
            \sim ;
            [
                \ppsim(\Varhat{W}_{\indexsecond{0}{0} \indexsecond{0}{0}})
            ]_{\indexThirdSC{D} \indexThirdSC{D}} ;
            0 ;
            \sim ;
            \sim ; 
            \sim ; 
            0
        }
        (\Qfirst{0})\tran 
        },
    \end{align*}
    where we abbreviate $(\Qfirst{0})\tran \Var{W}[0][0]^i \Qfirst{0}$ as $\Varhat{W}$. $\indexThirdSC{P}$ and $\indexThirdSC{D}$ divide $\ppsim(\Varhat{W}_{\indexsecond{0}{0} \indexsecond{0}{0}})$'s indices by the primal and dual part. Moreover, from Lemma~\ref{lem:soa:Q0-block-structure}, $\Qfirst{0}_{\indexSecondSC{P} \indexSecondSC{P}} \in \calO^{\sizeof{\indexSecondSC{P}}}(\Varbar{H}_{\indexZeroP \indexZeroP})$ and $\Qfirst{0}_{\indexSecondSC{D} \indexSecondSC{D}} \in \calO^{\sizeof{\indexSecondSC{D}}}(\Varbar{H}_{\indexZeroD \indexZeroD})$. Therefore, 
    \begin{align*}
        [\ppsim'(\Varbar{H}[0][0]; \Var{W}[0][0]^i)]_{\indexSecondSC{D} \indexSecondSC{D}} = \Qfirst{0}_{\indexSecondSC{D} \indexSecondSC{D}} \MatrixFour{
            [
            \ppsim(\Varhat{W}_{\indexsecond{0}{0} \indexsecond{0}{0}})
            ]_{\indexThirdSC{D} \indexThirdSC{D}}
            ; 
            0
            ;
            \sim
            ;
            0
        } (\Qfirst{0}_{\indexSecondSC{D} \indexSecondSC{D}})\tran \succeq 0.
    \end{align*}
    Combining (a) and (b), we prove~\eqref{eq:mdvZ-kernel:direction1-1}.

    On the other hand, we notice that $\inprod{\opeT{W^i} + \opeE}{\Varbar{S}} = 0$ regardless of the choice of  $W^i$. From Lemma~\ref{lem:soa:X-inrpod-S}, 
    \begin{align*}
        \abs{\inprod{\opeT{W^i} + \opeE}{\Ssc}} \le  \epsilon \normF{\Ssc - \Varbar{S}}.
    \end{align*}
    Since $[\Ssc]_{\MultiIndexFirstSC{3}{3}} \succ 0$, $[\opeT{W^i} + \opeE]_{\MultiIndexFirstSC{3}{3}} \succeq 0$, and
    \begin{align*}
        \inprod{\opeT{W^i} + \opeE}{\Ssc} 
        = \inprod{
            [\opeT{W^i} + \opeE]_{\MultiIndexFirstSC{3}{3}}
        }{
            [\Ssc]_{\MultiIndexFirstSC{3}{3}}
        }  
    \end{align*}
    from~\eqref{eq:mdvZ-kernel:direction1-1}, we have 
    \begin{align*}
        & \normF{[\opeT{W^i} + \opeE]_{\MultiIndexFirstSC{3}{3}}}
        \le \trace{[\opeT{W^i} + \opeE]_{\MultiIndexFirstSC{3}{3}}} \\
        \le & \frac{
            \inprod{
                [\opeT{W^i} + \opeE]_{\MultiIndexFirstSC{3}{3}}
            }{
                [\Ssc]_{\MultiIndexFirstSC{3}{3}}
            }
        }{\lambda_{\min}([\Ssc]_{\MultiIndexFirstSC{3}{3}})}
        \le \frac{\normF{\Ssc - \Varbar{S}}}{\lambda_{\min}([\Ssc]_{\MultiIndexFirstSC{3}{3}})} \cdot \epsilon.
    \end{align*}
    The first inequality comes from the property that $\normF{A} \le \trace{A}$ if $A \succeq 0$. The second inequality is because for any $A \succeq 0$ and $B \succ 0$, $\lambda_\min(B) \cdot \trace{A} \le \inprod{A}{B}$.
    Together with $[\ppsim'(\Varbar{H}[0][0]; \Var{W}[0][0]^i)]_{\indexSecondSC{D} \indexSecondSC{D}} \succeq 0$, $[2 \sum\limits_{c \in \totalsetfirst[+]} \frac{1}{\eigvalfirst{c}} \Varbar{H}[0][c] \Varbar{H}[c][0]]_{\indexSecondSC{D} \indexSecondSC{D}} \succeq 0$, and 
    \begin{align*}
        \simpleadjustbox{
            [\opeT{W^i} + \opeE]_{\indexZeroD \indexZeroD} = 
            [\ppsim'(\Varbar{H}[0][0]; \Var{W}[0][0]^i) 
            + 2 \sum\limits_{c \in \totalsetfirst[+]} \frac{1}{\eigvalfirst{c}} \Varbar{H}[0][c] \Varbar{H}[c][0]]_{\indexSecondSC{D} \indexSecondSC{D}},
        }
    \end{align*}
    we get 
    \begin{align*}
        \normlongF{
            [2 \sum\limits_{c \in \totalsetfirst[+]} \frac{1}{\eigvalfirst{c}} \Varbar{H}[0][c] \Varbar{H}[c][0]]_{\indexSecondSC{D} \indexSecondSC{D}}
        } \le \frac{\normF{\Ssc - \Varbar{S}}}{\lambda_{\min}([\Ssc]_{\indexSC{D} \indexSC{D}})} \cdot \epsilon.
    \end{align*}
    Observing that the above error bound does not contain $W^i$ and $\epsilon$ could be picked arbitrarily small, we get 
    \begin{align*}
        \normlongF{
            [2 \sum\limits_{c \in \totalsetfirst[+]} \frac{1}{\eigvalfirst{c}} \Varbar{H}[0][c] \Varbar{H}[c][0]]_{\indexSecondSC{D} \indexSecondSC{D}}
        } = 0.
    \end{align*}
    Due to the positive semi-definiteness of $\{\Varbar{H}[0][c] \Varbar{H}[c][0]\}_{c \in \totalsetfirst[+]}$'s, we get 
    \begin{align*}
        \normlongF{\Varbar{H}_{\indexZeroD \indexfirst{c}} \Varbar{H}_{\indexfirst{c} \indexZeroD }} = \normlongF{[\Varbar{H}[0][c] \Varbar{H}[c][0]]_{\indexSecondSC{D} \indexSecondSC{D}}} = 0, \quad \forall c \in \totalsetfirst[+].
    \end{align*}
    Therefore, $\Varbar{H}_{\indexfirst{+} \indexZeroD} = 0$. By primal--dual symmetry, $\Varbar{H}_{\indexfirst{-} \indexZeroP} = 0$.
    Finally, by Proposition~\ref{prop:soa:coneT} (2), we get $\Hbar \in \coneT$.
\end{proof}

\subsection{\titlemath{Proof of ``$\Hbar \in \coneT \Longrightarrow \mdvZ = 0$''}}
\label{sec:mdvZ-kernel-2}

Intuitively, if $\Varbar{H} \in \ri(\calT_{\Zopt}(\Varbar{Z}))$, then $\mdvZ = 0$. The following stronger result shows that $\mdvZ$ will vanish even when $\Varbar{H} \in \calT_{\Zopt}(\Varbar{Z}) \backslash \ri(\calT_{\Zopt}(\Varbar{Z}))$. 

\begin{lemma}[$\Hbar \in \coneT \Longrightarrow \mdvZ = 0$]
    \label{lemma:mdvZ-kernel:direction-2}
    Under Assumption~\ref{ass:soa:sc}, if $\Varbar{H} \in \calT_{\Zopt}(\Varbar{Z})$, then $\mdvZ = 0$.
\end{lemma}
\begin{proof}
    Proof by construction. Since $(\Xsc, \Ssc)$ is a strictly complementary pair,  $\Var{[\Xsc]}_{\indexZeroP \indexZeroP} \succ 0$ and $\Var{[\Ssc]}_{\indexZeroD \indexZeroD} \succ 0$. Define two constants
    \begin{align}
        \label{eq:mdvZ-kernel:kappa-PD}
        \kappa_\Primal = \frac{
            \lambda_{\max} (2 \sum_{c \in \totalsetfirst[+]}\frac{1}{\eigvalfirst{c}} \Varbar{H}_{\indexZeroP \indexfirst{c}} \Varbar{H}_{\indexfirst{c} \indexZeroP})
        }{
            \lambda_{\min} (\Var{[\Xsc]}_{\indexZeroP \indexZeroP})
        }, \quad 
        \kappa_\Dual = \frac{
            \lambda_{\max} (-2 \sum_{c \in \totalsetfirst[-]}\frac{1}{\eigvalfirst{c}} \Varbar{H}_{\indexZeroD \indexfirst{c}} \Varbar{H}_{\indexfirst{c} \indexZeroD})
        }{
            \lambda_{\min} (\Var{[\sigma \Ssc]}_{\indexZeroD \indexZeroD})
        }.
    \end{align}
    Construct $W$ as follows:
    \begin{align*}
        \simpleadjustbox{
        \MatrixSixteenSC{
            \kappa_\Primal \cdot [\Var{\Xsc} - \Varbar{X}]_{\indexfirst{+} \indexfirst{+}} ;
            \kappa_\Primal \cdot \Var{[\Xsc]}_{\indexfirst{+} \indexZeroP} ;
            0 ; 
            0 ;
            \sim ;
            \kappa_\Primal \cdot \Var{[\Xsc]}_{\indexZeroP \indexZeroP} - 2 \sum\limits_{c \in \totalsetfirst[+]}\frac{1}{\eigvalfirst{c}} \Varbar{H}_{\indexZeroP \indexfirst{c}} \Varbar{H}_{\indexfirst{c} \indexZeroP} ;
            0 ;
            0 ;
            \sim ;
            \sim ;
            -\kappa_\Dual \cdot \Var{[\sigma \Ssc]}_{\indexZeroD \indexZeroD} - 2 \sum\limits_{c \in \totalsetfirst[-]}\frac{1}{\eigvalfirst{c}} \Varbar{H}_{\indexZeroD \indexfirst{c}} \Varbar{H}_{\indexfirst{c} \indexZeroD} ;
            -\kappa_\Dual \cdot \Var{[ \sigma \Ssc]}_{\indexZeroD \indexfirst{-}} ;
            \sim ; 
            \sim ;
            \sim ;
            -\kappa_\Dual \cdot [\sigma \Var{\Ssc} - \sigma \Varbar{S}]_{\indexfirst{-} \indexfirst{-}}
        }.
        }
    \end{align*}
    We shall prove $\opeP = \PA \opeT{W} + \PAp \opeTP{W}$, which implies $\opeP \in \coneK$. 

    (i) Primal part. Since $\Varbar{H} \in \calT_{\Zopt}(\Varbar{Z})$, we have $\Varbar{H}_{\indexfirst{a} \indexZeroD} = 0, \forall a \in \totalsetfirst[+]$ and $\Varbar{H}_{\indexZeroP \indexfirst{b}} = 0, \forall b \in \totalsetfirst[-]$. In addition with  
    \begin{align*}
        \Varbar{H}[0][0] = \MatrixFourSC{
            [\Varbar{H}[0][0]]_{\indexSecondSC{P} \indexSecondSC{P}} ; 0 ;
            \sim ; [\Varbar{H}[0][0]]_{\indexSecondSC{D} \indexSecondSC{D}}
        }, \quad \text{where } [\Varbar{H}[0][0]]_{\indexSecondSC{P} \indexSecondSC{P}} \succeq 0, \ [\Varbar{H}[0][0]]_{\indexSecondSC{D} \indexSecondSC{D}} \preceq 0,
    \end{align*}
    we have 
    \begin{align*}
        -2 \frac{1}{\eigvalfirst{a}} \Varbar{H}[a][0] \ppsim(-\Varbar{H}[0][0])
        & = -2 \frac{1}{\eigvalfirst{a}} \VectorTwoRow{
            \Varbar{H}_{\indexfirst{a} \indexZeroP} ; 0
        } \MatrixFourSC{
            0 ; 0 ;
            \sim ; -[\Varbar{H}[0][0]]_{\indexSecondSC{D} \indexSecondSC{D}}
        } = 0, \quad \forall a \in \totalsetfirst[+], \\
        2 \frac{1}{-\eigvalfirst{b}} \ppsim(\Varbar{H}[0][0]) \Varbar{H}[0][b] 
        & = 2 \frac{1}{-\eigvalfirst{b}} \MatrixFourSC{
            [\Varbar{H}[0][0]]_{\indexSecondSC{P} \indexSecondSC{P}} ; 0 ;
            \sim ; 0
        } \VectorTwoColumn{
            0 ; \Varbar{H}_{\indexZeroD \indexfirst{b}}
        } = 0, \quad \forall b \in \totalsetfirst[-].
    \end{align*}
    In additional with 
    \begin{align*}
        \Varbar{H}[a][0] \Varbar{H}[0][b] = \VectorTwoRow{
            \Varbar{H}_{\indexfirst{a} \indexZeroP} ; 0
        } \VectorTwoColumn{
            0 ; \Varbar{H}_{\indexZeroD \indexfirst{b}}
        } = 0, \quad \forall a \in \totalsetfirst[+], b \in \totalsetfirst[-],
    \end{align*}
    we get  
    \begin{align*}
        \opeE = \MatrixNine{
            0 ; 0 ; 0 ; 
            \sim ; 2 \sum_{c \in \totalsetfirst[+]} \frac{1}{\eigvalfirst{c}} \Varbar{H}_{\indexfirst{0} \indexfirst{c}} \Varbar{H}_{\indexfirst{c} \indexfirst{0}} ; 0 ;
            \sim ; \sim ; 0
        } = \MatrixSixteenSC{
            0 ; 0 ; 0 ; 0 ;
            \sim ; 2 \sum_{c \in \totalsetfirst[+]} \frac{1}{\eigvalfirst{c}} \Varbar{H}_{\indexZeroP \indexfirst{c}} \Varbar{H}_{\indexfirst{c} \indexZeroP} ; 0 ; 0 ;
            \sim ; \sim ; 0 ; 0 ;
            \sim ; \sim ; \sim ; 0 
        }
    \end{align*}
    from $\opeE$'s definition in~\eqref{eq:soa:calE}.
    Now we calculate $\opeT{W}$. The most complex part is $\ppsim'(\Varbar{H}[0][0]; \Var{W}[0][0])$. With $\kappa_\Primal$ set as in~\eqref{eq:mdvZ-kernel:kappa-PD}, $\lambda_{\min}(\kappa_\Primal \cdot [\Var{\Xsc}]_{\indexZeroP \indexZeroP}) = \lambda_{\max}(2 \sum_{c \in \totalsetfirst[+]} \frac{1}{\eigvalfirst{c}} \Varbar{H}_{\indexZeroP \indexfirst{c}} \Varbar{H}_{\indexfirst{c} \indexZeroP})$. Thus, $\kappa_\Primal \cdot [\Var{\Xsc}]_{\indexZeroP \indexZeroP} - 2 \sum_{c \in \totalsetfirst[+]} \frac{1}{\eigvalfirst{c}} \Varbar{H}_{\indexZeroP \indexfirst{c}} \Varbar{H}_{\indexfirst{c} \indexZeroP} \succeq 0$. Symmetrically, $\kappa_\Dual \cdot \Var{[\sigma \Ssc]}_{\indexZeroD \indexZeroD} + 2 \sum_{c \in \totalsetfirst[-]}\frac{1}{\eigvalfirst{c}} \Varbar{H}_{\indexZeroD \indexfirst{c}} \Varbar{H}_{\indexfirst{c} \indexZeroD} \succeq 0$. Define 
    \begin{align*}
        & \Varhat{W} := (\Qfirst{0})\tran \Var{W}[0][0] \Qfirst{0} \\
        = &
        (\Qfirst{0})\tran \MatrixFourSC{
            \kappa_\Primal \cdot \Var{[\Xsc]}_{\indexZeroP \indexZeroP} - 2 \sum_{c \in \totalsetfirst[+]}\frac{1}{\eigvalfirst{c}} \Varbar{H}_{\indexZeroP \indexfirst{c}} \Varbar{H}_{\indexfirst{c} \indexZeroP} ; 0 ; 
            \sim ; -\kappa_\Dual \cdot \Var{[\sigma \Ssc]}_{\indexZeroD \indexZeroD} - 2 \sum_{c \in \totalsetfirst[-]}\frac{1}{\eigvalfirst{c}} \Varbar{H}_{\indexZeroD \indexfirst{c}} \Varbar{H}_{\indexfirst{c} \indexZeroD}
        } \Qfirst{0},
    \end{align*}
    with the block-diagonal $\Qfirst{0}$ defined in Lemma~\ref{lem:soa:Q0-block-structure}.
    Thus, $\Varhat{W}_{\indexSecondSC{P} \indexSecondSC{P}} \succeq 0$, $\Varhat{W}_{\indexSecondSC{D} \indexSecondSC{D}} \preceq 0$, and $\Varhat{W}_{\indexSecondSC{P} \indexSecondSC{D}} = 0$. Consequently,
    \begin{align*}
        \ppsim(\Varhat{W}_{\indexsecond{0}{0} \indexsecond{0}{0}}) 
        = \MatrixFourSC{
            \Varhat{W}_{\indexSecondZeroP \indexSecondZeroP} ; 0 ;
            \sim ; 0 
        }
    \end{align*}
    because $\Varhat{W}_{\indexSecondZeroP \indexSecondZeroP}$ (resp. $\Varhat{W}_{\indexSecondZeroD \indexSecondZeroD}$) is a principle submatrix of $\Varhat{W}_{\indexSecondSC{P} \indexSecondSC{P}}$ (resp. $\Varhat{W}_{\indexSecondSC{D} \indexSecondSC{D}}$).
    Now, 
    \begin{align*}
        & \ppsim'(\Varbar{H}[0][0]; \Var{W}[0][0]) = \\
        & \simpleadjustbox{
        \Qfirst{0}
        \MatrixSixteenSC{
            \Varhat{W}_{\indexsecond{0}{+} \indexsecond{0}{+}} ; 
            \Varhat{W}_{\indexsecond{0}{+} \indexSecondZeroP} ; 
            \Varhat{W}_{\indexsecond{0}{+} \indexSecondZeroD} ; 
            \left\{
                \frac{\eigvalsecond{0}{i}}{\eigvalsecond{0}{i} - \eigvalsecond{0}{j}}
                \Varhat{W}_{\indexsecond{0}{i} \indexsecond{0}{j}}
            \right\}_{\substack{
                i \in \totalsetsecond{0}[+] \\ j \in \totalsetsecond{0}[-]
            }} ;
            \sim ;
            [
                \ppsim(\Varhat{W}_{\indexsecond{0}{0} \indexsecond{0}{0}})
            ]_{\indexThirdSC{P} \indexThirdSC{P}} ;
            [
                \ppsim(\Varhat{W}_{\indexsecond{0}{0} \indexsecond{0}{0}})
            ]_{\indexThirdSC{P} \indexThirdSC{D}} ; 
            0 ;
            \sim ;
            \sim ;
            [
                \ppsim(\Varhat{W}_{\indexsecond{0}{0} \indexsecond{0}{0}})
            ]_{\indexThirdSC{D} \indexThirdSC{D}} ;
            0 ;
            \sim ;
            \sim ; 
            \sim ; 
            0
        }
        (\Qfirst{0})\tran 
        } \\
        = & \simpleadjustbox{
        \MatrixFourSC{\Qfirst{0}_{\indexSecondSC{P} \indexSecondSC{P}} ; 0 ; 0 ; \Qfirst{0}_{\indexSecondSC{D} \indexSecondSC{D}}}
        \MatrixSixteenSC{
            \Varhat{W}_{\indexsecond{0}{+} \indexsecond{0}{+}} ; 
            \Varhat{W}_{\indexsecond{0}{+} \indexSecondZeroP} ; 
            0 ; 
            0 ;
            \sim ;
            \Varhat{W}_{\indexSecondZeroP \indexSecondZeroP} ;
            0 ; 
            0 ;
            \sim ;
            \sim ;
            0 ;
            0 ;
            \sim ;
            \sim ; 
            \sim ; 
            0
        }
        \MatrixFourSC{\Qfirst{0}_{\indexSecondSC{P} \indexSecondSC{P}} ; 0 ; 0 ; \Qfirst{0}_{\indexSecondSC{D} \indexSecondSC{D}}}\tran 
        = 
        \MatrixFourSC{
            \Var{W}_{\indexZeroP \indexZeroP} ; 0 ; 
            \sim ; 0 
        }.
        }
    \end{align*}
    Therefore, from~\eqref{eq:soa:Theta},
    \begin{align*}
        \opeT{W} = \MatrixSixteenSC{
            \Var{W}_{\indexfirst{+} \indexfirst{+}} ; \Var{W}_{\indexfirst{+} \indexZeroP} ; 0 ; 0 ;
            \sim ; \Var{W}_{\indexZeroP \indexZeroP} ; 0 ; 0 ; 
            \sim ; \sim ; 0 ; 0 ;
            \sim ; \sim ; \sim ; 0 
        } = \MatrixFourSC{
            \Var{W}_{\indexSC{P} \indexSC{P}} ; 0 ; 
            \sim ; 0 
        }.
    \end{align*}
    Combining $\opeE$ and $\opeT{W}$:
    \begin{align*}
        \opeT{W} + \opeE 
        = \MatrixSixteenSC{
            \kappa_\Primal \cdot [\Var{\Xsc} - \Varbar{X}]_{\indexfirst{+} \indexfirst{+}} ;
            \kappa_\Primal \cdot \Var{[\Xsc]}_{\indexfirst{+} \indexZeroP} ;
            0 ; 
            0 ;
            \sim ;
            \kappa_\Primal \cdot \Var{[\Xsc]}_{\indexZeroP \indexZeroP}  ;
            0 ;
            0 ;
            \sim ; \sim ; 0 ; 0 ;
            \sim ; \sim ; \sim ; 0
        } = \kappa_\Primal \cdot (\Var{\Xsc} - \Varbar{X}).
    \end{align*}
    By the definition of $\Xsc$ and $\Varbar{X}$, $\PA \Xsc = \PA \Varbar{X} = \PA \Vartilde{X}$. Thus, 
    \begin{align*}
        \PA (\opeT{W} + \opeE) = \kappa_\Primal \cdot \PA (\Var{\Xsc} - \Varbar{X}) = 0.
    \end{align*}

    (ii) Dual part. Same as the proof procedure for the primal part, we can show that 
    \begin{align*}
        \opeEP = \MatrixSixteenSC{
            0 ; 0 ; 0 ; 0 ;
            \sim ; 0 ; 0 ; 0 ;
            \sim ; \sim ; 2 \sum_{c \in \totalsetfirst[-]} \frac{1}{\eigvalfirst{c}} \Varbar{H}_{\indexZeroP \indexfirst{c}} \Varbar{H}_{\indexfirst{c} \indexZeroP} ; 0 ;
            \sim ; \sim ; \sim ; 0 
        } \text{ and } \opeTP{W} = \MatrixFourSC{
            0 ; 0 ; 
            \sim ; \Var{W}_{\indexSC{D} \indexSC{D}}
        }.
    \end{align*}
    With the fact that $\PAp \Ssc = \PAp \Varbar{S} = \PAp C$:
    \begin{align*}
        \PAp (\opeTP{W} + \opeEP) = - \kappa_\Dual \cdot \PAp (\sigma \Ssc - \sigma \Varbar{S}) = 0.
    \end{align*}
    
    (iii) Combining the primal and dual part:
    \begin{align*}
        \opeP = -\PA \opeE - \PAp \opeEP = \PA \opeT{W} + \PAp \opeTP{W},
    \end{align*}
    which directly implies $\opeP \in \coneK$ and $\mdvZ = 0$.
\end{proof}

\subsection{\titlemath{Discussion: Small yet Non-Zero $\angle (\Delta \Vark{Z}, \Delta \Varkpo{Z})$}}
\label{sec:mdvZ-kernel-3}

From Proposition~\ref{prop:mdvZ-kernel:kernel}, as long as $\Zbar \in \Zopt$ and $\Vark{Z} - \Zbar \in \coneC \backslash \coneT$, $\mdvZ[\Zbar; \Vark{Z} - \Zbar]$ is guaranteed to be non-zero. In this case, the higher-order term in the second-order local limit dynamics~\eqref{eq:soa:limiting-dynamics-def} can be (transiently) omitted, and $\Delta \Vark{Z} \approx \frac{1}{2} \mdvZ[\Zbar; \Vark{Z} - \Zbar] \sim o(\normF{\Vark{Z} - \Zbar})$. The approximation becomes more and more accurate as $Z - \Zbar \rightarrow 0$. In this case,
\begin{align*}
    \angle (\Delta \Vark{Z}, \Delta \Varkpo{Z}) \approx \angle (\mdvZ[\Zbar; \Vark{Z} - \Zbar], \mdvZ[\Zbar; \Vark{Z} - \Zbar + \Delta \Vark{Z}]).
\end{align*}
Therefore, as long as $\mdvZmap$ can exhibit certain type of continuity at $\Vark{Z} - \Zbar$ (and the ``almost-sure'' type continuity will be established in \S\ref{sec:mdvZ-continuity-2}), one could expect $\angle (\Delta \Vark{Z}, \Delta \Varkpo{Z}) \rightarrow 0$ as $\Vark{Z} - \Zbar \rightarrow 0$. The ``small yet non-zero'' effect may be due to the presence of higher-order terms. We empirically verify our analysis with three SDP examples defined in \S\ref{sec:toy}. Across all examples, we fix $\sigma$ to $1$ and the tolerance for $r_{\max}$ is set to $10^{-14}$. The maximum three-step ADMM iteration number is set to $1000$. $\Hbar \in \coneC \backslash \coneT$ is chosen as: 
\begin{enumerate}
    \item For~\eqref{eq:toy:example-1}, $a = 1$ and $b = 1$ in~\eqref{eq:toy1:Hbar}. The corresponding $\mdvZ$ is defined in~\eqref{eq:toy1:mdvZ}.
    \item For~\eqref{eq:toy:example-2}, $\Hbar_{12} = 1, \Hbar_{22} = 1, \Hbar_{23} = 1$ in~\eqref{eq:toy2:Hbar}. The corresponding $\mdvZ$ is defined in~\eqref{eq:toy2:mdvZ}.
    \item For~\eqref{eq:toy:example-3}, $h = 1$ and $\epsilon = 0$ in~\eqref{eq:toy3:Hbar}. The corresponding $\mdvZ$ is defined in~\eqref{eq:toy3:mdvZ-1}.
\end{enumerate}
The initial guess $Z^{(0)}$ is set to $\Zbar + t \Hbar$ with different $t$'s. Accordingly, $X^{(0)} = \ppsim(Z^{(0)})$ and $S^{(0)} = -\frac{1}{\sigma} \npsim(Z^{(0)})$.
We check the trajectories of four quantities:
\begin{align*}
    \normF{\Delta \Vark{Z}}, \quad \angle (\Delta \Vark{Z}, \Delta \Varkpo{Z}), \quad \frac{
                \normF{0.5 \mdvZ[\Zbar; Z^{(0)}-\Zbar]-\Delta \Vark{Z}}
            }{
                \normF{\Delta \Vark{Z}}
            }, \quad \frac{
                \normF{0.5 \mdvZ[\Zbar; Z^{(k)}-\Zbar]-\Delta \Vark{Z}}
            }{
                \normF{\Delta \Vark{Z}}
            }.
\end{align*}

\paragraph{Discussion on $\normF{\Delta \Vark{Z}}$.}
The results are shown in Figure~\ref{fig:mdvZ-kernel:Znorm}. When $t$ is relatively large (\eg $\log_{10}(t) > -2$), ADMM still exhibits an observable linear convergence rate. However, as $t \downarrow 0$, this rate approaches to $1$ and the iterations nearly stall. On the other hand, when $t$ is sufficiently small (\eg $\log_{10}(t) < -3.5$), the second-order term starts to dominate the dynamics, and $\Delta \Vark{Z}$ transiently converges to $\frac{t^2}{2}\mdvZ$. This quadratic relationship is evident in Figure~\ref{fig:mdvZ-kernel:Znorm}: as $\log_{10}(t)$ decreases by $1$, the transiently convergent $\log_{10}(\normF{\Delta \Vark{Z}})$ decreases by approximately $2$ across all three SDP examples.
\input{figs/mdvZ-kernel/Znorm.tex}

\paragraph{Discussion on $\angle (\Delta \Vark{Z}, \Delta \Varkpo{Z})$.}
The results are shown in Figure~\ref{fig:mdvZ-kernel:Zang}. As $t \downarrow 0$, the transiently convergent $\angle (\Delta \Vark{Z}, \Delta \Varkpo{Z})$ tends to become smaller. One noticeable phenomenon is that the convergent angle does not appear to decrease monotonically: in all examples, as $t$ decreases from $10^{-4}$ to $10^{-5}$, the transiently convergent angle actually increases. This behavior may be caused by numerical issues when computing angles between two extremely small vectors in double precision.
\input{figs/mdvZ-kernel/Zang.tex}

\paragraph{Discussion on $\frac{
    \normF{0.5 \mdvZ[\Zbar; Z^{(0)}-\Zbar]-\Delta \Vark{Z}}
}{
    \normF{\Delta \Vark{Z}}
}$.}
The results are shown in Figure~\ref{fig:mdvZ-kernel:relerr}. As $t \downarrow 0$, $\Delta \Vark{Z}$ first transiently converges to $0.5\,\mdvZ[\Zbar; Z^{(0)}-\Zbar] = \frac{t^2}{2}\mdvZ$, and then gradually deviates from it. This deviation is caused by the change of $\Vark{Z}$ discussed in \S\ref{sec:soa:mdvZ}.
\input{figs/mdvZ-kernel/relerr.tex}

\paragraph{Discussion on $\frac{
    \normF{0.5 \mdvZ[\Zbar; Z^{(k)}-\Zbar]-\Delta \Vark{Z}}
}{
    \normF{\Delta \Vark{Z}}
}$.}
The results are shown in Figure~\ref{fig:mdvZ-kernel:instant-relerr}. Since the complete description of $\coneC$ and the corresponding $\mdvZmap$ is hard to obtain in~\eqref{eq:toy:example-3}, we report only the results for~\eqref{eq:toy:example-1} and~\eqref{eq:toy:example-2}. Unlike Figure~\ref{fig:mdvZ-kernel:relerr}, the second-order limit predictor $0.5\,\mdvZ[\Zbar; Z^{(k)}-\Zbar]$ stably tracks $\Delta \Vark{Z}$. Interestingly, as $\log_{10}(t)$ decreases by $1$, the log of relative tracking error also decreases by $1$. When $\log_{10}(t) \le -1.5$, the trajectories tend to be noisy. This is because when $t$ is relatively large, $\Delta \Vark{Z}$'s linearly converge to $0$ quickly. Therefore, the division becomes unstable in double precision. We early stop the trajectories as long as $\Vark{r}_{\max}$ approaches $10^{-14}$.

\input{figs/mdvZ-kernel/relerr_instant.tex}

%% file: figs/mdvZ-kernel/Znorm.tex

\begin{figure}[h]
    \centering

    \begin{minipage}{\textwidth}
        \centering
        \begin{tabular}{ccc}
            \begin{minipage}{0.32\textwidth}
                \centering
                \includegraphics[width=\columnwidth]{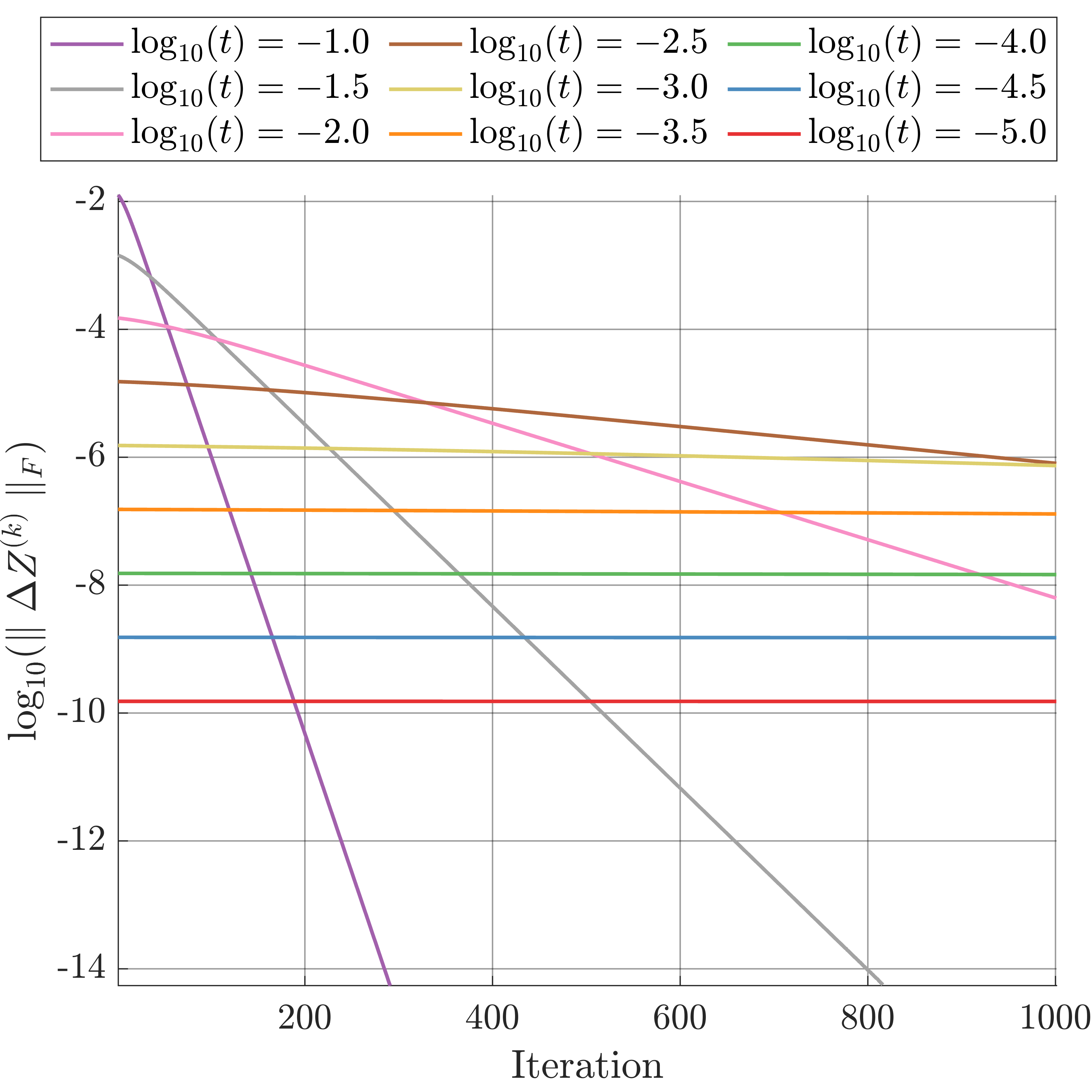}
                \eqref{eq:toy:example-1}
            \end{minipage}

            \begin{minipage}{0.32\textwidth}
                \centering
                \includegraphics[width=\columnwidth]{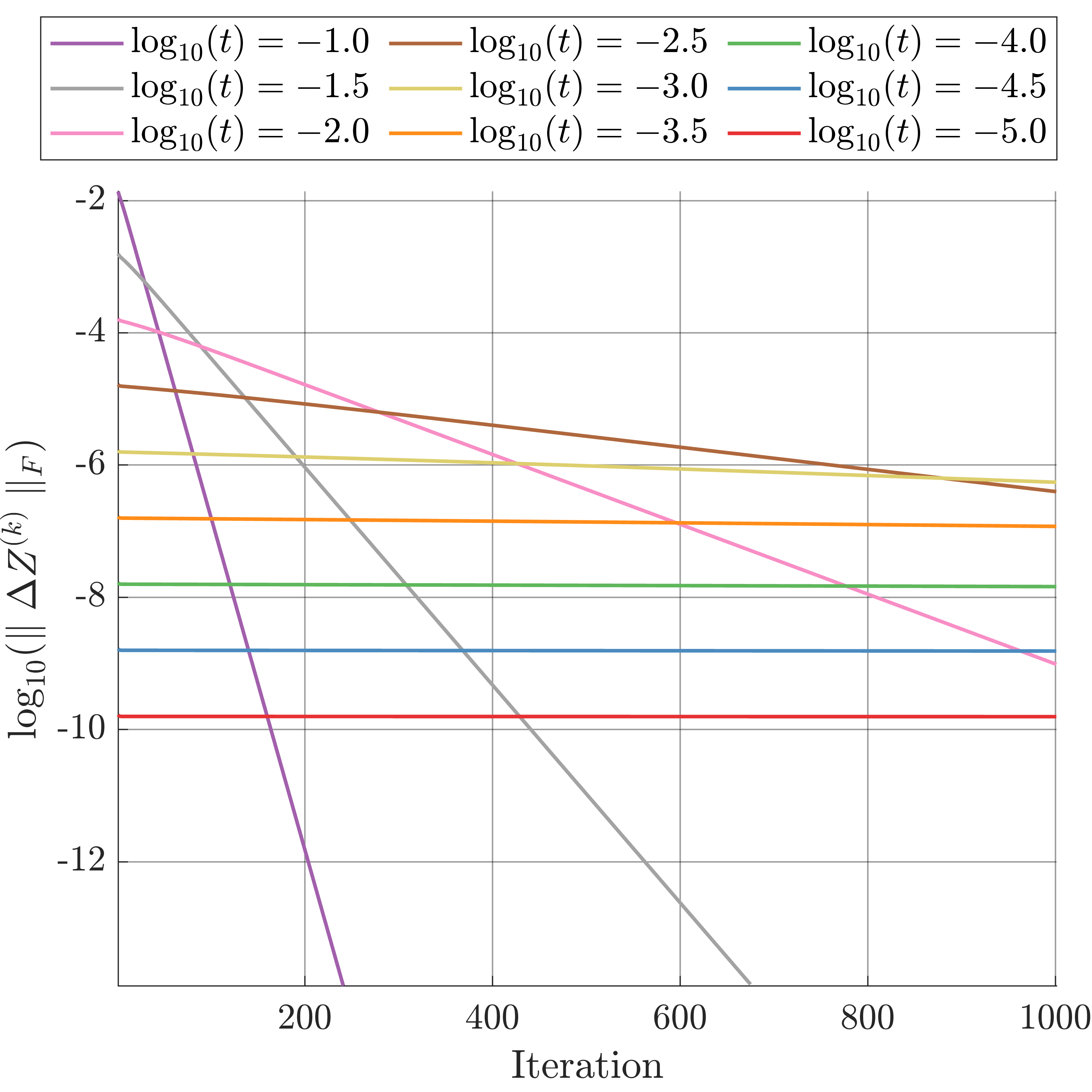}
                \eqref{eq:toy:example-2}
            \end{minipage}

            \begin{minipage}{0.32\textwidth}
                \centering
                \includegraphics[width=\columnwidth]{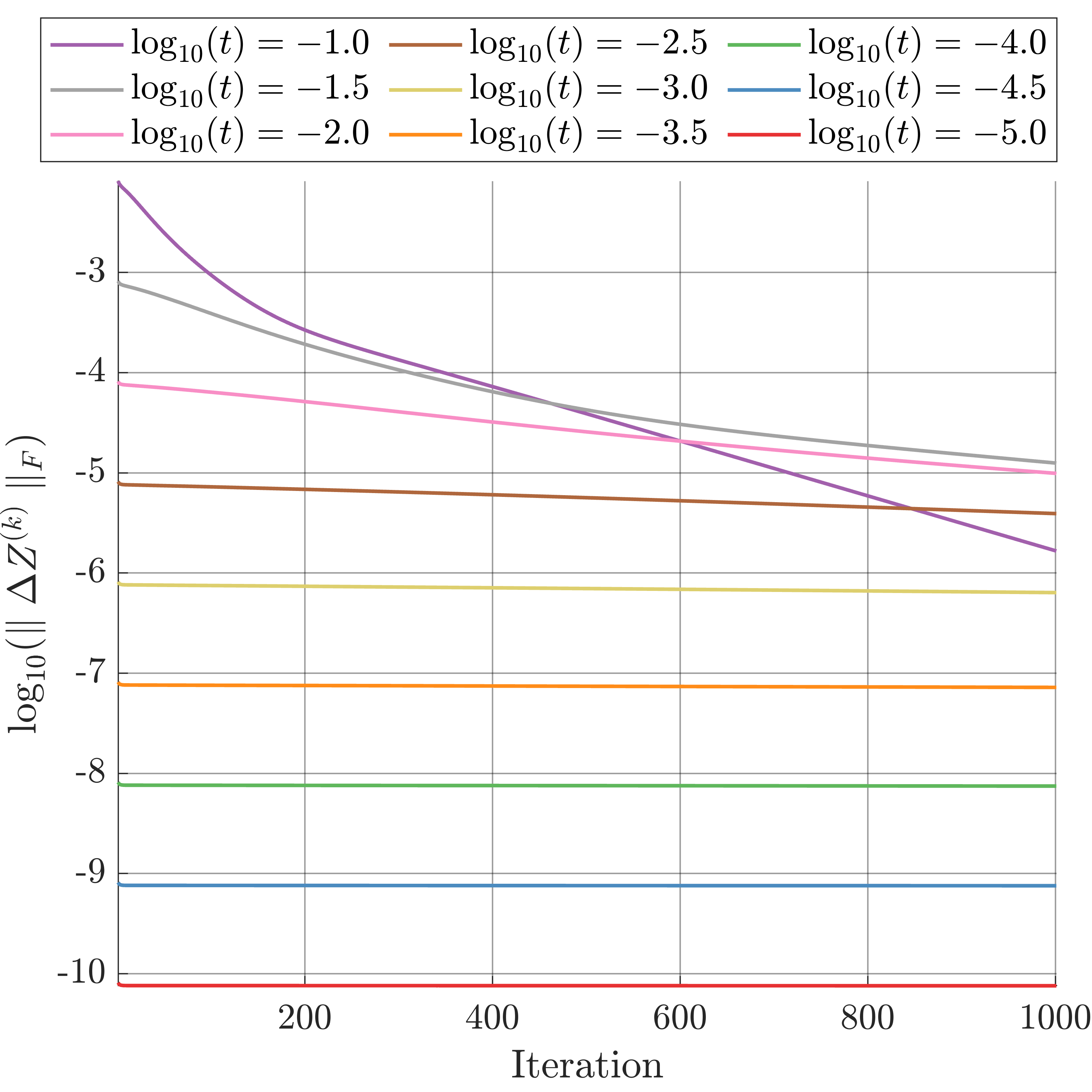}
                \eqref{eq:toy:example-3}
            \end{minipage}
        \end{tabular}
    \end{minipage}
    \caption{\label{fig:mdvZ-kernel:Znorm} $\log_{10}(\normF{\Delta \Vark{Z}})$ in three SDP examples. In each example, the initialization is chosen as $Z^{(0)}=\Zbar+t\Hbar$, where $\Zbar\in\Zopt$ and $\Hbar\in\coneC\backslash\coneT$, and we sweep $t$ from $10^{-1}$ to $10^{-5}$. $\sigma$ is fixed to $1$. 
    }
\end{figure}

%% file: figs/mdvZ-kernel/Zang.tex

\begin{figure}[h]
    \centering

    \begin{minipage}{\textwidth}
        \centering
        \begin{tabular}{ccc}
            \begin{minipage}{0.32\textwidth}
                \centering
                \includegraphics[width=\columnwidth]{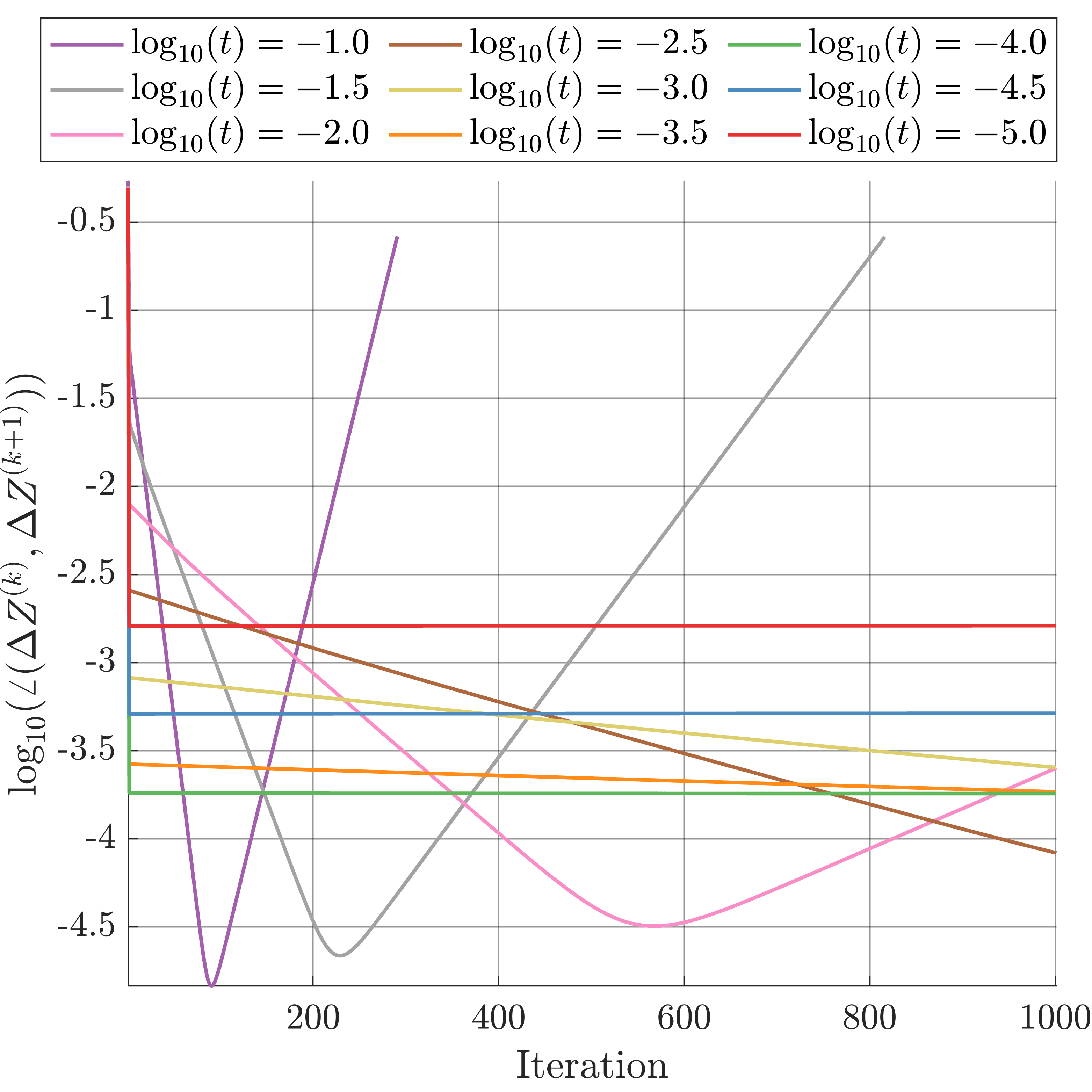}
                \eqref{eq:toy:example-1}
            \end{minipage}

            \begin{minipage}{0.32\textwidth}
                \centering
                \includegraphics[width=\columnwidth]{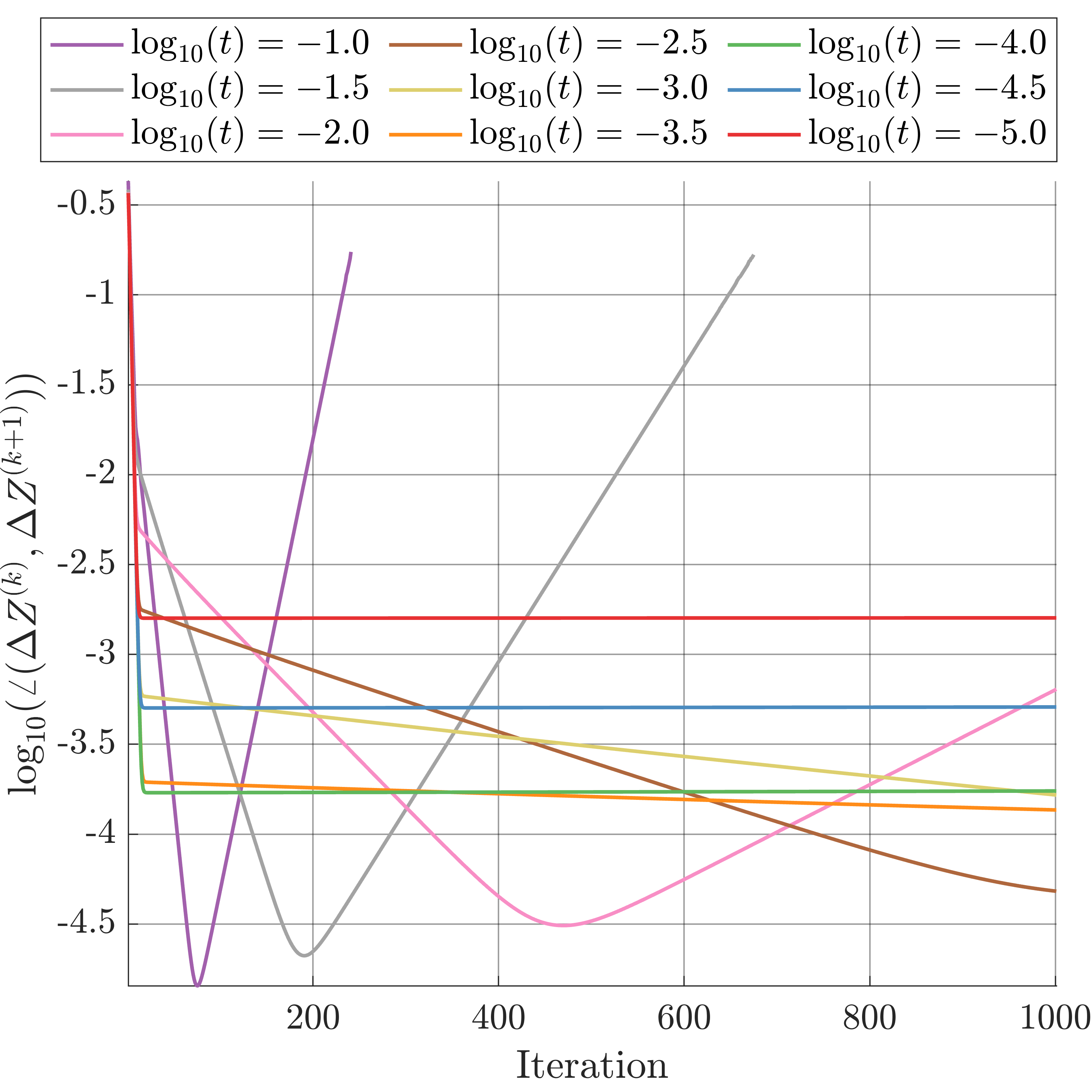}
                \eqref{eq:toy:example-2}
            \end{minipage}

            \begin{minipage}{0.32\textwidth}
                \centering
                \includegraphics[width=\columnwidth]{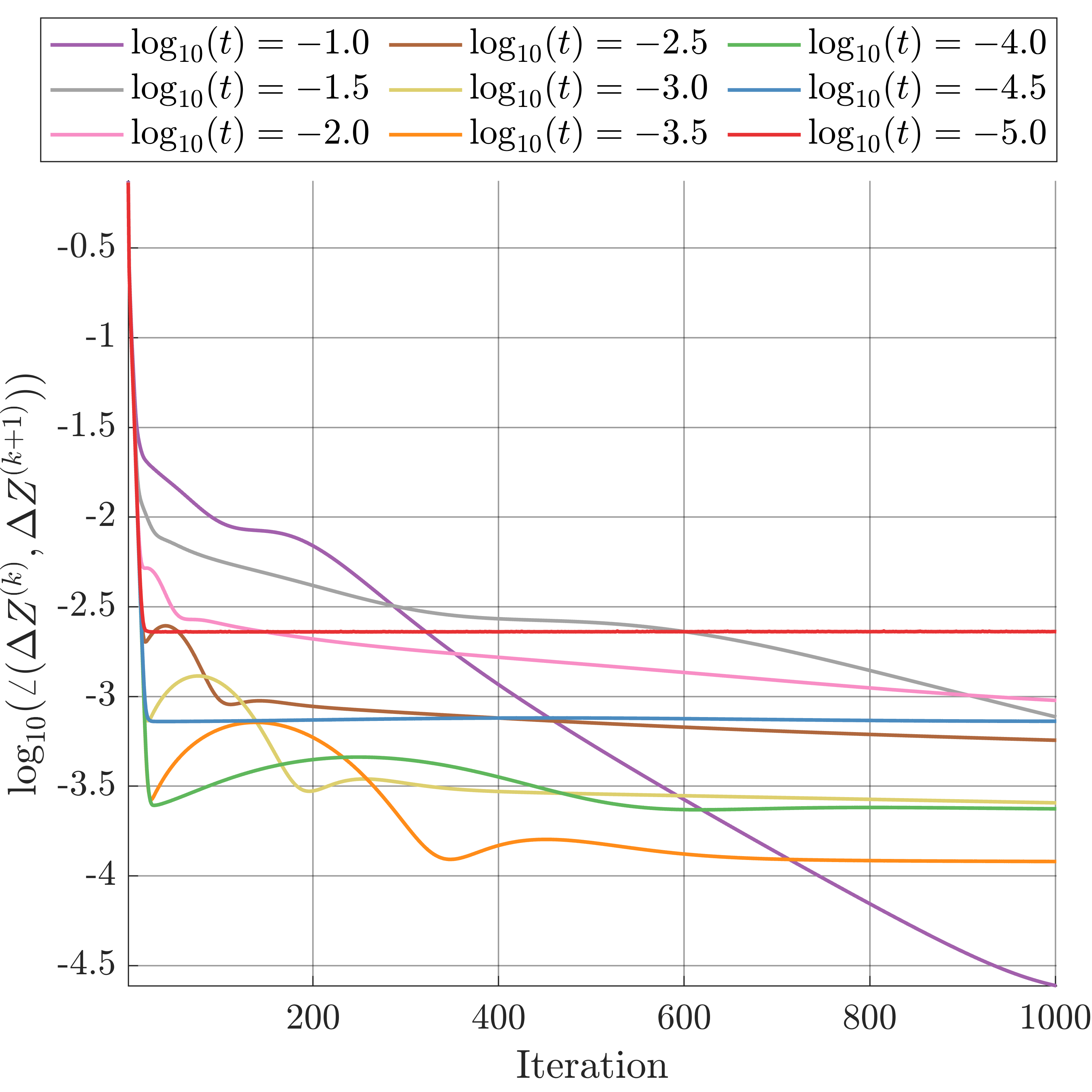}
                \eqref{eq:toy:example-3}
            \end{minipage}
        \end{tabular}
    \end{minipage}
    \caption{\label{fig:mdvZ-kernel:Zang} $\log_{10}(\angle (\Delta \Vark{Z}, \Delta \Varkpo{Z}))$ in three SDP examples. In each example, the initialization is chosen as $Z^{(0)}=\Zbar+t\Hbar$, where $\Zbar\in\Zopt$ and $\Hbar\in\coneC\backslash\coneT$, and we sweep $t$ from $10^{-1}$ to $10^{-5}$. $\sigma$ is fixed to $1$.
    }
\end{figure}

%% file: figs/mdvZ-kernel/relerr.tex

\begin{figure}[h]
    \centering

    \begin{minipage}{\textwidth}
        \centering
        \begin{tabular}{ccc}
            \begin{minipage}{0.32\textwidth}
                \centering
                \includegraphics[width=\columnwidth]{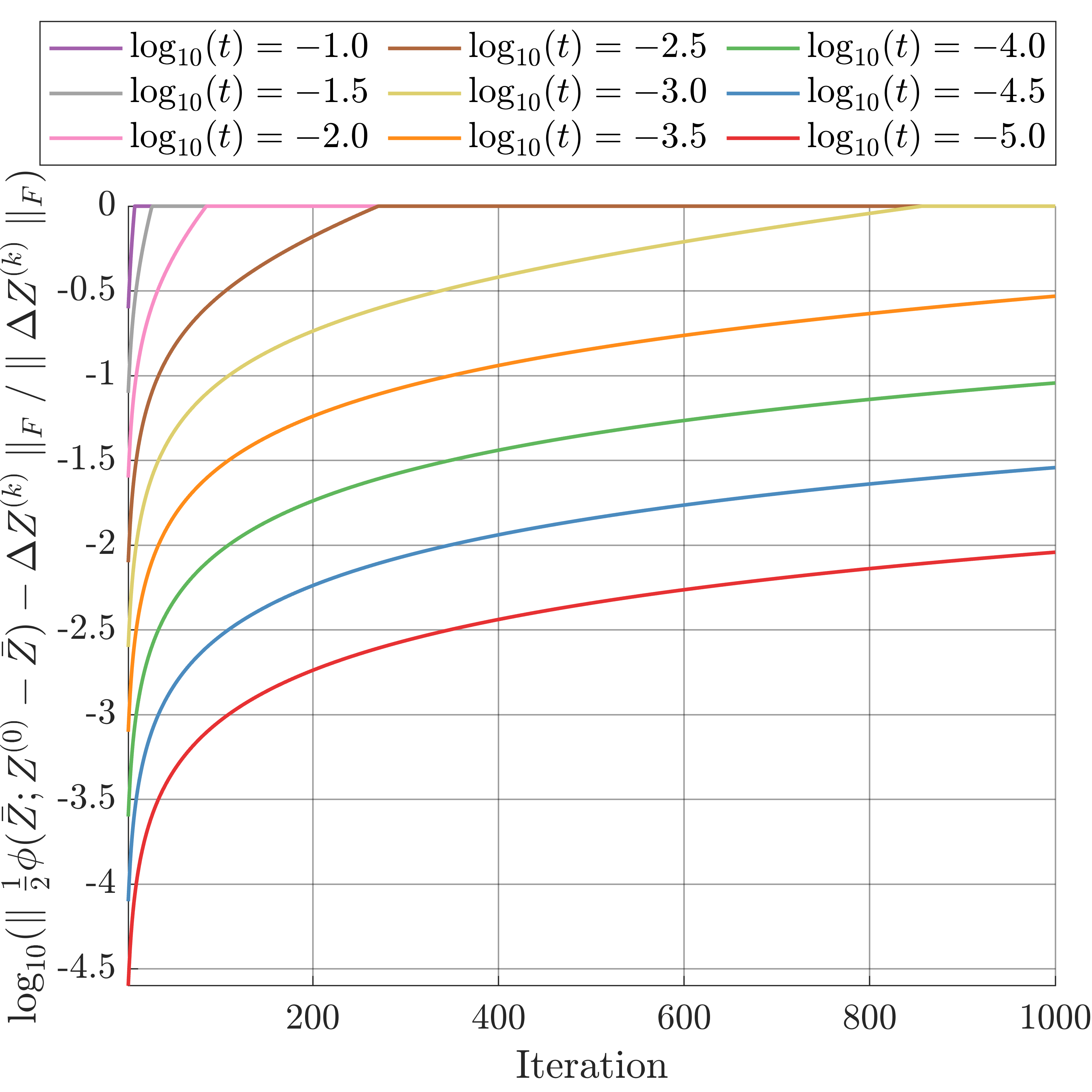}
                \eqref{eq:toy:example-1}
            \end{minipage}

            \begin{minipage}{0.32\textwidth}
                \centering
                \includegraphics[width=\columnwidth]{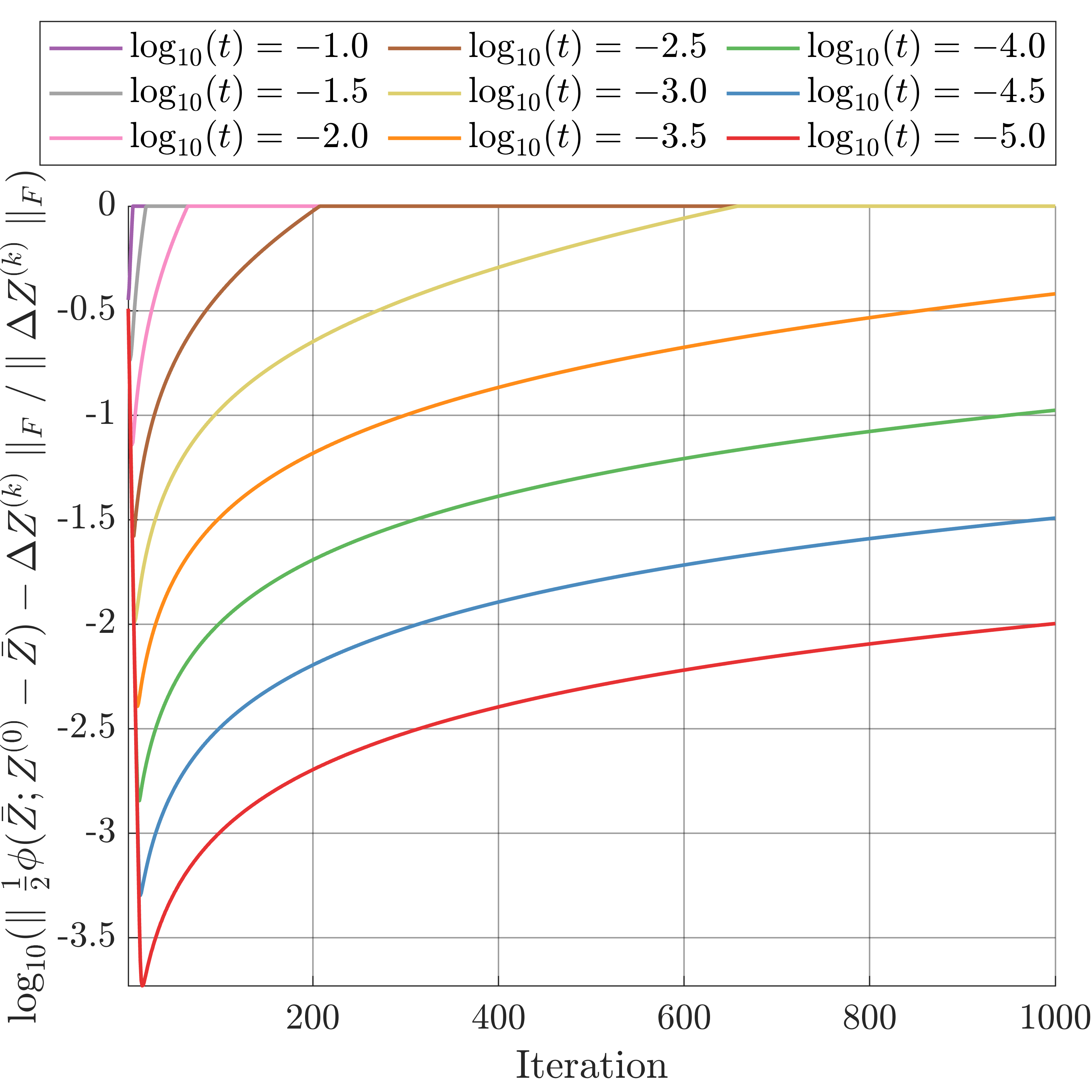}
                \eqref{eq:toy:example-2}
            \end{minipage}

            \begin{minipage}{0.32\textwidth}
                \centering
                \includegraphics[width=\columnwidth]{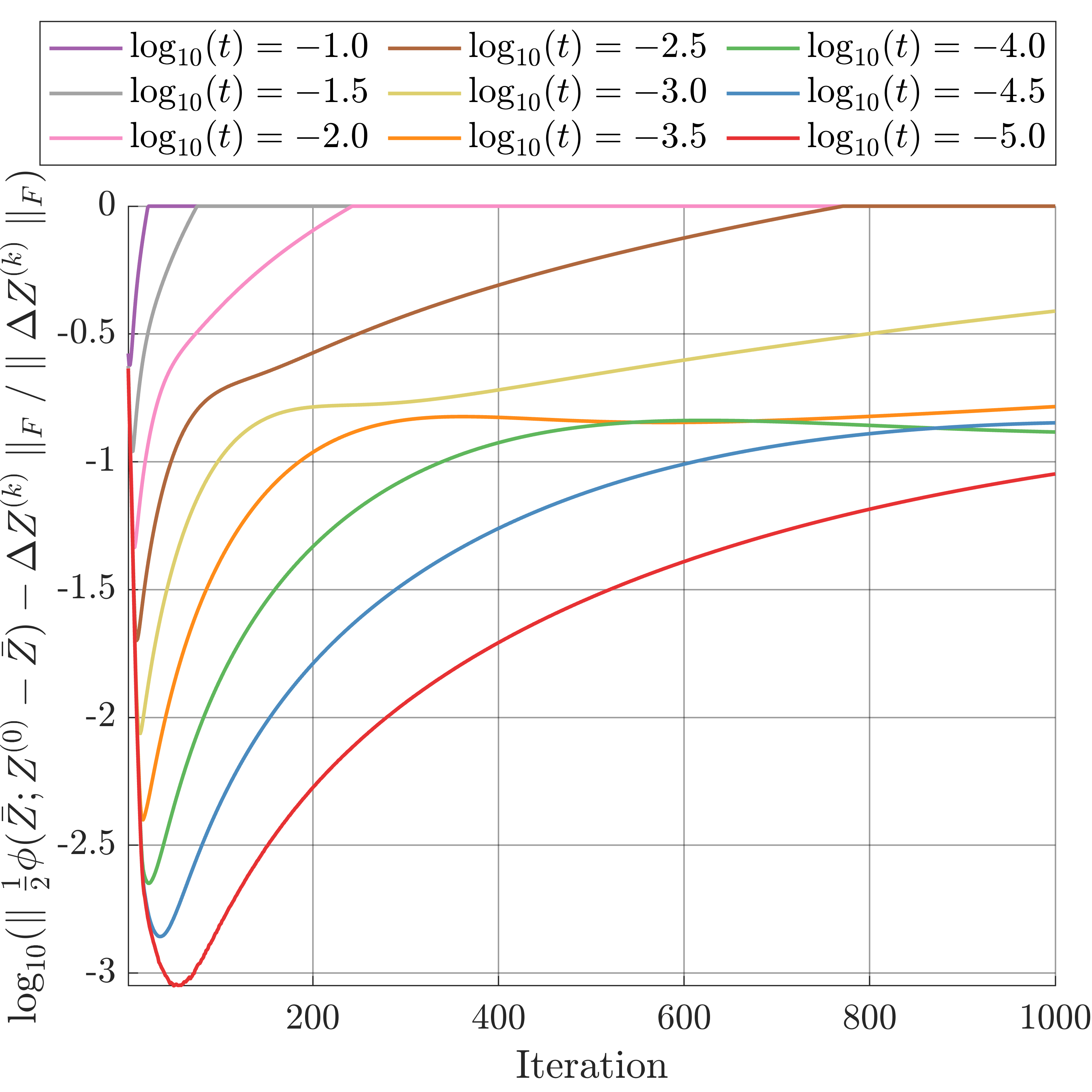}
                \eqref{eq:toy:example-3}
            \end{minipage}
        \end{tabular}
    \end{minipage}
    \caption{
        \label{fig:mdvZ-kernel:relerr} $\log_{10} (
            \frac{
                \normF{0.5\,\mdvZ[\Zbar; Z^{(0)}-\Zbar]-\Delta \Vark{Z}}
            }{
                \normF{\Delta \Vark{Z}}
            }
        )$ in the three SDP examples. For visualization, we upper-clamp the values at $1$. In each example, the initialization is chosen as $Z^{(0)}=\Zbar+t\Hbar$, where $\Zbar\in\Zopt$ and $\Hbar\in\coneC\backslash\coneT$, and we sweep $t$ from $10^{-1}$ to $10^{-5}$. $\sigma$ is fixed to $1$.
    }
\end{figure}

%% file: figs/mdvZ-kernel/relerr_instant.tex

\begin{figure}[h]
    \centering

    \begin{minipage}{\textwidth}
        \centering
        \begin{tabular}{cc}
            \begin{minipage}{0.32\textwidth}
                \centering
                \includegraphics[width=\columnwidth]{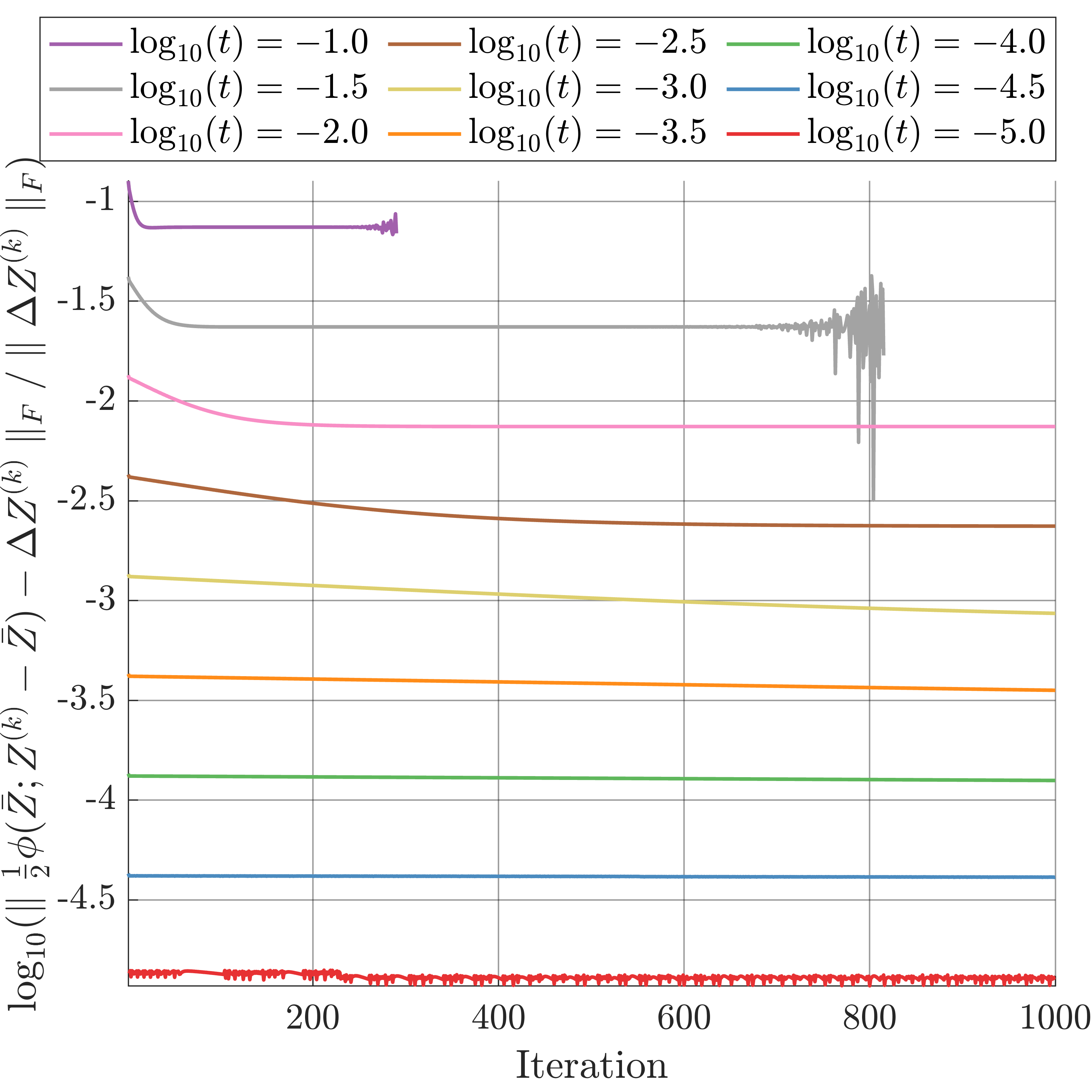}
                \eqref{eq:toy:example-1}
            \end{minipage}

            \begin{minipage}{0.32\textwidth}
                \centering
                \includegraphics[width=\columnwidth]{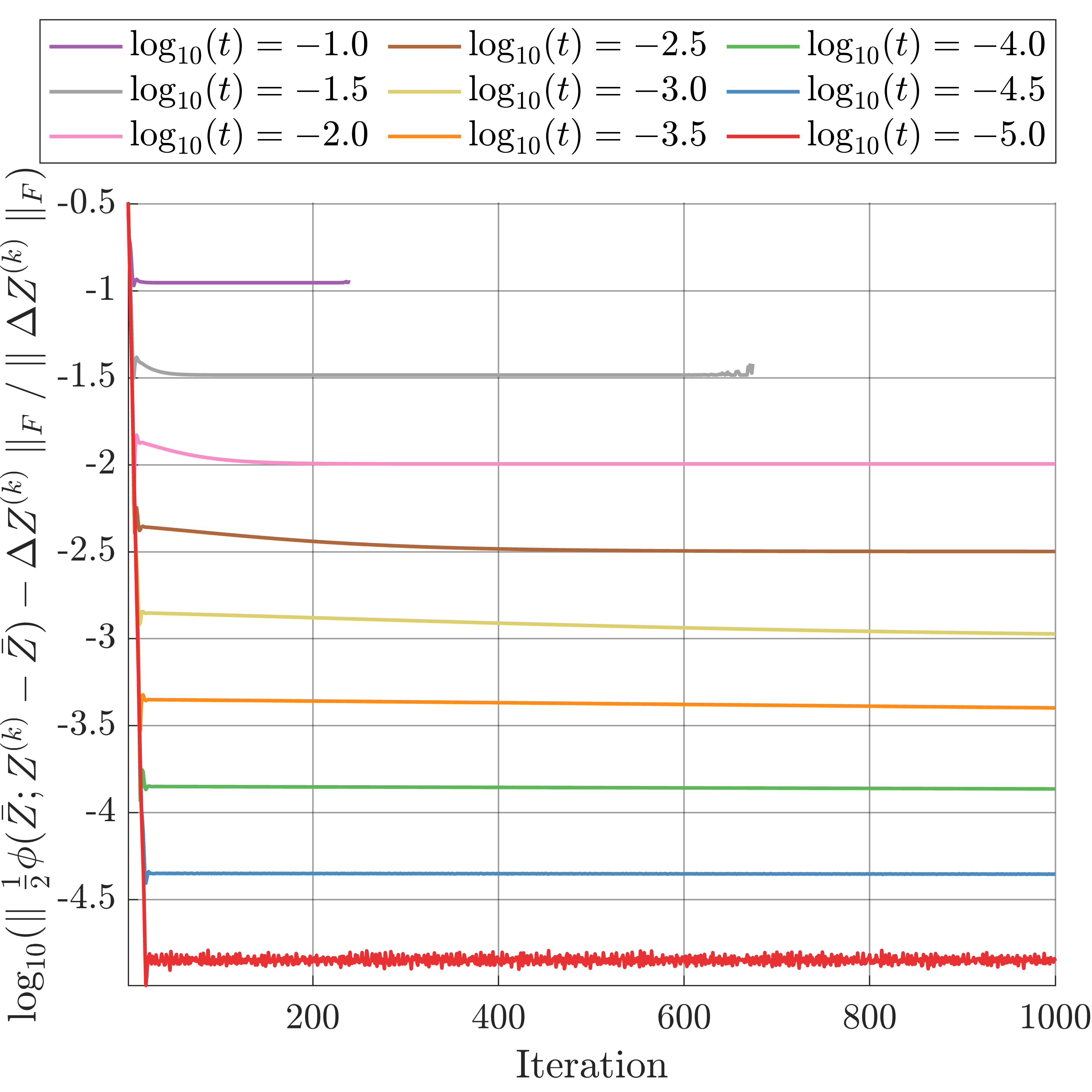}
                \eqref{eq:toy:example-2}
            \end{minipage}
        \end{tabular}
    \end{minipage}
    \caption{
        \label{fig:mdvZ-kernel:instant-relerr} $\log_{10} (
            \frac{
                \normF{0.5\,\mdvZ[\Zbar;\Vark{Z}-\Zbar]-\Delta \Vark{Z}}
            }{
                \normF{\Delta \Vark{Z}}
            }
        )$ in the first two SDP examples.  In each example, the initialization is chosen as $Z^{(0)}=\Zbar+t\Hbar$, where $\Zbar\in\Zopt$ and $\Hbar\in\coneC\backslash\coneT$, and we sweep $t$ from $10^{-1}$ to $10^{-5}$. $\sigma$ is fixed to $1$.
    }
\end{figure}

%% file: sections/mdvZ_range.tex

\section{\titlemath{Range of $\mdvZ[\Zbar; \cdot]$}}
\label{sec:mdvZ-range}

The second property of $\mdvZmap$ that we study is its range, \ie $\range(\mdvZmap)$. For the one-step ADMM iteration~\eqref{eq:intro:one-step-admm}, $\mdvZmap$ can be interpreted as a second-order local ``steady-state response'' of $\Delta \Vark{Z}$ from any initialization $Z^{(0)}=Z$ satisfying $Z-\Zbar\in\coneC$ and $Z\rightarrow \Zbar$, after filtering out all transient directions. It is therefore natural to expect that $\range(\mdvZmap)$ lies in a subset whose dimension is much lower than that of the ambient space $\Sym{n}$. We are particularly interested in how $\range(\mdvZmap)$ relates to $\coneC$ (up to an affine hull). For example, if one could establish that $\range(\mdvZmap)\subseteq \coneC$ under suitable conditions, then it would follow immediately that $\Zbar+\coneC$ is an invariant set for the local second-order limit dynamics~\eqref{eq:soa:limiting-dynamics-def} when the higher-order term $o(\normF{\Vark{Z}-\Zbar}^2)$ is neglected.

In \S\ref{sec:mdvZ-range-1}, we present a negative result: under Assumption~\ref{ass:soa:sc} alone, one cannot even guarantee $\range(\mdvZmap)\subseteq \affinehull(\coneC)$. In \S\ref{sec:mdvZ-range-2}, however, we show that the inclusion $\range(\mdvZmap)\subseteq \affinehull(\coneC)$ does hold once uniqueness of either the primal or the dual optimal solution is imposed. Finally, in \S\ref{sec:mdvZ-range-3}, we discuss almost-invariant sets around $\Zbar$, leveraging the structure of $\range(\mdvZmap)$.

\subsection{\titlemath{General Case: $\range(\mdvZmap) \nsubseteq \affinehull(\coneC)$}}
\label{sec:mdvZ-range-1}

From~\eqref{eq:toy:example-1} and~\eqref{eq:toy:example-2}, one may conjecture that $\range (\mdvZmap) \subseteq \coneC$, or at least $\range (\mdvZmap) \subseteq \affinehull(\coneC)$. However, this is not true in general. 

\begin{proposition}[$\range (\mdvZmap) \nsubseteq \affinehull(\coneC)$ in general]
    \label{prop:mdvZ-range:general-case}
    There exists an SDP data triplet $(\Asdp, b, C)$ satisfying Assumption~\ref{ass:soa:sc}, equipped with $\Varbar{Z} \in \Zopt$, such that $\range (\mdvZmap) \nsubseteq \affinehull(\coneC)$.
\end{proposition}
\begin{proof}
    From Proposition~\ref{prop:soa:coneC} (2), $\Var{H}_{\MultiIndexFirstSC{2}{3}} = 0$ for all $\Var{H} \in \coneC$. Thus, as long as we can construct an SDP satisfying Assumption~\ref{ass:soa:sc}, such that there exists $\Varbar{Z} \in \Zopt$ and $\Varbar{H} \in \coneC$ with $\Var{\mdvZ}_{\MultiIndexFirstSC{2}{3}} \ne 0$, the claim holds. Please see~\eqref{eq:toy:example-3} for a concrete construction. Specifically,
    consider $\Hbar = \Hbar(h, 0)$ in~\eqref{eq:toy3:Hbar} with $h > \sqrt{2}$. Then, from~\eqref{eq:toy3:mdvZ-2}:
    \begin{align*}
        \mdvZ = \MatrixNine{
            -\frac{4}{9 \sigma} ; 
            \mymatplain{-\frac{2\sqrt{2}}{9 \sigma} & 0 & 0 & -\frac{h}{3}} ;
            0 ;
            \sim ;
            \MatrixSixteenSC{
                \frac{2}{9 \sigma} ; \frac{4}{9 \sigma} ; 0 ; 0 ;
                \sim ; \frac{2}{9 \sigma} ; 0 ; \frac{\sqrt{2}h}{12} ;
                \sim ; \sim ; 0; -\frac{h}{3} ;
                \sim ; \sim ; \sim ; -\frac{h^2}{6} 
            } ;
            \mymatplain{-\frac{2}{3\sigma} \\ 0 \\ 0 \\ 0} ;
            \sim ;
            \sim ; 
            \frac{h^2}{6}
    }.
    \end{align*}
    Clearly, $\Var{\mdvZ}_{\MultiIndexFirstSC{2}{3}} \ne 0$.
\end{proof}

\subsection{\titlemath{Under One-Sided Uniqueness: $\range(\mdvZmap) \subseteq \affinehull(\coneC)$}}
\label{sec:mdvZ-range-2}

Although in general $\range(\mdvZmap)\nsubseteq \affinehull(\coneC)$, we will show that this inclusion does hold once additional conditions are imposed. Before proceeding, we first state a few elementary lemmas on relative interiors and affine hulls.
\begin{lemma}
    \label{lem:mdvZ-range:ri-aff-1}
    Given a finite-dimensional Hilbert space $\calH$. Assume $M$ is an affine set and $S \subseteq M$. If there exists $\bar{x} \in M$ and $\epsilon > 0$, such that $(\bar{x} + \epsilon \mathbb{B}) \cap M \subseteq S$, then $\affinehull(S) = M$. 
\end{lemma}
\begin{proof}
    Clearly, $\affinehull(S) \subseteq M$. 
    
    For the other direction, denote $M$ as $\bar{x} + L$, where $L = M - M$ is the linear subspace parallel to $M$. Then, by assumption, $((\bar{x} + \epsilon \mathbb{B}) \cap M) - \bar{x} = \epsilon \mathbb{B} \cap L \subseteq S - \bar{x}$. But $\affinehull(\epsilon \mathbb{B} \cap L) = L$: picking any $v \in L$ and $t \in \mathbb{R}$ large enough, we get $v/t \in \epsilon \mathbb{B}$. Thus, $v/t \in \epsilon \mathbb{B} \cap L$ and $v = t \cdot v/t \in \spanning(\epsilon \mathbb{B} \cap L)$. Since $0 \in \epsilon \mathbb{B} \cap L$, we get $\spanning(\epsilon \mathbb{B} \cap L) = \affinehull(\epsilon \mathbb{B} \cap L)$. By the minimality of the affine hull, $L = \affinehull(\epsilon \mathbb{B} \cap L) \subseteq \affinehull(S - \bar{x})$. The proof is closed by showing $\affinehull(S) = \bar{x} + \spanning(S - \bar{x}) \supseteq \bar{x} + L = M$.
\end{proof}

\begin{lemma}
    \label{lem:mdvZ-range:ri-aff-2}
    Given a finite-dimensional Hilbert space $\calH$ and two convex sets $C_1, C_2 \subset \calH$. If $\ri(C_1) \cap \ri(C_2) \ne \emptyset$, then $\affinehull(C_1) \cap \affinehull(C_2) = \affinehull(C_1 \cap C_2)$. 
\end{lemma}
\begin{proof}
    For ease of notation, set $M_1 := \affinehull(C_1), M_2 = \affinehull(C_2), M = M_1 \cap M_2$. Take any $\bar{x} \in \ri(C_1) \cap \ri(C_2)$. By definition, $\exists \epsilon_1, \epsilon_2 > 0$, s.t. $(\bar{x} + \epsilon_1 \mathbb{B}) \cap M_1 \subseteq C_1$ and $(\bar{x} + \epsilon_2 \mathbb{B}) \cap M_2 \subseteq C_2$. Set $\epsilon = \min \{\epsilon_1, \epsilon_2\}$, we get 
    \begin{align*}
        (\bar{x} + \epsilon \mathbb{B}) \cap M = ((\bar{x} + \epsilon \mathbb{B}) \cap M_1) \cap ((\bar{x} + \epsilon \mathbb{B}) \cap M_2) \subseteq C_1 \cap C_2
    \end{align*}
    Thus, by $(C_1 \cap C_2) \subseteq M$ and Lemma~\ref{lem:mdvZ-range:ri-aff-1}, we get $\affinehull(C_1 \cap C_2) = M$.
\end{proof}

\begin{lemma}
    \label{lem:mdvZ-range:ri-aff-3}
    Given a finite-dimensional Hilbert space $\calH$ and two sets $C_1, C_2 \subset \calH$. $\affinehull(C_1 + C_2) = \affinehull(C_1) + \affinehull(C_2)$.
\end{lemma}
\begin{proof}
    The ``$\subseteq$'' part. Since $C_1 \subseteq \affinehull(C_1), C_2 \subseteq \affinehull(C_2)$, we have $C_1 + C_2 \subseteq \affinehull(C_1) + \affinehull(C_2)$. Thus, by minimality of affine hull, $\affinehull(C_1 + C_2) \subseteq \affinehull(C_1) + \affinehull(C_2)$.

    The ``$\supseteq$'' part. Take any $u \in \affinehull(C_1)$ and $v \in \affinehull(C_2)$. By affine hull's definition, there exist $x_i \in C_1$ and $\sum_i \alpha_i = 1$, s.t. $u = \sum_i \alpha_i x_i$. Similarly, there exist $y_j \in C_2$ and $\sum_j \beta_j = 1$, s.t. $v = \sum_j \beta_j y_j$. Thus, 
    \begin{align*}
        u + v = \sum_i \alpha_i x_i + \sum_j \beta_j y_j = \sum_{i,j} (\alpha_i \beta_j) (x_i + y_j).
    \end{align*}
    Observing that $\sum_{i,j} \alpha_i \beta_j = 1$ and $x_i + y_j \in C_1 + C_2$, we get $u + v \in \affinehull(C_1 + C_2)$.
\end{proof}

\begin{proposition}[$\range (\mdvZmap) \subseteq \affinehull(\coneC)$ under one-sided uniqueness]
    \label{prop:mdvZ-range:one-side-unique}
    Under Assumption~\ref{ass:soa:sc}, if either the primal or the dual optimal solution is unique, then $\range(\mdvXmap) \subseteq \affinehull(\coneCX)$ and $\range(\mdvSmap) \subseteq \affinehull(\coneCS)$. Consequently, $\range (\mdvZmap) \subseteq \affinehull(\coneC)$.
\end{proposition}
\begin{proof}
    We only prove for the case when the primal solution is unique. The dual solution unique case can be proven symmetrically. Since the primal optimal solution is unique, $\Xsc = \Varbar{X}$. Thus, picking any $\Varbar{H} \in \coneC$, we have $\Varbar{H}_{\MultiIndexFirstSC{2}{2}} = 0, \Varbar{H}_{\MultiIndexFirstSC{2}{3}} = 0, \Varbar{H}_{\MultiIndexFirstSC{2}{4}} = 0$. Thus, $\Qfirst{0}$ is degraded to $\Qfirst{0}_{\indexSecondSC{D} \indexSecondSC{D}}$. The cones in~\eqref{eq:soa:coneC-XS} are degraded to:
    \begin{eqnarray*}
        & \coneCX = \left\{ 
            H = \MatrixNine{
                H_{\MultiIndexFirstSC{1}{1}} ; H_{\MultiIndexFirstSC{1}{2}} ; 0 ;
                \sim ; 0 ; 0 ;
                \sim ; \sim ; 0 
            } \mymid \PA H = 0
         \right\}, \\
        & \simpleadjustbox{\coneCS = \left\{ 
            H = \MatrixNine{
                0 ; 0 ; 0 ;
                \sim ; H_{\MultiIndexFirstSC{3}{3}} ; H_{\MultiIndexFirstSC{3}{4}} ;
                \sim ; \sim ; H_{\MultiIndexFirstSC{4}{4}} 
            } \mymid \mymatplain{
               \PAp H = 0, \\
               H_{\MultiIndexFirstSC{2}{2}} \preceq 0 
            } 
         \right\} = \underbrace{
            \left\{ H \mymid \PAp H = 0 \right\}
         }_{=: \calC_1} \cap \underbrace{
                \left\{ 
                H = \MatrixNine{
                    0 ; 0 ; 0 ;
                    \sim ; H_{\MultiIndexFirstSC{3}{3}} \preceq 0 ; H_{\MultiIndexFirstSC{3}{4}} ;
                    \sim ; \sim ; H_{\MultiIndexFirstSC{4}{4}} 
                }
             \right\}
            }_{=: \calC_2}}.
    \end{eqnarray*}
    Since $\coneCX$ is already affine, $\affinehull(\coneCX) = \coneCX$. For $\coneCS = \calC_1 \cap \calC_2$, we shall prove that 
    \begin{align*}
        -\Ssc + \Varbar{S} \in \ri(\calC_1) \cap \ri(\calC_2).
    \end{align*}
    To see this: since $\PAp \Ssc = \PAp \Varbar{S} = \PAp C$, then $-\Ssc + \Varbar{S} \in \calC_1 = \ri(\calC_1)$; since $[-\Ssc + \Varbar{S}]_{\MultiIndexFirstSC{3}{3}} = -[\Ssc]_{\MultiIndexFirstSC{3}{3}} \prec 0$, then $-\Ssc + \Varbar{S} \in \ri(\calC_2)$. Therefore, invoking Lemma~\ref{lem:mdvZ-range:ri-aff-2}:
    \begin{align*}
        \affinehull(\coneCS) = \affinehull(\calC_1) \cap \affinehull(\calC_2)
        = \left\{ H \mymid \PAp H = 0 \right\} \cap \left\{ 
                    H = \MatrixNine{
                        0 ; 0 ; 0 ;
                        \sim ; H_{\MultiIndexFirstSC{3}{3}} ; H_{\MultiIndexFirstSC{3}{4}} ;
                        \sim ; \sim ; H_{\MultiIndexFirstSC{4}{4}} 
                    }
                \right\}.
    \end{align*}

    On the other hand, the cones in~\eqref{eq:soa:polarK-XS} are degraded to:
    \begin{eqnarray*}
        & \coneKPX = \left\{ 
            W = \MatrixNine{
                W_{\MultiIndexFirstSC{1}{1}} ; W_{\MultiIndexFirstSC{1}{2}} ; 0 ;
                \sim ; 0 ; 0 ;
                \sim ; \sim ; 0
            } \mymid \PA W = 0
         \right\}, \\
        & \simpleadjustbox{\coneKPS = \left\{ 
            W = \MatrixNine{
                0 ; 0 ; 0 ;
                \sim ; \Qfirst{0}_{\indexSecondSC{D} \indexSecondSC{D}} \MatrixFour{
                    \Varhat{W}_{\MultiIndexSecondSC{3}{3}} ; \Varhat{W}_{\MultiIndexSecondSC{3}{4}} ; 
                    \sim ; \Varhat{W}_{\MultiIndexSecondSC{4}{4}}
                } (\Qfirst{0}_{\indexSecondSC{D} \indexSecondSC{D}})\tran ; W_{\MultiIndexFirstSC{3}{4}} ;
                \sim ; \sim ; W_{\MultiIndexFirstSC{4}{4}}
            } \mymid \mymatplain{
                \PAp W = 0, \\ 
                \Varhat{W} = (\Qfirst{0}_{\indexSecondSC{D} \indexSecondSC{D}})\tran W_{\MultiIndexFirstSC{3}{3}} \Qfirst{0}_{\indexSecondSC{D} \indexSecondSC{D}}, \\
                \Varhat{W}_{\MultiIndexSecondSC{3}{3}} \preceq 0 
            }
         \right\}.}
    \end{eqnarray*}
    Thus, 
    \begin{align*}
        \coneKPX = \affinehull(\coneCX), \quad \coneKPS \subseteq \affinehull(\coneCS).
    \end{align*}
    By Theorem~\ref{thm:soa:primal-dual-decouple}, $\mdvX \in \affinehull(\coneCX)$ and $-\sigma \mdvS \in \affinehull(\coneCS)$. Thus, 
    \begin{align*}
        \mdvZ = \mdvX - \sigma \mdvS \in \affinehull(\coneCX) + \affinehull(\coneCS) = \affinehull(\coneC),
    \end{align*}
    by Lemma~\ref{lem:mdvZ-range:ri-aff-3}.
\end{proof}

It remains unclear to us under what conditions the stronger inclusion $\range(\mdvZmap)\subseteq \coneC$ holds.

\begin{remark}
    In~\cite{ding2021siopt-simplicity-lowrank-sdp}, one-sided uniqueness of optimal solution in addition with Assumption~\ref{ass:soa:sc} is called the \emph{simplicity} condition. 
\end{remark}

\subsection{\titlemath{Discussion: Connections to Almost Invariant Sets}}
\label{sec:mdvZ-range-3}
Proposition~\ref{prop:mdvZ-range:one-side-unique} indicates that, under the local second-order limit dynamics~\eqref{eq:soa:limiting-dynamics-def}, $\Delta \Vark{Z}$ lies in $\affinehull(\coneC)$ whenever $\Vark{Z}\in\coneC$ (up to higher-order terms), provided the additional  uniqueness condition holds. Indeed, in both~\eqref{eq:toy:example-1} and~\eqref{eq:toy:example-2}, the stronger inclusion $\range(\mdvZmap)\subseteq \coneC$ holds. One can readily verify that dual uniqueness holds in both examples.

On the other hand, in the nonlinear dynamics literature there is the notion of an \emph{almost invariant set}~\cite{dellnitz99sina-approximation-complicated-dyanmical}. Informally, an almost invariant set is a region of the state space that trajectories tend to remain in for a long time, with only a small probability (or small ``leakage'') of leaving over a prescribed time horizon. This raises the question of whether $\coneC \cap \mathbb{B}_{r}(\Zbar)$, for some fixed small $r>0$, can serve as a \emph{local} almost invariant set for the one-step ADMM dynamics~\eqref{eq:intro:one-step-admm}. This question is difficult to answer in general, because two competing forces must be balanced: (\romannumeral1) the local first-order dynamics~\eqref{eq:soa:fod} tends to drive $\Vark{Z}$ toward $\coneC$ (Lemma~\ref{lem:soa:convergent-fod}); (\romannumeral2) the local second-order dynamics~\eqref{eq:soa:sod} may drive $\Vark{Z}$ outside $\coneC$, as suggested by Proposition~\ref{prop:mdvZ-range:general-case}.

\paragraph{A simple visualization.}
We illustrate the two-level effects using~\eqref{eq:toy:example-1}. The results are shown in Figure~\ref{fig:mdvZ-range:toy1}. Figure~\ref{fig:mdvZ-range:toy1}(a) depicts the vector field induced by $\mdvZmap$. Figures~\ref{fig:mdvZ-range:toy1}(b)--(e) show one-step ADMM trajectories initialized at $Z^{(0)}=\Zbar+tH$ for different choices of $t$ and $H$. Across all experiments, we fix $\sigma=1$ and set the maximum number of iterations to $1000$. We make three empirical observations: (\romannumeral1) Starting from any initialization, $\Vark{Z}$ collapses to $\coneC$ in a single ADMM step, regardless of the choice of $t$. (\romannumeral2) As $t\to 0$, the decrease in $\normF{\Delta \Vark{Z}}$ is much faster than $\normF{\Vark{Z}-\Zbar}$, which remains of order $O(t)$. (\romannumeral3) The trajectories of $\Vark{Z}$ closely resemble the theoretical vector field in (a), regardless of the choice of $t$. Taken together, these observations suggest that $\coneC$ in~\eqref{eq:toy:example-1} is very likely to be an almost invariant set. 
\input{figs/mdvZ-range/toy1.tex}

%% file: figs/mdvZ-range/toy1.tex

\begin{figure}[h]
    \centering
    \begin{minipage}{\textwidth}
        \centering
        \begin{tabular}{ccc}
            \begin{minipage}{0.32\textwidth}
                \centering
                \includegraphics[width=\columnwidth]{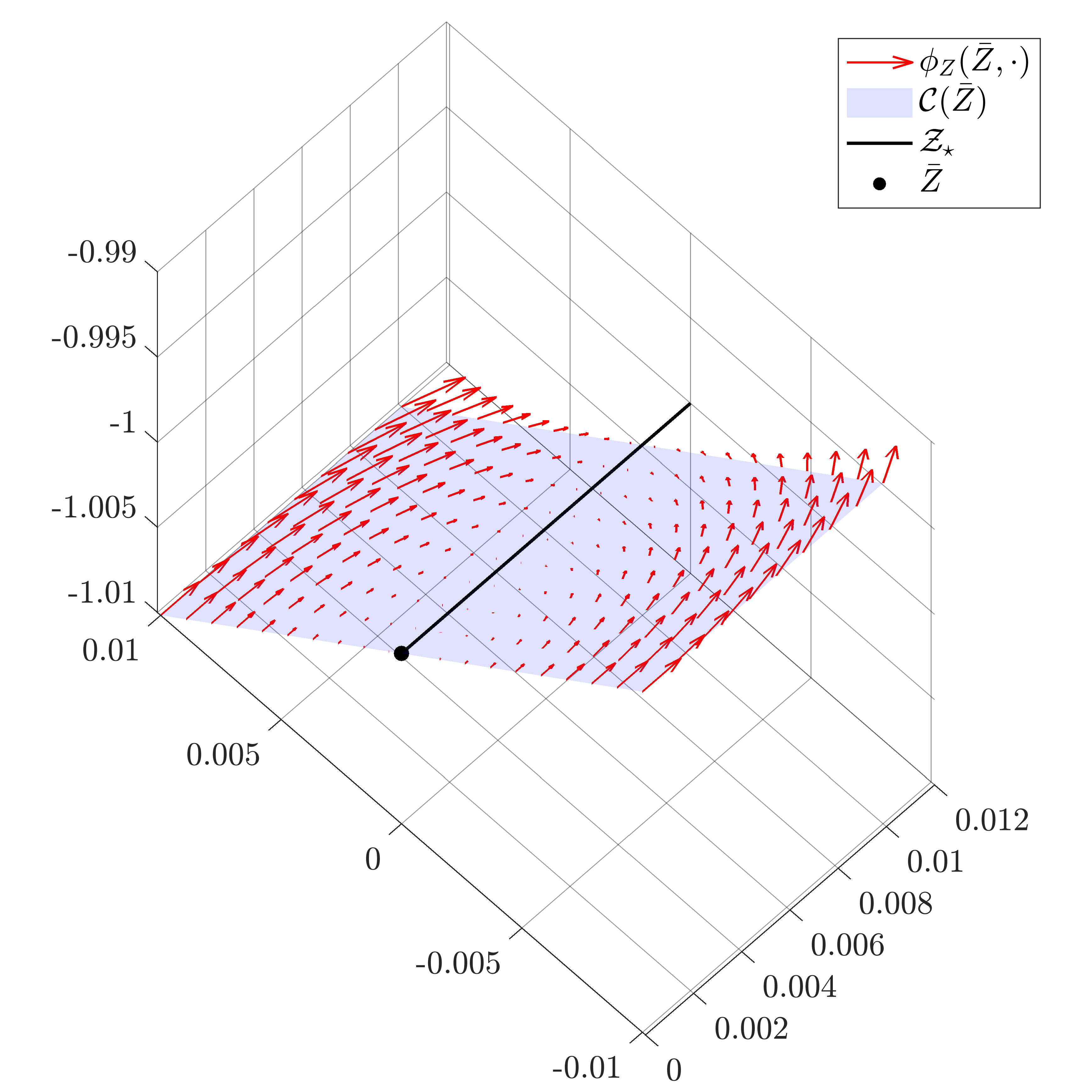}
                (a) Theoretical Results
            \end{minipage}

            \begin{minipage}{0.32\textwidth}
                \centering
                \includegraphics[width=\columnwidth]{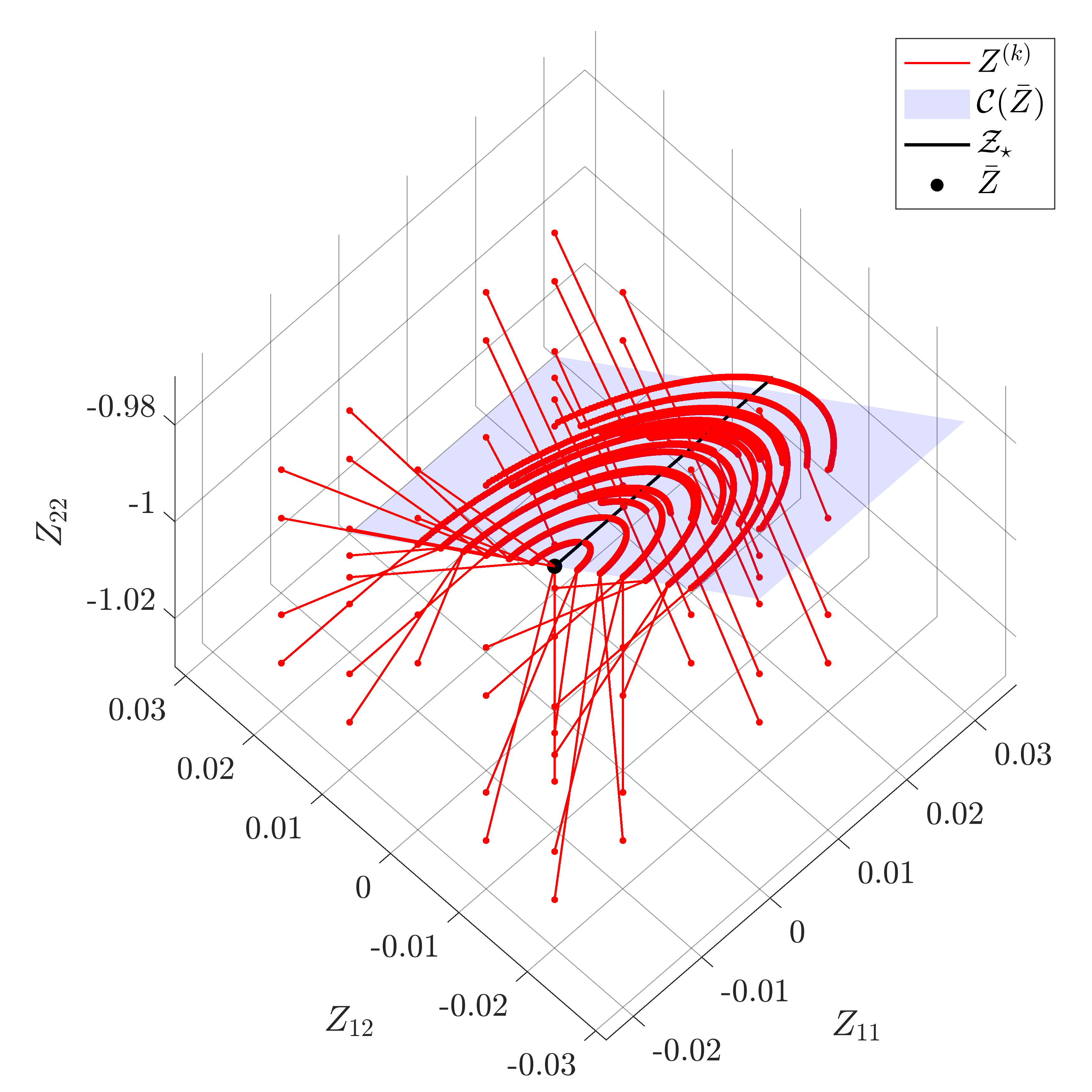}
                (b) $t = 10^{-2}$
            \end{minipage}

            \begin{minipage}{0.32\textwidth}
                \centering
                \includegraphics[width=\columnwidth]{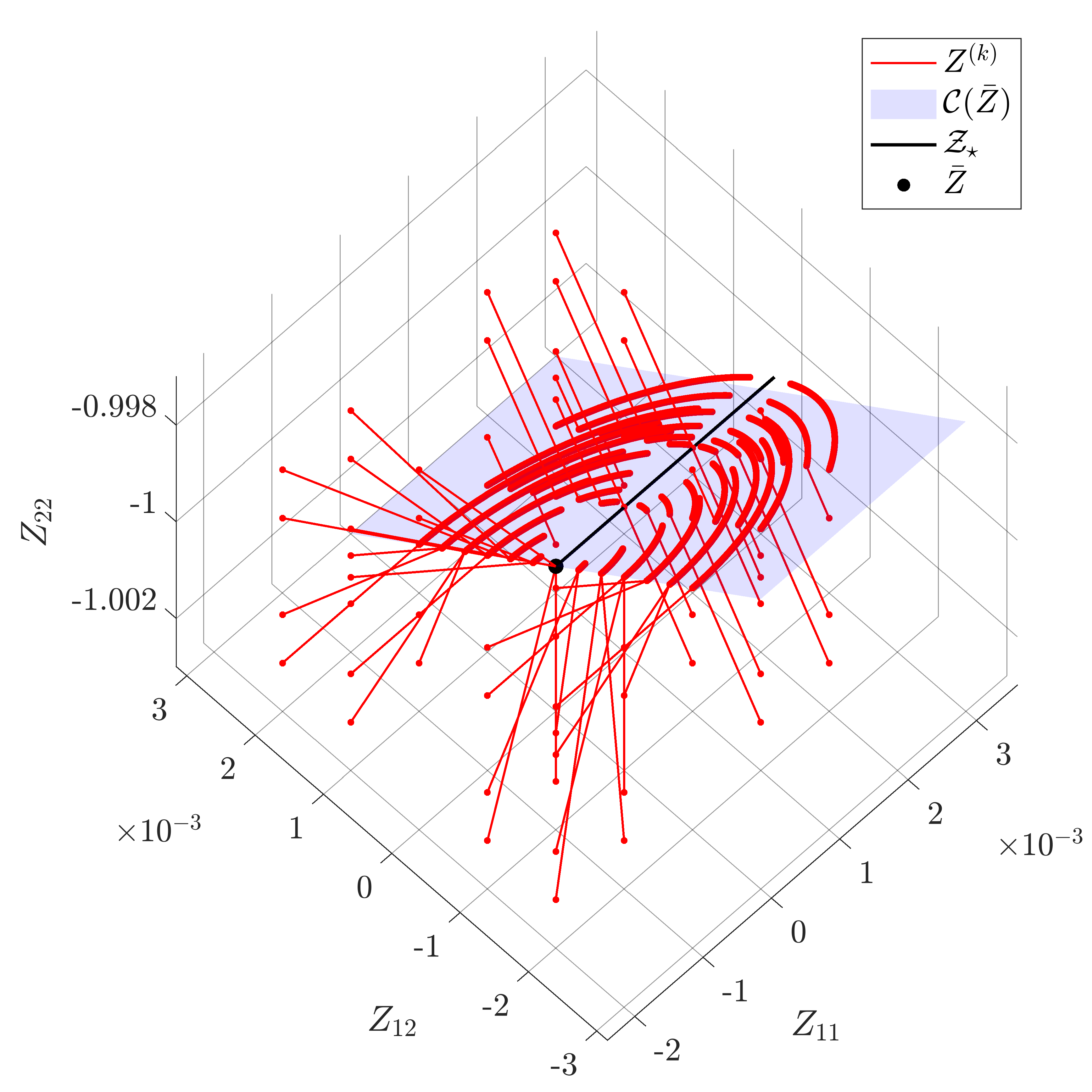}
                (c) $t = 10^{-3}$
            \end{minipage}
        \end{tabular}
    \end{minipage}

    \begin{minipage}{\textwidth}
        \centering
        \begin{tabular}{cc}
            \begin{minipage}{0.32\textwidth}
                \centering
                \includegraphics[width=\columnwidth]{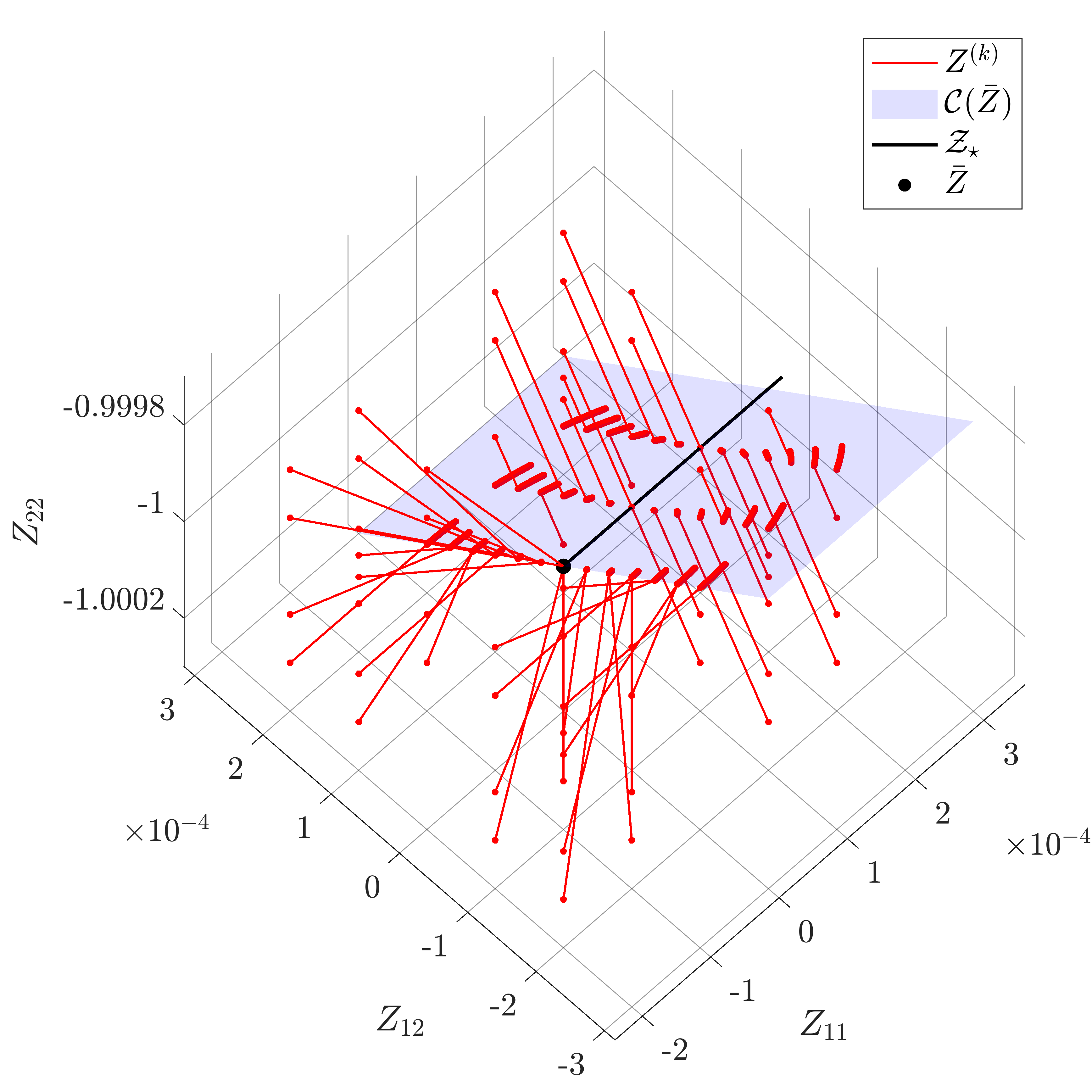}
                (d) $t = 10^{-4}$
            \end{minipage}

            \begin{minipage}{0.32\textwidth}
                \centering
                \includegraphics[width=\columnwidth]{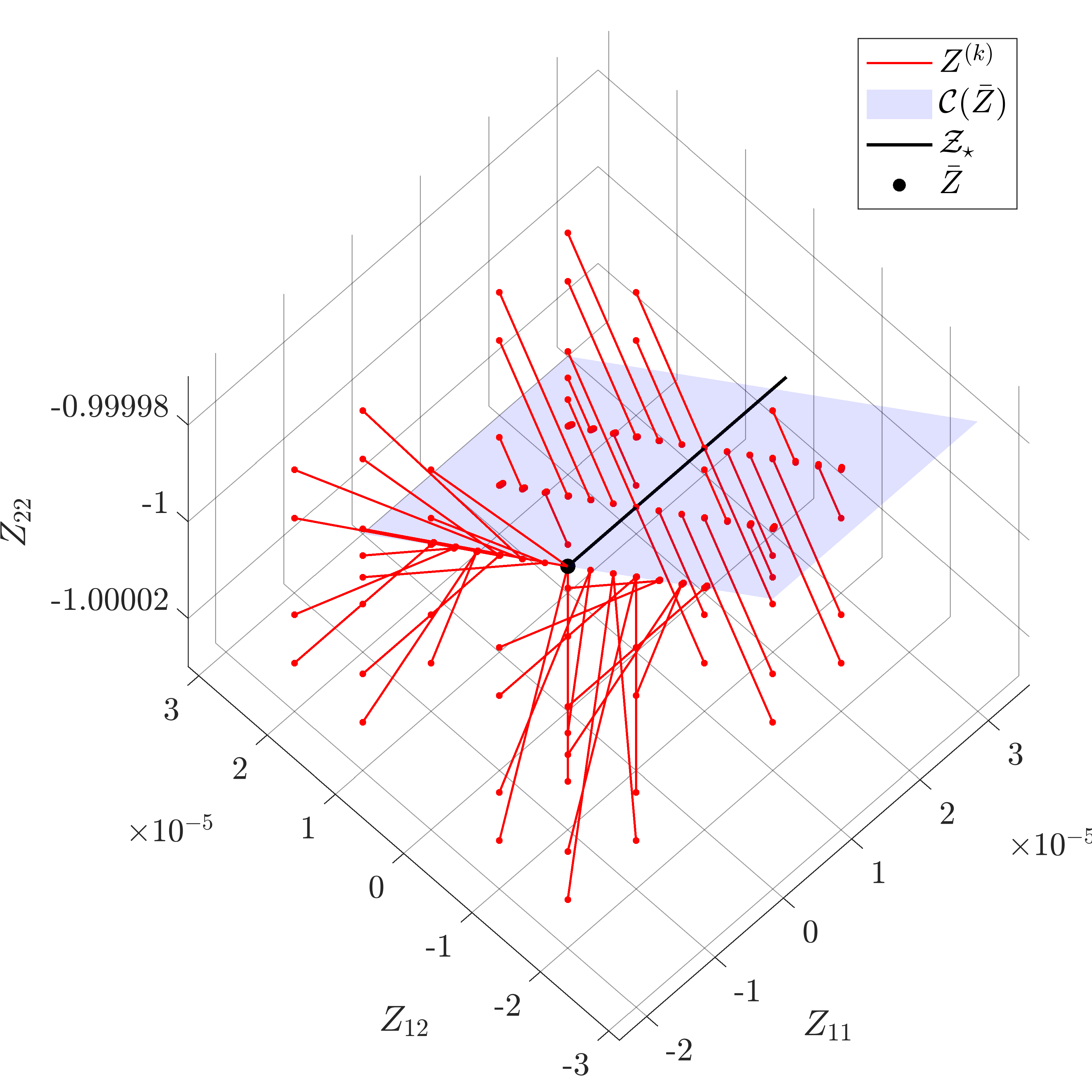}
                (e) $t = 10^{-5}$
            \end{minipage}
        \end{tabular}
    \end{minipage}

        \caption{\label{fig:mdvZ-range:toy1} (a) The theoretical vector field $\mdvZ[\Zbar; H]$ in~\eqref{eq:toy:example-1}, where $H \in \coneC$. (b)--(e) In~\eqref{eq:toy:example-1}, trajectories of $\Vark{Z}$ from different initializations $Z^{(0)}$ with varying $t$ and first-order perturbation $H$. We sweep $t$ from $10^{-2}$ to $10^{-5}$. For each fixed $t$, $(H_{11}, H_{12}, H_{22})$ is sampled from $\{-2,-1,1,2\}^3$, yielding $64$ initial points in total. 
    }
\end{figure}

%% file: sections/mdvZ_continuity.tex

\section{\titlemath{Continuity of $\mdvZ[\Zbar; \cdot]$}}
\label{sec:mdvZ-continuity}

The third property of $\mdvZmap$ that we study is its continuity on $\coneC$. Perhaps surprisingly, although the residual mapping of the one-step ADMM update~\eqref{eq:soa:residual} is continuous on the entire ambient space $\Sym{n}$, the induced second-order limit map can be discontinuous. Indeed, as defined in~\eqref{eq:soa:polarK-XS}, the cone-valued mapping $\coneKP[\cdot]$ may lose continuity at a point $\Hbar$ satisfying $\det(\Varbar{H}[0][0])=0$, which provides a potential source of discontinuity for $\mdvZmap$. We construct an explicit example in Proposition~\ref{prop:mdvZ-continuity:discontinuity} (\S\ref{sec:mdvZ-continuity-1}). On the positive side, we show that the set of discontinuity points of $\mdvZmap$ has Lebesgue measure zero on $\affinehull(\coneC)$ (\cf Proposition~\ref{prop:mdvZ-continuity:measure-zero} in \S\ref{sec:mdvZ-continuity-2}). Moreover, except for the trivial case $\coneC=\coneT$, the set $\coneC\backslash \coneT$---where $\mdvZmap$ is nonzero (\cf Proposition~\ref{prop:mdvZ-kernel:kernel})---has infinite Lebesgue measure (\cf Proposition~\ref{prop:mdvZ-continuity:coneC-divide-coneT} in \S\ref{sec:mdvZ-continuity-2}). Together, these results establish an ``almost-sure'' type continuity of $\mdvZmap$ on $\coneC$.

In \S\ref{sec:mdvZ-continuity-3}, we discuss a subtle phenomenon in slow-convergence regions. For most iterations, the angle $\angle(\Delta \Vark{Z}, \Delta \Varkpo{Z})$ tends to be small and varies smoothly, as described in \S\ref{sec:mdvZ-kernel-3}. Occasionally, however, $\angle(\Delta \Vark{Z}, \Delta \Varkpo{Z})$ can spike to a large value (often close to $\frac{\pi}{2}$) before quickly returning to a small value. We use the almost-sure continuity of $\mdvZmap$ to explain these ``sparse spikes'' in the slow-convergence regime. For small-scale SDP instances, our surrogate limiting model~\eqref{eq:soa:limiting-dynamics-def} can even accurately predict such microscopic phase transitions.

\subsection{\titlemath{Existence of Discontinuity}}
\label{sec:mdvZ-continuity-1}

\begin{proposition}[Discontinuity in $\mdvZmap$]
    \label{prop:mdvZ-continuity:discontinuity}
    There exists an SDP data triplet $(\Asdp, b, C)$ satisfying Assumption~\ref{ass:soa:sc} with $\Zbar \in \Zopt$, $\{H^i\}_{i=1}^\infty \in \ker(\deltafirst)$, and $\Hbar \in \ker(\deltafirst)$, s.t. 
    \begin{align*}
        \lim\limits_{i \rightarrow \infty} H^i = \Hbar,~\text{yet}~\lim\limits_{i \rightarrow \infty} \mdvZ[\Zbar, H^i] \ne \mdvZ.
    \end{align*}
\end{proposition}
\begin{proof}
    Please see~\eqref{eq:toy:example-3} for a constructive example. Concretely, under the SDP data provided by~\eqref{eq:toy:example-3}, we choose a real sequence $\epsilon_i \downarrow 0$ as $i \rightarrow \infty$. For $\Hbar(h, \epsilon)$ in~\eqref{eq:toy3:Hbar}, define $H^i := \Hbar(h, \epsilon_i)$, $\Hbar := \Hbar(h, 0)$. As long as $\epsilon_i \ge 0$, $\{H^i\}_{i=1}^\infty$ and $\Hbar$ all belong to $\coneC \backslash \coneT$. On the other hand, from~\eqref{eq:toy3:mdvZ-1} and~\eqref{eq:toy3:mdvZ-2}, as long as $h > \sqrt{2}$, we have 
    \begin{align*}
        \simpleadjustbox{
        \lim\limits_{i \rightarrow \infty} \mdvZ[\Zbar; H^i] = 
        \lim\limits_{i \rightarrow \infty} \MatrixNine{
            -\frac{4}{9 \sigma} ; 
            \mymatplain{-\frac{2\sqrt{2}}{9 \sigma} & 0 & -\frac{\epsilon_i}{3} & -\frac{h}{3}} ;
            0 ;
            \sim ;
            \MatrixSixteenSC{
                \frac{2}{9 \sigma} ; \frac{4}{9 \sigma} ; 0 ; 0 ;
                \sim ; \frac{2}{9 \sigma} ; 0 ; \frac{\sqrt{2}h}{12} ;
                \sim ; \sim ; \frac{h^2 - 2}{9}; -\frac{h}{3} ;
                \sim ; \sim ; \sim ; -\frac{2 h^2 - 1}{9}
            } ;
            \mymatplain{-\frac{2}{3\sigma} \\ 0 \\ \frac{2\epsilon_i}{3\sigma} \\ 0} ;
            \sim ;
            \sim ; 
            \frac{h^2 + 1}{9} 
        } = \MatrixNine{
            -\frac{4}{9 \sigma} ; 
            \mymatplain{-\frac{2\sqrt{2}}{9 \sigma} & 0 & 0 & -\frac{h}{3}} ;
            0 ;
            \sim ;
            \MatrixSixteenSC{
                \frac{2}{9 \sigma} ; \frac{4}{9 \sigma} ; 0 ; 0 ;
                \sim ; \frac{2}{9 \sigma} ; 0 ; \frac{\sqrt{2}h}{12} ;
                \sim ; \sim ; \frac{h^2 - 2}{9}; -\frac{h}{3} ;
                \sim ; \sim ; \sim ; -\frac{2 h^2 - 1}{9}
            } ;
            \mymatplain{-\frac{2}{3\sigma} \\ 0 \\ 0 \\ 0} ;
            \sim ;
            \sim ; 
            \frac{h^2 + 1}{9} 
        },
        }
    \end{align*}
    and 
    \begin{align*}
        \mdvZ = \MatrixNine{
            -\frac{4}{9 \sigma} ; 
            \mymatplain{-\frac{2\sqrt{2}}{9 \sigma} & 0 & 0 & -\frac{h}{3}} ;
            0 ;
            \sim ;
            \MatrixSixteenSC{
                \frac{2}{9 \sigma} ; \frac{4}{9 \sigma} ; 0 ; 0 ;
                \sim ; \frac{2}{9 \sigma} ; 0 ; \frac{\sqrt{2}h}{12} ;
                \sim ; \sim ; 0; -\frac{h}{3} ;
                \sim ; \sim ; \sim ; -\frac{h^2}{6} 
            } ;
            \mymatplain{-\frac{2}{3\sigma} \\ 0 \\ 0 \\ 0} ;
            \sim ;
            \sim ; 
            \frac{h^2}{6}
    }.
    \end{align*}
    Clearly, $\lim_{i \rightarrow \infty} \mdvZ[\Zbar, H^i] \ne \mdvZ$.
\end{proof}

\subsection{\titlemath{Almost-Sure Continuity}}
\label{sec:mdvZ-continuity-2}

In~\eqref{eq:toy:example-3}, the point $\Hbar$ at which $\mdvZmap$ loses continuity corresponds to $\Varbar{H}[0][0]$ being rank deficient, \ie $\det(\Varbar{H}[0][0])=0$. In the next lemma, we show that $\mdvZmap$ is continuous at every $\Hbar$ whose $\Varbar{H}[0][0]$ is nonsingular.

\begin{lemma}[\titlemath{Continuity of $\mdvZmap$ at nonsingular $\Varbar{H}[0][0]$}]
    \label{lem:mdvZ-continuity:nonsingular-H00}
    Under Assumption~\ref{ass:soa:sc}, suppose further that $\Zbar \in \Zopt$ is singular, \ie $\sizeof{\indexfirst{0}} > 0$. Then $\mdvZmap$ is continuous at every $\Hbar \in \coneC$ such that $\Varbar{H}[0][0]$ is nonsingular.
\end{lemma}
\begin{proof}
    We prove the continuity of $\mdvXmap$ and $\mdvSmap$ separately, and begin with the primal part.

    Since $\Varbar{H}[0][0]$ is nonsingular, we have $\sizeof{\indexsecond{0}{0}} = 0$. Hence, by~\eqref{eq:soa:polarK-XS}, the cone $\coneKPX$ reduces to
    \begin{align*}
        & \coneKPX = \left\{
                W = \MatrixNine{
                    W_{\MultiIndexFirst{1}{1}} ; W_{\MultiIndexFirst{1}{2}} ; 0 ;
                    \sim ;
                    {\Qfirst{0} \MatrixFourSC{
                        \Varhat{W}_{\MultiIndexSecondPD{1}{1}} ; 0 ;
                        \sim ; 0
                    } (\Qfirst{0})\tran}
                    ; 0 ;
                    \sim ; \sim ; 0
                } \mymid
                \mymatplain{
                    \PA W = 0, \\
                    \Varhat{W} = (\Qfirst{0})\tran W_{\MultiIndexFirst{2}{2}} \Qfirst{0}
                }
            \right\} \\
        =\ & \underbrace{\left\{
            W = \MatrixSixteenSC{
                W_{\MultiIndexFirstSC{1}{1}} ; W_{\MultiIndexFirstSC{1}{2}} ; W_{\MultiIndexFirstSC{1}{3}} ; 0 ;
                \sim ; W_{\MultiIndexFirstSC{2}{2}} ; 0 ; 0 ;
                \sim ; \sim ; 0 ; 0 ;
                \sim ; \sim ; \sim ; 0
            }
         \right\}}_{=: \calM_1} \cap \underbrace{
            \left\{ W \mymid \PA W = 0 \right\}
         }_{=:\calM_2},
    \end{align*}
    where the last equality uses Lemma~\ref{lem:soa:Q0-block-structure}. 
    By Proposition~\ref{prop:soa:coneC} (2), define 
    \begin{align*}
        \calC_1 := \left\{
            H = \MatrixSixteenSC{
                H_{\MultiIndexFirstSC{1}{1}} ; H_{\MultiIndexFirstSC{1}{2}} ; H_{\MultiIndexFirstSC{1}{3}} ; 0 ;
                \sim ; H_{\MultiIndexFirstSC{2}{2}} ; 0 ; 0 ;
                \sim ; \sim ; 0 ; 0 ;
                \sim ; \sim ; \sim ; 0
            } \mymid H_{\MultiIndexFirstSC{2}{2}} \succeq 0
         \right\}, \quad \calC_2 = \calM_2.
    \end{align*}
    Then, $\coneCX = \calC_1 \cap \calC_2$, $\affinehull(\calC_1) = \calM_1$, and $\affinehull(\calC_2) = \calM_2$. With the observation that $\Xsc - \Varbar{X} \in \ri(\calC_1) \cap \ri(\calC_2)$, we have $\coneKPX = \affinehull(\calC_1) \cap \affinehull(\calC_2) = \affinehull(\coneCX)$ by Lemma~\ref{lem:mdvZ-range:ri-aff-2}.
    In particular, the above description of $\coneKPX$ at nonsingular $\Varbar{H}[0][0]$ is independent of $\Hbar$.

    Next, by Weyl's theorem, for any fixed such $\Hbar$, there exists $\epsilon>0$ such that for all $H \in \mathbb{B}_{\epsilon}(\Hbar)$,
    we have $\det(\Var{H}[0][0]) \ne 0$. Therefore, for all $H \in \mathbb{B}_{\epsilon}(\Hbar)\cap \coneC$,
    \begin{align*}
        \mdvX[\Zbar; H]
        = \argmin_{W \in \coneKPX[H]} \normF{W + \opeEP[H]}^2
        = \argmin_{W \in \affinehull(\coneCX)} \normF{W + \opeEP[H]}^2
        = \Pi_{\affinehull(\coneCX)}(-\opeEP[H]),
    \end{align*}
    by Theorem~\ref{thm:soa:primal-dual-decouple}. 
    Since $\psdproj{n}(\cdot)$ (resp.\ $\nsdproj{n}(\cdot)$) is continuous on $\Sym{n}$, it follows from~\eqref{eq:soa:calE} that $-\opeEP[\cdot]$ is continuous on $\mathbb{B}_{\epsilon}(\Hbar)\cap \coneC$. Moreover, the projection mapping $\Pi_{\affinehull(\coneCX)}(\cdot)$ is continuous on $\mathbb{B}_{\epsilon}(\Hbar)\cap \coneC$. Therefore, $\mdvX[\Zbar;\cdot]$ is continuous on $\mathbb{B}_{\epsilon}(\Hbar)\cap \coneC$, and in particular continuous at $\Hbar$.

    The continuity for $\mdvSmap$ at such an $\Hbar$ can be deduced symmetrically. Thus, $\mdvZmap = \mdvXmap - \sigma \mdvSmap$ is continuous at $\Hbar$ with $\det(\Varbar{H}[0][0]) \ne 0$. 
\end{proof}

From now on, abbreviate $\affinehull(\coneC)$ as $\calL$. Suppose the dimension of $\calL$ is $d$. Let $\rho_d$ be the standard Lebesgue measure on $\Real{d}$. Fix $\calF$ as any linear isomorphism from $\Real{d}$ to $\calL$. Then, the $d$-dimension Lebesgue measure on $\calL$ is defined by 
\begin{align}
    \label{eq:mdvZ-continuity:Lebesgue-measure-L}
    \rho_\calL(A) = \rho_d(\calF^{-1}(A)), \quad \forall A \subset \calL~\text{Borel}.
\end{align} 
Please note that the choice of $\calF$ will only affect $\rho_\calL$ by a positive constant. We first show that the set of points making $\mdvZmap$ discontinuous is of measure zero in terms of $\rho_\calL$.

\begin{proposition}[Measure-zero discontinuity]
    \label{prop:mdvZ-continuity:measure-zero}
    Under Assumption~\ref{ass:soa:sc}, fix any $\Zbar \in \Zopt$. Suppose $\rho_\calL$ is defined in~\eqref{eq:mdvZ-continuity:Lebesgue-measure-L}. Then:
    \begin{align*}
        \rho_\calL\left( 
            \left\{ \Hbar \in \coneC \mymid \mdvZ~\text{is discontinuous at}~\Hbar \right\}
         \right) = 0.
    \end{align*}
\end{proposition}
\begin{proof}
    If $\Zbar$ is nonsingular, then from Corollary~\ref{cor:soa:coneC-equal-coneT} (2), $\coneC = \coneT$. We get $\mdvZ \equiv 0$ for all $\Hbar \in \coneC$ from Proposition~\ref{prop:mdvZ-kernel:kernel}. Thus, the claim trivially holds. Now let us consider the case when $\Zbar$ is singular. Invoking Lemma~\ref{lem:mdvZ-continuity:nonsingular-H00}, discontinuity only occurs when $\Varbar{H}[0][0]$ is singular. Denote the polynomial $p: \Sn \mapsto \mathbb{R}$ as $p(H) = \det(\Var{H}[0][0])$:
    \begin{align*}
        & \left\{ \Hbar \in \coneC \mymid \mdvZ~\text{is discontinuous at}~\Hbar \right\} \\
        \subseteq & \left\{ 
            \Hbar \in \coneC \mymid \det(\Varbar{H}[0][0]) = 0
         \right\} \subseteq \left\{ 
            \Hbar \in \calL \mymid p(\Varbar{H}) = 0
          \right\} =: \calD.
    \end{align*}
    All we need to prove is $\rho_\calL(\calD) = 0$. 

    (\romannumeral1) We first prove that there exists $\Vartilde{H} \in \coneC$, s.t. $p(\Vartilde{H}) \ne 0$. Set $\Vartilde{H}$ as 
    \begin{align*}
        \Vartilde{H} = (\Xsc - \Varbar{X}) - (\Ssc - \Varbar{S}) = \MatrixSixteenSC{
            \Var{[\Xsc - \Varbar{X}]}_{\MultiIndexFirstSC{1}{1}} ; \Var{[\Xsc]}_{\MultiIndexFirstSC{1}{2}} ; 0 ; 0 ; 
            \sim ; \Var{[\Xsc]}_{\MultiIndexFirstSC{2}{2}} \succ 0 ; 0 ; 0 ;
            \sim ; \sim ; -\Var{[\Ssc]}_{\MultiIndexFirstSC{3}{3}} \prec 0 ; -\Var{[\Ssc]}_{\MultiIndexFirstSC{3}{4}} ; 
            \sim ; \sim ; \sim ; -\Var{[\Ssc - \Varbar{S}]}_{\MultiIndexFirstSC{4}{4}}
        }.
    \end{align*}
    It is easy to verify that 
    \begin{align*}
        \PA \ppsim'(\Zbar; \Vartilde{H}) + \PAp \npsim'(\Zbar; \Vartilde{H}) = 
        \PA (\Xsc - \Varbar{X}) - \PAp (\Ssc - \Varbar{S}) = 0,
    \end{align*}
    and $\det(\Vartilde{H}[0][0]) \ne 0$. 

    (\romannumeral2) Consider the restriction $q = p_{\mid \calL}: \ \calL \mapsto \mathbb{R}$. Under the identification $\calL \simeq \Real{d}$ via $\calF$, $\Vartilde{q} = q \circ \calF: \ \Real{d} \mapsto \mathbb{R}$ is a polynomial on $\Real{d}$. Since $\Vartilde{q}(\calF^{-1}(\Vartilde{H})) = q(\Vartilde{H}) \ne 0$ with $\Vartilde{H} \in \calL$, $\Vartilde{q}$ is a nonzero polynomial on $\Real{d}$. From~\cite{caron05note-zero-set-polynomial}, the set $\Vartilde{\calD} := \calF^{-1}(\calD) = \{x \in \Real{d} \mid \Vartilde{q}(x) = 0 \}$ is of Lebesgue measure zero, \ie $\rho_d(\Vartilde{\calD}) = 0$. From~\eqref{eq:mdvZ-continuity:Lebesgue-measure-L},
    \begin{align*}
        \rho_\calL(\calD) = \rho_{d}(\calF^{-1}(\calD)) = \rho_{d}(\Vartilde{\calD}) = 0,
    \end{align*}
    proving the desired result.
\end{proof}

Proposition~\ref{prop:mdvZ-continuity:measure-zero} shows that the set of discontinuity points has measure zero. However, this does not rule out the possibility that $\coneC \backslash \coneT$---the set on which $\mdvZ$ does not vanish---also has measure zero. The following proposition dispels this concern.

\begin{proposition}[Measure of $\coneC \backslash \coneT$]
    \label{prop:mdvZ-continuity:coneC-divide-coneT}
    Under Assumption~\ref{ass:soa:sc}, let $\rho_\calL$ be defined in~\eqref{eq:mdvZ-continuity:Lebesgue-measure-L}. Then either of the two cases holds: (\romannumeral1) $\coneT = \coneC$; (\romannumeral2) $\coneT \subsetneq \coneC$ and $\rho_\calL(\coneC \backslash \coneT) = \infty$. 
\end{proposition}
\begin{proof}
    Case (\romannumeral1) is the trivial case, where $\mdvZ \equiv 0$ for all $\Hbar \in \coneC$.
    
    For case (\romannumeral2), there exists $\Hbar \in \coneC \backslash \coneT$, s.t. at least one of the following two conditions holds: 
    $\Hbar_{\MultiIndexFirstSC{1}{3}} \ne 0$ and $\Hbar_{\MultiIndexFirstSC{2}{4}} \ne 0$. 
    Otherwise, from Proposition~\ref{prop:soa:coneT} (2), $\Hbar \in \coneT$. Thus, from Proposition~\ref{prop:soa:coneT} (1), $\Hbar \notin \spanning(\coneT)$. This gives us $\spanning(\coneT) \subsetneq \spanning(\coneC)$, which implies $\dim(\coneT) < \dim(\coneC)$. Since $\calL = \affinehull(\coneC)$, $\rho_\calL(\coneT) = 0$. Thus, by countable additivity,
    \begin{align*}
        \rho_\calL(\coneC) = \rho_\calL(\coneC \backslash \coneT) + \rho_\calL(\coneT) = \rho_\calL(\coneC \backslash \coneT). 
    \end{align*}
    On the other hand, since $\coneC$ is a nonempty closed convex cone of dimension $d \ge 1$, $\ri(\coneC)$ at least contains a ray $\calR := \{t \Hbar \mid t > 0 \}$ with a nonzero $\Hbar \in \ri(\coneC)$. By definition, there exists $r > 0$, s.t. $A := \mathbb{B}_r(\Hbar) \cap \calL \subset \coneC$ and $\rho_\calL(A) > 0$. Define $tA := \{t H \mid H \in A\}$ for any $t > 0$, we get 
    \begin{align*}
        \rho_\calL(tA) = t^d \rho_\calL(A) \rightarrow \infty, \quad \text{as}~t \rightarrow \infty.
    \end{align*}
    Finally, since $tA\subset \coneC$ for all $t>0$, we conclude that $\rho_\calL(\coneC)=\infty$, and hence
    $\rho_\calL(\coneC\backslash \coneT)=\infty$.
\end{proof}
Proposition~\ref{prop:mdvZ-continuity:coneC-divide-coneT}, together with Proposition~\ref{prop:mdvZ-continuity:measure-zero}, describes the ``almost-sure'' continuity of $\mdvZmap$ over $\coneC$.

\subsection{\titlemath{Discussion: ``Spikes'' in $\angle (\Delta \Vark{Z}, \Delta \Varkpo{Z})$}}
\label{sec:mdvZ-continuity-3}

The almost-sure type discontinuity of $\mdvZmap$ provides a natural explanation for the microscopic phase transitions observed inside ADMM's slow-convergence regions: (\romannumeral1) When the iterates $\Vark{Z}$ pass through regions where $\mdvZ[\Zbar; \Vark{Z}-\Zbar]$ varies continuously, the angle $\angle(\Delta \Vark{Z}, \Delta \Varkpo{Z})$ remains small and evolves smoothly, as discussed in \S\ref{sec:mdvZ-kernel-3}; (\romannumeral2) when $\Vark{Z}$ hits a discontinuity point of $\mdvZmap$, the approximation $\Delta \Vark{Z} \approx \frac{1}{2}\mdvZ[\Zbar; \Vark{Z}-\Zbar]$ abruptly switches to a different displacement vector and quickly stabilizes again. Since $\Delta \Vark{Z}$ is closely related to the KKT residuals in ADMM~\cite[Lemma~4]{kang25arxiv-admm}, we also expect an observable jump in $\Vark{r}_{\max}$.

We use~\eqref{eq:toy:example-3} to illustrate the validity and accuracy of this explanation. By~\eqref{eq:toy3:mdvZ-1} and~\eqref{eq:toy3:mdvZ-2}, if $h>\sqrt{2}$, then $\Hbar(h,0)$ defined in~\eqref{eq:toy3:Hbar} is a discontinuity point of $\mdvZmap$, whereas if $h\le \sqrt{2}$, $\Hbar(h,0)$ is a continuity point. Now consider the initialization $Z^{(0)} := \Zbar + t\,\Hbar(h,\epsilon)$. When $t\rightarrow 0$ and $\epsilon$ is set to a small positive value, we expect $\angle(\Delta \Vark{Z}, \Delta \Varkpo{Z})$ to exhibit a spike when $h>\sqrt{2}$. Moreover, $\epsilon$ should affect the spike's arrival time: the smaller $\epsilon$ is, the earlier the spike occurs.

The results are shown in Figure~\ref{fig:mdvZ-continuity}. When $\epsilon=10^{-2}$ is too large, no spike is observed even when $h=1.6$. When $\epsilon=10^{-3}$, larger values of $h$ lead to earlier spike times. In addition, when $h \le 1.40$, no spike is observed within the first $1000$ iterations. Whenever $\angle(\Delta \Vark{Z}, \Delta \Varkpo{Z})$ spikes, there is also a clearly observable jump in $\Varkpo{r}_{\max}-\Vark{r}_{\max}$. By comparison, the jump in $\normF{\Delta \Varkpo{Z}}-\normF{\Delta \Vark{Z}}$ is less pronounced. The behavior for $\epsilon=10^{-4}$ is similar to that for $\epsilon=10^{-3}$. When $\epsilon=10^{-5}$, the spike occurs so early that it becomes indistinguishable from the initial transient phase before $\Vark{Z}$ has converged to $\coneC$.

\input{figs/mdvZ-continuity/toy3.tex}

%% file: figs/mdvZ-continuity/toy3.tex

\begingroup

\newcommand{\FieldFig}[3]{
    \begin{minipage}{#3\textwidth}
        \centering
        \includegraphics[width=\columnwidth]{\toyPrefix/toy3/taskVII/#1_#2_ang.png}
        \includegraphics[width=\columnwidth]{\toyPrefix/toy3/taskVII/#1_#2_kkt.png}
    \end{minipage}
}

\begin{figure}[htbp]
    \centering
    \begin{minipage}{\textwidth}
        \centering
        \includegraphics[width=0.6\columnwidth]{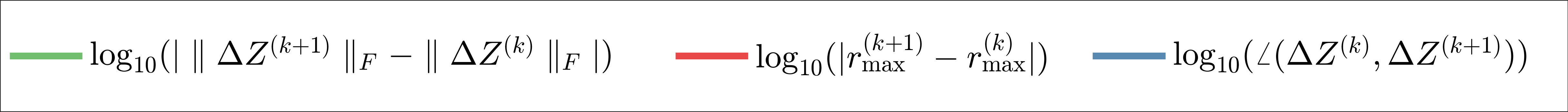}
    \end{minipage}
    \vspace{1mm}

    \begin{minipage}{\textwidth}
        \centering
        \begin{tabular}{cccc}
            \FieldFig{1}{1}{0.24}
            \FieldFig{1}{2}{0.24}
            \FieldFig{1}{3}{0.24}
            \FieldFig{1}{4}{0.24}
        \end{tabular}
    \end{minipage}

    \begin{minipage}{\textwidth}
        \centering
        \begin{tabular}{cccc}
            \FieldFig{2}{1}{0.24}
            \FieldFig{2}{2}{0.24}
            \FieldFig{2}{3}{0.24}
            \FieldFig{2}{4}{0.24}
        \end{tabular}
    \end{minipage}

    \begin{minipage}{\textwidth}
        \centering
        \begin{tabular}{cccc}
            \FieldFig{3}{1}{0.24}
            \FieldFig{3}{2}{0.24}
            \FieldFig{3}{3}{0.24}
            \FieldFig{3}{4}{0.24}
        \end{tabular}
    \end{minipage}

    \begin{minipage}{\textwidth}
        \centering
        \begin{tabular}{cccc}
            \FieldFig{4}{1}{0.24}
            \FieldFig{4}{2}{0.24}
            \FieldFig{4}{3}{0.24}
            \FieldFig{4}{4}{0.24}
        \end{tabular}
    \end{minipage}

    \caption{\label{fig:mdvZ-continuity} Trajectories of $\log_{10}(\angle (\Delta \Vark{Z}, \Delta \Varkpo{Z}) )$, $\log_{10}(| \normF{\Delta \Varkpo{Z}} - \normF{\Delta \Vark{Z}} |)$, and $\log_{10}(| \Varkpo{r}_{\max} - \Vark{r}_{\max} |)$ in~\eqref{eq:toy:example-3} with different $Z^{(0)} := \Zbar + t \Hbar(h, \epsilon)$. $t$ is fixed as $10^{-4}$ and $\sigma$ is fixed as $1$. The maximum iteration number of ADMM is $1000$. We sweep $(h, \epsilon)$ from $\{1.6, 1.5, 1.4, 1.3\} \times \{10^{-2}, 10^{-3}, 10^{-4}, 10^{-5}\}$, leading to $16$ points in total.}
\end{figure}

\endgroup

%% file: sections/mdvZ_sigma.tex

\section{\titlemath{$\sigma$'s Effect on $\mdvZ[\Zbar; \cdot]$}}
\label{sec:mdvZ-sigma}

The fourth property of $\mdvZmap$ that we study concerns its dependence on $\sigma$, the tunable penalty parameter in~\eqref{eq:intro:admm-three-step}. This issue is both theoretically and computationally important, since the choice of $\sigma$ can significantly affect ADMM's convergence behavior. In general, the dependence of the one-step residual $\delta(\cdot)$ in~\eqref{eq:soa:residual} on $\sigma$ is highly intricate. Under our local second-order limit dynamics model, however, this relationship becomes much simpler.

Specifically, in \S\ref{sec:mdvZ-sigma-1}, we show that when $\sigma$ is updated to $\sigma'$ (and $Z=\Zbar+t\Hbar+o(t)$ is updated to $Z'$), both the primal and dual iterates remain unchanged to first order as long as $\Hbar\in\coneC$. Moreover, the updated point $Z'$ continues to lie in $\Zbar' + \coneCnew$ up to first order, so a corresponding first-order direction $\Hbar'\in\coneCnew$ is well defined. At the second-order level, we obtain a clean scaling law: $\mdvX$ in~\eqref{eq:soa:mdvX} is updated to $\mdvXnew=\frac{\sigma'}{\sigma}\mdvX$, and $\mdvS$ in~\eqref{eq:soa:mdvS} is updated to $\mdvSnew=\frac{\sigma}{\sigma'}\mdvS$ (\cf \S\ref{sec:mdvZ-sigma-2}). An immediate corollary is that, under the second-order limiting model, both the primal and dual infeasibilities are invariant to $\sigma$ tuning. Finally, in \S\ref{sec:mdvZ-sigma-3}, we discuss practical strategies for updating $\sigma$ in the second-order-dominant regime.


\subsection{\titlemath{First-Order Effect}}
\label{sec:mdvZ-sigma-1}

Suppose we change $\sigma$ to $\sigma'$. Then, for $\Var{Z} = \Varbar{Z} + t \Varbar{H} + \frac{t^2}{2} \Var{W} + o(t^2)$ with $X = \ppsim(Z)$ and $S = -\frac{1}{\sigma} \npsim(Z)$, it is updated to:
\begin{align}
    \label{eq:mdvZ-sigma:Zbar-prime}
    \Var{Z}' := \Var{X}' - \sigma' \Var{S}' = X - \frac{\sigma'}{\sigma} S
    = \ppsim (\Var{Z}) + \frac{\sigma'}{\sigma} \npsim (\Var{Z}).
\end{align}
For the KKT point $\Varbar{Z} := \Varbar{X} - \sigma \Varbar{S}$, it is updated to $\Varbar{Z}' := \Varbar{X} - \sigma' \Varbar{S}$. Its corresponding eigenvalues are updated as follows:
\begin{align*}
    \eigvalfirst{k}' = \begin{cases}
        \frac{\sigma'}{\sigma} \eigvalfirst{k}, & \quad k \in \totalsetfirst[-] \\
        \eigvalfirst{k}, & \quad \text{otherwise}
    \end{cases}.
\end{align*}
The corresponding optimal set is changed to $\Zopt' := \{ X - \sigma' S \mid X \in \Xopt, S \in \Sopt \}$. We aim to expand the new $\Var{Z}'$ around the new KKT point $\Varbar{Z}'$ up to second-order, \ie $\Var{Z}' = \Varbar{Z}' + t \Varbar{H}' + \frac{t^2}{2} \Var{W}' + o(t^2)$.
From~\eqref{eq:mdvZ-sigma:Zbar-prime}, 
\begin{align}
    \label{eq:mdvZ-sigma:Hbar-prime}
    \Varbar{H}' = \ppsim'(\Varbar{Z}; \Varbar{H}) + \frac{\sigma'}{\sigma} \npsim'(\Varbar{Z}; \Varbar{H}).
\end{align} 
For an arbitrary $\Hbar \in \Sym{n}$, it does not hold in general that $\ppsim'(\Varbar{Z}'; \Varbar{H}') = \ppsim'(\Varbar{Z}; \Varbar{H})$ and $\npsim'(\Varbar{Z}'; \Varbar{H}') = \frac{\sigma'}{\sigma} \npsim'(\Varbar{Z}; \Varbar{H})$. However, as we will see, if $\Hbar \in \coneC$, these equalities do hold. 
\begin{lemma}[New partition for $\Hbar'$]
    \label{lem:mdvZ-sigma:H-new-partition}
    Under Assumption~\ref{ass:soa:sc}, if $\Varbar{H} \in \coneC$, then $\ppsim'(\Varbar{Z}'; \Varbar{H}') = \ppsim'(\Varbar{Z}; \Varbar{H})$ and $\npsim'(\Varbar{Z}'; \Varbar{H}') = \frac{\sigma'}{\sigma} \npsim'(\Varbar{Z}; \Varbar{H})$.
\end{lemma}
\begin{proof}
    From Proposition~\ref{prop:soa:coneC} (2), we have 
    \begin{align*}
        \Varbar{H} = \MatrixSixteenSC{
            \Varbar{H}[+][+]; 
            \Varbar{H}_{\indexfirst{+} \indexZeroP}; 
            \Varbar{H}_{\indexfirst{+} \indexZeroD};
            0;
            \sim;
            \Varbar{H}_{\indexZeroP \indexZeroP};
            0;
            \Varbar{H}_{\indexZeroP \indexfirst{-}};
            \sim;
            \sim;
            \Varbar{H}_{\indexZeroD \indexZeroD};
            \Varbar{H}_{\indexZeroD \indexfirst{-}};
            \sim;
            \sim;
            \sim;
            \Varbar{H}[-][-]
        }, \ \text{with } \Varbar{H}_{\indexZeroP \indexZeroP} \succeq 0, \ \Varbar{H}_{\indexZeroD \indexZeroD} \preceq 0.
    \end{align*}
    Thus,
    \begin{align}
        \label{eq:mdvZ-sigma:H-primal}
        \Varbar{H}' = \ppsim'(\Varbar{Z}; \Varbar{H}) + \frac{\sigma'}{\sigma} \npsim'(\Varbar{Z}; \Varbar{H}) 
        = \MatrixSixteenSC{
            \Varbar{H}[+][+]; 
            \Varbar{H}_{\indexfirst{+} \indexZeroP}; 
            \Varbar{H}_{\indexfirst{+} \indexZeroD};
            0;
            \sim;
            \Varbar{H}_{\indexZeroP \indexZeroP};
            0;
            \frac{\sigma'}{\sigma} \Varbar{H}_{\indexZeroP \indexfirst{-}};
            \sim;
            \sim;
            \frac{\sigma'}{\sigma} \Varbar{H}_{\indexZeroD \indexZeroD};
            \frac{\sigma'}{\sigma} \Varbar{H}_{\indexZeroD \indexfirst{-}};
            \sim;
            \sim;
            \sim;
            \frac{\sigma'}{\sigma} \Varbar{H}[-][-]
        },
    \end{align}
    and 
    \begin{align*}
        & \ppsim'(\Varbar{Z}'; \Varbar{H}') = 
        \MatrixSixteenSC{
            \Varbar{H}[+][+]; 
            \Varbar{H}_{\indexfirst{+} \indexZeroP}; 
            \Varbar{H}_{\indexfirst{+} \indexZeroD};
            0;
            \sim;
            \Varbar{H}_{\indexZeroP \indexZeroP};
            0;
            0;
            \sim;
            \sim;
            0;
            0;
            \sim;
            \sim;
            \sim;
            0
        }
        = \ppsim'(\Varbar{Z}; \Varbar{H}), \\
        & \npsim'(\Varbar{Z}'; \Varbar{H}') = 
        \MatrixSixteenSC{
            0;
            0;
            0;
            0;
            \sim;
            0;
            0;
            \frac{\sigma'}{\sigma} \Varbar{H}_{\indexZeroP \indexfirst{-}};
            \sim;
            \sim;
            \frac{\sigma'}{\sigma} \Varbar{H}_{\indexZeroD \indexZeroD};
            \frac{\sigma'}{\sigma} \Varbar{H}_{\indexZeroD \indexfirst{-}};
            \sim;
            \sim;
            \sim;
            \frac{\sigma'}{\sigma} \Varbar{H}[-][-]
        }
        = \frac{\sigma'}{\sigma} \npsim'(\Varbar{Z}; \Varbar{H}),
    \end{align*}
    proving the claim.
\end{proof}
Since $X' = \ppsim(Z')$ and $S' = -\frac{1}{\sigma'} \npsim(Z')$, from Lemma~\ref{lem:mdvZ-sigma:H-new-partition}:
\begin{align*}
    & X' = \ppsim(\Zbar') + t \ppsim'(\Zbar'; \Hbar') + o(t) = \Varbar{X} + t \ppsim'(\Zbar; \Hbar) + o(t) = X + o(t), \\
    & S' = -\frac{1}{\sigma'} \npsim(\Zbar') - t \cdot \frac{1}{\sigma'} \npsim'(\Zbar'; \Hbar') + o(t)
    = \Varbar{S} - t \cdot \frac{1}{\sigma} \npsim'(\Zbar; \Hbar) + o(t) = S + o(t).
\end{align*}
Therefore, both primal and dual iterates remain unchanged up to first-order. The following results show that if $\Varbar{H} \in \coneC \backslash \coneT$, then only updating $\sigma$ cannot escape from the second-order-dominant region. 
\begin{lemma}[$\Hbar'$ in $\coneCnew \backslash \coneTnew$]
    \label{lem:mdvZ-sigma:H-still-stall}
    Under Assumption~\ref{ass:soa:sc}, if $\Varbar{H} \in \coneC \backslash \coneT$, then $\Varbar{H}' \in \coneCnew \backslash \coneTnew$.
\end{lemma}
\begin{proof}
    (i) We first show that $\Varbar{H}' \in \coneCnew$. From Lemma~\ref{lem:mdvZ-sigma:H-new-partition}, $\ppsim'(\Varbar{Z}'; \Varbar{H}') = \ppsim'(\Varbar{Z}; \Varbar{H})$ and $\npsim'(\Varbar{Z}'; \Varbar{H}') = \frac{\sigma'}{\sigma} \npsim'(\Varbar{Z}; \Varbar{H})$. Thus,
    \begin{align*}
        & \PA \ppsim'(\Varbar{Z}'; \Varbar{H}') = \PA \ppsim'(\Varbar{Z}; \Varbar{H}) = 0, \\
        & \PAp \npsim'(\Varbar{Z}'; \Varbar{H}') = \PAp \frac{\sigma'}{\sigma} \npsim'(\Varbar{Z}; \Varbar{H}) = 0 .
    \end{align*}
    Thus, $\delta'(\Zbar'; \Hbar') = 0$.

    (ii) We second show that $\Varbar{H}' \notin \coneTnew$. Proof by contradiction. Suppose $\Varbar{H}' \in \coneTnew$. From Proposition~\ref{prop:soa:coneT} (1), $\Varbar{H}'_{\indexSC{P} \indexSC{D}} = 0$. Thus, $\Varbar{H}_{\indexfirst{+} \indexZeroD} = 0$ and $\Varbar{H}_{\indexZeroP \indexfirst{-}} = 0$ from~\eqref{eq:mdvZ-sigma:H-primal}. Combining Proposition~\ref{prop:soa:coneT} (2), $\Varbar{H} \in \coneT$, which results in a contradiction. 
\end{proof}
\subsection{\titlemath{Second-Order Effect}}
\label{sec:mdvZ-sigma-2}
We show the following scaling law.
\begin{proposition}[$\sigma$'s second-order effect]
    \label{prop:mdvZ-sigma:update}
    Under Assumption~\ref{ass:soa:sc}, suppose $\Zbar \in \Zopt$ and $\Hbar \in \coneC$. When $\sigma$ is updated to $\sigma'$:
    \begin{align*}
        \mdvXnew = \frac{\sigma'}{\sigma} \mdvX, \quad 
        \mdvSnew = \frac{\sigma}{\sigma'} \mdvS.
    \end{align*}
\end{proposition}
\begin{proof}
    (\romannumeral1) We first show that 
    \begin{align*}
        \coneKPXnew = \coneKPX, \quad  
        \coneKPSnew = \coneKPS.
    \end{align*}
    To see this: from Lemma~\ref{lem:mdvZ-sigma:H-new-partition}, 
    \begin{align*}
        \Varbar{H}[0][0]' = \MatrixFourSC{
            \Varbar{H}_{\MultiIndexFirstSC{2}{2}} \succeq 0 ; 0 ; 
            \sim ; \frac{\sigma'}{\sigma} \Varbar{H}_{\MultiIndexFirstSC{3}{3}} \preceq 0 
        }.
    \end{align*}
    Thus, 
    \begin{align*}
        (\Qfirst{0})' = \Qfirst{0}~\text{and}~\begin{cases}
            \eigvalsecond{0}{i}' = \eigvalsecond{0}{i}, & \ i \in \indexsecond{0}{+} \\
            \eigvalsecond{0}{j}' = \frac{\sigma'}{\sigma} \eigvalsecond{0}{j}, & \ j \in \indexsecond{0}{-} \\
            \eigvalsecond{0}{k}' = 0, & \ k \in \indexsecond{0}{0}
        \end{cases}.
    \end{align*}
    This directly implies that $(\indexsecond{0}{+}, \indexsecond{0}{0}, \indexsecond{0}{-})$ remains the same after $\sigma$'s update.
    Since $\coneKPX$ and $\coneKPS$ only depends on $\Qfirst{0}$ and the partition $(\indexsecond{0}{+}, \indexsecond{0}{0}, \indexsecond{0}{-})$, we finished the proof.

    (\romannumeral2) For the primal part: due to $\Varbar{H}[0][0]'$'s structure, $\ppsim(\Varbar{H}[0][0]') = \ppsim(\Varbar{H}[0][0])$ and $\npsim(\Varbar{H}[0][0]') = \frac{\sigma'}{\sigma} \npsim(\Varbar{H}[0][0])$. From~\eqref{eq:soa:calE},
    \begin{align*}
        \opeEPnew = & \MatrixNine{
            0 ; 
            \left\{ 
                2 \frac{1}{-\eigvalfirst{a}'} \Varbar{H}[a][0]' \npsim(\Varbar{H}[0][0]') 
             \right\}_{a \in \totalsetfirst[+]} ;
            \left\{ 
                2 \frac{1}{\eigvalfirst{a}' - \eigvalfirst{b}'} \Varbar{H}[a][0]' \Varbar{H}[0][b]'
             \right\}_{\substack{a \in \totalsetfirst[+] \\ b \in \totalsetfirst[-]}} ;
            \sim ; 
            2 \sum_{c \in \totalsetfirst[-]} \frac{1}{\eigvalfirst{c}'} \Varbar{H}[0][c]' \Varbar{H}[c][0]' ; 
            \left\{ 
                -2 \frac{1}{\eigvalfirst{b}'} \npsim(-\Varbar{H}[0][0]') \Varbar{H}[0][b]'
             \right\}_{b \in \totalsetfirst[-]} ;
            \sim ; 
            \sim ;
            0 
        } \\
        = & \MatrixNine{
            0 ; 
            \left\{ 
                2 \frac{\sigma'}{\sigma} \frac{1}{-\eigvalfirst{a}} \Varbar{H}[a][0] \npsim(\Varbar{H}[0][0]) 
             \right\}_{a \in \totalsetfirst[+]} ;
            \left\{ 
                2 \frac{\frac{\sigma'}{\sigma}}{\eigvalfirst{a} - \frac{\sigma'}{\sigma} \eigvalfirst{b}} \Varbar{H}[a][0] \Varbar{H}[0][b]
             \right\}_{\substack{a \in \totalsetfirst[+] \\ b \in \totalsetfirst[-]}} ;
            \sim ; 
            2 \frac{\sigma'}{\sigma} \sum_{c \in \totalsetfirst[-]} \frac{1}{\eigvalfirst{c}} \Varbar{H}[0][c] \Varbar{H}[c][0] ; 
            \left\{ 
                -2 \frac{1}{\eigvalfirst{b}} \npsim(-\Varbar{H}[0][0]) \Varbar{H}[0][b]
             \right\}_{b \in \totalsetfirst[-]} ;
            \sim ; 
            \sim ;
            0 
        }.
    \end{align*}
    Therefore, by Proposition~\ref{prop:soa:polarK-XS}:
    \begin{align*}
        & \mdvXnew = \argmin_{W \in \coneKPXnew} \normF{W + \opeEPnew}^2 
        = \argmin_{W \in \coneKPX} \normF{W + \opeEPnew}^2 \\
        = & \argmin_{W \in \coneKPX} 2 \sum_{a \in \totalsetfirst[+]} \normlongF{\Var{W}[a][0] + \frac{\sigma'}{\sigma} \cdot 2 \frac{1}{-\eigvalfirst{a}} \Varbar{H}[a][0] \npsim(\Varbar{H}[0][0])}^2 
        + \normlongF{\Varbar{W}[0][0] + \frac{\sigma'}{\sigma} \cdot 2 \sum_{c \in \totalsetfirst[-]} \frac{1}{\eigvalfirst{c}} \Varbar{H}[0][c] \Varbar{H}[c][0]}^2 \\
        = & \argmin_{W \in \coneKPX} 2 \sum_{a \in \totalsetfirst[+]} \normlongF{\Var{W}[a][0] + \frac{\sigma'}{\sigma} \cdot [\opeEP]_{\indexfirst{a} \indexfirst{0}}}^2 + \normlongF{\Varbar{W}[0][0] + \frac{\sigma'}{\sigma} \cdot [\opeEP]_{\MultiIndexFirst{2}{2}}}^2 \\
        = & \argmin_{W \in \coneKPX} \normlongF{W + \frac{\sigma'}{\sigma} \cdot \opeEP}^2 
        = \Pi_{\coneKPX}(-\frac{\sigma'}{\sigma} \cdot \opeEP)
        = \frac{\sigma'}{\sigma} \cdot \mdvX.
    \end{align*}
    One may notice that the key observation here is $\Var{W}[+][-] \equiv 0, \forall W \in \coneKPX$. The last equality comes from the fact that for a closed convex cone $\calC \subset \Sym{n}$, $\Pi_\calC(\alpha x) = \alpha \Pi_\calC(x)$ for all $\alpha > 0$. 

    (\romannumeral3) For the dual part, similar to the primal part:
    \begin{align*}
        \opeE = \MatrixNine{
            0 ; 
            \left\{ 
                -2 \frac{\sigma'}{\sigma} \frac{1}{\eigvalfirst{a}} \Varbar{H}[a][0] \ppsim(-\Varbar{H}[0][0]) 
             \right\}_{a \in \totalsetfirst[+]} ;
            \left\{ 
                -2 \frac{\frac{\sigma'}{\sigma}}{\eigvalfirst{a} - \frac{\sigma'}{\sigma} \eigvalfirst{b}} \Varbar{H}[a][0] \Varbar{H}[0][b]
             \right\}_{\substack{a \in \totalsetfirst[+] \\ b \in \totalsetfirst[-]}} ;
            \sim ; 
            2 \sum_{c \in \totalsetfirst[+]} \frac{1}{\eigvalfirst{c}} \Varbar{H}[0][c] \Varbar{H}[c][0] ; 
            \left\{ 
                2 \frac{1}{-\eigvalfirst{b}} \ppsim(\Varbar{H}[0][0]) \Varbar{H}[0][b]
             \right\}_{b \in \totalsetfirst[-]} ;
            \sim ; 
            \sim ;
            0 
        }.
    \end{align*}
    Thus, by Proposition~\ref{prop:soa:polarK-XS}:
    \begin{align*}
        & \mdvSnew = -\frac{1}{\sigma'} \argmin_{W \in \coneKPSnew} \normF{W + \opeEnew}^2 
        = -\frac{1}{\sigma'} \argmin_{W \in \coneKPS} \normF{W + \opeEnew}^2 \\
        = & -\frac{1}{\sigma'} \argmin_{W \in \coneKPS} 2 \sum_{b \in \totalsetfirst[-]} \normlongF{\Var{W}[0][b] + 2 \frac{1}{-\eigvalfirst{b}} \ppsim(\Varbar{H}[0][0]) \Varbar{H}[0][b]}^2 
        + \normlongF{\Var{W}[0][0] + 2 \sum_{c \in \totalsetfirst[+]} \frac{1}{\eigvalfirst{c}} \Varbar{H}[0][c] \Varbar{H}[c][0]}^2 \\
        = & -\frac{1}{\sigma'} \argmin_{W \in \coneKPS} \normlongF{W + \opeE}^2 
        =  \frac{\sigma}{\sigma'} \cdot -\frac{1}{\sigma} \argmin_{W \in \coneKPS} \normlongF{W + \opeE}^2 = \frac{\sigma}{\sigma'} \mdvS.
    \end{align*}
    Again, the key observation is $\Var{W}[+][-] \equiv 0, \forall W \in \coneKPS$.
\end{proof}

An immediate corollary from Proposition~\ref{prop:mdvZ-sigma:update} is that, the limiting behavior of primal/dual infeasibility is \emph{irrelevant} to $\sigma$ in the second-order-dominant regions. 
\begin{corollary}
    \label{cor:mdvZ-sigma:pdinf-irrelevant-sigma}
    Under Assumption~\ref{ass:soa:sc}, let $\Zbar \in \Zopt$ and $\Hbar \in \coneC$. Under the first- and second-order local dynamics models in Definition~\ref{def:soa:fod-sod}, when $\sigma$ is updated to $\sigma'$, the limits of both $\Vark{r}_p$ and $\Vark{r}_d$ in~\eqref{eq:intro:kkt-residual} remain unchanged up to second-order.
\end{corollary} 
\begin{proof}
    From~\cite[Corollary 1]{wen10mpc-admmsdp}, 
    \begin{align*}
        \Asdp \Vark{X} - b = \sigma \Asdp (\Varkpo{S} - \Vark{S}), \quad 
        \AsdpT \Vark{y} + \Vark{S} - C = \frac{1}{\sigma} (\Varkpo{X} - \Vark{X}). 
    \end{align*}
    Thus, 
    \begin{align*}
        \Vark{r}_p = \sigma \frac{
            \normtwo{\Asdp (\Varkpo{S} - \Vark{S})}
        }{1 + \norm{b}}, \quad \Vark{r}_d = \frac{1}{\sigma} \frac{
            \normF{\Varkpo{X} - \Vark{X}}
        }{1 + \normF{C}}.
    \end{align*}
    From Theorem~\ref{thm:soa:mdvX-dmvS-def}, the local second-order limit of $\Varkpo{X} - \Vark{X}$ (resp. $\Varkpo{S} - \Vark{S}$) is $\mdvX$ (resp. $\mdvS$). Thus,
    \begin{align*}
        \lim_{k \rightarrow \infty} \Vark{r}_p \propto \sigma \normtwo{\Asdp \mdvS}, \quad 
        \lim_{k \rightarrow \infty} \Vark{r}_d \propto \frac{1}{\sigma} \normF{\mdvX}.
    \end{align*}

    On the other hand, from Proposition~\ref{prop:mdvZ-sigma:update}, 
    \begin{align*}
        & \sigma' \mdvSnew = \sigma' \frac{\sigma}{\sigma'} \mdvS = \sigma \mdvS, \\
        & \frac{1}{\sigma'} \mdvXnew = \frac{1}{\sigma'} \frac{\sigma'}{\sigma} \mdvX = \frac{1}{\sigma} \mdvX.
    \end{align*}
    which closes the proof.
\end{proof}

\subsection{\titlemath{Discussion: $\sigma$'s Updating Rules}}
\label{sec:mdvZ-sigma-3}

Traditional $\sigma$-updating heuristics typically aim to balance the primal and dual infeasibilities, under the implicit assumption that $\Delta \Vark{X} := \Varkpo{X}-\Vark{X}$ and $\Delta \Vark{S} := \Varkpo{S}-\Vark{S}$ are nearly independent of $\sigma$~\cite{wen10mpc-admmsdp,kang24wafr-strom}. However, Proposition~\ref{prop:mdvZ-sigma:update} and Corollary~\ref{cor:mdvZ-sigma:pdinf-irrelevant-sigma} indicate that such heuristics become ineffective in second-order-dominant regions, since $\Vark{r}_p$ and $\Vark{r}_d$ are (locally) insensitive to $\sigma$.

To empirically verify Proposition~\ref{prop:mdvZ-sigma:update} and Corollary~\ref{cor:mdvZ-sigma:pdinf-irrelevant-sigma}, we fix $\Zbar$ and $\Hbar$ as in \S\ref{sec:mdvZ-kernel} for all three SDP examples. Starting from $Z^{(0)} := \Zbar + t\Hbar$ with different choices of $t$, we uniformly increase $\log_{10}(\sigma)$ from $0$ to $1$ over $1000$ ADMM iterations. Since the change in $\sigma$ is gradual and the iteration horizon is moderate, we may assume that $\Delta \Vark{X}$ (resp.\ $\Delta \Vark{S}$) steadily tracks its second-order limit as $t\downarrow 0$. The results for~\eqref{eq:toy:example-1},~\eqref{eq:toy:example-2}, and~\eqref{eq:toy:example-3} are shown in Figures~\ref{fig:mdvZ-sigma:toy1},~\ref{fig:mdvZ-sigma:toy2}, and~\ref{fig:mdvZ-sigma:toy3}, respectively. When $t=10^{-5}$, the dependence of $(\Delta \Vark{X}, \Delta \Vark{S}, \Vark{r}_p, \Vark{r}_d)$ on $\sigma$ is consistent across all three examples: (\romannumeral1) $\log_{10}(\normF{\Delta \Vark{X}})$ (resp.\ $\log_{10}(\normF{\Delta \Vark{S}})$) increases (resp.\ decreases) linearly with $\log_{10}(\sigma)$, with slope close to $+1$ (resp.\ $-1$); (\romannumeral2) $\Vark{r}_p$ and $\Vark{r}_d$ remain essentially unchanged as $\sigma$ varies.

\input{figs/mdvZ-sigma/toy1.tex}
\input{figs/mdvZ-sigma/toy2.tex}
\input{figs/mdvZ-sigma/toy3.tex}

\paragraph{Discussion on the one-sided uniqueness condition.}
It is difficult to design a ``universally'' good $\sigma$-updating strategy in the second-order-dominant regime. For instance, when both primal and dual constraint nondegeneracy fail, it is likely that $\Hbar_{\MultiIndexFirstSC{1}{3}}$ in the $\coneCX$ part and $\Hbar_{\MultiIndexFirstSC{2}{4}}$ in the $\coneCS$ part are simultaneously nonzero (\eg~\eqref{eq:toy3:Hbar} in~\eqref{eq:toy:example-3}). In this case, enlarging $\sigma$ amplifies $\Delta \Vark{X}$, which may help reduce $\Hbar_{\MultiIndexFirstSC{1}{3}}$. On the other hand, it also suppresses $\Delta \Vark{S}$, which may worsen the $\coneCS$ component. The situation can be more favorable when one-sided uniqueness holds in either the primal or the dual optimal solution set. For example, when the primal solution is unique, Proposition~\ref{prop:mdvZ-range:one-side-unique}'s proof implies $\Hbar_{\MultiIndexFirstSC{2}{4}}=0$. In this case, we only need to eliminate $\Hbar_{\MultiIndexFirstSC{1}{3}}$ in the $\coneCX$ part, and it may be beneficial to choose a large $\sigma$. Symmetrically, when the dual optimal solution is unique, we only need to eliminate $\Hbar_{\MultiIndexFirstSC{2}{4}}$ in the $\coneCS$ part, and it may be beneficial to choose a small $\sigma$.

We empirically verify this analysis using three SDP examples. The initial $\Zbar\in\Zopt$ and $\Hbar\in\coneC$ are the same as in \S\ref{sec:mdvZ-kernel}. We fix $t=10^{-4}$ and initialize $\sigma=1$. After running $1000$ ADMM iterations, we update $\sigma$ to a value in $\{10^{-2},10^{-1},1,10,10^2\}$. We then record the change in the maximum KKT residual $\Vark{r}_{\max}$. The results are shown in Figure~\ref{fig:mdvZ-sigma:accelerate}. In both~\eqref{eq:toy:example-1} and~\eqref{eq:toy:example-2}, we observe a significant acceleration when $\sigma$ is updated from $1$ to $10^{-2}$. This is consistent with our analysis, since the dual optimal solution is unique in both examples. For~\eqref{eq:toy:example-3}, changing $\sigma$ does not help the iterates escape the slow-convergence region. This is not surprising, since $\Hbar(1,0)$ defined in~\eqref{eq:toy3:Hbar} has nonzero entries in both $\Hbar_{\MultiIndexFirstSC{1}{3}}$ and $\Hbar_{\MultiIndexFirstSC{2}{4}}$.

\input{figs/mdvZ-sigma/accelerate.tex}

%% file: figs/mdvZ-sigma/toy1.tex

\begin{figure}[h]
    \centering

    \begin{minipage}{\textwidth}
        \centering
        \includegraphics[width=0.5\columnwidth]{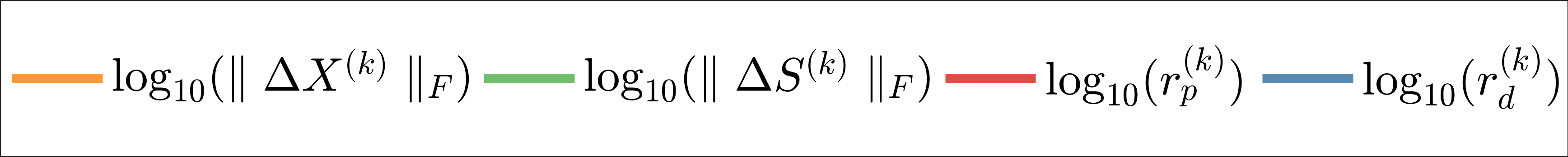}
    \end{minipage}
    \vspace{1mm}

    \begin{minipage}{\textwidth}
        \centering
        \begin{tabular}{cccc}
            \begin{minipage}{0.24\textwidth}
                \centering
                \includegraphics[width=\columnwidth]{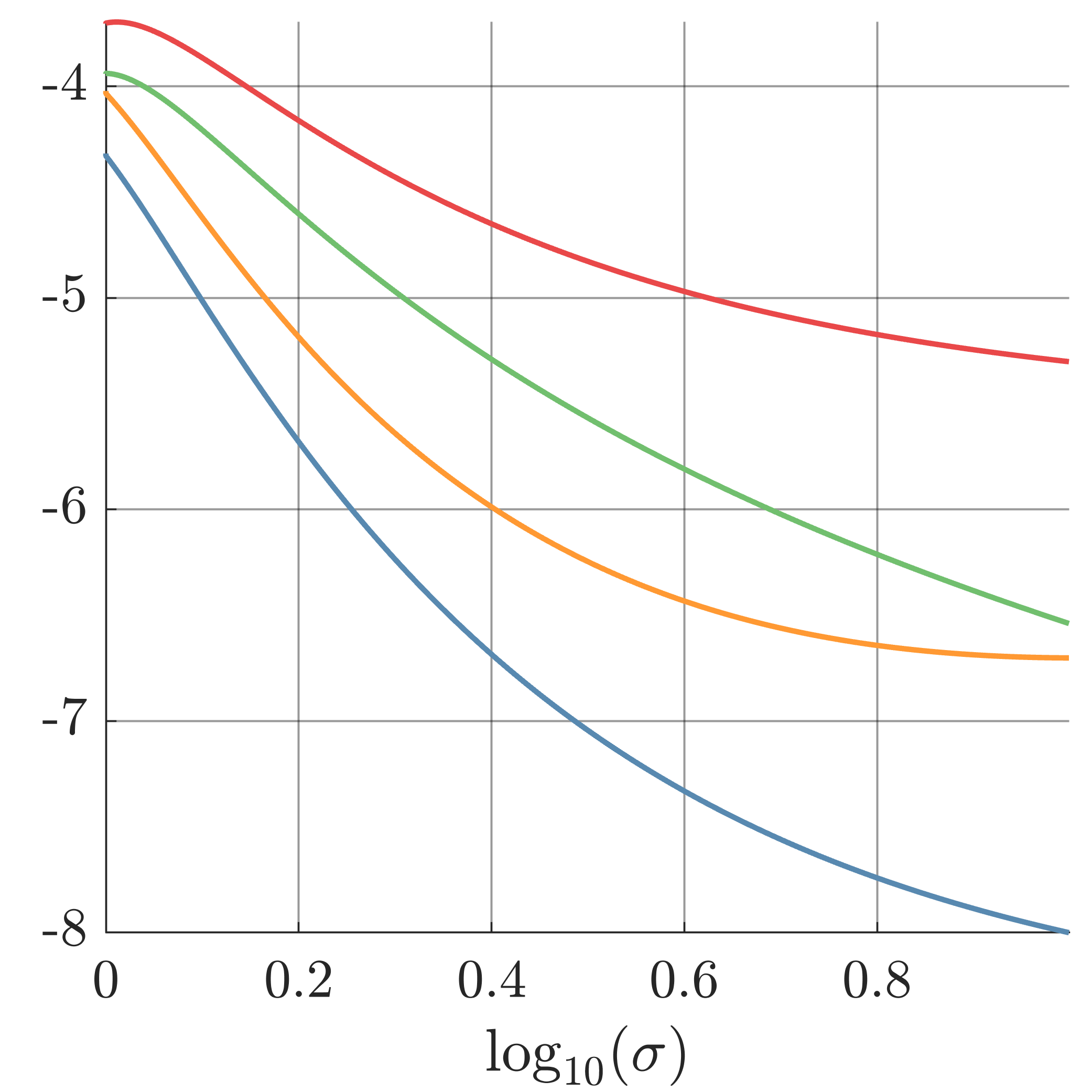}
                (a) $t = 10^{-2}$
            \end{minipage}

            \begin{minipage}{0.24\textwidth}
                \centering
                \includegraphics[width=\columnwidth]{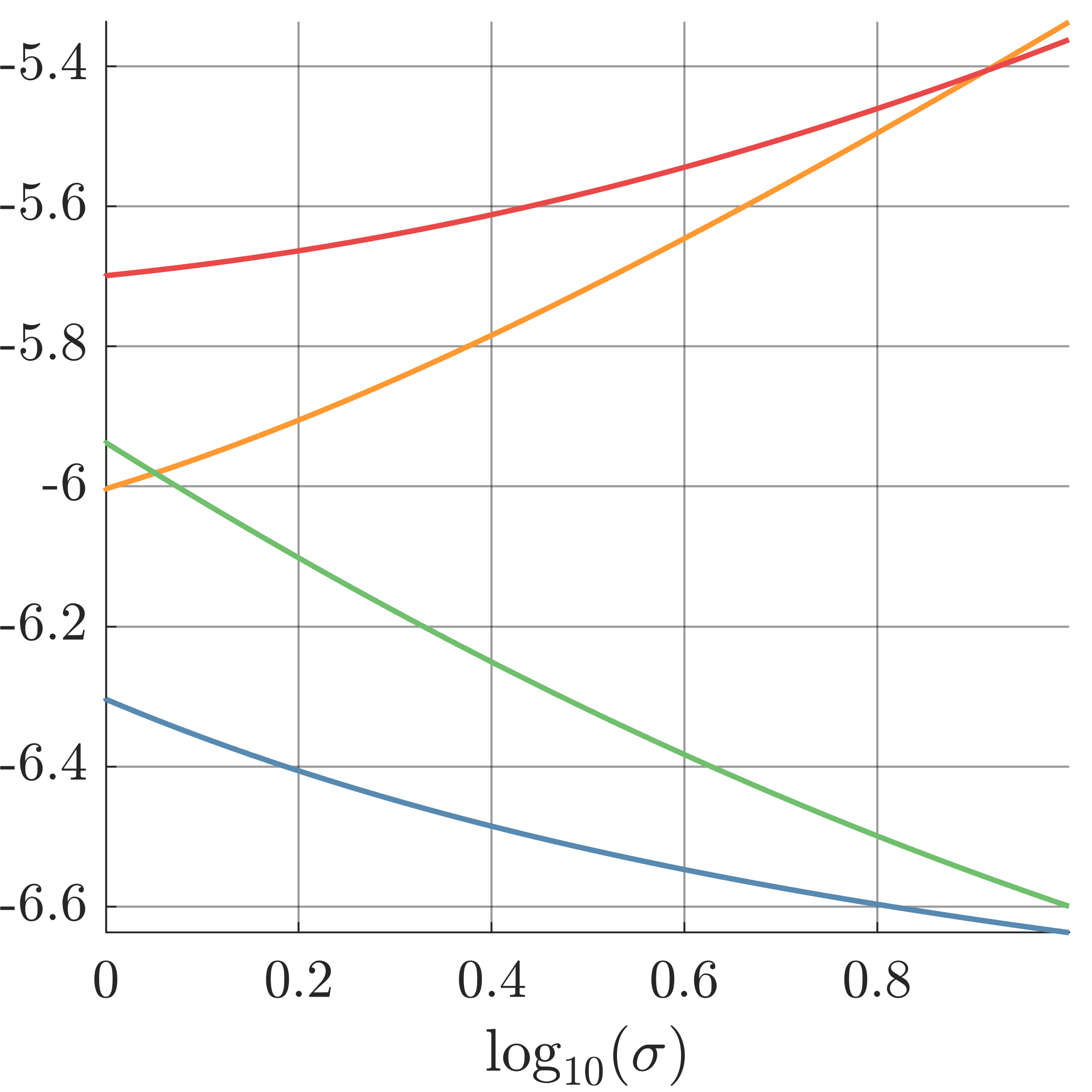}
                (b) $t = 10^{-3}$
            \end{minipage}

            \begin{minipage}{0.24\textwidth}
                \centering
                \includegraphics[width=\columnwidth]{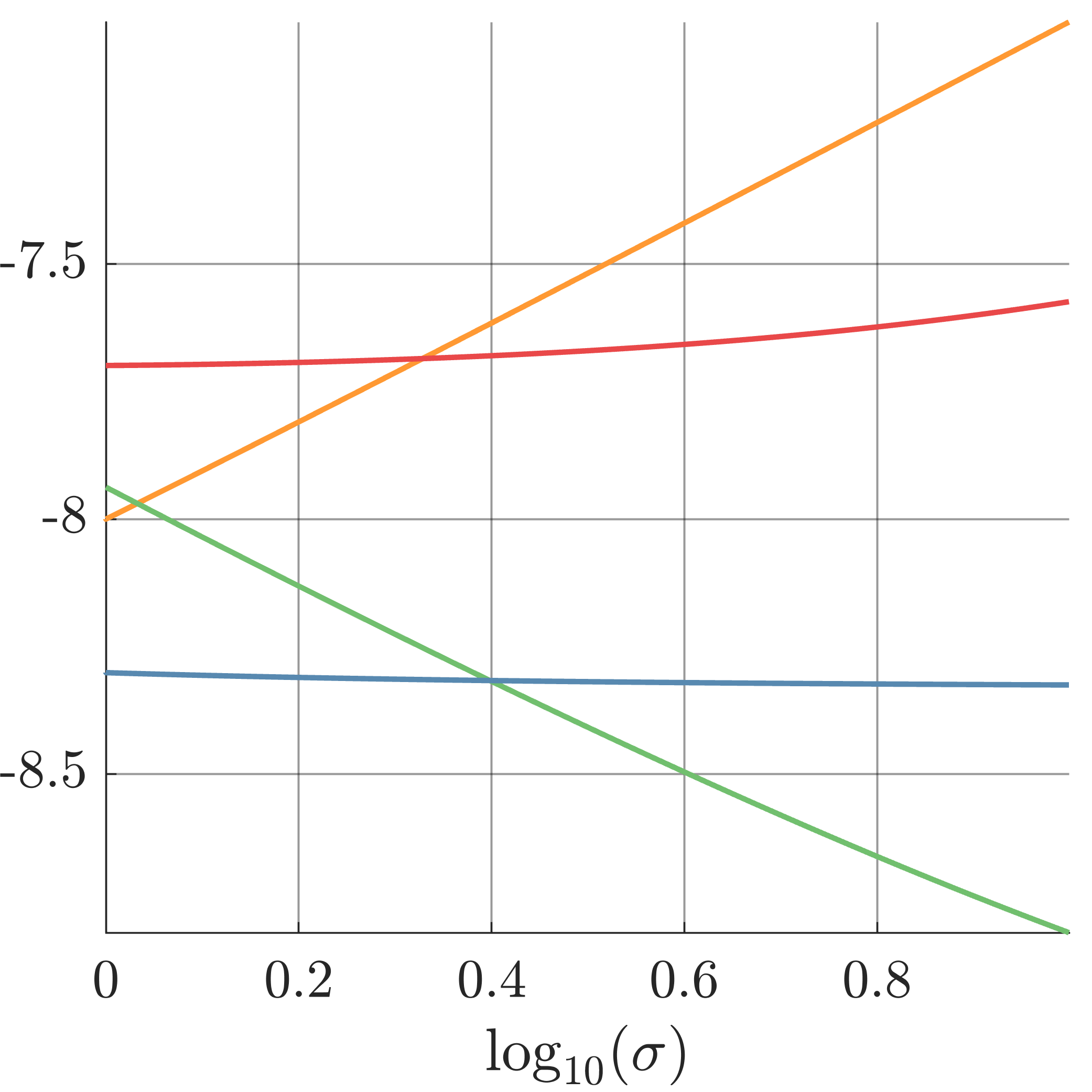}
                (c) $t = 10^{-4}$
            \end{minipage}

            \begin{minipage}{0.24\textwidth}
                \centering
                \includegraphics[width=\columnwidth]{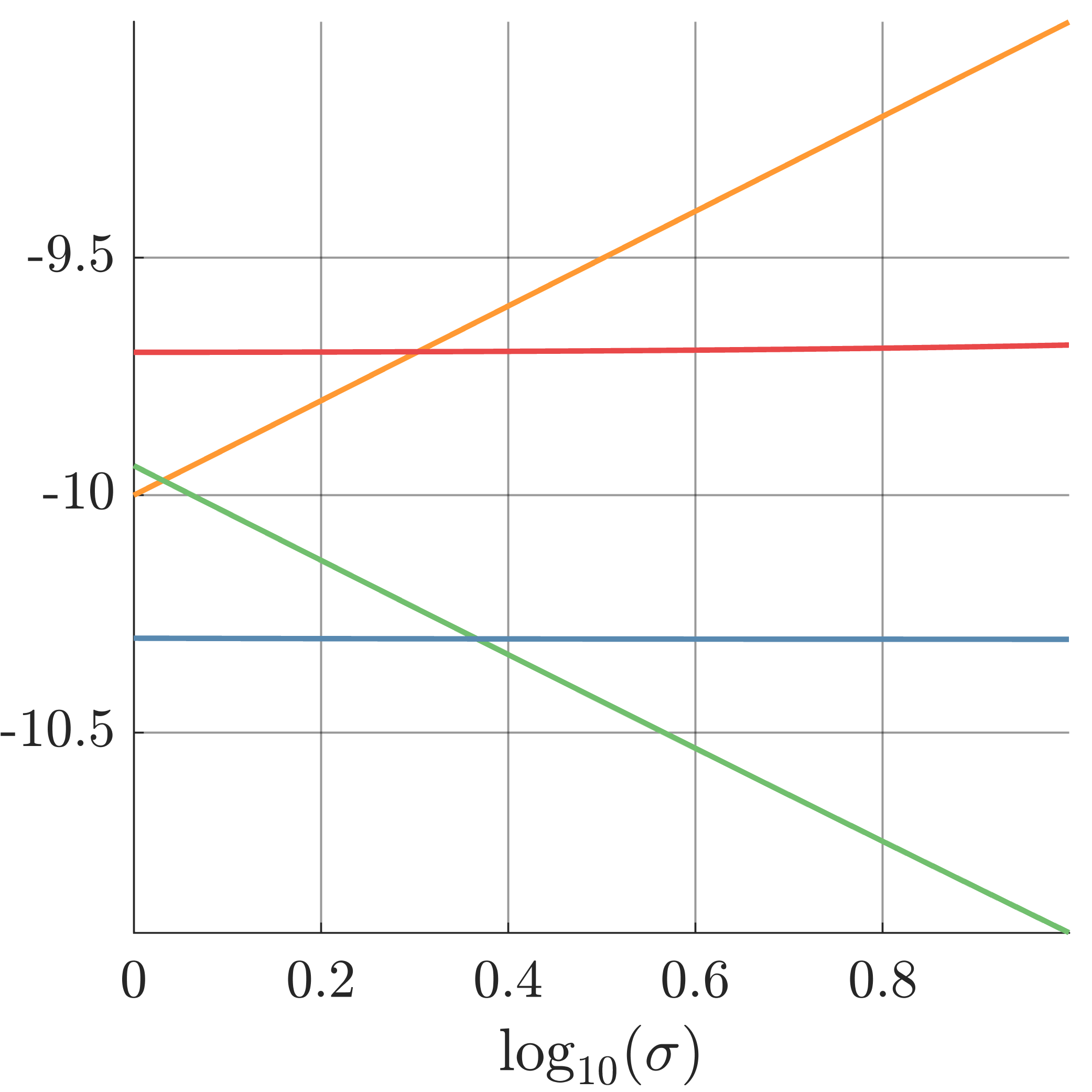}
                (d) $t = 10^{-5}$
            \end{minipage}
        \end{tabular}
    \end{minipage}
    \caption{\label{fig:mdvZ-sigma:toy1} Trajectories of $\normF{\Delta \Vark{X}}$, $\normF{\Delta \Vark{S}}$, $\Vark{r}_{p}$, and $\Vark{r}_{d}$ in~\eqref{eq:toy:example-1}. Fix $\Zbar \in \Zopt$ and $\Hbar \in \coneC$. Pick $t \in \{10^{-2}, 10^{-3}, 10^{-4}, 10^{-5}\}$ such that $Z^{(0)} = \Zbar + t\Hbar$. $\log_{10}(\sigma)$ is uniformly increased from $0$ to $1$ in $1000$ iterations.}
\end{figure}

%% file: figs/mdvZ-sigma/toy2.tex

\begin{figure}[h]
    \centering

    \begin{minipage}{\textwidth}
        \centering
        \includegraphics[width=0.5\columnwidth]{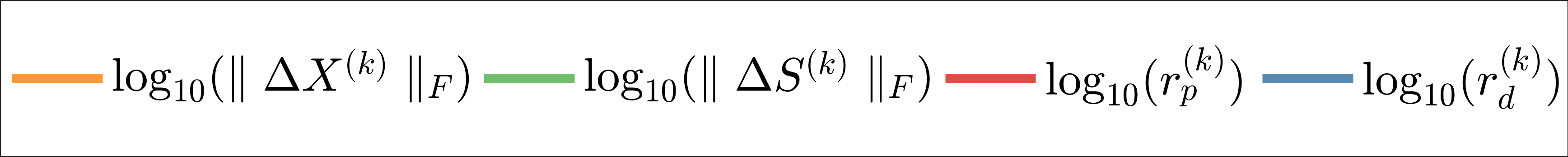}
    \end{minipage}
    \vspace{1mm}

    \begin{minipage}{\textwidth}
        \centering
        \begin{tabular}{cccc}
            \begin{minipage}{0.24\textwidth}
                \centering
                \includegraphics[width=\columnwidth]{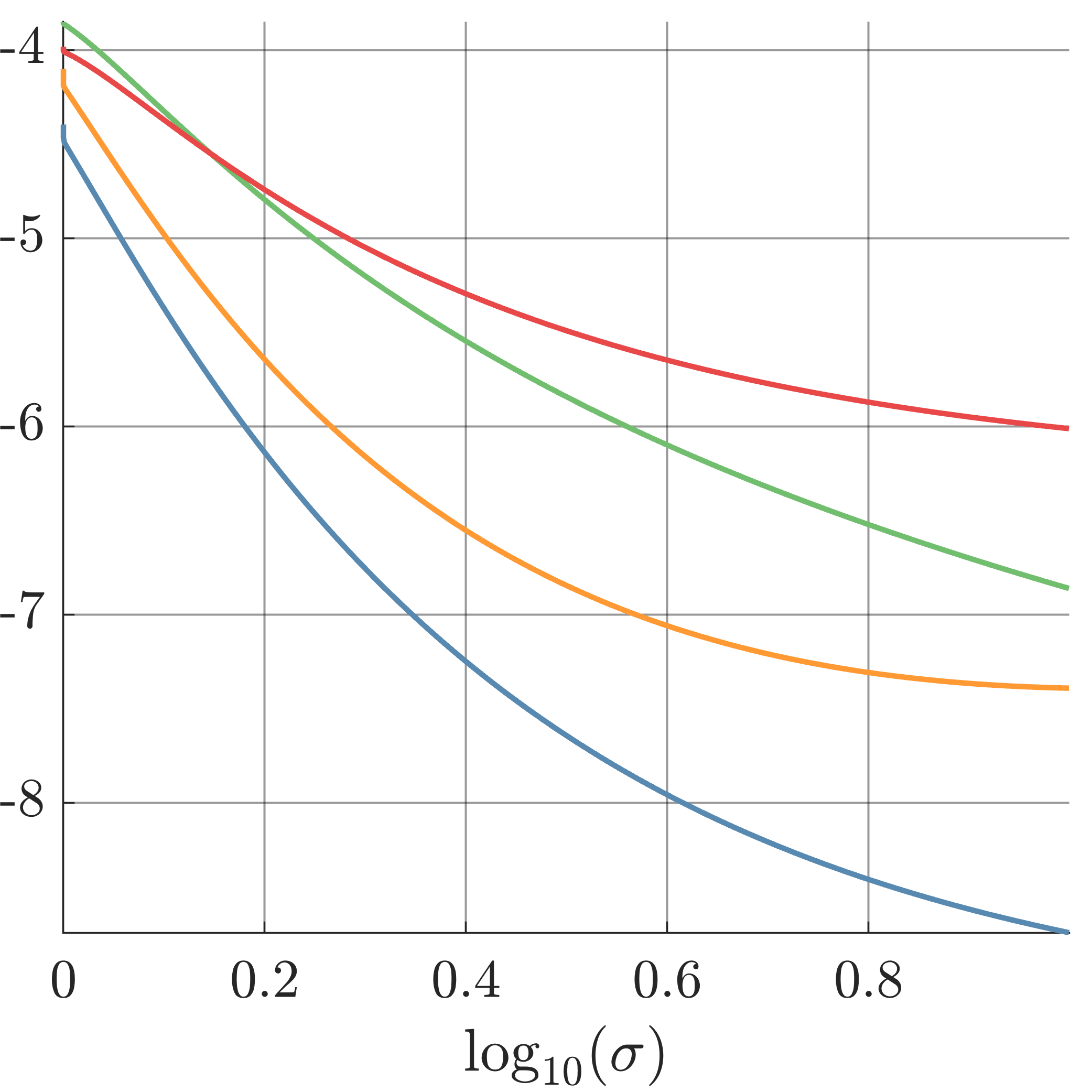}
                (a) $t = 10^{-2}$
            \end{minipage}

            \begin{minipage}{0.24\textwidth}
                \centering
                \includegraphics[width=\columnwidth]{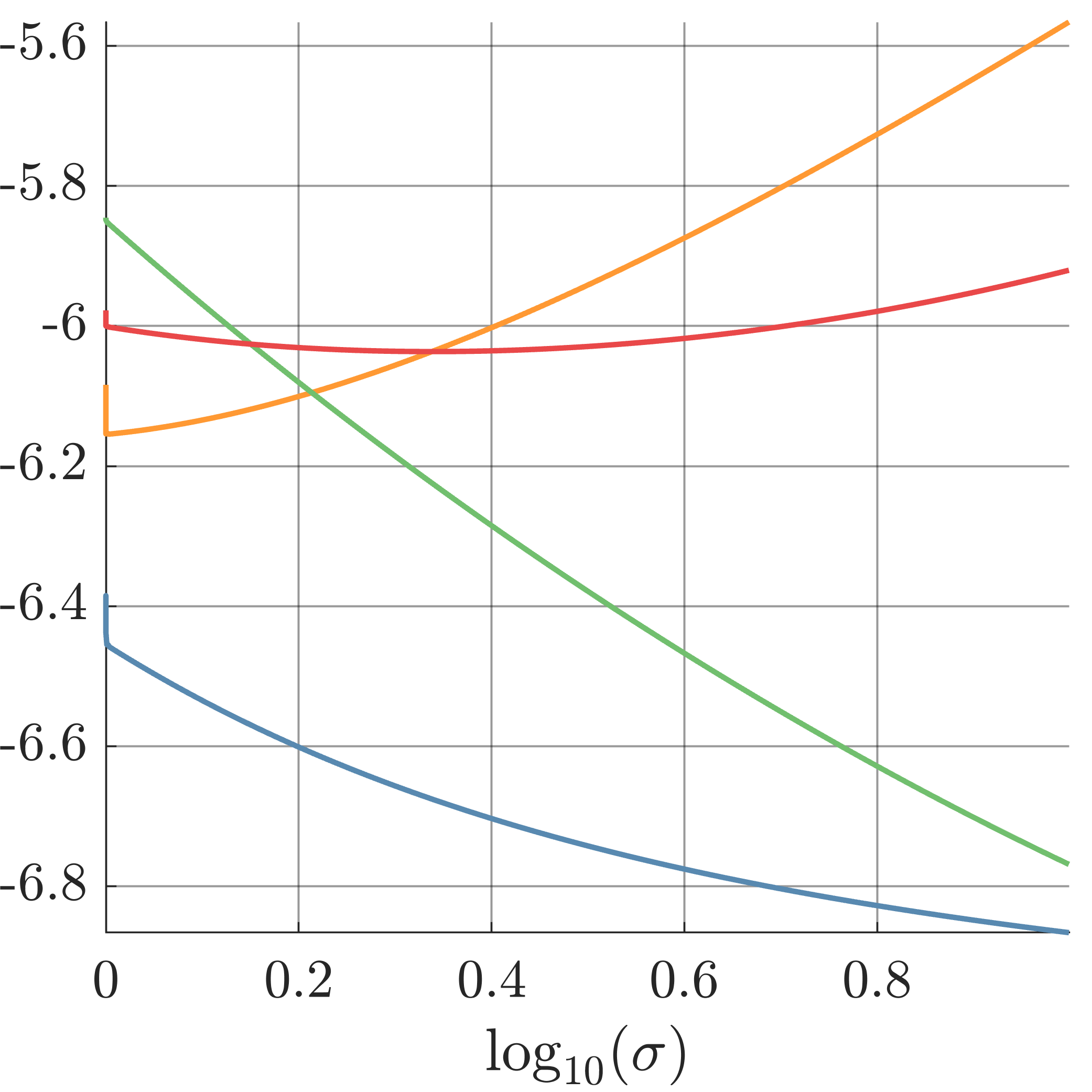}
                (b) $t = 10^{-3}$
            \end{minipage}

            \begin{minipage}{0.24\textwidth}
                \centering
                \includegraphics[width=\columnwidth]{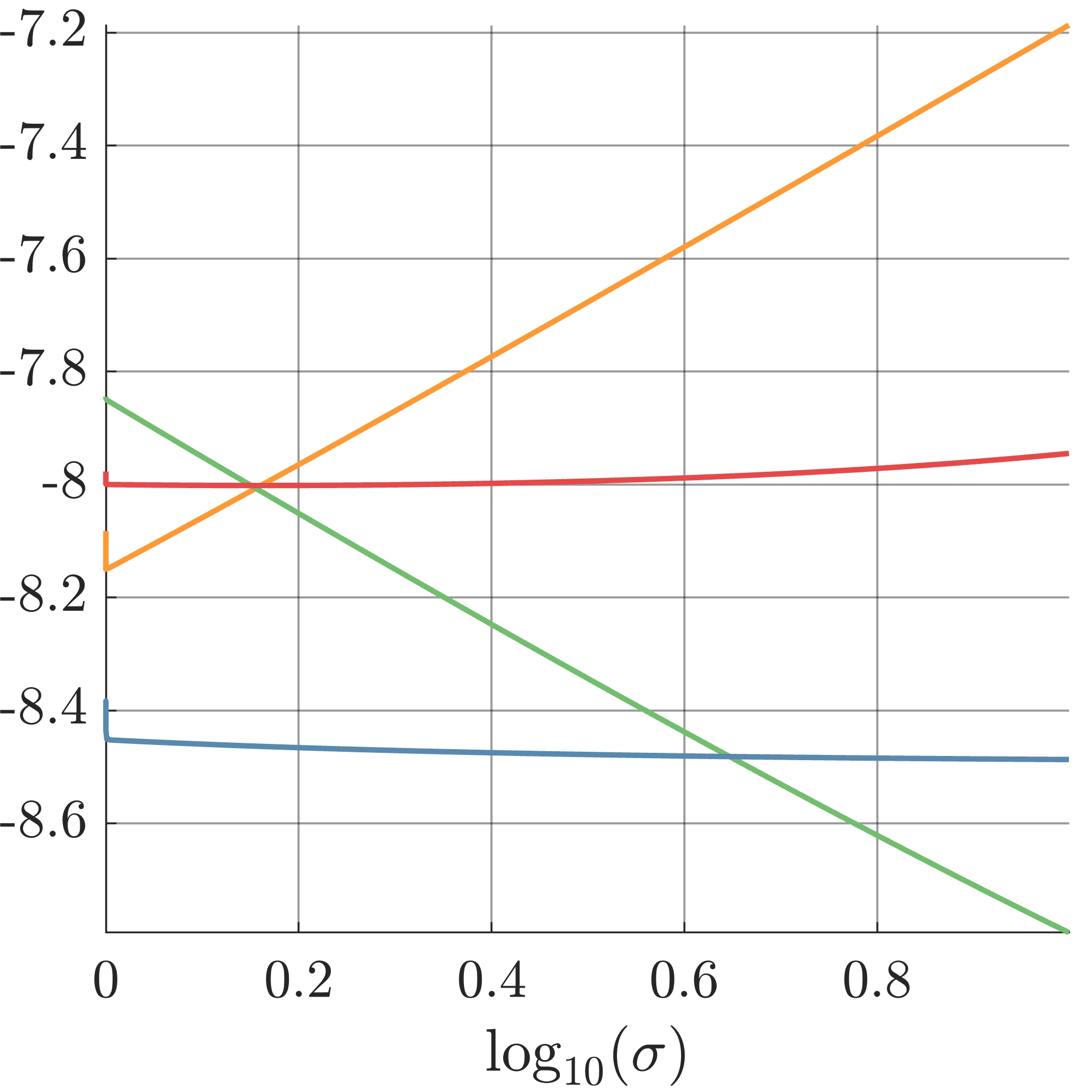}
                (c) $t = 10^{-4}$
            \end{minipage}

            \begin{minipage}{0.24\textwidth}
                \centering
                \includegraphics[width=\columnwidth]{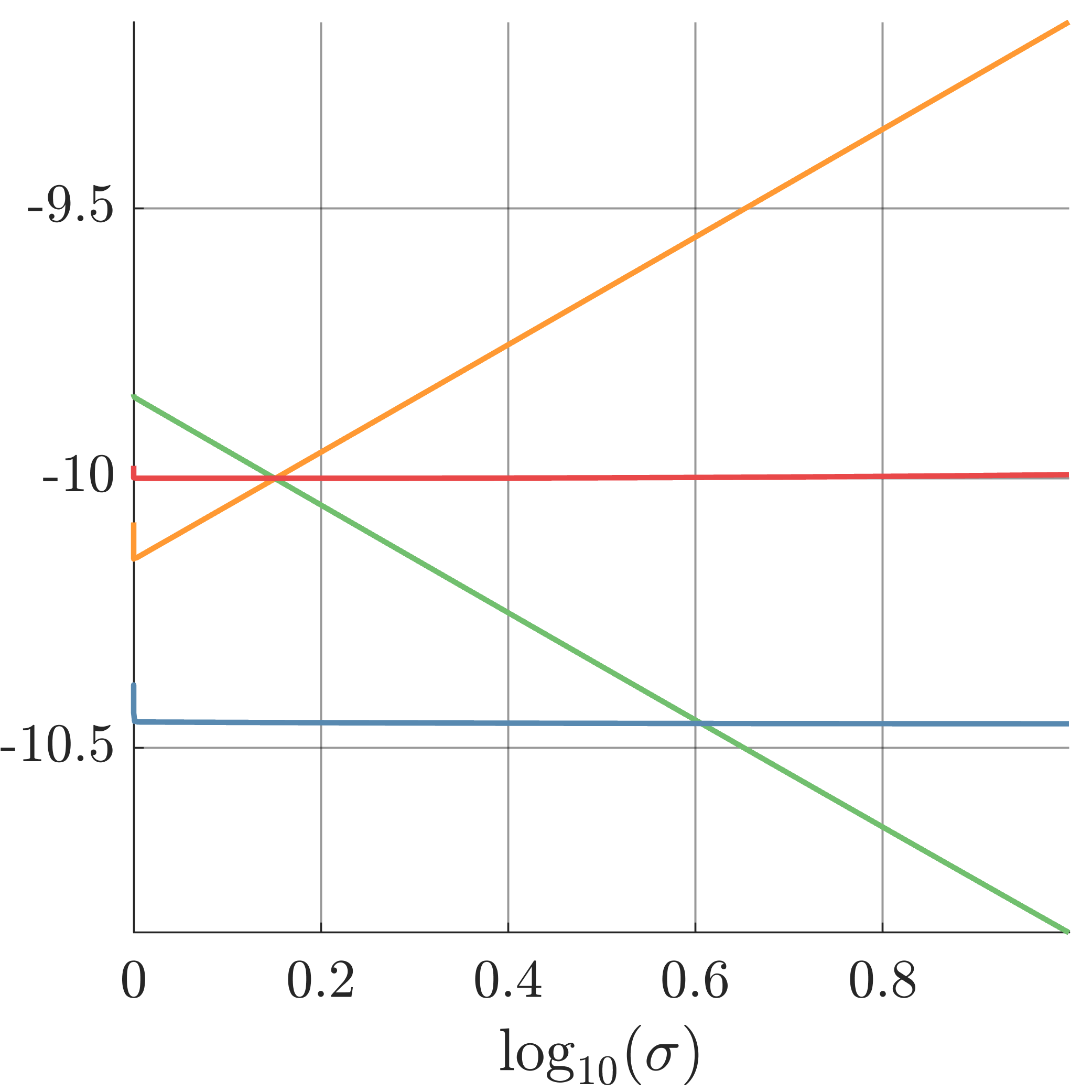}
                (d) $t = 10^{-5}$
            \end{minipage}
        \end{tabular}
    \end{minipage}
    \caption{\label{fig:mdvZ-sigma:toy2} Trajectories of $\normF{\Delta \Vark{X}}$, $\normF{\Delta \Vark{S}}$, $\Vark{r}_{p}$, and $\Vark{r}_{d}$ in~\eqref{eq:toy:example-2}. Fix $\Zbar \in \Zopt$ and $\Hbar \in \coneC$. Pick $t \in \{10^{-2}, 10^{-3}, 10^{-4}, 10^{-5}\}$ such that $Z^{(0)} = \Zbar + t\Hbar$. $\log_{10}(\sigma)$ is uniformly increased from $0$ to $1$ in $1000$ iterations.}
\end{figure}

%% file: figs/mdvZ-sigma/toy3.tex

\begin{figure}[h]
    \centering

    \begin{minipage}{\textwidth}
        \centering
        \includegraphics[width=0.5\columnwidth]{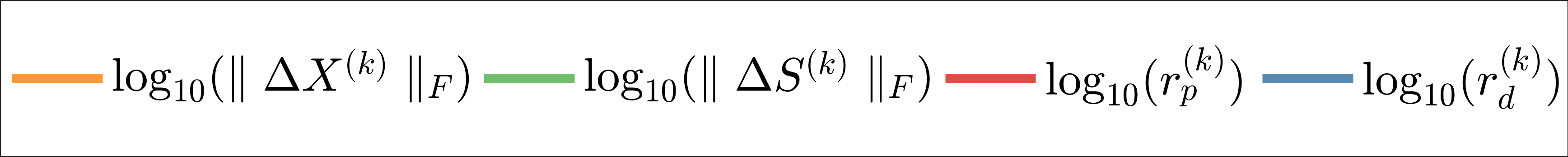}
    \end{minipage}
    \vspace{1mm}

    \begin{minipage}{\textwidth}
        \centering
        \begin{tabular}{cccc}
            \begin{minipage}{0.24\textwidth}
                \centering
                \includegraphics[width=\columnwidth]{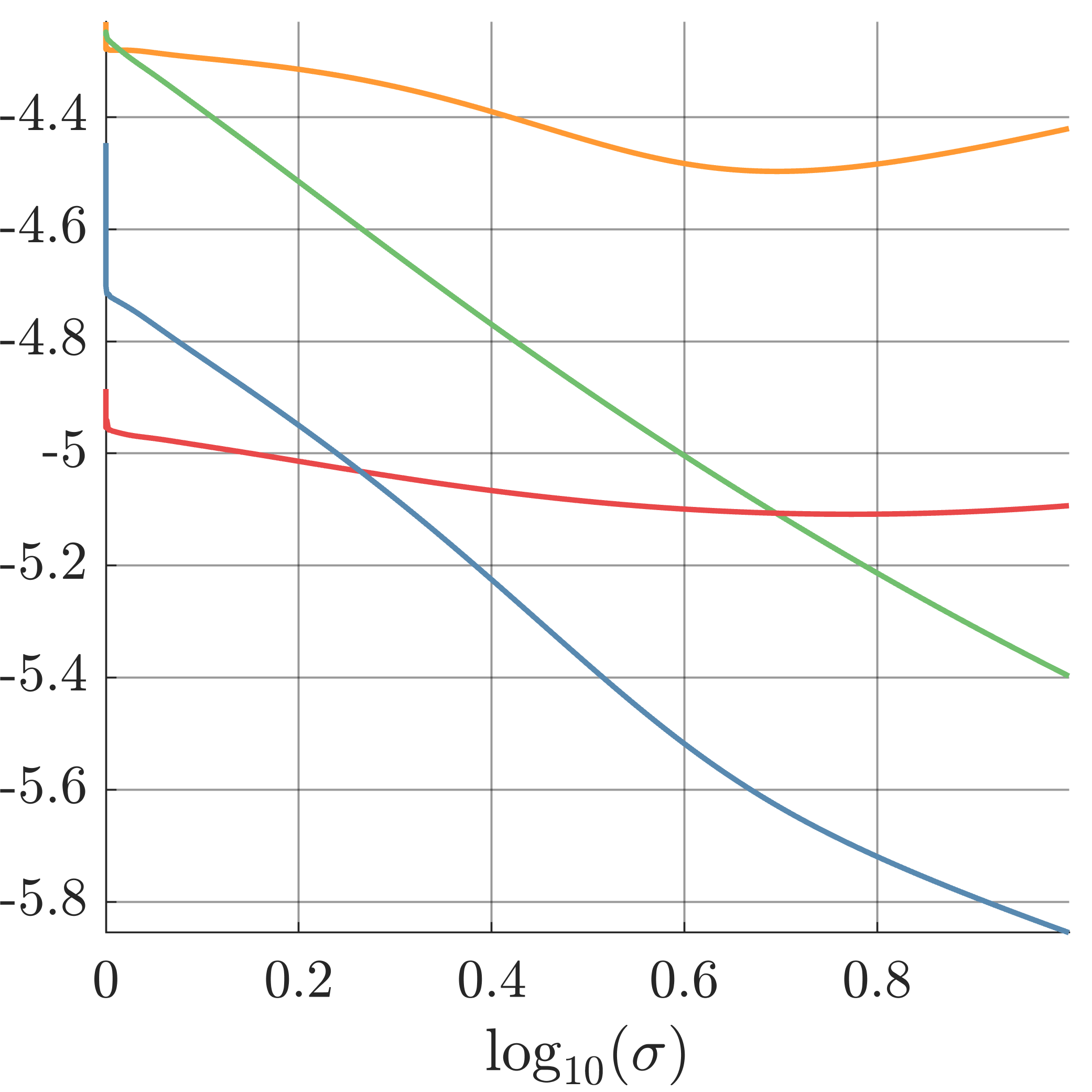}
                (a) $t = 10^{-2}$
            \end{minipage}

            \begin{minipage}{0.24\textwidth}
                \centering
                \includegraphics[width=\columnwidth]{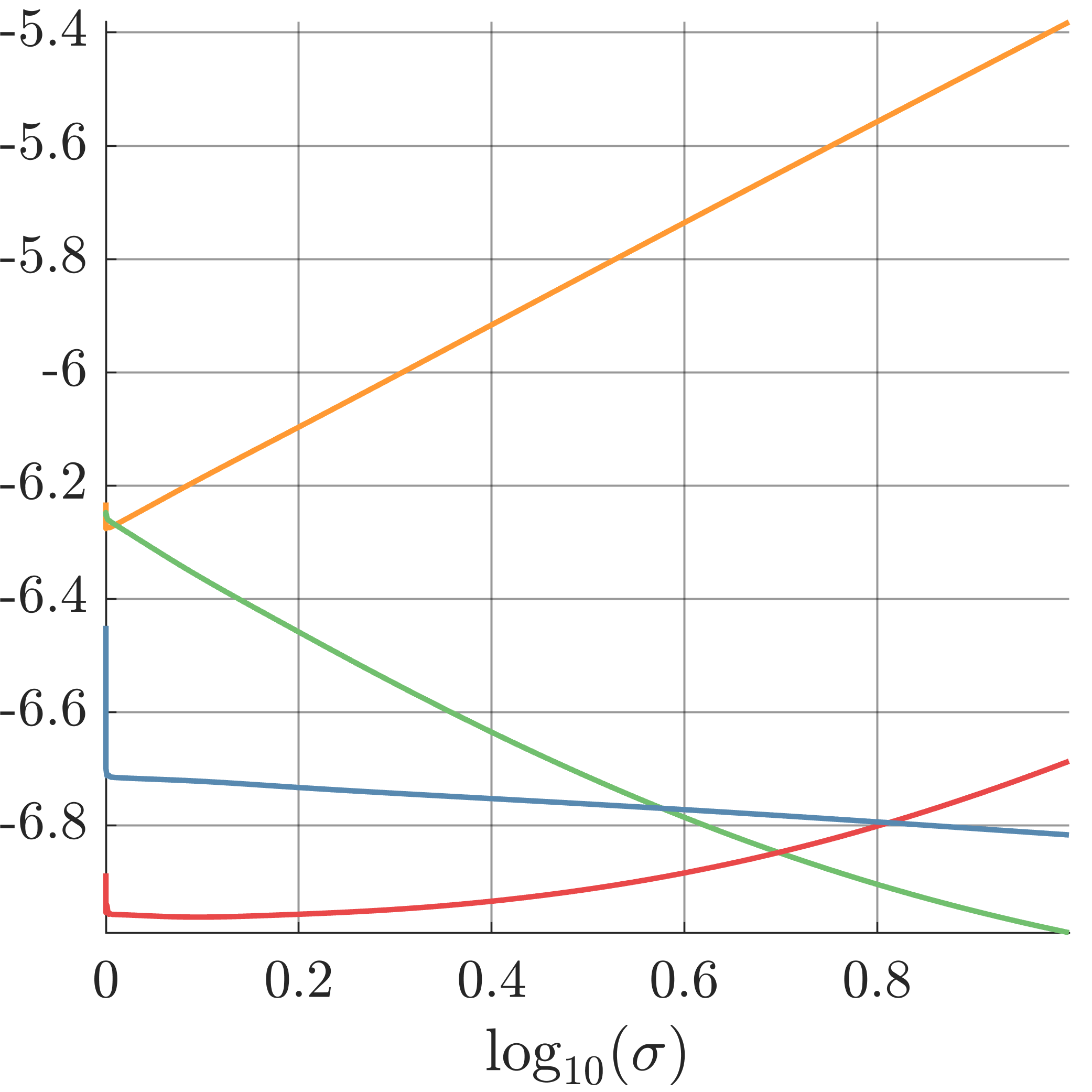}
                (b) $t = 10^{-3}$
            \end{minipage}

            \begin{minipage}{0.24\textwidth}
                \centering
                \includegraphics[width=\columnwidth]{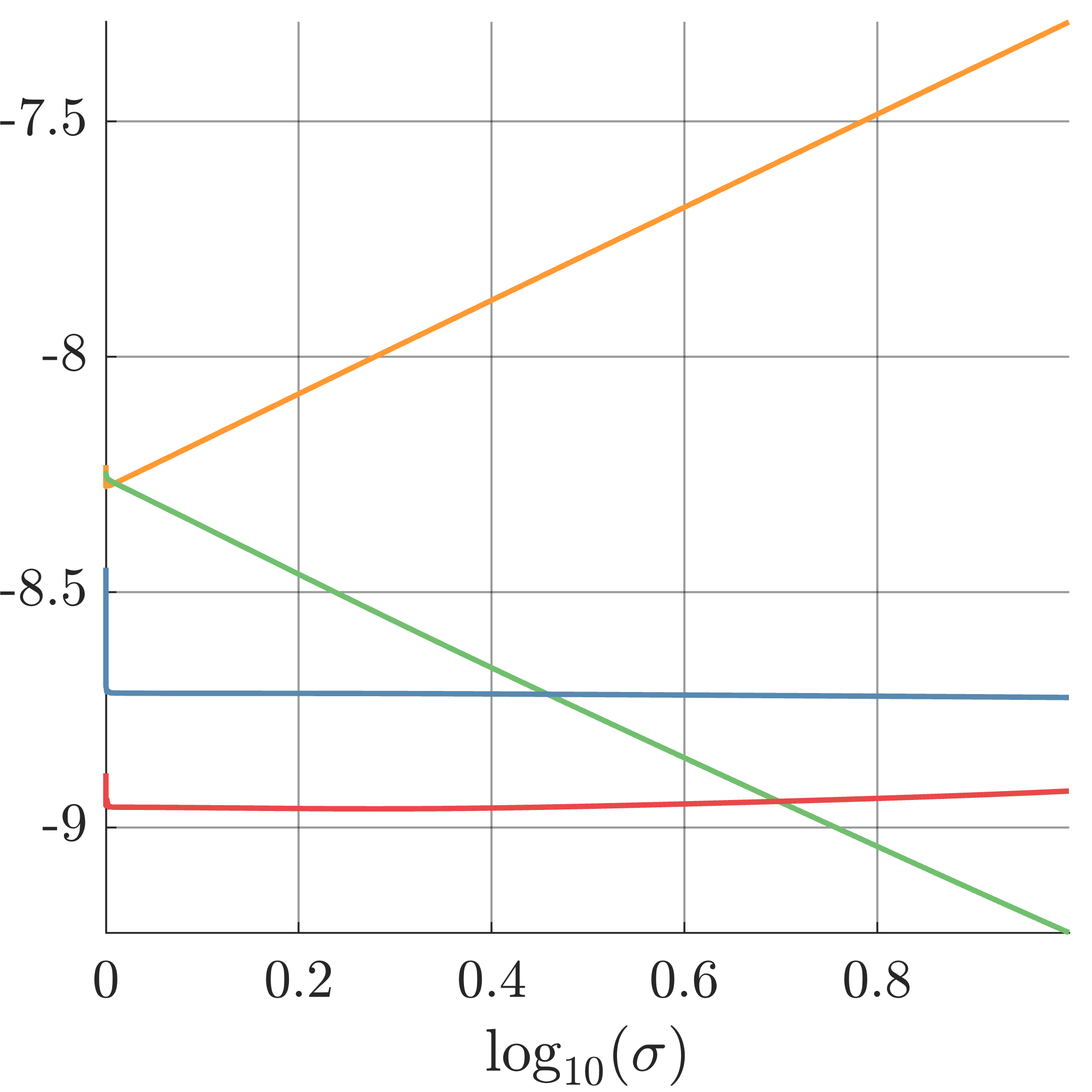}
                (c) $t = 10^{-4}$
            \end{minipage}

            \begin{minipage}{0.24\textwidth}
                \centering
                \includegraphics[width=\columnwidth]{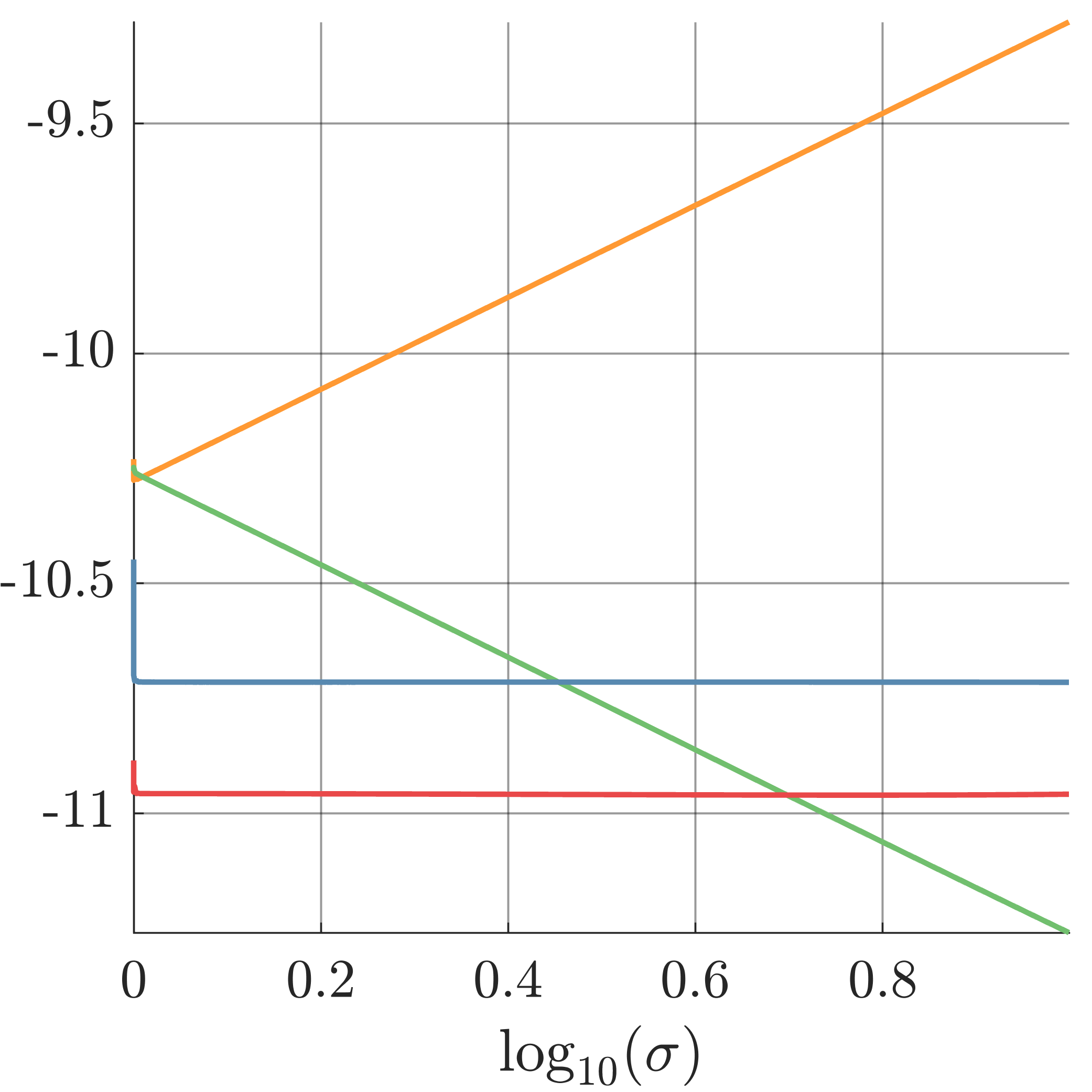}
                (d) $t = 10^{-5}$
            \end{minipage}
        \end{tabular}
    \end{minipage}
    \caption{\label{fig:mdvZ-sigma:toy3} Trajectories of $\normF{\Delta \Vark{X}}$, $\normF{\Delta \Vark{S}}$, $\Vark{r}_{p}$, and $\Vark{r}_{d}$ in~\eqref{eq:toy:example-3}. Fix $\Zbar \in \Zopt$ and $\Hbar \in \coneC$. Pick $t \in \{10^{-2}, 10^{-3}, 10^{-4}, 10^{-5}\}$ such that $Z^{(0)} = \Zbar + t\Hbar$. $\log_{10}(\sigma)$ is uniformly increased from $0$ to $1$ in $1000$ iterations.}
\end{figure}

%% file: figs/mdvZ-sigma/accelerate.tex

\begin{figure}[h]
    \centering

    \begin{minipage}{\textwidth}
        \centering
        \includegraphics[width=0.5\columnwidth]{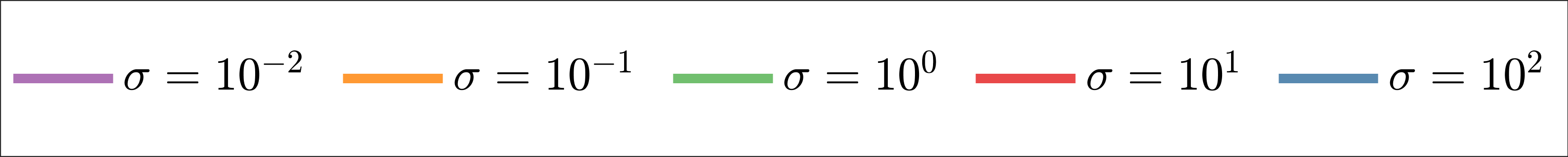}
    \end{minipage}
    \vspace{1mm}

    \begin{minipage}{\textwidth}
        \centering
        \begin{tabular}{ccc}
            \begin{minipage}{0.32\textwidth}
                \centering
                \includegraphics[width=\columnwidth]{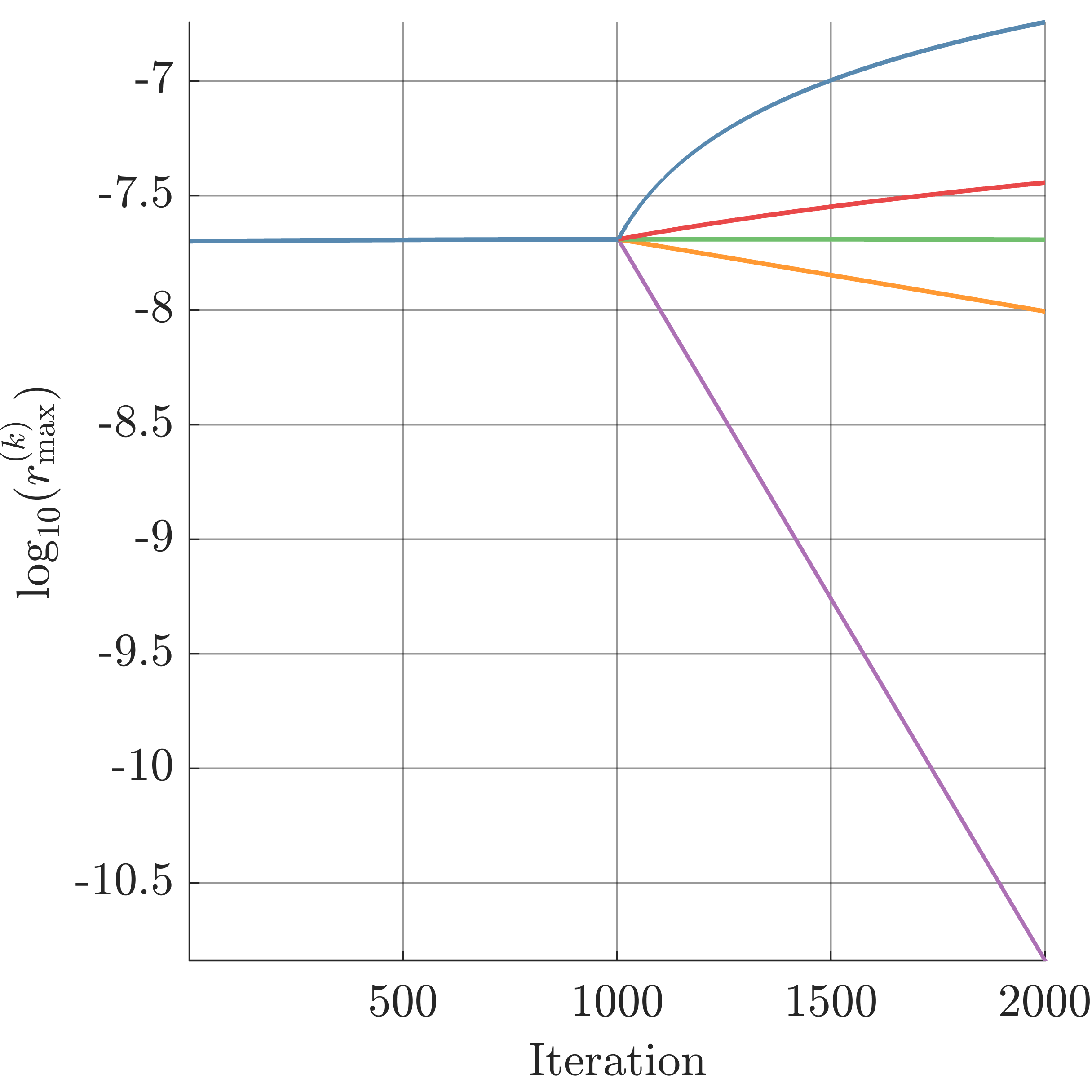}
                \eqref{eq:toy:example-1}
            \end{minipage}

            \begin{minipage}{0.32\textwidth}
                \centering
                \includegraphics[width=\columnwidth]{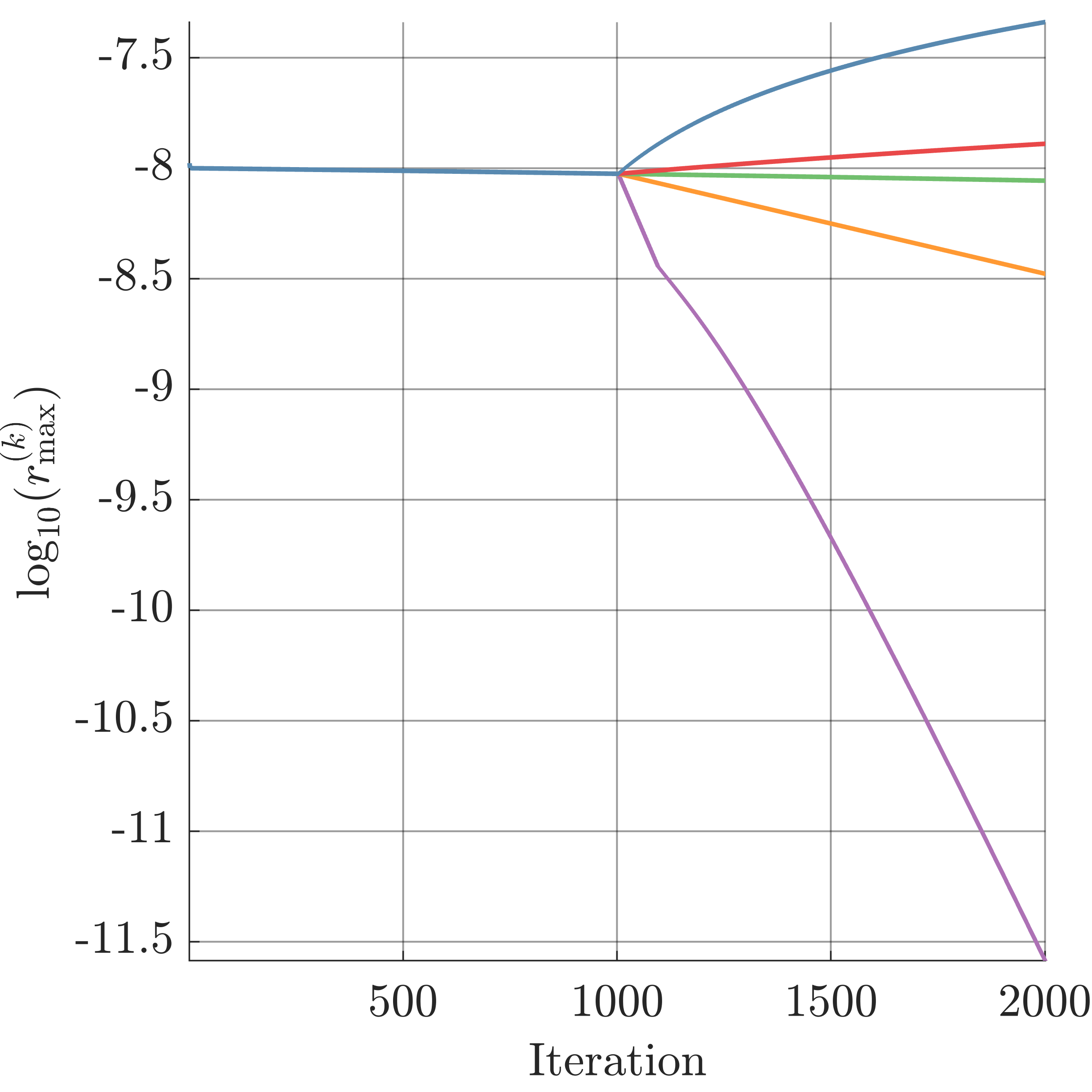}
                \eqref{eq:toy:example-2}
            \end{minipage}

            \begin{minipage}{0.32\textwidth}
                \centering
                \includegraphics[width=\columnwidth]{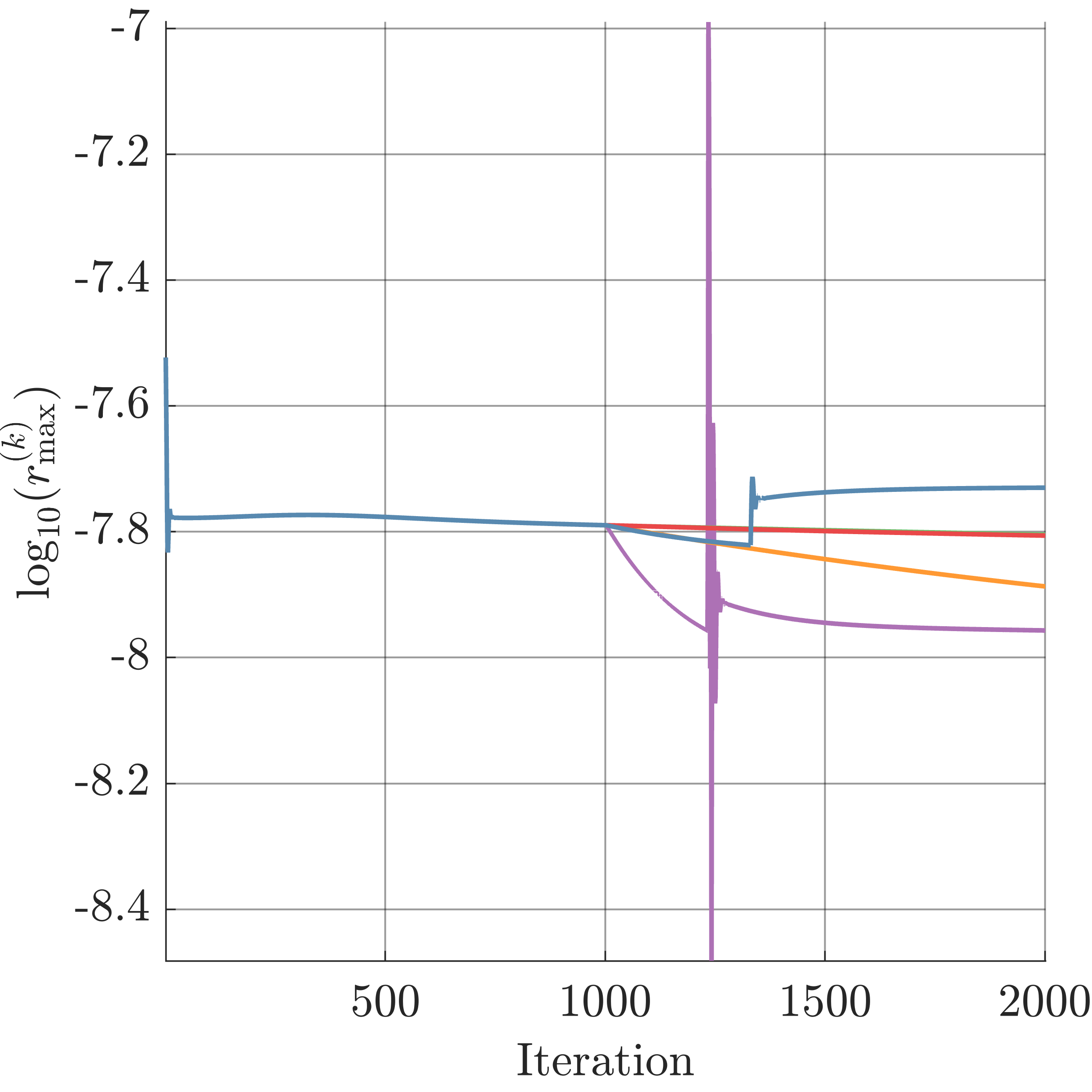}
                \eqref{eq:toy:example-3}
            \end{minipage}
        \end{tabular}
    \end{minipage}
        \caption{\label{fig:mdvZ-sigma:accelerate} Trajectories of $\log_{10}(\Vark{r}_{\max})$ in the three SDP examples. For each example, we fix $\Zbar \in \Zopt$, $\Hbar \in \coneC$, and $t = 10^{-4}$, and initialize $Z^{(0)} := \Zbar + t \Hbar$. The penalty parameter is initialized at $\sigma=1$. After $1000$ iterations, we update $\sigma$ to a value in $\{10^{-2}, 10^{-1}, 1, 10, 10^{2}\}$ and run an additional $1000$ ADMM iterations.
    }
\end{figure}

%% file: sections/toy_1.tex

\section{\titlemath{Examples}}
\label{sec:toy}

We present three SDP examples in this section. For each instance and the associated rank-deficient $\Zbar \in \Zopt$, we compute the relevant first-order objects (\eg $\coneC$, $\coneT$) and second-order objects (\eg $\coneK$, $\coneKP$, $\opeP$, $\mdvZ$). These calculations serve three purposes:
\begin{enumerate}
    \item These examples provide a sanity check for the validity of our second-order analysis.

    \item More importantly, the examples serve as constructive demonstrations of the key properties of $\mdvZmap$. For instance,~\eqref{eq:toy:example-3} simultaneously establishes Proposition~\ref{prop:mdvZ-range:general-case} and Proposition~\ref{prop:mdvZ-continuity:discontinuity}.

    \item When discussing the connection between properties of $\mdvZmap$ and ADMM's empirical behavior, we use numerical results on these examples for illustration.
\end{enumerate}

\subsection{\titlemath{Example \RomanNum{1}}}
\label{sec:toy1}

\paragraph{SDP data.} We consider a $2 \times 2$ SDP from~\cite[Example 1]{cui16arxiv-superlinear-alm-sdp}:
\begin{align}
    \tag{SDP-\RomanNum{1}}
    \label{eq:toy:example-1}
    C = \mymat{0 & 0 \\ \sim & 1}, \quad A_1 = \mymat{0 & 1 \\ \sim & -1}, \quad b = 0.
\end{align}

Its optimal sets are:
\begin{align*}
    \Xopt = \left\{ \mymat{X_{11} & 0 \\ \sim & 0} \mymid X_{11} \ge 0 \right\}, \ \Sopt = \left\{ \mymat{0 & 0 \\ \sim & 1} \right\}, \ \Zopt = \left\{ \mymat{Z_{11} & 0 \\ \sim & -\sigma} \mymid Z_{11} \ge 0 \right\}.
\end{align*}
Clearly, the primal optimal set is unbounded. Moreover, except $Z_{11} = 0$, all other optimal solutions satisfies strict complementarity. We are typically interested in the rank deficient optimal solution ($Z_{11} = 0$).

\paragraph{First-order information.} Denote $(x)_+$ (resp. $(x)_-$) as an abbreviation for $\psdproj{1}(x)$ (resp. $\nsdproj{1}(x)$).
Since 
\begin{align*}
    \PA X = \frac{1}{3} (2 X_{12} - X_{22}) \mymat{0 & 1 \\ \sim & -1}, \quad
    \PAp X = \mymat{X_{11} & \frac{1}{3}(X_{12} + X_{22}) \\ \sim & \frac{2}{3}(X_{12} + X_{22})},
\end{align*}
and 
\begin{align*}
    \ppsim'(\Zbar; H) = \mymat{(H_{11})_+ & 0 \\ \sim & 0}, \quad \npsim'(\Zbar; H) = \mymat{(H_{11})_- & H_{12} \\ \sim & H_{22}}.
\end{align*}
we have 
\begin{align*}
    \PA \ppsim'(\Zbar; H) = 0, \quad \PAp \npsim'(\Zbar; H) = \mymat{(H_{11})_- & \frac{1}{3}(H_{12} + H_{22}) \\ \sim & \frac{2}{3}(H_{12} + H_{22})}.
\end{align*}
Thus,
\begin{align}
    \label{eq:toy1:Hbar}
    \coneC = \left\{ \mymat{a & b \\ \sim & -b} \mymid a \ge 0 \right\}, \quad 
    \coneT = \coneC \cap \left\{ H \in \Sym{2} \mymid H_{12} = 0 \right\}
    = \left\{ \mymat{a & 0 \\ \sim & 0} \mymid a \ge 0 \right\}.
\end{align}
we choose an arbitrary $\Hbar \in \coneC \backslash \coneT$ with $b \ne 0$.

\paragraph{Second-order information.} For this simple $2 \times 2$ SDP example, we use $\coneK$'s formula in~\eqref{eq:soa:mdvZ-def} to calculate $\mdvZ$.
Through careful calculation:
\begin{align*}
    \opeE = \mymat{0 & 2 \frac{ab}{\sigma} \\ \sim & 0}, \quad 
    \opeEP = \mymat{-2 \frac{b^2}{\sigma} & -2 \frac{ab}{\sigma} \\ \sim & 0}, \quad 
    \opeP = \mymat{2 \frac{b^2}{\sigma} & -\frac{2}{3} \frac{ab}{\sigma} \\ \sim & \frac{8}{3} \frac{ab}{\sigma}}.
\end{align*}
For $\coneK$, there are two cases:
\begin{itemize}
    \item Case I: $a > 0$. In this case,
    \begin{align*}
        \opeT{W} = \mymat{W_{11} & 0 \\ \sim & 0}, \quad 
        \opeTP{W} = \mymat{0 & W_{12} \\ \sim & W_{22}}, \quad 
        \coneK = \left\{ \mymat{0 & \frac{1}{3}(W_{12} + W_{22}) \\ \sim & \frac{2}{3}(W_{12} + W_{22})} \right\}.
    \end{align*}
    Thus,
    \begin{align}
        \label{eq:toy1:mdvZ}
        \Pi_{\coneK}(\opeP) = \mymat{0 & \frac{2}{3} \frac{ab}{\sigma} \\ \sim & \frac{4}{3} \frac{ab}{\sigma}}, \quad \mdvZ = \mymat{2 \frac{b^2}{\sigma} & -\frac{4}{3} \frac{ab}{\sigma} \\ \sim & \frac{4}{3} \frac{ab}{\sigma}}.
    \end{align}

    \item Case II: $a = 0$. In this case,
     \begin{align*}
        \simpleadjustbox{
        \opeT{W} = \mymat{(W_{11})_+ & 0 \\ \sim & 0}, \
        \opeTP{W} = \mymat{(W_{11})_- & W_{12} \\ \sim & W_{22}}, \
        \coneK = \left\{ \mymat{(W_{11})_- & \frac{1}{3}(W_{12} + W_{22}) \\ \sim & \frac{2}{3}(W_{12} + W_{22})} \right\}.
        }
    \end{align*}
    Thus,
    \begin{align*}
        \Pi_{\coneK}(\opeP) = \mymat{0 & 0 \\ \sim & 0}, \quad \mdvZ = \mymat{2 \frac{b^2}{\sigma} & 0 \\ \sim & 0}.
    \end{align*}
\end{itemize}
Clearly, the two cases can be combined. 

\paragraph{$\sigma$ updating.} Consider updating $\sigma$ to $\sigma'$. From Lemma~\ref{lem:mdvZ-sigma:H-new-partition} and Lemma~\ref{lem:mdvZ-sigma:H-still-stall}: 
\begin{align*}
    \Hbar = \mymat{a & b \\ \sim & -b} \Longrightarrow 
    \Hbar' = \mymat{a' & b' \\ \sim & -b'} = \mymat{a &  \frac{\sigma'}{\sigma} b \\ \sim & -\frac{\sigma'}{\sigma} b}.
\end{align*}
Since 
\begin{align*}
    \mdvX = \mymat{2 \frac{b^2}{\sigma} & 0 \\ \sim & 0}, \quad 
    \mdvS = \mymat{0 & \frac{4}{3} \frac{ab}{\sigma^2} \\ \sim & -\frac{4}{3} \frac{ab}{\sigma^2}},
\end{align*}
we have 
\begin{align*}
    \mdvXnew = \frac{\sigma'}{\sigma} \mymat{2 \frac{b^2}{\sigma} & 0 \\ \sim & 0}
    = \frac{\sigma'}{\sigma} \mdvX, \quad 
    \mdvSnew = \frac{\sigma}{\sigma'} \mymat{0 & \frac{4}{3} \frac{ab}{\sigma^2} \\ \sim & -\frac{4}{3} \frac{ab}{\sigma^2}} = \frac{\sigma}{\sigma'} \mdvS.
\end{align*}

%% file: sections/toy_2.tex

\subsection{\titlemath{Example \RomanNum{2}}}
\label{sec:toy2}

\paragraph{SDP data.} Consider the following SDP instance:
\begin{align}
    \tag{SDP-\RomanNum{2}}
    \label{eq:toy:example-2}
    C = \begin{bmatrix}
        0 & 0 & 0 \\ 
        \sim & 0 & 0 \\
        \sim & \sim & 1
    \end{bmatrix}, \ 
    A_1 = \begin{bmatrix}
        1 & 0 & 0 \\
        \sim & 1 & 0 \\
        \sim & \sim & 1
    \end{bmatrix}, \  
    A_2 = \begin{bmatrix}
        0 & 0 & 0 \\
        \sim & 0 & 1 \\
        \sim & \sim & 0 
    \end{bmatrix}, \ 
    b = \begin{bmatrix}
        1 \\ 0 
    \end{bmatrix}.
\end{align}

The primal-dual optimal set:
\begin{align*}
    \Xopt = \left\{ \begin{bmatrix}
        a & u & 0 \\
        \sim & 1-a & 0 \\
        0 & 0 & 0
    \end{bmatrix} \mymid 0 \le a \le 1, \ u^2 \le a(1-a) \right\}, 
    \
    \Sopt = \left\{ \begin{bmatrix}
        0 & 0 & 0 \\ 
        \sim & 0 & 0 \\
        \sim & \sim & 1
    \end{bmatrix} \right\}, \ 
    \Zopt = \Xopt - \sigma \Sopt.
\end{align*}
For~\eqref{eq:toy:example-2}, we can easily check that it satisfies two-sided Slater conditions. Also, there exists a strictly complementary solution pair. We are typically interested in the rank-deficient solutions, \ie $\rank{\Varbar{X}} = 1$ of the form:
\begin{align*}
    \Varbar{X} = \Varbar{X}(a) = \mymat{
        a & \pm \sqrt{a} \sqrt{1-a} & 0 \\ 
        \sim & 1-a & 0 \\
        \sim & \sim & 0
    }, \quad a \in [0, 1].
\end{align*}
Our framework requires $\Varbar{X}$ to be diagonal. Therefore, some changes of basis are needed. Define the orthonormal matrix $Q$ as:
\begin{align*}
    Q = \mymat{
        \sqrt{a} & \mp \sqrt{1-a} & 0 \\
        \pm \sqrt{1-a} & \sqrt{a} & 0 \\
        0 & 0 & 1
    }.
\end{align*}
Under this basis:
\begin{eqnarray*}
    & C \leftarrow Q\tran C Q = \mymat{0 & 0 & 0 \\ \sim & 0 & 0 \\ \sim & \sim & 1}, \
    A_1 \leftarrow Q\tran A_1 Q = \mymat{1 & 0 & 0 \\ \sim & 1 & 0 \\ \sim & \sim & 1}, \ 
    A_2 \leftarrow Q\tran A_2 Q = \mymat{
        0 & 0 & \pm \sqrt{1-a} \\
        \sim & 0 & \sqrt{a} \\
        \sim & \sim & 0 
    }, \\
    & \Varbar{X} \leftarrow Q\tran \Varbar{X} Q = \mymat{1 & 0 & 0 \\ \sim & 0 & 0 \\ \sim & \sim & 0}, \ 
    \Varbar{S} \leftarrow Q\tran \Varbar{S} Q = \mymat{0 & 0 & 0 \\ \sim & 0 & 0 \\ \sim & \sim & 1}.
\end{eqnarray*}

\paragraph{First-order information.} For a fixed $a$, 
\begin{eqnarray*}
    & \PA X = \frac{1}{3} (X_{11} + X_{22} + X_{33}) \mymat{1 & 0 & 0 \\ \sim & 1 & 0 \\ \sim & \sim & 1} 
    + (\pm \sqrt{1-a} X_{13} + \sqrt{a} X_{23}) \mymat{0 & 0 & \pm \sqrt{1-a} \\ \sim & 0 & \sqrt{a} \\ \sim & \sim & 0}, \\
    & \ppsim'(\Zbar; H) = \mymat{
        H_{11} & H_{12} & \frac{1}{1 + \sigma} H_{13} \\
        \sim & (H_{22})_+ & 0 \\
        \sim & \sim & 0
    }, \quad \npsim'(\Zbar; H) = \mymat{
        0 & 0 & \frac{\sigma}{1 + \sigma} H_{13} \\
        \sim & (H_{22})_- & H_{23} \\
        \sim & \sim & H_{33} 
    }.
\end{eqnarray*}
Via careful calculation:
\begin{align}
    \label{eq:toy2:Hbar}
    \simpleadjustbox{
    \coneC = \begin{cases}
        \left\{ 
            \mymat{-H_{22} & H_{12} & 0 \\ \sim & H_{22} & 0 \\ \sim & \sim & 0} \mymid H_{22} \ge 0
         \right\}, & a \in [0, 1) \\
        \left\{ 
            \mymat{-H_{22} & H_{12} & 0 \\ \sim & H_{22} & H_{23} \\ \sim & \sim & 0} \mymid H_{22} \ge 0
         \right\}, & a = 1
    \end{cases}, \
    \coneT = \left\{ 
            \mymat{-H_{22} & H_{12} & 0 \\ \sim & H_{22} & 0 \\ \sim & \sim & 0} \mymid H_{22} \ge 0
         \right\}, \ \forall a \in [0, 1].
    }
\end{align}
Therefore, to pick up $\Hbar \in \coneC \backslash \coneT$, the only nontrivial case is $a = 1$ and $\Hbar_{23} \ne 0$.  

\paragraph{Second-order information.} 
We adopt the polar description in Proposition~\ref{prop:soa:polarK-XS} to calculate $\mdvZ$. Via careful calculation: 
\begin{align*}
    \opeE = \mymat{
        0 & 0 & \frac{2}{1+\sigma} \Hbar_{12} \Hbar_{23} \\
        \sim & 2 \Hbar_{12}^2 & \frac{2}{\sigma} \Hbar_{22} \Hbar_{23} \\
        \sim & \sim & 0 
    }, \quad \opeEP = \mymat{
        0 & 0 & -\frac{2}{1+\sigma} \Hbar_{12} \Hbar_{23} \\
        \sim & -\frac{2}{\sigma} \Hbar_{23}^2 & -\frac{2}{\sigma} \Hbar_{22} \Hbar_{23} \\
        \sim & \sim & 0 
    },
\end{align*}
and 
\begin{align*}
    \simpleadjustbox{
    \coneKPX = \begin{cases}
        \left\{ 
            W = \mymat{
                W_{11} & W_{12} & 0 \\
                \sim & W_{22} & 0 \\
                \sim & \sim & 0 
            } \mymid W_{11} + W_{22} = 0
         \right\}, & \ \Hbar_{22} > 0 \\
        \left\{ 
            W = \mymat{
                W_{11} & W_{12} & 0 \\
                \sim & W_{22} & 0 \\
                \sim & \sim & 0 
            } \mymid \mymat{
                W_{11} + W_{22} = 0, \\
                W_{22} \ge 0
            }
         \right\}, & \ \Hbar_{22} = 0
    \end{cases}, \quad 
    \coneKPS = \left\{ 
        W = \mymat{
            0 & 0 & 0 \\
            \sim & 0 & W_{23} \\
            \sim & \sim & 0 
        } 
    \right\}.}
\end{align*}

(\romannumeral1) For the primal part, we need to consider two cases:

(a) $\Hbar_{22} > 0$. In this case, from Theorem~\ref{thm:soa:primal-dual-decouple},
\begin{align*}
    \simpleadjustbox{
    \mdvX = \argmin\limits_{W \in \coneKPX} \normF{W + \opeEP}^2 
    = \argmin\limits_{W_{11} + W_{22} = 0} (W_{22} - \frac{2}{\sigma} \Hbar_{23}^2)^2 + 2 W_{12}^2 + W_{11}^2 
    = \mymat{
        -\frac{1}{\sigma} \Hbar_{23}^2 & 0 & 0 \\
        \sim & \frac{1}{\sigma} \Hbar_{23}^2 & 0 \\
        \sim & \sim & 0 
    }.}
\end{align*}

(b) $\Hbar_{22} = 0$. Similar to case (a),
\begin{align*}
    \simpleadjustbox{
    \mdvX = \argmin\limits_{W \in \coneKPX} \normF{W + \opeEP}^2 
    = \argmin\limits_{\substack{
        W_{11} + W_{22} = 0, \\
        W_{22} \ge 0 
    }} (W_{22} - \frac{2}{\sigma} \Hbar_{23}^2)^2 + 2 W_{12}^2 + W_{11}^2 
    = \mymat{
        -\frac{1}{\sigma} \Hbar_{23}^2 & 0 & 0 \\
        \sim & \frac{1}{\sigma} \Hbar_{23}^2 & 0 \\
        \sim & \sim & 0 
    }.}
\end{align*}
Clearly, case (a) and (b) can be combined. 

(\romannumeral2) For the dual part, from Theorem~\ref{thm:soa:primal-dual-decouple}:
\begin{align*}
    -\sigma \mdvS = \argmin\limits_{W \in \coneKPS} \normF{W + \opeE}^2 
    = \argmin\limits_{W_{23} \in \mathbb{R}} (W_{23} + \frac{2}{\sigma} \Hbar_{22} \Hbar_{23})^2
    = \mymat{
        0 & 0 & 0 \\
        \sim & 0 & -\frac{2}{\sigma} \Hbar_{22} \Hbar_{23} \\
        \sim & \sim & 0 
    }.
\end{align*}

(\romannumeral3) Combining the primal and dual part:
\begin{align}
    \label{eq:toy2:mdvZ}
    \mdvZ = \mdvX - \sigma \mdvS = \mymat{
        -\frac{1}{\sigma} \Hbar_{23}^2 & 0 & 0 \\
        \sim & \frac{1}{\sigma} \Hbar_{23}^2 & -\frac{2}{\sigma} \Hbar_{22} \Hbar_{23} \\
        \sim & \sim & 0 
    }.
\end{align}

\paragraph{$\sigma$ updating.} When $\sigma$ is updated to $\sigma'$, $\Hbar$ is updated to
\begin{align*}
    \Hbar' = \mymat{
        -\Hbar_{22} & \Hbar_{12} & 0 \\
        \sim & \Hbar_{22} & \frac{\sigma'}{\sigma} \Hbar_{23} \\
        \sim & \sim & 0 
    }
\end{align*}
from~\eqref{eq:mdvZ-sigma:Hbar-prime}. Thus,
\begin{align*}
    & \mdvXnew = \mymat{
        -\frac{1}{\sigma'} (\frac{\sigma'}{\sigma}\Hbar_{23})^2 & 0 & 0 \\
        \sim & \frac{1}{\sigma'} (\frac{\sigma'}{\sigma}\Hbar_{23})^2 & 0 \\
        \sim & \sim & 0 
    } = \frac{\sigma'}{\sigma} \mdvX, \\
    & \mdvSnew = -\frac{1}{\sigma'} \mymat{
        0 & 0 & 0 \\
        \sim & 0 & -\frac{2}{\sigma'} \Hbar_{22} (\frac{\sigma'}{\sigma}\Hbar_{23}) \\
        \sim & \sim & 0 
    } = \frac{\sigma}{\sigma'} \mdvS.
\end{align*}

%% file: sections/toy_3.tex

\subsection{\titlemath{Example \RomanNum{3}}}
\label{sec:toy3}

\paragraph{SDP data.} Consider a $6$ by $6$ SDP example. For ease of notation, define $E_{ij} \in \Sym{6} \ (1 \le i, j \le 6)$ as:
\begin{align*}
    E_{ij}(m, n) := \begin{cases}
        1, & \quad m = i, n = j \\
        0, & \quad \text{otherwise}
    \end{cases}
\end{align*}
Moreover, $0_{m \times n}$ is an abbreviation of all-zero matrix of size $m \times n$ and $I_{m}$ is an identity matrix of size $m \times m$.
Define an orthonormal matrix $Q$ as 
\begin{align*}
    Q := \mymat{q_1 & q_2 & q_3} = \mymat{
        \frac{1}{\sqrt{3}} & \frac{1}{\sqrt{2}} & \frac{1}{\sqrt{6}} \\
        \frac{1}{\sqrt{3}} & -\frac{1}{\sqrt{2}} & \frac{1}{\sqrt{6}} \\
        \frac{1}{\sqrt{3}} & 0 & -\frac{2}{\sqrt{6}}
    }.
\end{align*}
The SDP data is 
\begin{align}
    \label{eq:toy:example-3}
    & \simpleadjustbox{
    b = \mymat{6 \\ 0 \\ \vdots \\ 0} \in \Real{15}, \ C = \mymat{0_{3 \times 3} & 0_{3 \times 3} \\ \sim & I_3}, \ A_1 = I_6, \ A_2 = \mymat{Q\tran \mymat{1 & 0 & 0 \\ \sim & -1 & 0 \\ \sim & \sim & 0} Q & 0_{3 \times 3} \\ \sim & 0_{3 \times 3}}, \ A_3 = \mymat{Q\tran \mymat{1 & 0 & 0 \\ \sim & 0 & 0 \\ \sim & \sim & -1} Q & 0_{3 \times 3} \\ \sim & 0_{3 \times 3}}
    }, \nonumber \\
    & \simpleadjustbox{A_4 = E_{44} - E_{55}, \ A_5 = E_{55} - E_{66}, \ A_6 = E_{24} + E_{42}, \ A_7 = E_{25} + E_{52}, \ A_8 = E_{26} + E_{62}, \ A_9 = E_{34} + E_{43},} \nonumber \\
    & \simpleadjustbox{A_{10} = E_{35} + E_{53}, \ A_{11} = E_{36} + E_{63}, \ A_{12} = E_{45} + E_{54}, \ A_{13} = E_{46} + E_{64}, \ A_{14} = E_{56} + E_{65}, \ A_{15} = E_{16} + E_{61}.} \tag{SDP-\RomanNum{3}}
\end{align}
One can verify that for~\eqref{eq:toy:example-3}, there exist strictly complementary and rank-deficient solution pairs:
\begin{align*}
    (\Xsc, \Ssc) = \left( 
        \mymat{
            2 I_3 & 0_{3 \times 3} \\ 
            \sim & 0_{3 \times 3}
        }, \ \mymat{
            0_{3 \times 3} & 0_{3 \times 3} \\ 
            \sim & I_3
        }
     \right), \quad (\Varbar{X}, \Varbar{S}) = \left( 
        6 E_{11}, \ 3 E_{66}
      \right).
\end{align*}
Thus, we pick $\Varbar{Z} = 6 E_{11} - 3 \sigma E_{66}$.

\paragraph{First-order information.} $H \in \coneC$ if and only if:
\begin{align*}
    \simpleadjustbox{
    \PA \MatrixSixteenSC{
        H_{11} ; 
        \mymatplain{H_{12} & H_{13}} ; 
        \mymatplain{H_{14} & H_{15}} ;
        0 ; 
        \sim ; 
        \mymat{
            H_{22} & H_{23} \\
            \sim & H_{33}
        } \succeq 0 ; 
        0_{2 \times 2} ;
        0_{2 \times 1} ;
        \sim ; 
        \sim ;
        0_{2 \times 2} ;
        0_{2 \times 1} ;
        \sim ; 
        \sim ; 
        \sim ; 
        0
    } = 0~\text{and}~\PAp \MatrixSixteenSC{
        0 ; 
        0_{1 \times 2} ; 
        0_{1 \times 2} ;
        0 ; 
        \sim ; 
        0_{2 \times 2} ; 
        0_{2 \times 2} ;
        \mymatplain{H_{26} \\ H_{36}} ;
        \sim ; 
        \sim ;
        \mymat{
            H_{44} & H_{45} \\
            \sim & H_{55}
        } \preceq 0 ; 
        \mymatplain{H_{46} \\ H_{56}} ;
        \sim ; 
        \sim ; 
        \sim ; 
        H_{66}
    } = 0.
    }
\end{align*}
Via calculation, we find a family of $\Hbar$'s belonging to $\coneC \backslash \coneT$:
\begin{align}
    \label{eq:toy3:Hbar}
    \Hbar = \Hbar(h, \epsilon) = \MatrixNine{
        -1 ; \mymatplain{0 & -\frac{\sqrt{2}}{4} & 1 & h} ; 0 ;
        \sim ; \MatrixSixteenSC{
            1 ; 0 ; 0 ; 0 ; 
            \sim ; 0 ; 0 ; 0 ; 
            \sim ; \sim ; -\epsilon ; 0 ;
            \sim ; \sim ; \sim ; -1 
        } ; \mymatplain{1 \\ 1 \\ 1 \\ 1} ;
        \sim ; \sim ; 1 + \epsilon
    } \in \coneC \backslash \coneT, \quad \forall \epsilon \ge 0, h \in \mathbb{R}.
\end{align}

\paragraph{Second-order information.} We adopt the polar description in Proposition~\ref{prop:soa:polarK-XS} to calculate $\mdvZ[\Zbar; \Hbar(h, \epsilon)]$.

(\romannumeral1) For the primal part, calculating $\opeE$ from~\eqref{eq:soa:calE} as:
\begin{align*}
    \opeEP = \MatrixNine{
        0 ; \mymatplain{0 & 0 & \frac{\epsilon}{3} & \frac{h}{3}} ; -\frac{4h-\sqrt{2}+4}{6(\sigma+2)} ;
        \sim ; \MatrixSixteenSC{
            -\frac{2}{3\sigma} ; -\frac{2}{3\sigma} ; -\frac{2}{3\sigma} ; -\frac{2}{3\sigma} ; 
            \sim ; -\frac{2}{3\sigma} ; -\frac{2}{3\sigma} ; -\frac{2}{3\sigma} ; 
            \sim ; \sim ; -\frac{2}{3\sigma} ; -\frac{2}{3\sigma} ;
            \sim ; \sim ; \sim ; -\frac{2}{3\sigma} 
        } ; \mymatplain{-\frac{2}{3\sigma}  \\ 0 \\ \frac{2\epsilon}{3\sigma} \\ 0} ;
        \sim ; \sim ; 0
    }.
\end{align*}
Calculate $\coneKPX$ from~\eqref{eq:soa:polarK-XS}:
\begin{align*}
    & \coneKPX = \left\{ 
        U = \MatrixNine{
            U_{11} ; 
            \mymatplain{U_{12} & U_{13} & U_{14} & U_{15}} ;
            0 ;
            \sim ;
            \MatrixSixteenSC{
                U_{22} ; U_{23} ; U_{24} ; 0 ;
                \sim ; U_{33} \ge 0 ; 0 ; 0 ;
                \sim ; \sim ; 0 ; 0 ;
                \sim ; \sim ; \sim ; 0
            } ;
            \mymatplain{0 \\ 0 \\ 0 \\ 0} ;
            \sim ;
            \sim ; 
            0 
        } \mymid
        \PA U  = 0
        \right\} \\
    = & \left\{ 
        U = \MatrixNine{
            U_{11} ; 
            \mymatplain{U_{12} & U_{13} & U_{14} & U_{15}} ;
            0 ;
            \sim ;
            \MatrixSixteenSC{
                U_{22} ; U_{23} ; U_{24} ; 0 ;
                \sim ; U_{33} ; 0 ; 0 ;
                \sim ; \sim ; 0 ; 0 ;
                \sim ; \sim ; \sim ; 0
            } ;
            \mymatplain{0 \\ 0 \\ 0 \\ 0} ;
            \sim ;
            \sim ; 
            0 
        } \mymid \mymatplain{
            \Vartilde{U} = Q \mymat{
                U_{11} & U_{12} & U_{13} \\
                \sim & U_{22} & U_{23} \\
                \sim & \sim & U_{33} 
            } Q\tran, \\
            \Vartilde{U}_{11} = \Vartilde{U}_{22} = \Vartilde{U}_{33} = 0, \\
            U_{33} \ge 0, \ U_{24} = 0, \ U_{14} \in \mathbb{R}, \ U_{15} \in \mathbb{R}
        }
        \right\}.
\end{align*}
Thus, from Theorem~\ref{thm:soa:primal-dual-decouple}:
\begin{align*}
    & \mdvX = \argmin_{U \in \coneKPX} \normF{U + \opeEP}^2 \\
    = & \argmin_{\substack{
        U_{33} \ge 0, U_{24} = 0, \\
        \Vartilde{U}_{11} = \Vartilde{U}_{22} = \Vartilde{U}_{33} = 0
    }} \normlongF{
        \Vartilde{U} + Q \mymat{
            0 & 0 & 0 \\
            \sim & -\frac{2}{3 \sigma} & -\frac{2}{3 \sigma} \\
            \sim & \sim & -\frac{2}{3 \sigma}
        } Q\tran 
    }^2 + 2 (U_{14} + \frac{\epsilon}{3})^2 + 2 (U_{15} + \frac{h}{3})^2 + 2 (U_{24} - \frac{2}{3\sigma})^2 \\
    = & \MatrixNine{
    -\frac{4}{9 \sigma} ; 
    \mymatplain{-\frac{2\sqrt{2}}{9 \sigma} & 0 & -\frac{\epsilon}{3} & -\frac{h}{3}} ;
    0 ;
    \sim ;
    \MatrixSixteenSC{
        \frac{2}{9 \sigma} ; \frac{4}{9 \sigma} ; 0 ; 0 ;
        \sim ; \frac{2}{9 \sigma} ; 0 ; 0 ;
        \sim ; \sim ; 0 ; 0 ;
        \sim ; \sim ; \sim ; 0
    } ;
    \mymatplain{0 \\ 0 \\ 0 \\ 0} ;
    \sim ;
    \sim ; 
    0 
    }.
\end{align*}

(ii) For the dual part: from~\eqref{eq:soa:calE},
\begin{align*}
    \opeE = \MatrixNine{
        0 ; \mymatplain{0 & 0 & -\frac{\epsilon}{3} & -\frac{h}{3}} ; \frac{4h-\sqrt{2}+4}{6(\sigma+2)} ;
        \sim ; \MatrixSixteenSC{
            0 ; 0 ; 0 ; 0 ; 
            \sim ; \frac{1}{24} ; -\frac{\sqrt{2}}{12} ; -\frac{\sqrt{2}h}{12} ; 
            \sim ; \sim ; \frac{1}{3} ; \frac{h}{3} ;
            \sim ; \sim ; \sim ; \frac{h^2}{3} 
        } ; \mymatplain{\frac{2}{3\sigma}  \\ 0 \\ -\frac{2\epsilon}{3\sigma} \\ 0} ;
        \sim ; \sim ; 0
    }.
\end{align*}
However, $\coneKPS[\Hbar(h, \epsilon)]$ shows discontinuity at $\epsilon = 0$. 

(a) When $\epsilon = 0$: in this case, 
\begin{align*}
    & \coneKPS = \left\{ 
        V = \MatrixNine{
            0 ; 
            \mymatplain{0 & 0 & 0 & 0} ;
            0 ;
            \sim ;
            \MatrixSixteenSC{
                0 ; 0 ; 0 ; 0 ;
                \sim ; 0 ; 0 ; V_{35} ;
                \sim ; \sim ; V_{44} \le 0 ; V_{45} ;
                \sim ; \sim ; \sim ; V_{55}
            } ;
            \mymatplain{V_{26} \\ V_{36} \\ V_{46} \\ V_{56}} ;
            \sim ;
            \sim ; 
            V_{66}
        } \mymid \PAp V = 0
        \right\} \\
    = & \left\{ 
        V = \MatrixNine{
            0 ; 
            \mymatplain{0 & 0 & 0 & 0} ;
            0 ;
            \sim ;
            \MatrixSixteenSC{
                0 ; 0 ; 0 ; 0 ;
                \sim ; 0 ; 0 ; V_{35} ;
                \sim ; \sim ; V_{44}; V_{45} ;
                \sim ; \sim ; \sim ; V_{55}
            } ;
            \mymatplain{V_{26} \\ V_{36} \\ V_{46} \\ V_{56}} ;
            \sim ;
            \sim ; 
            V_{66}
        } \mymid \mymatplain{
            V_{44} \le 0, \\
            \mymat{V_{44} \\ V_{55} \\ V_{66}} \in a \mymat{1 \\ -1 \\ 0} + b \mymat{0 \\ 1 \\ -1}, \ a, b \in \mathbb{R}, \\
            V_{26}, V_{35}, V_{36}, V_{45}, V_{46}, V_{56} \in \mathbb{R}
        }
        \right\}.
\end{align*}
Thus, from Theorem~\ref{thm:soa:primal-dual-decouple}, we get 
\begin{align*}
    & -\sigma \mdvS = \argmin_{V \in \coneKPS} \normF{V + \opeE}^2 \\
    = & \simpleadjustbox{\argmin\limits_{\substack{
        V_{44} \le 0, \\ 
        V_{44} = a, \\
        V_{55} = -a + b, \\
        V_{66} = -b
    }} \normlongF{
        \mymat{V_{44} \\ V_{55} \\ V_{66}} + \mymat{\frac{1}{3} \\ \frac{h^2}{3} \\ 0}
    }^2 + 
    2 (V_{26} + \frac{2}{3\sigma})^2 + 2 (V_{35} - \frac{\sqrt{2}h}{12})^2 + 2 (V_{45} + \frac{h}{3})^2 
    + 2 V_{36}^2 + 2 V_{46}^2 + 2 V_{56}^2} \\
    = & \MatrixNine{
            0 ; 
            \mymatplain{0 & 0 & 0 & 0} ;
            0 ;
            \sim ;
            \MatrixSixteenSC{
                0 ; 0 ; 0 ; 0 ;
                \sim ; 0 ; 0 ; \frac{\sqrt{2}h}{12} ;
                \sim ; \sim ; a^\star; -\frac{h}{3} ;
                \sim ; \sim ; \sim ; -a^\star + b^\star 
            } ;
            \mymatplain{-\frac{2}{3\sigma} \\ 0 \\ 0 \\ 0} ;
            \sim ;
            \sim ; 
            -b^\star
    }.
\end{align*}
where $(a^\star, b^\star)$ is defined as:
\begin{align*}
    (a^\star, b^\star) = \argmin_{a \le 0} (a + \frac{1}{3})^2 + (-a + b + \frac{h^2}{3})^2 + b^2 
    = \begin{cases}
        (\frac{1}{9}(h^2 - 2), \ -\frac{1}{9} (h^2 + 1)), & \quad \abs{h} \le \sqrt{2} \\
        (0, \ -\frac{1}{6} h^2), & \quad \abs{h} > \sqrt{2}
    \end{cases}
\end{align*}

(b) When $\epsilon > 0$: in this case, 
\begin{align*}
    & \coneKPS = \left\{ 
        V = \MatrixNine{
            0 ; 
            \mymatplain{0 & 0 & 0 & 0} ;
            0 ;
            \sim ;
            \MatrixSixteenSC{
                0 ; 0 ; 0 ; 0 ;
                \sim ; 0 ; 0 ; V_{35} ;
                \sim ; \sim ; V_{44} ; V_{45} ;
                \sim ; \sim ; \sim ; V_{55}
            } ;
            \mymatplain{V_{26} \\ V_{36} \\ V_{46} \\ V_{56}} ;
            \sim ;
            \sim ; 
            V_{66}
        } \mymid \PAp V = 0
        \right\} \\
    = & \left\{ 
        V = \MatrixNine{
            0 ; 
            \mymatplain{0 & 0 & 0 & 0} ;
            0 ;
            \sim ;
            \MatrixSixteenSC{
                0 ; 0 ; 0 ; 0 ;
                \sim ; 0 ; 0 ; V_{35} ;
                \sim ; \sim ; V_{44}; V_{45} ;
                \sim ; \sim ; \sim ; V_{55}
            } ;
            \mymatplain{V_{26} \\ V_{36} \\ V_{46} \\ V_{56}} ;
            \sim ;
            \sim ; 
            V_{66}
        } \mymid \mymatplain{
            \mymat{V_{44} \\ V_{55} \\ V_{66}} \in a \mymat{1 \\ -1 \\ 0} + b \mymat{0 \\ 1 \\ -1}, \ a, b \in \mathbb{R}, \\
            V_{26}, V_{35}, V_{36}, V_{45}, V_{46}, V_{56} \in \mathbb{R}
        }
        \right\}.
\end{align*}
Thus, 
\begin{align*}
    & -\sigma \mdvS = \argmin_{V \in \coneKPS} \normF{V + \opeE}^2 \\
    = & \simpleadjustbox{\argmin\limits_{\substack{
        V_{44} = a, \\
        V_{55} = -a + b, \\
        V_{66} = -b
    }} \normlongF{
        \mymat{V_{44} \\ V_{55} \\ V_{66}} + \mymat{\frac{1}{3} \\ \frac{h^2}{3} \\ 0}
    }^2 + 
    2 (V_{26} + \frac{2}{3\sigma})^2 + 2 (V_{35} - \frac{\sqrt{2}h}{12})^2 + 2 (V_{45} + \frac{h}{3})^2 
    + 2 V_{36}^2 + 2 (V_{46} - \frac{2\epsilon}{3\sigma})^2 + 2 V_{56}^2} \\
    = & \MatrixNine{
            0 ; 
            \mymatplain{0 & 0 & 0 & 0} ;
            0 ;
            \sim ;
            \MatrixSixteenSC{
                0 ; 0 ; 0 ; 0 ;
                \sim ; 0 ; 0 ; \frac{\sqrt{2}h}{12} ;
                \sim ; \sim ; \frac{1}{9}(h^2 - 2); -\frac{h}{3} ;
                \sim ; \sim ; \sim ; -\frac{1}{9}(2 h^2 - 1) 
            } ;
            \mymatplain{-\frac{2}{3\sigma} \\ 0 \\ \frac{2\epsilon}{3\sigma} \\ 0} ;
            \sim ;
            \sim ; 
            \frac{1}{9} (h^2 + 1)
    }.
\end{align*}

(\romannumeral3) Combining (\romannumeral1) and (\romannumeral2), we get:

(a) If case \RomanNum{1}: $\abs{h} \le \sqrt{2}$; or case \RomanNum{2}: $\abs{h} > \sqrt{2}$ and $\epsilon > 0$,
\begin{align}
    \label{eq:toy3:mdvZ-1}
    \mdvZ = \mdvX - \sigma \mdvS 
    = \MatrixNine{
        -\frac{4}{9 \sigma} ; 
        \mymatplain{-\frac{2\sqrt{2}}{9 \sigma} & 0 & -\frac{\epsilon}{3} & -\frac{h}{3}} ;
        0 ;
        \sim ;
        \MatrixSixteenSC{
            \frac{2}{9 \sigma} ; \frac{4}{9 \sigma} ; 0 ; 0 ;
            \sim ; \frac{2}{9 \sigma} ; 0 ; \frac{\sqrt{2}h}{12} ;
            \sim ; \sim ; \frac{h^2 - 2}{9}; -\frac{h}{3} ;
            \sim ; \sim ; \sim ; -\frac{2 h^2 - 1}{9}
        } ;
        \mymatplain{-\frac{2}{3\sigma} \\ 0 \\ \frac{2\epsilon}{3\sigma} \\ 0} ;
        \sim ;
        \sim ; 
        \frac{h^2 + 1}{9} 
    }.
\end{align}

(b) If $\abs{h} > \sqrt{2}$ and $\epsilon = 0$,
\begin{align}
    \label{eq:toy3:mdvZ-2}
    \mdvZ = \mdvX - \sigma \mdvS 
    = \MatrixNine{
            -\frac{4}{9 \sigma} ; 
            \mymatplain{-\frac{2\sqrt{2}}{9 \sigma} & 0 & 0 & -\frac{h}{3}} ;
            0 ;
            \sim ;
            \MatrixSixteenSC{
                \frac{2}{9 \sigma} ; \frac{4}{9 \sigma} ; 0 ; 0 ;
                \sim ; \frac{2}{9 \sigma} ; 0 ; \frac{\sqrt{2}h}{12} ;
                \sim ; \sim ; 0; -\frac{h}{3} ;
                \sim ; \sim ; \sim ; -\frac{h^2}{6} 
            } ;
            \mymatplain{-\frac{2}{3\sigma} \\ 0 \\ 0 \\ 0} ;
            \sim ;
            \sim ; 
            \frac{h^2}{6}
    }.
\end{align}

\paragraph{$\sigma$ updating.} Following the exact same procedure in \S\ref{sec:toy1} and \S\ref{sec:toy2}, we can get $\mdvX[\Zbar'; \Hbar'] = \frac{\sigma'}{\sigma} \mdvX$ and $\mdvS[\Zbar'; \Hbar'] = \frac{\sigma}{\sigma'} \mdvS$ as we update $\sigma$ to $\sigma'$ for all $\Hbar = \Hbar(h, \epsilon)$. We omit the details here. 

%% file: sections/experiments.tex

\section{\titlemath{Numerical Experiments}}
\label{sec:exp}

\paragraph{Experiment setup.}
To further qualitatively evaluate our analysis framework, we run the three-step ADMM~\eqref{eq:intro:admm-three-step} on the \Mittelmann\ dataset, a widely used benchmark for SDP solvers~\cite{aps19ugrm-mosek-sdpsolver,sturm99oms-sedumi-sdpsolver,tutuncu03mp-sdpt3-sdpsolver,han24arxiv-culoras}. For concreteness, we select the single-block instances with block size no greater than $3000$, yielding $25$ instances in total. All experiments were conducted on the Harvard University Faculty of Arts and Sciences Research Computing (FASRC) cluster. Jobs were submitted to the \texttt{seas\_compute} partition, and each run requested 48 CPU cores and 64\,GB of memory.\footnote{\href{https://docs.rc.fas.harvard.edu/kb/running-jobs/}{https://docs.rc.fas.harvard.edu/kb/running-jobs/}}

\paragraph{Experiment \RomanNum{1}.}
After rescaling the SDP data and applying diagonal preconditioning to the constraint matrix $\Asdp$, we start three-step ADMM~\eqref{eq:intro:admm-three-step} with all-zero initial guesses using a \emph{fixed} $\sigma$-updating strategy that aims to balance the primal and dual infeasibilities~\cite{wen10mpc-admmsdp} for $20000$ iterations. After that, $\sigma$ is kept unchanged. At the $40000$th iteration, we record the current penalty parameter as $\sigma_0$ and set $(X^{(40000)}, y^{(40000)}, S^{(40000)})$ as $(X_0, y_0, S_0)$ for subsequent use. For each instance, we set the maximum number of iterations to $10^6$ and the maximum running time to $168$ hours. We terminate once the maximum KKT residual satisfies $r_{\max}\le 10^{-10}$.

We report the trajectories of $\angle(\Delta \Vark{Z}, \Delta \Varkpo{Z})$, $\normF{\Delta \Vark{Z}}$, and $\Vark{r}_{\max}$. The results are shown in Figure~\ref{fig:exp:original}. Based on whether ADMM solves an instance to $r_{\max}$ below $10^{-10}$, we divide the $25$ SDPs into two groups:
\begin{itemize}
    \item ``Easy'' SDPs: \dataset{1et2048}, \dataset{1zc1024}, \dataset{cphil12}, \dataset{G48mb}, \dataset{G48mc}, \dataset{hamming8}, \dataset{hamming9}, \dataset{theta12}, \dataset{theta102}, \dataset{theta123}. ($10$ instances.)

    \item ``Hard'' SDPs: \dataset{1dc1024}, \dataset{1tc2048}, \dataset{cancer100}, \dataset{cnhil10}, \dataset{foot}, \dataset{G40mb}, \dataset{hand}, \dataset{neosfbr25}, \dataset{neosfbr30e8}, \dataset{neu1g}, \dataset{neu2g}, \dataset{neu3g}, \dataset{r12000}, \dataset{swissroll}, \dataset{texture}. ($15$ instances.)
\end{itemize}
Across the slow-convergence regions of all ``hard'' instances, we observe that $\angle(\Delta \Vark{Z}, \Delta \Varkpo{Z})$ remains small yet nonzero for most iterations (typically around $10^{-3}$ to $10^{-5}$), except for several sparse spikes. This behavior is consistent with our local second-order limiting dynamics model~\eqref{eq:soa:limiting-dynamics-def}, as discussed in \S\ref{sec:mdvZ-kernel-3} and \S\ref{sec:mdvZ-continuity-3}.
\input{figs/experiments/fig_original_kkt_ang.tex}

\paragraph{Experiment \RomanNum{2}.}
To further probe the slow-convergence regimes of the $15$ ``hard'' instances, we restart from $(X_0,y_0,S_0)$ and $\sigma_0$, and then uniformly increase $\log_{10}(\sigma)$ from $\log_{10}(\sigma_0)$ to $\log_{10}(10\sigma_0)$ over the next $5000$ iterations. As in \S\ref{sec:mdvZ-sigma-3}, we plot the resulting ``response curves'' of $\normF{\Delta \Vark{X}}$, $\normF{\Delta \Vark{S}}$, $\Vark{r}_p$, and $\Vark{r}_d$ as functions of $\sigma$. The results are shown in Figure~\ref{fig:exp:sigwave}. We further divide them into three groups:
\begin{itemize}
    \item Group \RomanNum{1}: \dataset{cnhil10}, \dataset{foot}, \dataset{neu1g}, \dataset{neu3g}, \dataset{texture}. ($5$ instances.)
    \item Group \RomanNum{2}: \dataset{1dc1024}, \dataset{G40mb}, \dataset{hand}, \dataset{neosfbr25}, \dataset{r12000}, \dataset{swissroll}. ($6$ instances.)
    \item Group \RomanNum{3}: \dataset{1tc2048}, \dataset{cancer100}, \dataset{neosfbr30e8}, \dataset{neu2g}. ($4$ instances.)
\end{itemize}
The response curves of Group \RomanNum{1} are compatible with our limit dynamics model. For these instances, $\log_{10}(\normF{\Delta \Vark{X}})$ (resp.\ $\log_{10}(\normF{\Delta \Vark{S}})$) increases (resp.\ decreases) approximately linearly with $\log_{10}(\sigma)$, with slope close to $+1$ (resp.\ $-1$). For these $5$ SDPs, the slow-convergence regions are therefore likely driven by the second-order limit dynamics~\eqref{eq:soa:limiting-dynamics-def}.

For Group \RomanNum{2}, the response curves partially resemble those in Group \RomanNum{1}. When $\sigma$ is small, the slope of $\log_{10}(\normF{\Delta \Vark{X}})$ (resp.\ $\log_{10}(\normF{\Delta \Vark{S}})$) is close to $+1$ (resp.\ $-1$). However, as $\sigma$ increases, the curves distort, either smoothly or abruptly. For these $6$ SDPs, we conjecture that updating $\sigma$ helps the iterates escape the current second-order-dominant regions.

For Group \RomanNum{3}, the response curves deviate substantially from those predicted by the local second-order limit dynamics model. The mechanisms underlying the slow-convergence behavior of these $4$ instances therefore remain unclear to us.
\input{figs/experiments/fig_sigwave_four_items.tex}

%% file: figs/experiments/fig_original_kkt_ang.tex

\begingroup

\newcommand{\FieldFig}[1]{
    \begin{minipage}{0.18\textwidth}
        \centering
        \dataset{#1}
        \includegraphics[width=\columnwidth]{\accPrefix/#1/original_ang.png}
        \includegraphics[width=\columnwidth]{\accPrefix/#1/original_kkt.png}
    \end{minipage}
}

\begin{figure}[htbp]
    \centering
    \begin{minipage}{\textwidth}
        \centering
        \includegraphics[width=0.5\columnwidth]{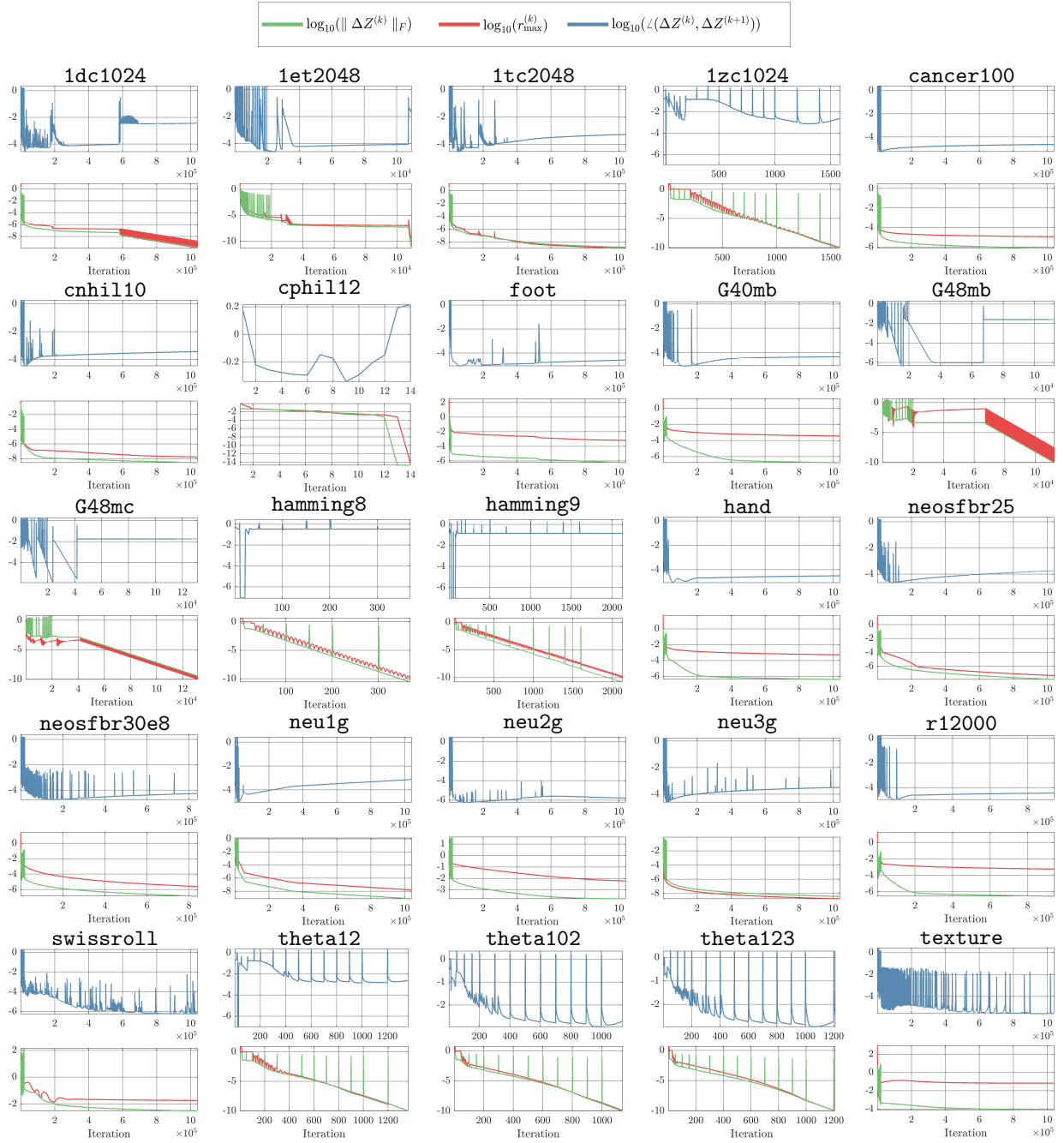}
        \vspace{3mm}
    \end{minipage}

    \begin{minipage}{\textwidth}
        \centering
        \begin{tabular}{ccccc}
            \FieldFig{1dc1024}
            \FieldFig{1et2048}
            \FieldFig{1tc2048}
            \FieldFig{1zc1024}
            \FieldFig{cancer100}
        \end{tabular}
    \end{minipage}

    \begin{minipage}{\textwidth}
        \centering
        \begin{tabular}{ccccc}
            \FieldFig{cnhil10}
            \FieldFig{cphil12}
            \FieldFig{foot}
            \FieldFig{G40mb}
            \FieldFig{G48mb}
        \end{tabular}
    \end{minipage}

    \begin{minipage}{\textwidth}
        \centering
        \begin{tabular}{ccccc}
            \FieldFig{G48mc}
            \FieldFig{hamming8}
            \FieldFig{hamming9}
            \FieldFig{hand}
            \FieldFig{neosfbr25}
        \end{tabular}
    \end{minipage}

    \begin{minipage}{\textwidth}
        \centering
        \begin{tabular}{ccccc}
            \FieldFig{neosfbr30e8}
            \FieldFig{neu1g}
            \FieldFig{neu2g}
            \FieldFig{neu3g}
            \FieldFig{r12000}
        \end{tabular}
    \end{minipage}

    \begin{minipage}{\textwidth}
        \centering
        \begin{tabular}{ccccc}
            \FieldFig{swissroll}
            \FieldFig{theta12}
            \FieldFig{theta102}
            \FieldFig{theta123}
            \FieldFig{texture}
        \end{tabular}
    \end{minipage}

    \caption{\label{fig:exp:original} $\Vark{r}_{\max}$, $\normF{\Delta \Vark{Z}}$, and $\angle(\Delta \Vark{Z}, \Delta \Varkpo{Z})$ for $25$ \Mittelmann~SDP datasets (single-block, block size $\le 3000$). For \emph{all} instances exhibiting slow convergence, $\angle(\Delta \Vark{Z}, \Delta \Varkpo{Z})$ remains small for an extended period, except for a few sparse spikes. In contrast, for many instances that display an observable sharp linear convergence phase, $\angle(\Delta \Vark{Z}, \Delta \Varkpo{Z})$ is large.}
\end{figure}

\endgroup

%% file: figs/experiments/fig_sigwave_four_items.tex

\begingroup

\newcommand{\FieldFig}[1]{
    \begin{minipage}{0.18\textwidth}
        \centering
        \dataset{#1}
        \includegraphics[width=\columnwidth]{\sigPrefix/#1/four_items.png}
    \end{minipage}
}

\begin{figure}[h]
    \centering
    \begin{minipage}{\textwidth}
        \centering
        \includegraphics[width=0.5\columnwidth]{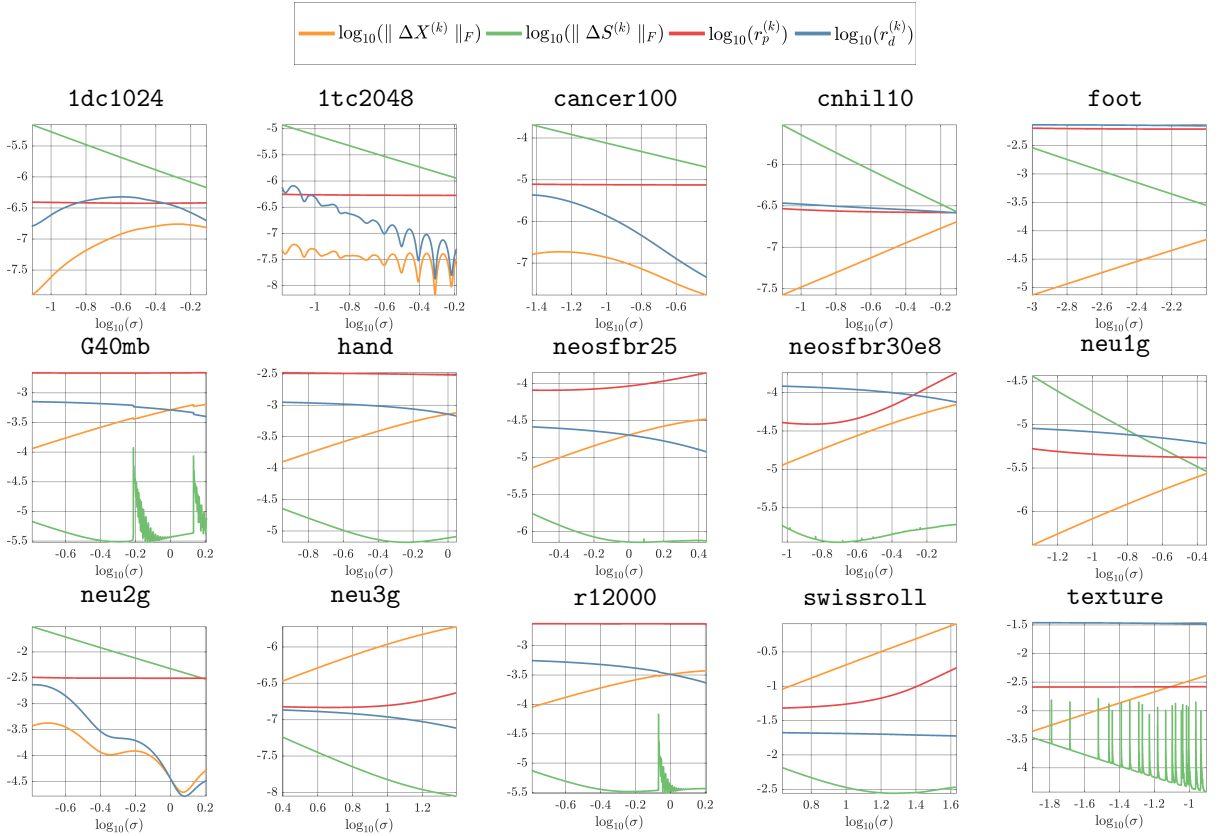}
        \vspace{3mm}
    \end{minipage}

    \begin{minipage}{\textwidth}
        \centering
        \begin{tabular}{ccccc}
            \FieldFig{1dc1024}
            \FieldFig{1tc2048}
            \FieldFig{cancer100}
            \FieldFig{cnhil10}
            \FieldFig{foot}
        \end{tabular}
    \end{minipage}

    \begin{minipage}{\textwidth}
        \centering
        \begin{tabular}{ccccc}
            \FieldFig{G40mb}
            \FieldFig{hand}
            \FieldFig{neosfbr25}
            \FieldFig{neosfbr30e8}
            \FieldFig{neu1g}
        \end{tabular}
    \end{minipage}

    \begin{minipage}{\textwidth}
        \centering
        \begin{tabular}{ccccc}
            \FieldFig{neu2g}
            \FieldFig{neu3g}
            \FieldFig{r12000}
            \FieldFig{swissroll}
            \FieldFig{texture}
        \end{tabular}
    \end{minipage}

    \caption{\label{fig:exp:sigwave} For the selected $15$ ``hard'' SDP instances in the \Mittelmann\ dataset, we run ADMM iterations with initial guess $(X_0, y_0, S_0)$ and initial $\sigma = \sigma_0$. We uniformly increase $\log_{10}(\sigma)$ from $\log_{10}(\sigma_0)$ to $\log_{10}(10\sigma_0)$ over the next $5000$ iterations. We report the trajectories of $\normF{\Delta \Vark{X}}$, $\normF{\Delta \Vark{S}}$, $\Vark{r}_p$, and $\Vark{r}_d$.}
\end{figure}

\endgroup

%% file: sections/conclusion.tex

\section{Future Directions}
\label{sec:future}

Our work opens several directions for future research.

\paragraph{Approximation error of the limit dynamics.}
As discussed in \S\ref{sec:intro:limitations}, a main open issue is to quantify the approximation error between the local second-order limit dynamics model~\eqref{eq:soa:limiting-dynamics-def} and the true ADMM dynamics~\eqref{eq:intro:one-step-admm}. This appears challenging in general, since the model relies on three coupled approximation layers (cf.\ Remark~\ref{rem:soa:three-level-approx}). Developing a principled error-control theory---for example, identifying regimes where the limit dynamics provides uniform guarantees over time horizons relevant to slow-convergence behavior---would substantially strengthen the framework.

\paragraph{Extension to singularity degree $>1$.}
The current framework relies on the existence of a strictly complementary solution pair (i.e., the singularity degree of the optimal set is $1$~\cite{sturm00siopt-error-bound-lmi}) to simplify the block structure and the ensuing analysis (cf.\ Remark~\ref{rem:soa:sc-partition}). This assumption may fail in general. Extending the analysis to singularity degree $d>1$ is an important direction. A natural conjecture is that one may need to understand local expansions to order $2^d$ to capture the correct limit behavior, although the precise order and the right notion of limit map remain open.

\paragraph{Characterization of the almost-invariant set.}
Our current discussion of local almost-invariant sets is qualitative (cf.\ \S\ref{sec:mdvZ-range-3}). A quantitative theory that bounds the leakage from $\coneC$ would provide a more complete picture of slow-convergence regions. Achieving this likely requires new technical tools to balance the two competing effects: the first-order dynamics that pulls iterates toward $\coneC$ and the second-order drift that may push them away.

\paragraph{Algorithmic acceleration.}
Several components of our analysis have direct implications for algorithm design, especially the dependence on $\sigma$ (cf.\ \S\ref{sec:mdvZ-sigma}). A promising direction is to detect second-order-dominant regimes using the $\sigma$--$(\Vark{r}_p,\Vark{r}_d)$ response curves developed in Experiment~\RomanNum{2} (cf.\ \S\ref{sec:exp}). Once such a regime is detected, one could design region-wise $\sigma$-adaptation rules informed by additional problem structure (e.g., one-sided uniqueness) and by the predicted scaling of the limit drifts.

\paragraph{Extension to other splitting methods and conic programs.}
It would also be interesting to extend the present approach beyond ADMM to other splitting schemes (e.g., sGS-ADMM and PDHG), and to other conic programs where projection operators admit comparable second-order structure. Such extensions could help clarify whether second-order limit dynamics is a general mechanism underlying slow convergence across a broader class of operator-splitting algorithms.


\section{\titlemath{Conclusion}}
\label{sec:conclusion}

This paper developed a transient, region-wise perspective on the slow-convergence behavior observed in ADMM for SDPs with multiple KKT points. We refined and streamlined the (parabolic) second-order directional derivative formula for the PSD projection operator, and leveraged it to derive a detailed second-order expansion of the ADMM dynamics around an arbitrary KKT point $\Zbar$. This expansion isolates the cone $\coneC$ of first-order stalled directions and leads to the central object of the paper: the local second-order limit map $\mdvZmap:\coneC\mapsto\Sym{n}$ and its induced limit dynamics. This limit dynamics serves as a local surrogate for the nonlinear ADMM update after transient effects have decayed.

We then analyzed four structural properties of $\mdvZmap$: its kernel, range, continuity, and dependence on the penalty parameter $\sigma$. These results explain or predict three empirical slow-convergence patterns:
\begin{itemize}
    \item Using the characterization $\ker(\mdvZmap)=\coneT$ together with the almost-sure type continuity of the limit map, we showed that $\angle(\Delta \Vark{Z},\Delta \Varkpo{Z})$ tends to be small yet nonzero, except for sparse spikes.
    \item By relating $\range(\mdvZmap)$ to $\affinehull(\coneC)$, we showed that $\Vark{Z}$ can be transiently trapped in a low-dimensional subspace for an extended period of time.
    \item Exploiting a primal--dual decoupling of $\mdvZmap$, we showed that primal/dual infeasibilities are locally insensitive to $\sigma$ in the second-order-dominant regimes, clarifying why classical balancing heuristics can become ineffective.
\end{itemize}
Extensive experiments on the \Mittelmann\ dataset corroborate these theoretical predictions. We hope our results motivate a broader and more systematic study of the ubiquitous slow-convergence behavior encountered in first-order splitting methods for SDPs.

%% file: appendices/psdproj_proof.tex

\section{Proof of Theorem~\ref{thm:psd:second-dd}}
\label{app:sec:proof-second-dd}

Let $F: \Sym{n} \mapsto \Sym{n}$ be a (parabolically) second-order directional differentiable spectral function generated by a (parabolically) second-order directional differentiable scalar function $f: \mathbb{R} \mapsto \mathbb{R}$.
For $x \ne y \ne z$, define first- and (parabolic) second-order divided difference of $f$ as
\begin{align*}
    & f^{[1]}(x, y) := \frac{f(x) - f(y)}{x - y} = \frac{f(x)}{x - y} + \frac{f(y)}{y - x}, \\
    & f^{[2]}(x, y, z) := \frac{f^{[1]}(x,y) - f^{[1]}(x,z)}{y - z} = \frac{f(x)}{(x-y)(x-z)} + \frac{f(y)}{(y-x)(y-z)} + \frac{f(z)}{(z-x)(z-y)}.
\end{align*} 
Now let $(Z, H, W)$ be given by the three-level description in \S\ref{sec:psd}.
For any $k \in \totalsetfirst$, denote $\Phi_k: \Sym{\sizeof{\indexfirst{k}}} \mapsto \Sym{\sizeof{\indexfirst{k}}}$ as the spectral function generated by the scalar function as $f'(\eigvalfirst{k}; \cdot)$. For any $k \in \totalsetfirst, i \in \totalsetsecond{k}$, denote $\Psi_{k,i}: \Sym{\sizeof{\indexsecond{k}{i}}} \mapsto \Sym{\sizeof{\indexsecond{k}{i}}}$, such that $\Psi_{k,i}$ is generated by the scalar function is $f''(\eigvalfirst{k}; \eigvalsecond{k}{i}, \cdot)$. For any $a, b \in \totalsetfirst$, define 
\begin{align}
    \label{app:eq:psd:Gamma-1}
    \simpleadjustbox{
        \Var{\Gamma_1(H, W)}[a][b] 
            := \begin{cases}
                f^{[1]}(\eigvalfirst{a}, \eigvalfirst{b}) W_{\indexfirst{a} \indexfirst{b}} + \sum\limits_{c \ne \{a, b\}} 2 f^{[2]} (\eigvalfirst{a}, \eigvalfirst{b}, \eigvalfirst{c}) H_{\indexfirst{a} \indexfirst{c}} H_{\indexfirst{c} \indexfirst{b}} - \frac{2 (f(\eigvalfirst{a}) - f(\eigvalfirst{b}))}{(\eigvalfirst{a} - \eigvalfirst{b})^2} (H_{\indexfirst{a} \indexfirst{a}} H_{\indexfirst{a} \indexfirst{b}} - H_{\indexfirst{a} \indexfirst{b}} H_{\indexfirst{b} \indexfirst{b}}), &\ a \ne b \\
                - \sum\limits_{c \in \totalsetfirst \backslash \{a\}} \frac{
                    2 (f(\eigvalfirst{a}) - f(\eigvalfirst{c}))
                }{
                    (\eigvalfirst{a} - \eigvalfirst{c})^2
                } H_{\indexfirst{a} \indexfirst{c}} H_{\indexfirst{c} \indexfirst{a}}, & \ a = b 
            \end{cases},
    }
\end{align}
and 
\begin{align}
    \label{app:eq:psd:Gamma-2}
    \Var{\Gamma_2(H, W)}[a][b] = \begin{cases}
        \Qfirst{a} [ 
            \Omega^a \circ \left( \Qfirst{a} \right)\tran V_a \Qfirst{a}
        ] \left( \Qfirst{a} \right)\tran, & \ a = b \\
        \Phi_a(H_{\indexfirst{a} \indexfirst{a}}) \frac{
                2 H_{\indexfirst{a} \indexfirst{b}}
            }{
                \eigvalfirst{a} - \eigvalfirst{b}
            } + 
            \frac{
                2 H_{\indexfirst{a} \indexfirst{b}}
            }{
                \eigvalfirst{b} - \eigvalfirst{a}
            } \Phi_b(H_{\indexfirst{b} \indexfirst{b}}), & \ a \ne b
    \end{cases},
\end{align}
and 
\begin{align}
    \label{app:eq:psd:Gamma-3}
    \Var{\Gamma_3(H, W)}[a][b] = \begin{cases}
        \Qfirst{a} \diag{
            \left\{ 
                \Psi_{a,i}(\hat{V}_a^{i,i})
            \right\}_{i \in \totalsetsecond{a}}
        } \left( \Qfirst{a} \right)\tran, & \ a = b \\
        0, & \ a \ne b
    \end{cases}.
\end{align}

To prove Theorem~\ref{thm:psd:second-dd}, we first need the (parabolic) second-order directional derivative's formula for a general spectral function $F$ from~\cite[Theorem 4.1]{zhang13svva-second-order-directional-derivative-symmatric-matrix-valued}.

\begin{theorem}[$F''(Z; H, W)$]
    \label{app:thm:general-second-dd}
    Let the triplet $(Z, H, W)$ be given by the three-level description in \S\ref{sec:psd}. Then, for any $a, b \in \totalsetfirst$,
    \begin{align*}
        \Var{F''(Z; H, W)}[a][b] = \Var{\Gamma_1(H, W)}[a][b] + \Var{\Gamma_2(H, W)}[a][b] + \Var{\Gamma_3(H, W)}[a][b],
    \end{align*}
    where $\Gamma_1, \Gamma_2, \Gamma_3$ is defined in~\eqref{app:eq:psd:Gamma-1} to~\eqref{app:eq:psd:Gamma-3}. \footnote{
        In~\cite[Eq. (4.4) - (4.5)]{zhang13svva-second-order-directional-derivative-symmatric-matrix-valued}, the authors drop the multiplier $2$, and it should be $-\frac{f(\mu_l) - f(\mu_k)}{(\mu_l - \mu_k)^2}$, instead of $\frac{f(\mu_l) - f(\mu_k)}{(\mu_l - \mu_k)^2}$. In~\cite[Theorem 4.1]{zhang13svva-second-order-directional-derivative-symmatric-matrix-valued}, $\left[ F''(Z; H, W) \right]_{\indexfirst{a} \indexfirst{a}}$ drops the term $C(H, W)_{\indexfirst{a} \indexfirst{a}}$.
    }
\end{theorem}

Now we are ready to prove Theorem~\ref{thm:psd:second-dd}. For the PSD cone projection operator $\psdproj{n}(\cdot)$, $f(\eigvalfirst{k}) = \max\left\{ \eigvalfirst{k},0 \right\}$. Thus,
\begin{align*}
    f'(\eigvalfirst{k}; \eigvalsecond{k}{i}) = \begin{cases}
        \eigvalsecond{k}{i}, & \quad \eigvalfirst{k} > 0 \\
        \max\left\{ \eigvalsecond{k}{i}, 0 \right\}, & \quad \eigvalfirst{k} = 0 \\
        0, & \quad \eigvalfirst{k} < 0
    \end{cases},
\end{align*}
and 
\begin{align*}
    f''(\eigvalfirst{k}; \eigvalsecond{k}{i}, \eigvalthird{k}{i}{i'}) = \begin{cases}
        \eigvalthird{k}{i}{i'}, & \quad \eigvalfirst{k} > 0 \\
        \begin{cases}
            \eigvalthird{k}{i}{i'}, & \quad \eigvalsecond{k}{i} > 0 \\
            \max\left\{ \eigvalthird{k}{i}{i'}, 0 \right\}, & \quad \eigvalsecond{k}{i} = 0 \\
            0, & \quad \eigvalsecond{k}{i} < 0
        \end{cases}, & \quad \eigvalfirst{k} = 0 \\
        0, & \quad \eigvalfirst{k} < 0 
    \end{cases}.
\end{align*}

\textbf{Case (1)(i): $a \in \totalsetfirst[+], b \in \totalsetfirst[+]$ and $a \ne b$}. For $\Var{\Gamma_1(H,W)}[a][b]$, $f^{[1]}(\eigvalfirst{a}, \eigvalfirst{b}) = 1$, $f(\eigvalfirst{a}) - f(\eigvalfirst{b}) = \eigvalfirst{a} - \eigvalfirst{b}$, and
\begin{align*}
    f^{[2]}(\eigvalfirst{a}, \eigvalfirst{b}, \eigvalfirst{c}) = \begin{cases}
        0, & \quad c \in \totalsetfirst[+] \cup \totalsetfirst[0] \\
        \frac{1}{\eigvalfirst{a} - \eigvalfirst{b}} \left( \frac{\eigvalfirst{a}}{\eigvalfirst{a} - \eigvalfirst{c}} - \frac{\eigvalfirst{b}}{\eigvalfirst{b} - \eigvalfirst{c}} \right) = \frac{-\eigvalfirst{c}}{(\eigvalfirst{c} - \eigvalfirst{a})(\eigvalfirst{c} - \eigvalfirst{b})}, & \quad c \in \totalsetfirst[-]
    \end{cases}
\end{align*}
Thus,~\footnote{
    In~\cite[Eq. (10a)]{liu22svva-second-order-sdcmpcc}, the $\mu_j$ in the numerator should be $-\mu_j$.
}
\begin{align*}
    \Var{\Gamma_1(H,W)}[a][b] = W_{\indexfirst{a} \indexfirst{b}}
    + 2 \sum_{c \in \totalsetfirst[-]} \frac{-\eigvalfirst{c}}{(\eigvalfirst{c} - \eigvalfirst{a})(\eigvalfirst{c} - \eigvalfirst{b})} H_{\indexfirst{a} \indexfirst{c}} H_{\indexfirst{c} \indexfirst{b}}
    - \frac{2}{\eigvalfirst{a} - \eigvalfirst{b}} (H_{\indexfirst{a} \indexfirst{a}} H_{\indexfirst{a} \indexfirst{b}} - H_{\indexfirst{a} \indexfirst{b}} H_{\indexfirst{b} \indexfirst{b}}).
\end{align*}
For $\Var{\Gamma_2(H,W)}[a][b]$, we have $\Phi_a(H_{\indexfirst{a} \indexfirst{a}}) = H_{\indexfirst{a} \indexfirst{a}}$ since $f'(\eigvalfirst{a}; \eigvalsecond{a}{i}) = \eigvalsecond{a}{i}$. Symmetrically, $\Phi_b(H_{\indexfirst{b} \indexfirst{b}}) = H_{\indexfirst{b} \indexfirst{b}}$. Thus,
\begin{align*}
    \Var{\Gamma_2(H,W)}[a][b] = \frac{2}{\eigvalfirst{a} - \eigvalfirst{b}} H_{\indexfirst{a} \indexfirst{a}} H_{\indexfirst{a} \indexfirst{b}} - \frac{2}{\eigvalfirst{a} - \eigvalfirst{b}} H_{\indexfirst{a} \indexfirst{b}} H_{\indexfirst{b} \indexfirst{b}}.
\end{align*}
For $\Var{\Gamma_3(H,W)}[a][b]$, it is $0$.
Thus,
\begin{align*}
    \psdproj{n}''(Z; H, W)_{\indexfirst{a} \indexfirst{b}}
    = W_{\indexfirst{a} \indexfirst{b}}
    + 2 \sum_{c \in \totalsetfirst[-]} \frac{-\eigvalfirst{c}}{(\eigvalfirst{c} - \eigvalfirst{a})(\eigvalfirst{c} - \eigvalfirst{b})} H_{\indexfirst{a} \indexfirst{c}} H_{\indexfirst{c} \indexfirst{b}}.
\end{align*}

\textbf{Case (1)(ii): $a \in \totalsetfirst[+], b \in \totalsetfirst[+]$ and $a = b$}. For $\Var{\Gamma_1(H,W)}[a][a]$, 
\begin{align*}
    \Var{\Gamma_1(H,W)}[a][a] = & -2 \sum_{c \in \totalsetfirst \backslash \{a\}} \frac{f(\eigvalfirst{a}) - f(\eigvalfirst{c})}{(\eigvalfirst{a} - \eigvalfirst{c})^2} H_{\indexfirst{a} \indexfirst{c}} H_{\indexfirst{c} \indexfirst{a}} \\
    = & -2 \sum_{c \in \totalsetfirst[+] \backslash \{a\}} \frac{1}{\eigvalfirst{a} - \eigvalfirst{c}} H_{\indexfirst{a} \indexfirst{c}} H_{\indexfirst{c} \indexfirst{a}}
    -2 \sum_{c \in \totalsetfirst[0]} \frac{1}{\eigvalfirst{a}} H_{\indexfirst{a} \indexfirst{c}} H_{\indexfirst{c} \indexfirst{a}}
    -2 \sum_{c \in \totalsetfirst[-]} \frac{\eigvalfirst{a}}{(\eigvalfirst{a} - \eigvalfirst{c})^2} H_{\indexfirst{a} \indexfirst{c}} H_{\indexfirst{c} \indexfirst{a}}.
\end{align*}
For $\Var{\Gamma_2(H,W)}[a][a]$, since $\left[ f'(\eigvalfirst{a}, \cdot) \right]^{[1]}(\eigvalsecond{a}{i}, \eigvalsecond{a}{j}) = 1$, we have 
\begin{align*}
    \Omega^a_{\indexsecond{a}{i}, \indexsecond{a}{j}} = \begin{cases}
        E_{\sizeof{\indexsecond{a}{i}} \times \sizeof{\indexsecond{a}{j}}}, & \quad i \ne j \\
        0, & \quad i = j 
    \end{cases}
\end{align*}
Thus,
\begin{align*}
    \Var{\Gamma_2(H,W)}[a][a] = & \Qfirst{a} \left[ \Omega^a \circ \left( (\Qfirst{a})\tran V_a(H, W) \Qfirst{a} \right) \right] (\Qfirst{a})\tran \\
    = & V_a(H, W) - \Qfirst{a} \diag{\left\{ 
        \hat{V}_a^{i,i}(H, W)
        \right\}_{i \in \totalsetsecond{a}}} (\Qfirst{a})\tran \\
    = & W_{\indexfirst{a} \indexfirst{a}} + \sum_{c \in \totalsetfirst \backslash \{a\}} \frac{2}{\eigvalfirst{a} - \eigvalfirst{c}} H_{\indexfirst{a} \indexfirst{c}} H_{\indexfirst{c} \indexfirst{a}} 
    - \Qfirst{a} \diag{\left\{ 
        \hat{V}_a^{i,i}(H, W)
        \right\}_{i \in \totalsetsecond{a}}} (\Qfirst{a})\tran. 
\end{align*}
For $\Var{\Gamma_3(H,W)}[a][a]$, since $f''(\eigvalfirst{a}; \eigvalsecond{a}{i}, \eigvalthird{a}{i}{i'}) = \eigvalthird{a}{i}{i'}$, $\Psi_{a,i}(\hat{V}_a^{i,i}) = \hat{V}_a^{i,i}$. Thus, 
\begin{align*}
    \Var{\Gamma_3(H,W)}[a][a] = \Qfirst{a} \diag{\left\{ 
        \hat{V}_a^{i,i}(H, W)
        \right\}_{i \in \totalsetsecond{a}}} (\Qfirst{a})\tran. 
\end{align*}
Summing up all three terms:
\begin{align*}
    & \psdproj{n}''(Z; H, W)_{\indexfirst{a} \indexfirst{a}} 
    = \Var{\Gamma_1(H,W)}[a][a] + \Var{\Gamma_2(H,W)}[a][a] + \Var{\Gamma_3(H,W)}[a][a] \\
    = & W_{\indexfirst{a} \indexfirst{a}} + 2 \sum_{c \in \totalsetfirst[-]} \frac{1}{\eigvalfirst{a} - \eigvalfirst{c}} H_{\indexfirst{a} \indexfirst{c}} H_{\indexfirst{c} \indexfirst{a}} 
    -2 \sum_{c \in \totalsetfirst[-]} \frac{\eigvalfirst{a}}{(\eigvalfirst{a} - \eigvalfirst{c})^2} H_{\indexfirst{a} \indexfirst{c}} H_{\indexfirst{c} \indexfirst{a}} \\
    = & W_{\indexfirst{a} \indexfirst{a}} + 2 \sum_{c \in \totalsetfirst[-]} \frac{-\eigvalfirst{c}}{(\eigvalfirst{a} - \eigvalfirst{c})^2} H_{\indexfirst{a} \indexfirst{c}} H_{\indexfirst{c} \indexfirst{a}}.
\end{align*}
Clearly, Case (1)(i) and Case (1)(ii)'s results can be merged.

\textbf{Case (2): $a \in \totalsetfirst[+], b \in \totalsetfirst[0]$}. $\Var{\Gamma_1(H,W)}[a][b]$ is the same as Case (1)(i): $a \in \totalsetfirst[+], b \in \totalsetfirst[+]$ and $a \ne b$, except that $\eigvalfirst{b} = 0$:
\begin{align*}
    \Var{\Gamma_1(H,W)}[a][b] = W_{\indexfirst{a} \indexfirst{b}}
    + 2 \sum_{c \in \totalsetfirst[-]} \frac{1}{\eigvalfirst{a} - \eigvalfirst{c}} H_{\indexfirst{a} \indexfirst{c}} H_{\indexfirst{c} \indexfirst{b}}
    - \frac{2}{\eigvalfirst{a}} (H_{\indexfirst{a} \indexfirst{a}} H_{\indexfirst{a} \indexfirst{b}} - H_{\indexfirst{a} \indexfirst{b}} H_{\indexfirst{b} \indexfirst{b}}).
\end{align*}
For $\Var{\Gamma_2(H,W)}[a][b]$, we have $f'(\eigvalfirst{a}; \eigvalsecond{a}{i}) = \eigvalsecond{a}{i}$ and $f'(\eigvalfirst{b}; \eigvalsecond{b}{i}) = f'(0; \eigvalsecond{b}{i}) = \max\left\{ \eigvalsecond{b}{i}, 0 \right\}$. Therefore, 
\begin{align*}
    \Phi_a(H_{\indexfirst{a} \indexfirst{a}}) = H_{\indexfirst{a} \indexfirst{a}}, \quad \Phi_b(H_{\indexfirst{b} \indexfirst{b}}) = \ppsim(H_{\indexfirst{b} \indexfirst{b}}).
\end{align*}
Consequently, 
\begin{align*}
    \Var{\Gamma_2(H,W)}[a][b] = H_{\indexfirst{a} \indexfirst{a}} \frac{2H_{\indexfirst{a} \indexfirst{b}}}{\eigvalfirst{a}} - \frac{2 H_{\indexfirst{a} \indexfirst{b}}}{\eigvalfirst{a}} \ppsim(H_{\indexfirst{b} \indexfirst{b}}).
\end{align*}
For $\Var{\Gamma_3(H,W)}[a][b]$, it is $0$ since $a \ne b$. Thus, 
\begin{align*}
    & \psdproj{n}''(Z; H, W)_{\indexfirst{a} \indexfirst{b}} \\
    = & W_{\indexfirst{a} \indexfirst{b}}
    + 2 \sum_{c \in \totalsetfirst[-]} \frac{1}{\eigvalfirst{a} - \eigvalfirst{c}} H_{\indexfirst{a} \indexfirst{c}} H_{\indexfirst{c} \indexfirst{b}}
    + \frac{2}{\eigvalfirst{a}} H_{\indexfirst{a} \indexfirst{b}} H_{\indexfirst{b} \indexfirst{b}} - \frac{2 }{\eigvalfirst{a}} H_{\indexfirst{a} \indexfirst{b}} \ppsim(H_{\indexfirst{b} \indexfirst{b}}) \\
    = & W_{\indexfirst{a} \indexfirst{b}}
    + 2 \sum_{c \in \totalsetfirst[-]} \frac{1}{\eigvalfirst{a} - \eigvalfirst{c}} H_{\indexfirst{a} \indexfirst{c}} H_{\indexfirst{c} \indexfirst{b}}
    - 2 \frac{1}{\eigvalfirst{a}} H_{\indexfirst{a} \indexfirst{b}} \ppsim(-H_{\indexfirst{b} \indexfirst{b}}).
\end{align*}

\textbf{Case (3): $a \in \totalsetfirst[+], b \in \totalsetfirst[-]$}. In this case, $f^{[1]}(\eigvalfirst{a}, \eigvalfirst{b}) = \frac{\eigvalfirst{a}}{\eigvalfirst{a} - \eigvalfirst{b}}$, and 
\begin{align*}
    f^{[2]}(\eigvalfirst{a}, \eigvalfirst{b}, \eigvalfirst{c}) 
    = \begin{cases}
        \frac{1 - \frac{-\eigvalfirst{c}}{\eigvalfirst{b} - \eigvalfirst{c}}}{\eigvalfirst{a} - \eigvalfirst{b}} = \frac{-\eigvalfirst{b}}{(\eigvalfirst{b} - \eigvalfirst{a})(\eigvalfirst{b} - \eigvalfirst{c})}, & \quad c \in \totalsetfirst[+] \backslash \{a\} \\
        \frac{1}{\eigvalfirst{a} - \eigvalfirst{b}}, & \quad c \in \totalsetfirst[0] \\
        \frac{\frac{\eigvalfirst{a}}{\eigvalfirst{a} - \eigvalfirst{c}}}{\eigvalfirst{a} - \eigvalfirst{b}} = \frac{\eigvalfirst{a}}{(\eigvalfirst{a} - \eigvalfirst{b})(\eigvalfirst{a} - \eigvalfirst{c})}, & \quad c \in \totalsetfirst[-] \backslash \{b\}
    \end{cases}
\end{align*}
Thus, 
\begin{align*}
    & \Var{\Gamma_1(H,W)}[a][b] =
    \frac{\eigvalfirst{a}}{\eigvalfirst{a} - \eigvalfirst{b}} W_{\indexfirst{a} \indexfirst{b}} 
    + 2 \sum_{c \in \totalsetfirst[+] \backslash \{a\}} \frac{-\eigvalfirst{b}}{(\eigvalfirst{b} - \eigvalfirst{a})(\eigvalfirst{b} - \eigvalfirst{c})} H_{\indexfirst{a} \indexfirst{c}} H_{\indexfirst{c} \indexfirst{b}} \\
    & + 2\ \sum_{c \in \totalsetfirst[0]} \frac{1}{\eigvalfirst{a} - \eigvalfirst{b}} H_{\indexfirst{a} \indexfirst{c}} H_{\indexfirst{c} \indexfirst{b}} 
    + 2 \sum_{c \in \totalsetfirst[-] \backslash \{b\}} \frac{\eigvalfirst{a}}{(\eigvalfirst{a} - \eigvalfirst{b})(\eigvalfirst{a} - \eigvalfirst{c})} H_{\indexfirst{a} \indexfirst{c}} H_{\indexfirst{c} \indexfirst{b}} \\
    & -  2\frac{\eigvalfirst{a}}{(\eigvalfirst{a} - \eigvalfirst{b})^2} (H_{\indexfirst{a} \indexfirst{a}} H_{\indexfirst{a} \indexfirst{b}} - H_{\indexfirst{a} \indexfirst{b}} H_{\indexfirst{b} \indexfirst{b}}).
\end{align*}
For $\Var{\Gamma_2(H,W)}[a][b]$: since $\Phi_b(H_{\indexfirst{b} \indexfirst{b}}) = 0$, we have 
\begin{align*}
    \Var{\Gamma_2(H,W)}[a][b] = 2 \frac{1}{\eigvalfirst{a} - \eigvalfirst{b}} H_{\indexfirst{a} \indexfirst{a}} H_{\indexfirst{a} \indexfirst{b}}.
\end{align*}
$\Var{\Gamma_3(H,W)}[a][b] = 0$ since $a \ne b$. Thus, 
\begin{align*}
    \psdproj{n}''(Z; H, W)_{\indexfirst{a} \indexfirst{b}} 
    = & \frac{\eigvalfirst{a}}{\eigvalfirst{a} - \eigvalfirst{b}} W_{\indexfirst{a} \indexfirst{b}} 
    + 2 \sum_{c \in \totalsetfirst[+] \backslash \{a\}} \frac{-\eigvalfirst{b}}{(\eigvalfirst{b} - \eigvalfirst{a})(\eigvalfirst{b} - \eigvalfirst{c})} H_{\indexfirst{a} \indexfirst{c}} H_{\indexfirst{c} \indexfirst{b}} \\ 
    & + 2\ \sum_{c \in \totalsetfirst[0]} \frac{1}{\eigvalfirst{a} - \eigvalfirst{b}} H_{\indexfirst{a} \indexfirst{c}} H_{\indexfirst{c} \indexfirst{b}} 
    + 2 \sum_{c \in \totalsetfirst[-] \backslash \{b\}} \frac{\eigvalfirst{a}}{(\eigvalfirst{a} - \eigvalfirst{b})(\eigvalfirst{a} - \eigvalfirst{c})} H_{\indexfirst{a} \indexfirst{c}} H_{\indexfirst{c} \indexfirst{b}} \\
    & +  2\frac{-\eigvalfirst{b}}{(\eigvalfirst{a} - \eigvalfirst{b})^2} H_{\indexfirst{a} \indexfirst{a}} H_{\indexfirst{a} \indexfirst{b}} + 2\frac{\eigvalfirst{a}}{(\eigvalfirst{a} - \eigvalfirst{b})^2} H_{\indexfirst{a} \indexfirst{b}} H_{\indexfirst{b} \indexfirst{b}} \\
    = & \frac{\eigvalfirst{a}}{\eigvalfirst{a} - \eigvalfirst{b}} W_{\indexfirst{a} \indexfirst{b}} 
    + 2 \sum_{c \in \totalsetfirst[+]} \frac{-\eigvalfirst{b}}{(\eigvalfirst{b} - \eigvalfirst{a})(\eigvalfirst{b} - \eigvalfirst{c})} H_{\indexfirst{a} \indexfirst{c}} H_{\indexfirst{c} \indexfirst{b}} \\
    & + 2\ \sum_{c \in \totalsetfirst[0]} \frac{1}{\eigvalfirst{a} - \eigvalfirst{b}} H_{\indexfirst{a} \indexfirst{c}} H_{\indexfirst{c} \indexfirst{b}} 
    + 2 \sum_{c \in \totalsetfirst[-]} \frac{\eigvalfirst{a}}{(\eigvalfirst{a} - \eigvalfirst{b})(\eigvalfirst{a} - \eigvalfirst{c})} H_{\indexfirst{a} \indexfirst{c}} H_{\indexfirst{c} \indexfirst{b}}. 
\end{align*}

\textbf{Case (4): $a \in \totalsetfirst[0], b \in \totalsetfirst[0]$}. This case implies $a = b$. For $\Var{\Gamma_1(H,W)}[a][a]$: 
\begin{align*}
    \Var{\Gamma_1(H,W)}[a][a] = 2 \sum_{c \in \totalsetfirst \backslash \{a\}} \frac{f(\eigvalfirst{c})}{\eigvalfirst{c}^2} H_{\indexfirst{a} \indexfirst{c}} H_{\indexfirst{c} \indexfirst{a}}
    = 2 \sum_{c \in \totalsetfirst[+]} \frac{1}{\eigvalfirst{c}} H_{\indexfirst{a} \indexfirst{c}} H_{\indexfirst{c} \indexfirst{a}}.
\end{align*}
For $\Var{\Gamma_2(H,W)}[a][a]$,
\begin{align*}
    \left[ f'(\eigvalfirst{a}; \cdot) \right]^{[1]}(\eigvalsecond{a}{i}, \eigvalsecond{a}{j})
    = \frac{
        \max\left\{ \eigvalsecond{a}{i}, 0 \right\} - \max\left\{ \eigvalsecond{a}{j}, 0 \right\}
    }{
        \eigvalsecond{a}{i} - \eigvalsecond{a}{j}
    }.
\end{align*}
Thus, 
\begin{align*}
    \Omega^a_{\indexsecond{a}{i}, \indexsecond{a}{j}} = \begin{cases}
        \frac{
        \max\left\{ \eigvalsecond{a}{i}, 0 \right\} - \max\left\{ \eigvalsecond{a}{j}, 0 \right\}
    }{
        \eigvalsecond{a}{i} - \eigvalsecond{a}{j}
    } E_{\sizeof{\indexsecond{a}{i}} \times \sizeof{\indexsecond{a}{j}}}, & \quad i \ne j \\
        0, & \quad i = j 
    \end{cases}
\end{align*}
For $\Var{\Gamma_3(H,W)}[a][a]$, since 
\begin{align*}
    f''(\eigvalfirst{a}; \eigvalsecond{a}{i}, \eigvalthird{a}{i}{i'}) =
    \begin{cases}
        \eigvalthird{a}{i}{i'}, & \quad \eigvalsecond{a}{i} > 0 \\
        \max\left\{ \eigvalthird{a}{i}{i'}, 0 \right\}, & \quad \eigvalsecond{a}{i} = 0 \\
        0, & \quad \eigvalsecond{a}{i} < 0
    \end{cases}
\end{align*}
We have 
\begin{align*}
    \Psi^{a,i}(\hat{V}_a^{i,i}) = \begin{cases}
        \hat{V}_a^{i,i}, & \quad i \in \totalsetsecond{a}[+] \\
        \ppsim(\hat{V}_a^{i,i}), & \quad i \in \totalsetsecond{a}[0] \\
        0, & \quad i \in \totalsetsecond{a}[-]
    \end{cases}
\end{align*}
Now we simplify $\Var{\Gamma_2(H,W)}[a][a] + \Var{\Gamma_3(H,W)}[a][a]$. Notice that from~\eqref{eq:psd:first-dd-nondiagonal-Z}, 
\begin{align*}
    \ppsim'(H_{\indexfirst{a} \indexfirst{a}}; V_a)
    = \Qfirst{a} \Upsilon^a (\Qfirst{a})\tran,
\end{align*}
where 
\begin{align*}
    \Upsilon^a_{\indexsecond{a}{i} \indexsecond{a}{j}} = \begin{cases}
        \frac{
            \max\left\{ \eigvalsecond{a}{i}, 0 \right\} - \max\left\{ \eigvalsecond{a}{j}, 0 \right\}
        }{
            \eigvalsecond{a}{i} - \eigvalsecond{a}{j}
        } \hat{V}_a^{i,j} = \Omega^a_{\indexsecond{a}{i} \indexsecond{a}{j}} \circ \hat{V}_a^{i,j}, & \quad i \ne j \\
        \Psi^{a,i}(\hat{V}_a^{i,i}), & \quad i = j 
    \end{cases}
\end{align*}
Therefore, 
\begin{align*}
    \Var{\Gamma_2(H,W)}[a][a] + \Var{\Gamma_3(H,W)}[a][a] =  
        \ppsim'(H_{\indexfirst{a} \indexfirst{a}}; V_a).
\end{align*}
Also,
\begin{align*}
    \psdproj{n}''(Z; H, W)_{\indexfirst{a} \indexfirst{a}} 
    = 2 \sum_{c \in \totalsetfirst[+]} \frac{1}{\eigvalfirst{c}} H_{\indexfirst{a} \indexfirst{c}} H_{\indexfirst{c} \indexfirst{a}} 
    + \ppsim'(H_{\indexfirst{a} \indexfirst{a}}; V_a(H, W)),
\end{align*}
where 
\begin{align*}
    V_a(H, W) = & W_{\indexfirst{a} \indexfirst{a}} + 2 \sum_{c \in \totalsetfirst \backslash \{a\}}
    \frac{1}{\eigvalfirst{a} - \eigvalfirst{c}} H_{\indexfirst{a} \indexfirst{c}} H_{\indexfirst{c} \indexfirst{a}} \\
    = & W_{\indexfirst{a} \indexfirst{a}} 
    - 2 \sum_{c \in \totalsetfirst[+]} \frac{1}{\eigvalfirst{c}} H_{\indexfirst{a} \indexfirst{c}} H_{\indexfirst{c} \indexfirst{a}} 
    + 2 \sum_{c \in \totalsetfirst[-]} \frac{1}{-\eigvalfirst{c}} H_{\indexfirst{a} \indexfirst{c}} H_{\indexfirst{c} \indexfirst{a}} .
\end{align*}

\textbf{Case (5): $a \in \totalsetfirst[0], b \in \totalsetfirst[-]$}. In this case, $\eigvalfirst{a} = 0$, $f^{[1]}(\eigvalfirst{a}, \eigvalfirst{b}) = 0$, and 
\begin{align*}
    f^{[2]}(\eigvalfirst{a}, \eigvalfirst{b}, \eigvalfirst{c}) 
    = \begin{cases}
        \frac{-\eigvalfirst{b}}{(\eigvalfirst{b} - \eigvalfirst{a})(\eigvalfirst{b} - \eigvalfirst{c})} = \frac{1}{\eigvalfirst{c} - \eigvalfirst{b}}, & \quad c \in \totalsetfirst[+] \\
        \frac{\eigvalfirst{a}}{(\eigvalfirst{a} - \eigvalfirst{b})(\eigvalfirst{a} - \eigvalfirst{c})} = 0, & \quad c \in \totalsetfirst[-] \backslash \{b\}
    \end{cases}
\end{align*}
Thus, 
\begin{align*}
    \Var{\Gamma_1(H,W)}[a][b] = 2 \sum_{c \in \totalsetfirst[+]} \frac{1}{\eigvalfirst{c} - \eigvalfirst{b}} H_{\indexfirst{a} \indexfirst{c}} H_{\indexfirst{c} \indexfirst{b}}.
\end{align*}
For $\Var{\Gamma_2(H,W)}[a][b]$, 
$\Phi_a(H_{\indexfirst{a} \indexfirst{a}}) = \ppsim(H_{\indexfirst{a} \indexfirst{a}})$ 
and $\Phi_b(H_{\indexfirst{b} \indexfirst{b}}) = 0$. Thus,
\begin{align*}
    \Var{\Gamma_2(H,W)}[a][b] = \frac{2}{-\eigvalfirst{b}} \ppsim(H_{\indexfirst{a} \indexfirst{a}}) H_{\indexfirst{a} \indexfirst{b}}.
\end{align*}
$\Var{\Gamma_3(H,W)}[a][b] = 0$ since $a \ne b$. Thus,
\begin{align*}
    \psdproj{n}''(Z; H, W)_{\indexfirst{a} \indexfirst{b}} 
    = 2 \sum_{c \in \totalsetfirst[+]} \frac{1}{\eigvalfirst{c} - \eigvalfirst{b}} H_{\indexfirst{a} \indexfirst{c}} H_{\indexfirst{c} \indexfirst{b}} 
    + 2 \frac{1}{-\eigvalfirst{b}} \ppsim(H_{\indexfirst{a} \indexfirst{a}}) H_{\indexfirst{a} \indexfirst{b}}.
\end{align*}

\textbf{Case (6)(i): $a \in \totalsetfirst[-], b \in \totalsetfirst[-]$ and $a \ne b$}. In this case, $\Var{\Gamma_2(H,W)}[a][b] = \Var{\Gamma_3(H,W)}[a][b] = 0$, $f(\eigvalfirst{a}) = f(\eigvalfirst{b}) = 0$. Thus, 
\begin{align*}
    & \psdproj{n}''(Z; H, W)_{\indexfirst{a} \indexfirst{b}} = \Var{\Gamma_1(H,W)}[a][b]
    = 2 \sum_{c \in \totalsetfirst[+]} f^{[2]}(\eigvalfirst{a}, \eigvalfirst{b}, \eigvalfirst{c}) H_{\indexfirst{a} \indexfirst{c}} H_{\indexfirst{c} \indexfirst{b}} \\
    = & 2 \sum_{c \in \totalsetfirst[+]} \frac{
        \frac{\eigvalfirst{c}}{\eigvalfirst{c} - \eigvalfirst{a}} - \frac{\eigvalfirst{c}}{\eigvalfirst{c} - \eigvalfirst{b}}
    }{\eigvalfirst{a} - \eigvalfirst{b}} H_{\indexfirst{a} \indexfirst{c}} H_{\indexfirst{c} \indexfirst{b}}
    = 2 \sum_{c \in \totalsetfirst[+]} \frac{
        \eigvalfirst{c}
    }{
        (\eigvalfirst{c} - \eigvalfirst{a})(\eigvalfirst{c} - \eigvalfirst{b})
    } H_{\indexfirst{a} \indexfirst{c}} H_{\indexfirst{c} \indexfirst{b}}.
\end{align*}

\textbf{Case (6)(ii): $a \in \totalsetfirst[-], b \in \totalsetfirst[-]$ and $a = b$}. In this case, $\Var{\Gamma_2(H,W)}[a][b] = \Var{\Gamma_3(H,W)}[a][b] = 0$, $\Omega^a = 0$, and $\Psi_{a,i}(\hat{V}_a^{i,i}) = 0$. Thus, 
\begin{align*}
    & \psdproj{n}''(Z; H, W)_{\indexfirst{a} \indexfirst{a}} = \Var{\Gamma_1(H,W)}[a][a] 
    = -2 \sum_{c \in \totalsetfirst \backslash \{a\}} \frac{
        f(\eigvalfirst{a}) - f(\eigvalfirst{c})
    }{
        (\eigvalfirst{a} - \eigvalfirst{c})^2
    } H_{\indexfirst{a} \indexfirst{c}} H_{\indexfirst{c} \indexfirst{a}} \\
    = & 2 \sum_{c \in \totalsetfirst[+]} \frac{\eigvalfirst{c}}{
        (\eigvalfirst{a} - \eigvalfirst{c})^2
    } H_{\indexfirst{a} \indexfirst{c}} H_{\indexfirst{c} \indexfirst{a}}.
\end{align*}
Clearly, Case (6)(i) and Case (6)(ii)'s results can be merged. This closes the proof of Theorem~\ref{thm:psd:second-dd}.